\documentclass[10pt,a4paper,reqno]{amsart}
\usepackage{amsthm,amsfonts,amsmath,amssymb,latexsym}
\usepackage{epsfig,graphics,color}

\usepackage[all]{xy}
\usepackage[breaklinks=true]{hyperref}

\usepackage{multicol}
\usepackage{mathtools}
\usepackage{stmaryrd}
\usepackage{verbatim}
\usepackage{bm}
\usepackage{mathabx}

\usepackage{enumitem}
\usepackage{pgf,amsmath,tikz,type1cm,fix-cm}
\usetikzlibrary{positioning}

\setlist[itemize,1]{leftmargin=\dimexpr 26pt-.1in}

\newtheorem{PARA}{}[section]
\newtheorem{theorem}[PARA]{Theorem}
\newtheorem{corollary}[PARA]{Corollary}
\newtheorem{lemma}[PARA]{Lemma}
\newtheorem{proposition}[PARA]{Proposition}
\newtheorem{definition}[PARA]{Definition}
\newtheorem{definition-proposition}[PARA]{Definition-Proposition}
\newtheorem{conjecture}[PARA]{Conjecture}

\theoremstyle{definition}
\newtheorem{remark}[PARA]{Remark}

\theoremstyle{theorem}
\newtheorem{example}[PARA]{Example}
\newcommand{\para}{\begin{PARA}\rm}
\newcommand{\arap}{\end{PARA}\rm}
\newcommand{\dfn}{\begin{definition}\rm}
\newcommand{\nfd}{\end{definition}\rm}
\newcommand{\rmk}{\begin{remark}\rm}
\newcommand{\kmr}{\end{remark}\rm}
\newcommand{\xmpl}{\begin{example}\rm}
\newcommand{\lpmx}{\end{example}\rm}

\newcommand{\cC}{\mathcal{C}}
\newcommand{\cD}{\mathcal{D}}

\newcommand{\cH}{\mathcal{H}}

\newcommand{\cM}{\mathcal{M}}

\newcommand{\cP}{\mathcal{P}}

\renewcommand{\H}{{\mathbb{H}}}

\newcommand{\N}{{\mathbb{N}}}

\newcommand{\R}{{\mathbb{R}}}

\newcommand{\Z}{{\mathbb{Z}}}
\newcommand{\Ccal}{{\mathcal{C}}}
\newcommand{\Mcal}{{\mathcal{M}}}
\newcommand{\Ncal}{{\mathcal{N}}}
\newcommand{\coker}{\mathrm{ coker }}  
\newcommand{\colim}{\mathrm{ colim}\, }  
\newcommand{\im}{\mathrm{im}\,}        

\newcommand{\ev}{\mathrm{ev}}
\newcommand{\coev}{\mathrm{coev}}

\newcommand{\Hom}{\mathrm{Hom}}

\newcommand{\out}{\mathrm{out}}
\newcommand{\inn}{\mathrm{in}}
\newcommand{\eps}{{\varepsilon}}

\newcommand{\CZ}{\mathrm{CZ}}

\def\NABLA#1{{\mathop{\nabla\kern-.5ex\lower1ex\hbox{$#1$}}}}
\def\Nabla#1{\nabla\kern-.5ex{}_{#1}}
\def\Tabla#1{\Tilde\nabla\kern-.5ex{}_{#1}}
\renewcommand{\Tilde}{\widetilde}

\newcommand{\p}{{\partial}}

\newcommand{\ol}{\overline}
\newcommand{\wh}{\widehat}

\newcommand{\simplex}{{\Delta \hspace{-8.1pt}\Delta}}

\newcommand{\bk}{\mathbf{k}}

\newcommand{\hatotimes}{{\widehat\otimes}}

\usetikzlibrary{arrows}
\usetikzlibrary{shapes}

\newcommand{\boldDelta}{{\boldsymbol{\Delta}}}
\newcommand{\boldlambda}{{\boldsymbol{\lambda}}}
\newcommand{\boldmu}{{\boldsymbol{\mu}}}
\newcommand{\boldeta}{{\boldsymbol{\eta}}}
\newcommand{\boldeps}{{\boldsymbol{\eps}}}
\newcommand{\boldc}{{\boldsymbol{c}}}
\newcommand{\boldp}{{\boldsymbol{p}}}

\renewcommand{\comment}[1]{}

\begin{document}

\title[BV bialgebra structures]{BV bialgebra structures in Floer theory and string topology}
\author{Janko Latschev}
\address{Universit\"at Hamburg \newline Fachbereich Mathematik,
Bundesstrasse 55 (Geomatikum), 20146 Hamburg, Germany}
\email{janko.latschev@uni-hamburg.de} 
\author{Alexandru Oancea}
\address{Universit\'e de Strasbourg \newline Institut de Recherche Math\'ematique Avanc\'ee \newline Strasbourg, France}
\email{oancea@unistra.fr}
\date{\today}


\begin{abstract} 
We derive the notions of BV unital infinitesimal bialgebra and BV Frobenius algebra from the topology of suitable compactifications of moduli spaces of decorated genus 0 curves. We construct these structures respectively on reduced symplectic homology and Rabinowitz Floer homology. As an application, we construct these structures in nonequivariant string topology. We also show how the Lie bialgebra structure in equivariant string topology, and more generally on $S^1$-equivariant symplectic homology, is obtained as a formal consequence.  
\end{abstract}

\maketitle

{\footnotesize{\tableofcontents}}


\section{Introduction}\label{sec:introduction}

This paper is situated at the crossroads of quantum algebra, symplectic topology, and string topology. We \emph{define} two new algebraic structures, which we call ``(odd) BV unital infinitesimal bialgebra" and ``(odd) BV Frobenius algebra", and we show that the former is the correct nonequivariant counterpart of the notion of Lie bialgebra, and the latter is the outcome of the Drinfel'd double construction on the former. We \emph{derive} the definition of these algebraic structures from the geometry of suitable compactified moduli spaces of curves. As a consequence, we \emph{construct} these algebraic structures on symplectic- and Rabinowitz Floer homology, which are Floer theoretic invariants of convex exact symplectic manifolds. Using the dictionary between the Floer theory of cotangent bundles and string topology, we \emph{specialize} these algebraic structures to the loop homology, respectively to the Rabinowitz loop homology, of closed orientable smooth manifolds. As an illustration, we \emph{compute} these structures for the string topology of spheres of odd dimension $\ge 3$.      

We use field coefficients throughout the paper, but it is worth pointing out that the definitions and basic statements in~\S\ref{ssec:definitions} also make sense with coefficients in a unital ring.

In \cite{CHO-PD, CO-reduced}, the second author together with Cieliebak and Hingston have introduced new versions of bialgebras, i.e., graded unital infinitesimal bialgebras and graded Frobenius algebras, and constructed instances of these structures on the symplectic homology of a suitable Liouville domain and on Rabinowitz Floer homology of a Liouville filled contact manifold. As a consequence, these structures were also shown to exist in the string topology of a smooth manifold. As the relevant homology groups are typically infinite dimensional, these authors have adapted the formalism of Tate vector spaces for Floer theory in order to phrase duality theorems. We recall this formalism in Appendix~\ref{sec:Tate}, and we will use it as well when dealing with infinite-dimensional structures.

In both these contexts, there is a natural unary operator of square zero present, which combines with the product of the respective bialgebra into a BV algebra structure. The purpose of this work is to clarify the interaction of all three operations: product, coproduct and BV operator. Two new algebraic structures arise, which we call
\begin{center}
{\em (odd) BV unital infinitesimal bialgebras}, 
\end{center}
and, respectively,
\begin{center}
{\em (odd) BV Frobenius algebras}.
\end{center}
Their precise definitions appear in~\S\ref{sec:BVuiBVFrob} in purely algebraic form. In that section, we use the graphical language of trees to give alternative representations for algebraic relations. 
As might be expected, besides the standard 7-term relation~\eqref{eq:7_term_for_Delta_and_mu} of a BV algebra, the definition of a BV unital infinitesimal bialgebra involves a similar 7-term relation~\eqref{eq:7_term_for_Delta_and_lambda} for a ``BV coalgebra" and a new relation~\eqref{eq:11term} involving all three structures, which in its most general form has 11 terms. 

The next statement explains the relation between the two structures.

{\bf Theorem~A.} {\it (i) (Lemma~\ref{lem:BVFrob-BVui}) Every odd BV Frobenius algebra is an odd BV unital infinitesimal bialgebra.

(ii) (Proposition~\ref{prop:BVui-BVFrob}) Under the assumption that the coproduct of the unit vanishes, an odd BV unital infinitesimal bialgebra determines an odd BV Frobenius algebra via the Drinfel'd double construction. 
}

The Drinfel'd double construction is a classical tool which originates in the study of Lie bialgebras and which allows to understand bialgebra structures as algebra structures~\cite{Drinfeld-quantum-groups}. Loosely stated, it associates to a vector space $A$ the vector space $D(A)=A\oplus A^\vee$, with $A^\vee$ the dual of $A$. As such, a bialgebra structure of some kind on $A$, expressed by the data of a product and a coproduct, can be equivalently phrased as an algebra structure of the same kind on $D(A)$. See~\S\ref{sec:stringtop} for more details.  

The last statement in Theorem~A generalizes to the graded and BV case previous results of Zhelyabin~\cite{Zhelyabin1997}, Aguiar~\cite{Aguiar2001} and Bai~\cite{Bai2010}. It answers a question left open in~\cite{CO-algebra} and sheds light on the passage from $A_2^+$-structures to $A_2$-structures on the cone in~\cite[\S7]{CO-cones}. 

An important aspect of this work is that the basic algebraic relations in the two algebraic structures that we introduce are derived purely geometrically from the study of suitable moduli spaces of curves, see~\S\ref{sec:moduli}. For more algebraically minded readers this might look like a detour, but in fact this geometric approach is really crucial for the construction of our algebraic structures in Floer theory. The overall strategy is not new in itself, but we depart from the classical literature in at least two ways: we take a first step in the direction of \emph{dioperads} and we make essential use of the \emph{weights} in deriving algebraic identities. 

To explain this, let us recall that the relationship of the framed little 2-disk operad to BV algebras is well-known~\cite{Getzler-BV}. It has also been known since the work of Kimura, Stasheff and Voronov~\cite{Kimura_Stasheff_Voronov_1995} that a topological 
model for this operad can be built from moduli spaces of spheres with one output and many input punctures: the existence of this model forms the basis for the many constructions of such structures arising in Floer-type theories~\cite{Seidel07,Seidel-disjoinable,Ekholm-Oancea,Abouzaid-Groman-Varolgunes}, and these moduli spaces and their higher genus analogues also underlie structures and operations present in real enumerative geometry~\cite{LiuMelissa}, symplectic field theory~\cite{EGH}, and string topology~\cite{Sullivan-background}. We review the relevant parts of the discussion of these moduli spaces in~\S\ref{ssec:moduli_kto1} and~\S\ref{ssec:moduli_2or3to1}. In \S\ref{ssec:moduli_1toell}, we give an analogous description of the geometry of the spaces relevant for the BV coalgebra structures, where
\emph{the curves are enriched with weights at the output punctures}. While the appearance of these weights might seem innocuous, they do in fact play an important role. In general, they give rise to certain correction terms in algebraic identities, which we avoid below by restricting to high enough dimensions. 
In a discussion of the complete dioperad structure these corrections would have to be taken more seriously. Here we only take a first step beyond the operadic framework by looking at a moduli space of spheres with two input and two output punctures. In this context, the weights are central in our geometric derivation of the 11-term relation involving product, coproduct and BV operator in \S\ref{ssec:moduli_2to2}.
As was also suggested in~\cite{CO-cones} and~\cite{Ekholm-Oancea}, this points towards a hierarchy of moduli spaces of curves with multiple inputs, multiple outputs, and weights at the outputs, which in genus zero should yield a generalization of (and a geometric perspective on) the structure described by Poirier and Tradler~\cite{PoirierTradler-V_d}.
 
In~\S\S\ref{sec:reduced}--\ref{sec:RFH} we use the geometric results on these moduli spaces of domains discussed in \S\ref{sec:moduli} to construct versions of our new algebraic structures in Floer theory. This theory, originating in the groundbreaking work of Floer~\cite{Floer-SympFixedPoints,Floer-Unregularized}, is the main source of symplectic invariants in symplectic topology, such as \emph{symplectic homology} and \emph{Rabinowitz Floer homology}, which feature in our next theorem. The first statement applies to a class of symplectic manifolds of fundamental importance called \emph{Weinstein domains}~\cite{Cieliebak-Eliashberg-book}, subject to a technical condition called ``strong essentiality" which was introduced in~\cite{CO-reduced}. Key for our later applications to string topology is the fact that the class of strongly essential Weinstein domains includes unit disc cotangent bundles of closed orientable manifolds. The second statement applies to the class of \emph{Liouville domains}, which can be thought of as an enlargement of the class of Weinstein domains~\cite{Eliashberg-Gromov}.

{\bf Theorem~B.}
{\it (i) (Theorem~\ref{thm:BV_unital_infinitesimal_reduced}) Given a strongly essential Weinstein domain $W$ of dimension $2n\ge 8$, reduced symplectic homology $\ol{S\H}_*(W)$ carries the structure of an odd BV unital infinitesimal bialgebra. 

(ii) (Theorem~\ref{thm:BVFrob}) Given a contact manifold $\p W$ with Liouville filling $W$, its Rabinowitz Floer homology $S\H_*(\p W)$ carries the structure of an odd BV Frobenius algebra. 
}

We refer to {\it loc. cit.} for the relevant definitions, and we remind the reader again that in this
infinite dimensional case the notion of a Frobenius algebra has to be understood in the context of Tate vector spaces~\cite{CO-algebra,CO-Tate}, see also Appendix~\ref{sec:Tate}.

In~\S\ref{sec:stringtop} we explain how these results specialize to statements in string topology, understood broadly as the study of algebraic structures on the homology and cohomology of loop- and path spaces~\cite{CS}. The dictionary is provided by the Viterbo isomorphism: given a closed manifold $Q$, let $\Lambda=\Lambda Q$ be its free loop space, and $D^*Q, S^*Q$ its unit disc, resp. unit sphere, cotangent bundle. We then have $\ol{S\H}_*(D^*Q;\sigma)\simeq \ol\H_*\Lambda$, the reduced homology of the free loop space of $Q$, where $\sigma$ is a suitable local system on $\Lambda$~\cite{Viterbo-cotangent,AS-Legendre,Abouzaid-cotangent,Kragh,SW}, and $S\H_*(S^*Q;\sigma)\simeq \widehat \H_*\Lambda$, the Rabinowitz loop homology of the free loop space of $Q$~\cite{CHO-MorseFloerGH}. The group $\ol\H_*\Lambda$ should be thought of as a variant of the homology $\H_*\Lambda$ of the free loop space of $Q$, and the group $\widehat \H_*\Lambda$ can be thought of as a Drinfel'd double construction on $\H_*\Lambda$.

{\bf Theorem C.} 
{\it (i) (Theorem~\ref{thm:H-BV-bialgebra}) Given a closed oriented manifold of dimension $\ge 4$,  reduced loop homology $\ol\H_*\Lambda$ carries the structure of an odd BV unital infinitesimal bialgebra. 

(ii) (Theorem~\ref{thm:RFH-BV-Frobenius}) Given a closed oriented manifold, its Rabinowitz loop homology $\widehat\H_*\Lambda$ carries the structure of an odd BV Frobenius algebra. 
}

These statements are the BV counterparts of the results proved by Cieliebak, Hingston and the second author in~\cite{CO-reduced,CHO-MorseFloerGH}. In the statement about reduced loop homology we restrict to orientable manifolds in order to ensure that their disc cotangent bundles are strongly essential Weinstein domains in the sense of~\cite{CO-reduced}. The statement holds for arbitrary closed manifolds with $\Z/2$-coefficients. 

The existence of these BV bialgebra structures has interesting consequences for $S^1$-equivariant loop homology. Given a closed orientable manifold $Q$, we define its \emph{reduced string homology} 
$$
\ol\H_*^{S^1}\Lambda=\H_*^{S^1}\Lambda/\chi(Q)[\mathrm{pt}],
$$ 
where $[\mathrm{pt}]$ is the class of a point in $S^1$-equivariant loop homology (with respect to the $S^1$-action given by reparametrization in the source). 

{\bf Theorem~D.} {\it (Theorem~\ref{thm:HS1-Lie-bialgebra}) 
Reduced string homology $\ol\H_*^{S^1}\Lambda$ of a closed orientable manifold of dimension $\ge 4$ carries the structure of a Lie bialgebra. 
} 

This Lie bialgebra structure is a simple formal consequence of the BV unital infinitesimal bialgebra structure from Theorem~C. The same formal reasoning implies a similar statement for reduced $S^1$-equivariant symplectic homology of strongly essential Weinstein domains (Theorem~\ref{thm:SHS1-Lie-bialgebra}). Theorem~D generalizes the Lie bialgebra structure constructed by Chas and Sullivan~\cite{CS-Lie} on $\H_*^{S^1}(\Lambda,\Lambda_0)$, with $\Lambda_0\subset \Lambda$ the constant loops.  

At the end of~\S\ref{sec:stringtop} we make some conjectures on $S^1$-equivariant Rabinowitz Floer homology, respectively Rabinowitz string homology, 
based on the Drinfel'd double construction (Conjectures~\ref{conj:H-doubles} and~\ref{conj:HS1-doubles})

We illustrate in~\S\ref{sec:spheres} the intricacy and nontriviality of these algebraic structures in the case of odd spheres of dimension $\ge 3$.

Appendix~\ref{app:BVui-BVTateFrob} contains the somewhat technical proof of part (ii) of Theorem~A (Proposition~\ref{prop:BVui-BVFrob} below), which due to its length would have disrupted the flow of the text if included in the main body. 

{\bf Relation to other work.} A number of recent constructions provide conjectural algebraic counterparts to Rabinowitz Floer homology or reduced symplectic homology. These constructions are flavors of Hochschild homology groups, and as such they carry Connes operators. We expect that the Connes operator corresponds to the BV operator that we study in this paper, and we expect the BV bialgebra structures that we  unveil 
in this paper to exist on the algebraic side as well. We also expect the corresponding cyclic homology groups to correspond  to the $S^1$-equivariant Rabinowitz Floer or reduced symplectic homology groups. This should apply to the singular Hochschild cohomology groups of Frobenius dg algebras of Rivera-Wang~\cite{Rivera-Wang} (Rabinowitz loop homology), to the Hochschild homology groups of a pre-Calabi-Yau smooth $A_\infty$-category of Rivera-Takeda-Wang~\cite{Rivera-Takeda-Wang}, and to the Hochschild homology groups arising from the Rabinowitz Fukaya category of Ganatra, Gao, Venkatesh~\cite{GGV} and Bae-Jeong-Kim~\cite{BJK}. Parts of the odd BV unital infinitesimal bialgebra structure are known for the Hochschild homology of Frobenius dg algebras (BV algebra by F\'elix-Thomas~\cite{Felix-Thomas} and Tradler~\cite{Tradler}, BV coalgebra by Chen-Gan~\cite{Chen-Gan} and Abbaspour~\cite{Abbaspour}).

The BV bialgebra structures from this paper should have counterparts in topological contexts such as those of Naef-Willwacher~\cite{Naef-Willwacher} or Hingston-Wahl~\cite{Hingston-Wahl}. 

In this paper we exhibit the Lie bialgebra structure on $S^1$-equivariant reduced symplectic homology as a formal consequence of the BV infinitesimal bialgebra structure on the non-equivariant theory. In this way, one can view the latter as a ``nonequivariant" refinement of the former. It appears reasonable to expect that a full nonequivariant chain level theory at all genera including BV-operators will in the same way refine the ${\rm IBL}_\infty$ structures of~\cite{Cieliebak-Fukaya-Latschev}. This might then give a blueprint for the construction of a nonequivariant version of (rational) symplectic field theory.

{\bf Acknowledgements.} The second author acknowledges partial funding from the ANR grant 21-CE40-0002 (COSY), and from a Fellowship of the University of Strasbourg Institute for Advanced Study (USIAS), within the French national programme ``Investment for the future" (IdEx-Unistra).

\section{BV unital infinitesimal and BV Frobenius bialgebras}\label{sec:BVuiBVFrob}

\subsection{Main definitions} \label{ssec:definitions}

In this section, we introduce the two main structures we consider.
We work with a field of coefficients $R$. The definitions in this section also make sense for $R$ a unital ring, and so we will allow ourselves to speak of $R$-modules rather than $R$-vector spaces. The construction of coproducts on Floer homology groups nevertheless requires field coefficients.\footnote{This is because we need the K\"unneth isomorphism to hold in order to induce a coproduct in homology from a coproduct at chain level.}

Given a graded $R$-module we denote $1$ its identity map. 
Given graded $R$-modules $A$ and $B$, the \emph{twist map} $\tau:A\otimes B\to B\otimes A$ acts by $\tau(a\otimes b)=(-1)^{|a||b|}b\otimes a$. We denote $\sigma= (\tau \otimes 1)(1 \otimes \tau):A^{\otimes 3} \to A^{\otimes 3}$, so $\sigma^2=(1 \otimes \tau)(\tau \otimes 1)$. As permutations we have 
$$
\sigma=\begin{pmatrix} 1 & 2 & 3 \\ 2 & 3 & 1 \end{pmatrix}
\qquad \mbox{and} \qquad 
\sigma^2=\begin{pmatrix} 1 & 2 & 3 \\ 3 & 1 & 2 \end{pmatrix}.
$$
We now describe the BV refinement of commutative and cocommutative unital infinitesimal anti-symmetric bialgebras in the sense of \cite{CO-algebra}.
Stated in words, this consists of a unital BV algebra structure and a non-unital BV coalgebra structure with the same BV operator, related by the unital infinitesimal relation (which involves the product, coproduct, and unit,  but not the BV operator) and a new ``11-term'' relation that involves the product, coproduct, unit, and BV operator.
\begin{definition} \label{defi:BVui}
An {\em odd BV unital infinitesimal bialgebra} 
is a tuple $(A, \mu, \eta, \lambda, \Delta)$,
where
\begin{itemize}
\item $A$ is a $\Z$-graded $R$-bimodule, and the following maps are maps of $R$-bimodules,
\item $\mu:A \otimes A \to A$ is an associative, commutative product of degree $0$, i.e.
$$
\mu(\mu \otimes 1)= \mu(1 \otimes \mu) \qquad\text{and} \qquad \mu \tau = \mu,
$$
\item $\eta:R \to A$ is a map such that
$$
\mu( \eta \otimes 1): R \otimes A \to A \qquad\text{and} \qquad  \mu (1\otimes \eta):A \otimes R \to A
$$
coincide with the left and right $R$-module structure on $A$ (in other words, $\eta(1)$ is a unit for $\mu$),
\item $\lambda:A \to A \otimes A$ is a coassociative, cocommutative product of odd degree, i.e.
$$
(\lambda \otimes 1)\lambda =-(1 \otimes \lambda)\lambda \qquad \text{and} \qquad \tau \lambda = -\lambda,
$$
and
\item $\Delta:A \to A$ is a map of degree 1 with $\Delta^2=0$,
\end{itemize}
satisfying the following additional relations:
\begin{itemize}
\item \emph{(unital infinitesimal relation)}
$$
\lambda \mu = (1 \otimes \mu)(\lambda \otimes 1) + (\mu \otimes 1)(1 \otimes \lambda) - (\mu \otimes \mu)(1 \otimes \lambda\eta \otimes 1),
$$
\item \emph{(7-term relation for $\Delta$ and $\mu$)} 
\begin{equation} \label{eq:7_term_for_Delta_and_mu}
\Delta \mu(\mu \otimes 1) = \left[ \mu (\Delta\mu \otimes 1) - \mu(\mu \otimes 1)(\Delta \otimes 1 \otimes 1)\right](1+\sigma + \sigma^2) ,
\end{equation}
\item \emph{(7-term relation for $\Delta$ and $\lambda$)} 
\begin{equation}\label{eq:7_term_for_Delta_and_lambda}
(\lambda \otimes 1)\lambda \Delta  = - (1+\sigma + \sigma^2) \left[ (\Delta \otimes 1 \otimes 1)(\lambda \otimes 1)\lambda + (\lambda\Delta \otimes 1)\lambda\right], 
\end{equation}
and
\item \emph{(11-term relation for $\Delta$, $\mu$, $\lambda$ and $\eta$)} 
\begin{align} \label{eq:11term}
\lambda \Delta \mu
&=\lambda \mu (\Delta \otimes 1) + \lambda \mu (1 \otimes \Delta)
- (\Delta \otimes 1)\lambda \mu  - (1 \otimes \Delta) \lambda \mu \\
&\phantom{=}\,\, + (\mu \otimes 1)(1 \otimes \Delta \otimes 1)(1 \otimes \lambda)(1+\tau) \nonumber\\
&\phantom{=}\,\,+ (1 \otimes \mu)(1 \otimes \Delta \otimes 1)(\lambda \otimes 1)(1+\tau)\nonumber\\
&\phantom{=}\,\, - (\mu\otimes\mu)(1 \otimes 1 \otimes \Delta \otimes 1)(1 \otimes \lambda\eta \otimes 1)(1 + \tau). \nonumber
\end{align}
\end{itemize}
\end{definition}
To represent the various structures and their relations, we introduce the following graphical language:
\begin{center}
\begin{tikzpicture}[scale=.25]

\draw (0,0) -- (0,-2.5);
\node at (0,-5.3) {$1$};

\pgftransformxshift{4cm}

\draw (0,0) -- (1,-1) -- (2,0);
\draw (1,-1) -- (1,-2.5);
\node at (1,-5.5) {$\mu$};

\pgftransformxshift{6cm}

\draw (0,0) -- (1,-1) -- (2,0);
\draw (1,-1) -- (1,-2.5);
\node at (0,1) {\tiny$2$};
\node at (2,1) {\tiny$1$};
\node at (1,-5.5) {$\mu\tau$};

\pgftransformxshift{6cm}

\draw (1,-1) -- (1,-2);
\node[circle,draw,inner sep=1pt] at (1,-1) {};
\node at (1,-5.5) {$\eta$};

\pgftransformxshift{6cm}
\draw (1,0) -- (1,-1.5);
\draw (0,-2.5) -- (1,-1.5) -- (2,-2.5);
\node at (1,-5.3) {$\lambda$};

\pgftransformxshift{6cm}
\draw (1,0) -- (1,-1.5);
\draw (0,-2.5) -- (1,-1.5) -- (2,-2.5);
\node at (0, -3.5) {\tiny$2$};
\node at (2, -3.5) {\tiny$1$};
\node at (1,-5.3) {$\tau\lambda$};


\end{tikzpicture}
\end{center}
The unital infinitesimal relation then takes the following form:
\begin{center}
\begin{tikzpicture}[scale=.25]
\draw (0,0) -- (1,-1) -- (2,0);
\draw (1,-1) -- (1,-3);
\draw (0,-4) -- (1,-3) -- (2,-4);
\node at (4,-2) {$=$};

\pgftransformxshift{6cm}

\draw (1,0) -- (1,-1) -- (0,-2) -- (0,-4);
\draw (1,-1) -- (3,-3);
\draw (4,0) -- (4,-2) -- (3,-3) -- (3,-4);
\node at (6,-2) {$+$};

\pgftransformxshift{8cm}

\draw (0,0) -- (0,-2) -- (1,-3) -- (1,-4);
\draw (3,-1) -- (1,-3);
\draw (3,0) -- (3,-1) -- (4,-2) -- (4,-4);
\node at (6,-2) {$-$};

\pgftransformxshift{8cm}

\draw (0,0) -- (0,-2) -- (1,-3) -- (1,-4);
\draw (1,-3) -- (2.5,-1.5) -- (4,-3);
\draw (2.5,-1.5) -- (2.5,-0.5);
\draw (5,0) -- (5,-2) -- (4,-3) -- (4,-4);
\node[circle,draw,inner sep=1pt] at (2.5,-0.5) {};

\end{tikzpicture}
\end{center}
The 7-term relation for $\Delta$ and $\mu$ is represented graphically as
\begin{center}
\begin{tikzpicture}[scale=.23]
\draw (0,0) -- (1,-1) -- (2,0);
\draw (1,-1) -- (1,-2);
\draw (1,-2) -- (2,-3) -- (4,-1);
\draw (2,-3) -- (2,-4);
\node at (2,-5) {\tiny$\Delta$};

\node at (6,-2) {$=$};
\pgftransformxshift{8cm}

\draw (0,0) -- (1,-1) -- (2,0);
\draw (1,-1) -- (1,-2);
\draw (1,-2) -- (2,-3) -- (4,-1);
\draw (2,-3) -- (2,-4);
\node at (0.4,-2.4) {\tiny$\Delta$};

\node at (6,-2) {$+$};
\pgftransformxshift{8cm}

\draw (0,0) -- (1,-1) -- (2,0);
\draw (1,-1) -- (1,-2);
\draw (1,-2) -- (2,-3) -- (4,-1);
\draw (2,-3) -- (2,-4);
\node at (0.4,-2.4) {\tiny$\Delta$};
\node at (0,1) {\tiny$2$};
\node at (2,1) {\tiny$3$};
\node at (4,0) {\tiny$1$};
\node at (6,-2) {$+$};
\pgftransformxshift{8cm}

\draw (0,0) -- (1,-1) -- (2,0);
\draw (1,-1) -- (1,-2);
\draw (1,-2) -- (2,-3) -- (4,-1);
\draw (2,-3) -- (2,-4);
\node at (0,1) {\tiny$3$};
\node at (2,1) {\tiny$1$};
\node at (4,0) {\tiny$2$};
\node at (0.4,-2.4) {\tiny$\Delta$};

\node at (6,-2) {$-$};
\pgftransformxshift{8cm}

\draw (0,0) -- (1,-1) -- (2,0);
\draw (1,-1) -- (1,-2);
\draw (1,-2) -- (2,-3) -- (4,-1);
\draw (2,-3) -- (2,-4);
\node at (0,1) {\tiny$\Delta$};

\node at (6,-2) {$-$};
\pgftransformxshift{8cm}

\draw (0,0) -- (1,-1) -- (2,0);
\draw (1,-1) -- (1,-2);
\draw (1,-2) -- (2,-3) -- (4,-1);
\draw (2,-3) -- (2,-4);
\node at (0,.8) {\tiny$\Delta$};
\node at (0,2) {\tiny$2$};
\node at (2,1) {\tiny$3$};
\node at (4,0) {\tiny$1$};

\node at (6,-2) {$-$};
\pgftransformxshift{8cm}

\draw (0,0) -- (1,-1) -- (2,0);
\draw (1,-1) -- (1,-2);
\draw (1,-2) -- (2,-3) -- (4,-1);
\draw (2,-3) -- (2,-4);
\node at (0,.8) {\tiny$\Delta$};
\node at (0,2) {\tiny$3$};
\node at (2,1) {\tiny$1$};
\node at (4,0) {\tiny$2$};

\node at (6,-2) {.};
\end{tikzpicture}
\end{center}

Similarly, the 7-term relation for $\Delta$ and $\lambda$ can be represented as
\begin{center}
\begin{tikzpicture}[scale=.23]
\draw (0,0) -- (0,-1) -- (-1,-2) -- (-1,-3) -- (-2,-4);
\draw (0,-1) -- (2,-3);
\draw (-1,-3) -- (0,-4);
\node at (0,.8) {\tiny$\Delta$};

\node at (4,-2) {$=$};
\node at (6,-2) {$-$};
\pgftransformxshift{9.5cm}

\draw (0,0) -- (0,-1) -- (-1,-2) -- (-1,-3) -- (-2,-4);
\draw (0,-1) -- (2,-3);
\draw (-1,-3) -- (0,-4);
\node at (-1.5,-1.4) {\tiny$\Delta$};

\node at (4,-2) {$-$};
\pgftransformxshift{8cm}
\draw (0,0) -- (0,-1) -- (-1,-2) -- (-1,-3) -- (-2,-4);
\draw (0,-1) -- (2,-3);
\draw (-1,-3) -- (0,-4);
\node at (-1.5,-1.4) {\tiny$\Delta$};
\node at (-2,-4.8) {\tiny $2$};
\node at (0,-4.8) {\tiny $3$};
\node at (2,-3.8) {\tiny $1$};

\node at (4,-2) {$-$};
\pgftransformxshift{8cm}
\draw (0,0) -- (0,-1) -- (-1,-2) -- (-1,-3) -- (-2,-4);
\draw (0,-1) -- (2,-3);
\draw (-1,-3) -- (0,-4);
\node at (-1.5,-1.4) {\tiny$\Delta$};
\node at (-2,-4.8) {\tiny $3$};
\node at (0,-4.8) {\tiny $1$};
\node at (2,-3.8) {\tiny $2$};

\node at (4,-2) {$-$};
\pgftransformxshift{7cm}
\draw (0,0) -- (0,-1) -- (-1,-2) -- (-1,-3) -- (-2,-4);
\draw (0,-1) -- (2,-3);
\draw (-1,-3) -- (0,-4);
\node at (-2,-4.8) {\tiny$\Delta$};

\node at (4,-2) {$-$};
\pgftransformxshift{7cm}
\draw (0,0) -- (0,-1) -- (-1,-2) -- (-1,-3) -- (-2,-4);
\draw (0,-1) -- (2,-3);
\draw (-1,-3) -- (0,-4);
\node at (-2,-4.8) {\tiny$\Delta$};
\node at (-2,-5.8) {\tiny $2$};
\node at (0,-4.8) {\tiny $3$};
\node at (2,-3.8) {\tiny $1$};

\node at (4,-2) {$-$};
\pgftransformxshift{7cm}
\draw (0,0) -- (0,-1) -- (-1,-2) -- (-1,-3) -- (-2,-4);
\draw (0,-1) -- (2,-3);
\draw (-1,-3) -- (0,-4);
\node at (-2,-4.8) {\tiny$\Delta$};
\node at (-2,-5.8) {\tiny $3$};
\node at (0,-4.8) {\tiny $1$};
\node at (2,-3.8) {\tiny $2$};

\node at (4,-2) {,};
\end{tikzpicture}
\end{center}
and the 11-term relation for $\Delta$, $\mu$ and $\lambda$ in graphical form reads
\begin{center}
\begin{tikzpicture}[scale=.25]

\draw (0,0) -- (1,-1) -- (2,0);
\draw (1,-1) -- (1,-3);
\node at (0,-2) {\tiny $\Delta$};
\draw (0,-4) -- (1,-3) -- (2,-4);
\node at (4,-2) {$=$};

\pgftransformxshift{6cm}
\draw (0,0) -- (1,-1) -- (2,0);
\draw (1,-1) -- (1,-3);
\draw (0,-4) -- (1,-3) -- (2,-4);
\node at (0,0.8) {\tiny$\Delta$};
\node at (4,-2) {$+$};

\pgftransformxshift{6cm}
\draw (0,0) -- (1,-1) -- (2,0);
\draw (1,-1) -- (1,-3);
\draw (0,-4) -- (1,-3) -- (2,-4);
\node at (2,0.8) {\tiny$\Delta$};
\node at (4,-2) {$-$};

\pgftransformxshift{6cm}
\draw (0,0) -- (1,-1) -- (2,0);
\draw (1,-1) -- (1,-3);
\draw (0,-4) -- (1,-3) -- (2,-4);
\node at (0,-4.8) {\tiny$\Delta$};
\node at (4,-2) {$-$};

\pgftransformxshift{6cm}
\draw (0,0) -- (1,-1) -- (2,0);
\draw (1,-1) -- (1,-3);
\draw (0,-4) -- (1,-3) -- (2,-4);
\node at (2,-4.8) {\tiny$\Delta$};

\pgftransformxshift{-18cm}
\pgftransformyshift{-8cm}

\node at (0,-2) {$+$};
\pgftransformxshift{2cm}

\draw (0,0) -- (0,-2) -- (1,-3) -- (1,-4);
\draw (1,-3) -- (3,-1);
\node at (1.5,-1.5) {\tiny$\Delta$};
\draw (3,0) -- (3,-1) -- (4,-2) -- (4,-4);
\node at (7,-2.5) {$+$};

\pgftransformxshift{9cm}
\node at (0,.8) {\tiny$2$};
\node at (3,.8) {\tiny$1$};
\draw (0,0) -- (0,-2) -- (1,-3) -- (1,-4);
\draw (1,-3) -- (3,-1);
\node at (1.5,-1.5) {\tiny$\Delta$};
\draw (3,0) -- (3,-1) -- (4,-2) -- (4,-4);
\node at (7,-2.5) {$+$};

\pgftransformxshift{9cm}

\draw (1,0)  -- (1,-1) -- (0,-2) -- (0,-4);
\draw (1,-1) -- (3,-3);
\node at (2.5,-1.5) {\tiny$\Delta$};
\draw (4,0) -- (4,-2) -- (3,-3) -- (3,-4);
\node at (7,-2.5) {$+$};

\pgftransformxshift{9cm}
\node at (1,.8) {\tiny$2$};
\node at (4,.8) {\tiny$1$};
\draw (1,0)  -- (1,-1) -- (0,-2) -- (0,-4);
\draw (1,-1) -- (3,-3);
\node at (2.5,-1.5) {\tiny$\Delta$};
\draw (4,0) -- (4,-2) -- (3,-3) -- (3,-4);
\pgftransformxshift{9cm}

\pgftransformxshift{-38cm}
\pgftransformyshift{-8cm}

\node at (0,-2) {$-$};
\pgftransformxshift{2cm}

\draw (0,0) -- (0,-2) -- (1,-3) -- (1,-4);
\draw (1,-3) -- (2.5,-1.5) -- (4,-3);
\draw (2.5,-1.5) -- (2.5,-0.5);
\node[circle,draw,inner sep=1pt] at (2.5,-.5) {};
\draw (5,0) -- (5,-2) -- (4,-3) -- (4,-4);
\node at (3.7,-1.5) {\tiny$\Delta$};
\node at (8,-2.5) {$-$};

\pgftransformxshift{11cm}
\node at (0,0.8) {\tiny$2$};
\node at (5,0.8) {\tiny$1$};
\draw (0,0) -- (0,-2) -- (1,-3) -- (1,-4);
\draw (1,-3) -- (2.5,-1.5) -- (4,-3);
\draw (2.5,-1.5) -- (2.5,-0.5);
\node[circle,draw,inner sep=1pt] at (2.5,-.5) {};
\draw (5,0) -- (5,-2) -- (4,-3) -- (4,-4);
\node at (3.7,-1.5) {\tiny$\Delta$};

\node at (8,-2.5) {.};
\end{tikzpicture}
\end{center}
\begin{remark}\label{rem:7_term_and_permutation}
Using (co)associativity and (co)commutativity of $\mu$ and $\lambda$, one checks that
$$
\mu (\mu \otimes 1) \sigma = \mu (\mu \otimes 1)
\quad \text{and} \quad
\sigma (\lambda \otimes 1)\lambda = (\lambda \otimes 1)\lambda.
$$
It follows that
\begin{align*}
\mu (\mu \otimes 1) (\Delta \otimes 1 \otimes 1) \sigma &=
\mu (\mu \otimes 1) (1 \otimes 1 \otimes \Delta),\\
\mu (\mu \otimes 1) (\Delta \otimes 1 \otimes 1) \sigma^2 &=
\mu (\mu \otimes 1) (1 \otimes \Delta \otimes 1),\\
\sigma (\Delta \otimes 1 \otimes 1)(\lambda \otimes 1)\lambda &= (1 \otimes \Delta \otimes 1)(\lambda \otimes 1)\lambda \quad \text{and}\\
\sigma^2 (\Delta \otimes 1 \otimes 1)(\lambda \otimes 1)\lambda &= (1 \otimes 1 \otimes \Delta)(\lambda \otimes 1)\lambda.
\end{align*}
Using these relations, one could reformulate the 7-term relations for $\Delta$ and $\mu$ (resp. $\Delta$ and $\lambda$), trading some uses of cyclic permutations for moving $\Delta$ to different inputs (respectively outputs). For example, the 7-term relation for $\Delta$ and $\mu$ is equivalent to the relation
\begin{center}
\begin{tikzpicture}[scale=.23]
\draw (0,0) -- (1,-1) -- (2,0);
\draw (1,-1) -- (1,-2);
\draw (1,-2) -- (2,-3) -- (4,-1);
\draw (2,-3) -- (2,-4);
\node at (2,-5) {\tiny$\Delta$};

\node at (6,-2) {$=$};
\pgftransformxshift{8cm}

\draw (0,0) -- (1,-1) -- (2,0);
\draw (1,-1) -- (1,-2);
\draw (1,-2) -- (2,-3) -- (4,-1);
\draw (2,-3) -- (2,-4);
\node at (0.4,-2.4) {\tiny$\Delta$};

\node at (6,-2) {$+$};
\pgftransformxshift{8cm}

\draw (0,0) -- (1,-1) -- (2,0);
\draw (1,-1) -- (1,-2);
\draw (1,-2) -- (2,-3) -- (4,-1);
\draw (2,-3) -- (2,-4);
\node at (0.4,-2.4) {\tiny$\Delta$};
\node at (0,1) {\tiny$2$};
\node at (2,1) {\tiny$3$};
\node at (4,0) {\tiny$1$};
\node at (6,-2) {$+$};
\pgftransformxshift{8cm}

\draw (0,0) -- (1,-1) -- (2,0);
\draw (1,-1) -- (1,-2);
\draw (1,-2) -- (2,-3) -- (4,-1);
\draw (2,-3) -- (2,-4);
\node at (0,1) {\tiny$3$};
\node at (2,1) {\tiny$1$};
\node at (4,0) {\tiny$2$};
\node at (0.4,-2.4) {\tiny$\Delta$};

\node at (6,-2) {$-$};
\pgftransformxshift{8cm}

\draw (0,0) -- (1,-1) -- (2,0);
\draw (1,-1) -- (1,-2);
\draw (1,-2) -- (2,-3) -- (4,-1);
\draw (2,-3) -- (2,-4);
\node at (0,0.8) {\tiny$\Delta$};

\node at (6,-2) {$-$};
\pgftransformxshift{8cm}

\draw (0,0) -- (1,-1) -- (2,0);
\draw (1,-1) -- (1,-2);
\draw (1,-2) -- (2,-3) -- (4,-1);
\draw (2,-3) -- (2,-4);
\node at (2,.8) {\tiny$\Delta$};

\node at (6,-2) {$-$};
\pgftransformxshift{8cm}

\draw (0,0) -- (1,-1) -- (2,0);
\draw (1,-1) -- (1,-2);
\draw (1,-2) -- (2,-3) -- (4,-1) -- (4,0);
\draw (2,-3) -- (2,-4);
\node at (4,.8) {\tiny$\Delta$};

\node at (6,-2) {.};
\end{tikzpicture}
\end{center}

\end{remark}
\begin{remark}
By applying the 7-term relation for $\Delta$ and $\mu$ to $\eta \otimes \eta \otimes \eta$, we find that
\begin{equation}\label{eq:Delta_and_eta}
\Delta \eta =0.
\end{equation}
By applying the 11-term relation to $\eta \otimes \eta$ and using \eqref{eq:Delta_and_eta} we get
\begin{equation}\label{eq:Delta_and_lambda_eta}
(\Delta \otimes 1) \lambda \eta - (1 \otimes \Delta) \lambda \eta =0.
\end{equation}
Applying the 7-term relation for $\Delta$ and $\lambda$ to $\eta$ does not yield any new identities.
\end{remark} 

\begin{remark}\label{rem:counit_Delta}
Suppose that the BV unital infinitesimal bialgebra has a counit $\eps$ (this is not part of the definition, but happens in the Frobenius case discussed below). Such a counit has odd degree and satisfies $(\eps\otimes 1)\lambda = 1 = -(1\otimes\eps)\lambda$. By postcomposing the 7-term relation for $\Delta$ and $\lambda$ with $\eps \otimes \eps \otimes \eps$, we find that
$$
\eps \Delta=0.
$$
Moreover, applying $\eps \otimes \eps$ to the 11-term relation and using the above identity yields
$$
\eps \mu (\Delta \otimes 1 + 1 \otimes\Delta) = (\eps\mu \otimes \eps \mu)((1 \otimes 1 \otimes \Delta \otimes 1)(1 \otimes \lambda \eta \otimes 1) (1+\tau).
$$
This equation simplifies to a tautology in the Frobenius case. 
\end{remark}
From the primary structures in the definition one can define two additional structures, a \emph{bracket} $\beta$ of degree 1 and a \emph{cobracket} $\gamma$ of even degree $|\gamma|=|\lambda|+1$ as
\begin{equation} \label{eq:beta-gamma}
\beta = [\Delta,\mu] \qquad \text{and} \qquad \gamma = [\Delta,\lambda],
\end{equation}
where $\Delta$ is extended to $A\otimes A$ in the usual way, i.e., as $\Delta \otimes 1 + 1 \otimes \Delta$. We represent them graphically as follows:

\begin{center}
\begin{tikzpicture}[scale=.25]

\draw (0,0) -- (1,-1) -- (2,0);
\draw (1,-1) -- (1,-2.5);
\node[circle,draw,fill ,inner sep=1pt] at (1,-1) {};
\node at (4,-1) {$=$};

\pgftransformxshift{6cm}

\draw (0,0) -- (1,-1) -- (2,0);
\draw (1,-1) -- (1,-2.5);
\node at (1,-3.2) {\tiny $\Delta$};
\node at (4,-1) {$-$};

\pgftransformxshift{6cm}

\draw (0,0) -- (1,-1) -- (2,0);
\draw (1,-1) -- (1,-2.5);
\node at (0,.7) {\tiny $\Delta$};
\node at (4,-1) {$-$};

\pgftransformxshift{6cm}

\draw (0,0) -- (1,-1) -- (2,0);
\draw (1,-1) -- (1,-2.5);
\node at (2,.7) {\tiny $\Delta$};
\node at (5.5,-1) {and};

\pgftransformxshift{10cm}

\draw (0,-2.5) -- (1,-1.5) -- (2,-2.5);
\draw (1,0) -- (1,-1.5);
\node[circle,draw,fill ,inner sep=1pt] at (1,-1.5) {};
\node at (4,-1) {$=$};

\pgftransformxshift{6cm}

\draw (0,-2.5) -- (1,-1.5) -- (2,-2.5);
\draw (1,0) -- (1,-1.5);
\node at (1,.7) {\tiny$\Delta$};
\node at (4,-1) {$+$};

\pgftransformxshift{6cm}

\draw (0,-2.5) -- (1,-1.5) -- (2,-2.5);
\draw (1,0) -- (1,-1.5);
\node at (0,-3.2) {\tiny$\Delta$};
\node at (4,-1) {$+$};

\pgftransformxshift{6cm}

\draw (0,-2.5) -- (1,-1.5) -- (2,-2.5);
\draw (1,0) -- (1,-1.5);
\node at (2,-3.2) {\tiny$\Delta$};

\end{tikzpicture}
\end{center}
From the symmetry properties of $\mu$ and $\lambda$ we conclude that $\beta$ and $\gamma$ are graded anti-symmetric, i.e.,
$$
\beta \tau = \beta \qquad \text{and} \qquad \tau \gamma = - \gamma.
$$
The graded Jacobi relation for $\beta$ is the same as in the ungraded case, i.e., 
\begin{equation} \label{eq:Jacobi-graded}
\beta(1\otimes\beta)(1+\sigma+\sigma^2)=0,
\end{equation}
and the coJacobi relation for $\gamma$ also takes the usual form
\begin{equation} \label{eq:coJacobi}
(1+\sigma+\sigma^2)(\gamma\otimes 1)\gamma=0.
\end{equation}

By a standard argument, the 7-term relation for $\Delta$ and $\mu$ implies the \emph{Poisson relation} 
\begin{equation} \label{eq:Poisson}
\beta (\mu \otimes 1) = \mu (1 \otimes \beta) + \mu(\beta\otimes 1)(1 \otimes \tau)
\end{equation}
for $\beta$ and $\mu$, graphically represented by
\begin{center}
\begin{tikzpicture}[scale=.25]
\draw (0,0) -- (1,-1) -- (2,0);
\draw (1,-1) -- (1,-2) -- (2,-3) -- (4,-1);
\draw (2,-3) -- (2,-4);
\node[circle,draw,fill ,inner sep=1pt] at (2,-3) {};

\node at (7,-2) {$=$};
\pgftransformxshift{10cm}

\draw (0,-1) -- (2,-3) -- (2,-4);
\draw (2,0) -- (3,-1) -- (4,0);
\draw (3,-1) -- (3,-2) -- (2,-3);
\node[circle,draw,fill ,inner sep=1pt] at (3,-1) {};

\node at (7,-2) {$+$};
\pgftransformxshift{10cm}

\draw (0,0) -- (1,-1) -- (2,0);
\draw (1,-1) -- (1,-2) -- (2,-3) -- (4,-1);
\draw (2,-3) -- (2,-4);
\node[circle,draw,fill ,inner sep=1pt] at (1,-1) {};
\node at (0,0.8) {\tiny$1$};
\node at (2,0.8) {\tiny$3$};
\node at (4,-0.2) {\tiny$2$};
\end{tikzpicture}
\end{center}
Similarly, the 7-term relation for $\Delta$ and $\lambda$ implies the \emph{coPoisson relation}
\begin{equation} \label{eq:coPoisson}
(\lambda \otimes 1)\gamma = (1 \otimes \gamma) \lambda + (1 \otimes \tau)(\gamma \otimes 1)\lambda 
\end{equation}
for $\gamma$ and $\lambda$, graphically represented as
\begin{center}
\begin{tikzpicture}[scale=.25]
\draw (1,0) -- (1,-1) -- (2,-2);
\node[circle,draw,fill ,inner sep=1pt] at (1,-1) {};
\draw (1,-1) -- (0,-2) -- (0,-3) -- (-1,-4) -- (0,-3) -- (1,-4);

\node at (7,-2) {$=$};
\pgftransformxshift{10cm}

\draw (1,0) -- (1,-1) -- (-1,-3) -- (1,-1) -- (2,-2) -- (2,-3) -- (1,-4) -- (2,-3) -- (3,-4);
\node[circle,draw,fill ,inner sep=1pt] at (2,-3) {};

\node at (7,-2) {$+$};
\pgftransformxshift{10cm}

\draw (1,0) -- (1,-1) -- (3,-3) -- (1,-1) -- (0,-2) -- (0,-3) -- (-1,-4) -- (0,-3) -- (1,-4);
\node[circle,draw,fill ,inner sep=1pt] at (0,-3) {};
\node at (-1,-4.5) {\tiny$1$};
\node at (1,-4.5) {\tiny$3$};
\node at (3,-3.5) {\tiny$2$};

\node at (6,-2) {$.$};

\end{tikzpicture}
\end{center}

For completeness, we include a proof here. We have
\begin{align*}
(\lambda \otimes 1) \gamma
&= (\lambda \otimes 1)\lambda \Delta + (\lambda \Delta \otimes 1) \lambda + (\lambda \otimes 1)(1 \otimes \Delta) \lambda\\
&= (\lambda \otimes 1)\lambda \Delta + (\lambda \Delta \otimes 1) \lambda -(1 \otimes 1 \otimes \Delta) (\lambda \otimes 1) \lambda.
\end{align*}
Similarly, we compute
\begin{align*}
(1 & \otimes \gamma) \lambda\\
&= (1 \otimes \lambda \Delta) \lambda + (1 \otimes \Delta \otimes 1)(1 \otimes \lambda) \lambda + (1 \otimes 1 \otimes \Delta)(1 \otimes \lambda) \lambda \\
&= (1 \otimes \lambda \Delta) \lambda + (1 \otimes \Delta \otimes 1)(1 \otimes \lambda) \lambda + (1 \otimes 1 \otimes \Delta)(1 \otimes \lambda) \lambda \\
&= \sigma^2 (\lambda\Delta \otimes 1)\tau\lambda + \sigma^2(\Delta \otimes 1 \otimes 1)(\lambda \otimes 1)\tau \lambda + (1 \otimes 1 \otimes \Delta)(1 \otimes \lambda) \lambda\\
&=- \sigma^2 (\lambda\Delta \otimes 1)\lambda - \sigma^2 (\Delta \otimes 1 \otimes 1)(\lambda \otimes 1)\lambda - (1 \otimes 1 \otimes \Delta)(\lambda \otimes 1) \lambda
\end{align*}
and
\begin{align*}
(1 & \otimes \tau)(\gamma \otimes 1)\lambda\\
&= (1 \otimes \tau)(\lambda\Delta \otimes 1)\lambda + (1 \otimes \tau)(\Delta \otimes 1 \otimes 1) (\lambda \otimes 1)\lambda\\
& \qquad \qquad \qquad \qquad \qquad + (1 \otimes \tau) (1 \otimes \Delta \otimes 1) (\lambda \otimes 1)\lambda\\
&=-\sigma (\lambda \Delta \otimes 1)\lambda - (\Delta \otimes 1 \otimes 1) (1 \otimes \tau)(1 \otimes \lambda)\lambda\\
& \qquad \qquad \qquad \qquad \qquad  - (1 \otimes \tau) (1 \otimes \Delta \otimes 1) (\tau\lambda \otimes 1)\lambda\\
&=-\sigma (\lambda \Delta \otimes 1)\lambda + (\Delta \otimes 1 \otimes 1) (1 \otimes \lambda)\lambda - \sigma (\Delta \otimes 1 \otimes 1) (\lambda \otimes 1)\lambda\\
&=-\sigma (\lambda \Delta \otimes 1)\lambda - (\Delta \otimes 1 \otimes 1) (\lambda \otimes 1)\lambda - \sigma (\Delta \otimes 1 \otimes 1) (\lambda \otimes 1)\lambda.
\end{align*}
So subtracting the last two terms from the first gives zero by the 7-term relation for $\Delta$ and $\lambda$.

Another consequence of the definition is
\begin{lemma}\label{lemma:rel_mu_lambda_beta_gamma}
Let $(A, \mu, \eta, \lambda, \Delta)$ be an odd BV unital infinitesimal bialgebra. 
The following equation holds:
\begin{eqnarray*} 
\lefteqn{\gamma \mu - \lambda \beta} \\
&=& (\beta \otimes 1)(1\otimes \lambda) + (\mu \otimes 1)(1\otimes \gamma) +(1 \otimes \mu)(\gamma \otimes 1) + (1 \otimes \beta)(\lambda \otimes 1) \\
& & -(\mu \otimes \beta)(1 \otimes \lambda \eta \otimes 1) - (\beta \otimes \mu)(1 \otimes \lambda \eta \otimes 1) -(\mu \otimes \mu)(1 \otimes \gamma\eta \otimes 1). 
\end{eqnarray*}
\end{lemma}

\begin{proof}[Picture proof]
We present a graphical proof, from which a formal version can be extrapolated without difficulty. The graphical version of the identity is
\begin{center}
\begin{tikzpicture}[scale=.25]

\draw (0,0) -- (1,-1) -- (2,0);
\draw (1,-1) -- (1,-3);
\draw (0,-4) -- (1,-3) -- (2,-4);
\node[circle,draw,fill,inner sep=1pt] at (1,-3) {};
\node at (4,-2) {$-$};

\pgftransformxshift{6cm}
\draw (0,0) -- (1,-1) -- (2,0);
\draw (1,-1) -- (1,-3);
\draw (0,-4) -- (1,-3) -- (2,-4);
\node[circle,draw,fill,inner sep=1pt] at (1,-1) {};
\node at (4,-2) {$=$};

\pgftransformxshift{6cm}

\draw (0,0) -- (0,-2) -- (1,-3) -- (3,-1)--(3,0);
\draw (1,-3) -- (1,-4);
\draw (3,-1) -- (4,-2) -- (4,-4);
\node[circle,draw,fill,inner sep=1pt] at (1,-3) {};
\node at (6,-2.5) {$+$};

\pgftransformxshift{8cm}
\draw (0,0) -- (0,-2) -- (1,-3) -- (3,-1)--(3,0);
\draw (1,-3) -- (1,-4);
\draw (3,-1) -- (4,-2) -- (4,-4);
\node[circle,draw,fill,inner sep=1pt] at (3,-1) {};
\node at (6,-2.5) {$+$};

\pgftransformxshift{8cm}

\draw (1,0)  -- (1,-1) -- (0,-2) -- (0,-4);
\draw (1,-1) -- (3,-3) -- (3,-4);
\draw (4,0) -- (4,-2) -- (3,-3);
\node[circle,draw,fill,inner sep=1pt] at (1,-1) {};
\node at (6,-2.5) {$+$};

\pgftransformxshift{8cm}

\draw (1,0)  -- (1,-1) -- (0,-2) -- (0,-4);
\draw (1,-1) -- (3,-3) -- (3,-4);
\draw (4,0) -- (4,-2) -- (3,-3);
\node[circle,draw,fill,inner sep=1pt] at (3,-3) {};

\pgftransformxshift{-24cm}
\pgftransformyshift{-6cm}
\node at (0,-2.5) {$-$};

\pgftransformxshift{2cm}
\draw (0,0)  -- (0,-2) -- (1,-3) -- (2.5,-1.5) -- (4,-3) -- (5,-2) -- (5,0);
\draw (1,-3) -- (1,-4);
\draw (4,-3) -- (4,-4);
\draw (2.5,-1.5) -- (2.5,-.5);
\node[circle,draw,inner sep=1pt] at (2.5,-.5) {};
\node[circle,draw,fill,inner sep=1pt] at (4,-3) {};
\node at (7,-2.5) {$-$};

\pgftransformxshift{9cm}

\draw (0,0)  -- (0,-2) -- (1,-3) -- (2.5,-1.5) -- (4,-3) -- (5,-2) -- (5,0);
\draw (1,-3) -- (1,-4);
\draw (4,-3) -- (4,-4);
\draw (2.5,-1.5) -- (2.5,-.5);
\node[circle,draw,inner sep=1pt] at (2.5,-.5) {};
\node[circle,draw,fill,inner sep=1pt] at (1,-3) {};
\node at (7,-2.5) {$-$};

\pgftransformxshift{9cm}

\draw (0,0)  -- (0,-2) -- (1,-3) -- (2.5,-1.5) -- (4,-3) -- (5,-2) -- (5,0);
\draw (1,-3) -- (1,-4);
\draw (4,-3) -- (4,-4);
\draw (2.5,-1.5) -- (2.5,-.5);
\node[circle,draw,inner sep=1pt] at (2.5,-.5) {};
\node[circle,draw,fill,inner sep=1pt] at (2.5,-1.5) {};
\end{tikzpicture}
\end{center}
Next we draw the pictures for the expansions of the seven terms on the right hand side:
\begin{center}
\begin{tikzpicture}[scale=.25]
\draw (0,0) -- (0,-2) -- (1,-3) -- (3,-1)--(3,0);
\draw (1,-3) -- (1,-4);
\draw (3,-1) -- (4,-2) -- (4,-4);
\node[circle,draw,fill,inner sep=1pt] at (1,-3) {};
\node at (6,-2.5) {$=$};

\pgftransformxshift{8cm}

\draw (0,0) -- (0,-2) -- (1,-3) -- (3,-1)--(3,0);
\draw (1,-3) -- (1,-4);
\draw (3,-1) -- (4,-2) -- (4,-4);
\node at (1,-4.8) {\tiny$\Delta$};
\node at (6,-2.5) {$+$};

\pgftransformxshift{8cm}

\draw (0,0) -- (0,-2) -- (1,-3) -- (3,-1)--(3,0);
\draw (1,-3) -- (1,-4);
\draw (3,-1) -- (4,-2) -- (4,-4);
\node at (0,.8) {\tiny$\Delta$};
\node at (6,-2.5) {$-$};

\pgftransformxshift{8cm}

\draw (0,0) -- (0,-2) -- (1,-3) -- (3,-1)--(3,0);
\draw (1,-3) -- (1,-4);
\draw (3,-1) -- (4,-2) -- (4,-4);
\node at (1.5,-1.3) {\tiny$\Delta$};

\end{tikzpicture}
\end{center}
The plus in the second term comes from switching the order of $\Delta$ and $\lambda$ in the composition. Similarly, we get
\begin{center}
\begin{tikzpicture}[scale=.25]
\draw (0,0) -- (0,-2) -- (1,-3) -- (3,-1)--(3,0);
\draw (1,-3) -- (1,-4);
\draw (3,-1) -- (4,-2) -- (4,-4);
\node[circle,draw,fill,inner sep=1pt] at (3,-1) {};
\node at (6,-2.5) {$=$};

\pgftransformxshift{8cm}

\draw (0,0) -- (0,-2) -- (1,-3) -- (3,-1)--(3,0);
\draw (1,-3) -- (1,-4);
\draw (3,-1) -- (4,-2) -- (4,-4);
\node at (3,.8) {\tiny$\Delta$};
\node at (6,-2.5) {$+$};

\pgftransformxshift{8cm}

\draw (0,0) -- (0,-2) -- (1,-3) -- (3,-1)--(3,0);
\draw (1,-3) -- (1,-4);
\draw (3,-1) -- (4,-2) -- (4,-4);
\node at (4,-4.8) {\tiny$\Delta$};
\node at (6,-2.5) {$+$};

\pgftransformxshift{8cm}

\draw (0,0) -- (0,-2) -- (1,-3) -- (3,-1)--(3,0);
\draw (1,-3) -- (1,-4);
\draw (3,-1) -- (4,-2) -- (4,-4);
\node at (1.5,-1.3) {\tiny$\Delta$};

\end{tikzpicture}
\end{center}

\begin{center}
\begin{tikzpicture}[scale=.25]
\draw (1,0)  -- (1,-1) -- (0,-2) -- (0,-4);
\draw (1,-1) -- (3,-3) -- (3,-4);
\draw (4,0) -- (4,-2) -- (3,-3);
\node[circle,draw,fill,inner sep=1pt] at (1,-1) {};
\node at (6,-2.5) {$=$};

\pgftransformxshift{8cm}

\draw (1,0)  -- (1,-1) -- (0,-2) -- (0,-4);
\draw (1,-1) -- (3,-3) -- (3,-4);
\draw (4,0) -- (4,-2) -- (3,-3);
\node at (1,.8) {\tiny$\Delta$};
\node at (6,-2.5) {$+$};

\pgftransformxshift{8cm}

\draw (1,0)  -- (1,-1) -- (0,-2) -- (0,-4);
\draw (1,-1) -- (3,-3) -- (3,-4);
\draw (4,0) -- (4,-2) -- (3,-3);
\node at (0,-4.8) {\tiny$\Delta$};
\node at (6,-2.5) {$+$};

\pgftransformxshift{8cm}

\draw (1,0)  -- (1,-1) -- (0,-2) -- (0,-4);
\draw (1,-1) -- (3,-3) -- (3,-4);
\draw (4,0) -- (4,-2) -- (3,-3);
\node at (2.5,-1.3) {\tiny$\Delta$};
\end{tikzpicture}
\end{center}

\begin{center}
\begin{tikzpicture}[scale=.25]
\draw (1,0)  -- (1,-1) -- (0,-2) -- (0,-4);
\draw (1,-1) -- (3,-3) -- (3,-4);
\draw (4,0) -- (4,-2) -- (3,-3);
\node[circle,draw,fill,inner sep=1pt] at (3,-3) {};
\node at (6,-2.5) {$=$};

\pgftransformxshift{8cm}

\draw (1,0)  -- (1,-1) -- (0,-2) -- (0,-4);
\draw (1,-1) -- (3,-3) -- (3,-4);
\draw (4,0) -- (4,-2) -- (3,-3);
\node at (3,-4.8) {\tiny$\Delta$};
\node at (6,-2.5) {$+$};

\pgftransformxshift{8cm}

\draw (1,0)  -- (1,-1) -- (0,-2) -- (0,-4);
\draw (1,-1) -- (3,-3) -- (3,-4);
\draw (4,0) -- (4,-2) -- (3,-3);
\node at (4,.8) {\tiny$\Delta$};
\node at (6,-2.5) {$-$};

\pgftransformxshift{8cm}

\draw (1,0)  -- (1,-1) -- (0,-2) -- (0,-4);
\draw (1,-1) -- (3,-3) -- (3,-4);
\draw (4,0) -- (4,-2) -- (3,-3);
\node at (2.5,-1.3) {\tiny$\Delta$};
\end{tikzpicture}
\end{center}

\begin{center}
\begin{tikzpicture}[scale=.25]

\draw (0,0)  -- (0,-2) -- (1,-3) -- (2.5,-1.5) -- (4,-3) -- (5,-2) -- (5,0);
\draw (1,-3) -- (1,-4);
\draw (4,-3) -- (4,-4);
\draw (2.5,-1.5) -- (2.5,-.5);
\node[circle,draw,inner sep=1pt] at (2.5,-.5) {};
\node[circle,draw,fill,inner sep=1pt] at (4,-3) {};
\node at (7,-2.5) {$=$};

\pgftransformxshift{9cm}

\draw (0,0)  -- (0,-2) -- (1,-3) -- (2.5,-1.5) -- (4,-3) -- (5,-2) -- (5,0);
\draw (1,-3) -- (1,-4);
\draw (4,-3) -- (4,-4);
\draw (2.5,-1.5) -- (2.5,-.5);
\node[circle,draw,inner sep=1pt] at (2.5,-.5) {};
\node at (4,-4.8) {\tiny$\Delta$};
\node at (7,-2.5) {$+$};

\pgftransformxshift{9cm}

\draw (0,0)  -- (0,-2) -- (1,-3) -- (2.5,-1.5) -- (4,-3) -- (5,-2) -- (5,0);
\draw (1,-3) -- (1,-4);
\draw (4,-3) -- (4,-4);
\draw (2.5,-1.5) -- (2.5,-.5);
\node[circle,draw,inner sep=1pt] at (2.5,-.5) {};
\node at (5,0.8) {\tiny$\Delta$};
\node at (7,-2.5) {$-$};

\pgftransformxshift{9cm}

\draw (0,0)  -- (0,-2) -- (1,-3) -- (2.5,-1.5) -- (4,-3) -- (5,-2) -- (5,0);
\draw (1,-3) -- (1,-4);
\draw (4,-3) -- (4,-4);
\draw (2.5,-1.5) -- (2.5,-.5);
\node[circle,draw,inner sep=1pt] at (2.5,-.5) {};
\node at (3.7,-1.7) {\tiny$\Delta$};
\end{tikzpicture}
\end{center}

\begin{center}
\begin{tikzpicture}[scale=.25]

\draw (0,0)  -- (0,-2) -- (1,-3) -- (2.5,-1.5) -- (4,-3) -- (5,-2) -- (5,0);
\draw (1,-3) -- (1,-4);
\draw (4,-3) -- (4,-4);
\draw (2.5,-1.5) -- (2.5,-.5);
\node[circle,draw,inner sep=1pt] at (2.5,-.5) {};
\node[circle,draw,fill,inner sep=1pt] at (1,-3) {};
\node at (7,-2.5) {$=$};

\pgftransformxshift{9cm}

\draw (0,0)  -- (0,-2) -- (1,-3) -- (2.5,-1.5) -- (4,-3) -- (5,-2) -- (5,0);
\draw (1,-3) -- (1,-4);
\draw (4,-3) -- (4,-4);
\draw (2.5,-1.5) -- (2.5,-.5);
\node[circle,draw,inner sep=1pt] at (2.5,-.5) {};
\node at (1,-4.8) {\tiny$\Delta$};
\node at (7,-2.5) {$+$};

\pgftransformxshift{9cm}

\draw (0,0)  -- (0,-2) -- (1,-3) -- (2.5,-1.5) -- (4,-3) -- (5,-2) -- (5,0);
\draw (1,-3) -- (1,-4);
\draw (4,-3) -- (4,-4);
\draw (2.5,-1.5) -- (2.5,-.5);
\node[circle,draw,inner sep=1pt] at (2.5,-.5) {};
\node at (0,.8) {\tiny$\Delta$};
\node at (7,-2.5) {$-$};

\pgftransformxshift{9cm}

\draw (0,0)  -- (0,-2) -- (1,-3) -- (2.5,-1.5) -- (4,-3) -- (5,-2) -- (5,0);
\draw (1,-3) -- (1,-4);
\draw (4,-3) -- (4,-4);
\draw (2.5,-1.5) -- (2.5,-.5);
\node[circle,draw,inner sep=1pt] at (2.5,-.5) {};
\node at (1.3,-1.7) {\tiny$\Delta$};
\end{tikzpicture}
\end{center}

\begin{center}
\begin{tikzpicture}[scale=.25]

\draw (0,0)  -- (0,-2) -- (1,-3) -- (2.5,-1.5) -- (4,-3) -- (5,-2) -- (5,0);
\draw (1,-3) -- (1,-4);
\draw (4,-3) -- (4,-4);
\draw (2.5,-1.5) -- (2.5,-.5);
\node[circle,draw,inner sep=1pt] at (2.5,-.5) {};
\node[circle,draw,fill,inner sep=1pt] at (2.5,-1.5) {};
\node at (7,-2.5) {$=$};

\pgftransformxshift{9cm}

\draw (0,0)  -- (0,-2) -- (1,-3) -- (2.5,-1.5) -- (4,-3) -- (5,-2) -- (5,0);
\draw (1,-3) -- (1,-4);
\draw (4,-3) -- (4,-4);
\draw (2.5,-1.5) -- (2.5,-.5);
\node[circle,draw,inner sep=1pt] at (2.5,-.5) {};
\node at (1.3,-1.7) {\tiny$\Delta$};
\node at (7,-2.5) {$+$};

\pgftransformxshift{9cm}

\draw (0,0)  -- (0,-2) -- (1,-3) -- (2.5,-1.5) -- (4,-3) -- (5,-2) -- (5,0);
\draw (1,-3) -- (1,-4);
\draw (4,-3) -- (4,-4);
\draw (2.5,-1.5) -- (2.5,-.5);
\node[circle,draw,inner sep=1pt] at (2.5,-.5) {};
\node at (3.7,-1.7) {\tiny$\Delta$};
\end{tikzpicture}
\end{center}
In the last equation, we have used $\Delta \eta=0$.
We see that in summing up the first four terms, the third summands cancel in pairs. Similarly, the third summands of the fifth and sixth expression cancel with the last expression. Using the unital infinitesimal relation, the remaining summands combine to give
\begin{center}
\begin{tikzpicture}[scale=.25]

\draw (0,0) -- (1,-1) -- (2,0);
\draw (1,-1) -- (1,-3);
\draw (0,-4) -- (1,-3) -- (2,-4);
\node at (0,-5) {\tiny$\Delta$};
\node at (4,-2) {$+$};

\pgftransformxshift{6cm}
\draw (0,0) -- (1,-1) -- (2,0);
\draw (1,-1) -- (1,-3);
\draw (0,-4) -- (1,-3) -- (2,-4);
\node at (2,-5) {\tiny$\Delta$};
\node at (4,-2) {$+$};

\pgftransformxshift{6cm}
\draw (0,0) -- (1,-1) -- (2,0);
\draw (1,-1) -- (1,-3);
\draw (0,-4) -- (1,-3) -- (2,-4);
\node at (0,1) {\tiny$\Delta$};
\node at (4,-2) {$+$};

\pgftransformxshift{6cm}
\draw (0,0) -- (1,-1) -- (2,0);
\draw (1,-1) -- (1,-3);
\draw (0,-4) -- (1,-3) -- (2,-4);
\node at (2,1) {\tiny$\Delta$};
\node at (4,-2) {$=$};

\pgftransformxshift{6cm}
\draw (0,0) -- (1,-1) -- (2,0);
\draw (1,-1) -- (1,-3);
\draw (0,-4) -- (1,-3) -- (2,-4);
\node[circle,draw,fill,inner sep=1pt] at (1,-3) {};
\node at (4,-2) {$-$};

\pgftransformxshift{6cm}
\draw (0,0) -- (1,-1) -- (2,0);
\draw (1,-1) -- (1,-3);
\draw (0,-4) -- (1,-3) -- (2,-4);
\node[circle,draw,fill,inner sep=1pt] at (1,-1) {};
\end{tikzpicture}
\end{center}
as claimed.
\end{proof}

In odd BV unital infinitesimal bialgebras $(A, \mu, \eta, \lambda, \Delta)$ such that 
$$
\lambda\eta=0
$$
some of the algebraic identities simplify considerably:
\begin{itemize}
\item the 4-term unital infinitesimal relation reduces to the 3-term \emph{infinitesimal relation} $\lambda\mu=(1 \otimes \mu)(\lambda \otimes 1) + (\mu \otimes 1)(1 \otimes \lambda)$,
\item the 11-term relation for $\Delta$, $\mu$, $\lambda$ and $\eta$ reduces to a 9-term relation for $\Delta$, $\mu$ and $\lambda$, and
\item the identity in Lemma~\ref{lemma:rel_mu_lambda_beta_gamma} simplifies to a 6-term relation.
\end{itemize}

\begin{remark} Note that, although $\lambda\eta$ may not vanish for reduced symplectic homology or reduced loop homology, which are our main sources of examples from~\S\ref{sec:reduced} and~\S\ref{sec:stringtop}, the 11-term relation still reduces to the 9-term relation in that case (Remark~\ref{rmk:11-term-is-9-term-for-SH}). 
\end{remark} 

\begin{remark}
The notion of a \emph{BV unital infinitesimal bialgebra} from Definition~\ref{defi:BVui} also makes sense when the underlying $R$-module $A$
is a Tate vector space. We will not insist on this aspect because, in our context, the structure arises on reduced symplectic homology and the underlying linearly topologized vector space is discrete. See also Remark~\ref{rmk:TateBVui}. 
\end{remark}

\begin{remark}\label{rem:no-intermediate1}
BV algebras are refinements of Gerstenhaber algebras, which are associative algebras with an odd Lie bracket, satisfying a Poisson identity. Nevertheless, a corresponding intermediate algebraic structure between unital infinitesimal bialgebras and BV unital infinitesimal bialgebras does not seem to exist. In particular, as far as we are aware, the 11-term relation does not imply any relation just involving the product, coproduct, bracket and cobracket.\footnote{Note that the identity in Lemma~\ref{lemma:rel_mu_lambda_beta_gamma} is derived using only the 4-term unital infinitesimal relation, so it is independent of the 11-term relation.}  We give a geometric reason for this phenomenon in Remark~\ref{rem:no-intermediate2}.
\end{remark}

\begin{definition} \label{defi:odd-Frobenius} 
An {\em odd BV Frobenius algebra} is a tuple\break $(A, \mu, \lambda, \eta, \eps,\Delta)$, where
\begin{itemize}
\item $(A, \mu, \lambda, \eta, \eps)$ is a Frobenius algebra with $|\mu|=0$ and $|\lambda|$ odd,
\item $(A, \mu, \eta, \Delta)$ is a BV algebra, i.e., $\Delta:A \to A$ has degree $1$ and satisfies $\Delta^2=0$ and the 7-term relation with respect to $\mu$, and
\item equation \eqref{eq:Delta_and_lambda_eta}, which we call \emph{BV Frobenius relation} holds, i.e.,
\begin{flalign} \label{eq:BVFrobenius-relation}
\mbox{\sc (BV Frobenius relation)} \quad 
(\Delta\otimes 1) \lambda \eta = (1 \otimes \Delta) \lambda \eta. &&
\end{flalign}
\end{itemize}
\end{definition}

The Frobenius property for $\mu$, $\lambda$, $\eta$, $\eps$ means that $\mu$ is associative with unit $\eta$, $\lambda$ is coassociative with counit $\eps$, and the following \emph{Frobenius relation} is satisfied,  
\begin{flalign}\label{eq:Frobenius}
\mbox{\sc (Frobenius relation)} \  \lambda \mu = (\mu \otimes 1) (1 \otimes \lambda) = (1 \otimes \mu)(\lambda \otimes 1), &&
\end{flalign}
which in particular implies 
\begin{equation}\label{eq:Frobenius_with_eta}
\lambda=(\mu \otimes 1)(1 \otimes \lambda\eta) = (1 \otimes \mu)(\lambda\eta \otimes 1)
\end{equation}
and
\begin{equation}\label{eq:Frobenius_with_epsilon}
\mu=(\eps\mu \otimes 1)(1 \otimes \lambda) = -(1 \otimes \eps\mu)(\lambda \otimes 1).
\end{equation}
This definition of Frobenius algebra is equivalent to that of biunital coFrobenius bialgebra from~\cite{CO-algebra}. 

The BV Frobenius property \eqref{eq:BVFrobenius-relation} can be equivalently rewritten as
\begin{equation} \label{eq:BVFrobenius-counit}
\eps\mu(\Delta\otimes 1)=\eps\mu(1\otimes\Delta),
\end{equation}
as can be seen by a direct computation using \eqref{eq:Frobenius_with_epsilon}.

\begin{remark} Let $(A, \mu, \lambda, \eta, \eps, \Delta)$ be a BV Frobenius algebra. Then $(A,\lambda,\eps,\Delta)$ is a BV coalgebra, i.e., the 7-term relation for $\Delta$ and $\lambda$ is satisfied. See Lemma~\ref{lem:BVFrob-BVui} below. Also, the notion is self-dual in the sense that $(A^\vee[-|\lambda|], \lambda^\vee, \mu^\vee, \eps^\vee, \eta^\vee, \Delta^\vee)$ is again a BV Frobenius algebra. See also~\cite[Lemma~4.8(i)]{CO-algebra}.
\end{remark}

\begin{remark}
When $A$ is infinite dimensional, the notion of a Frobenius algebra has to be understood in the context of Tate vector spaces as in~\cite{CO-algebra,CO-Tate}. See Appendix~\ref{sec:Tate}. 
\end{remark}

\subsection{Relations between the two structures}
 
The next result is the BV analogue of~\cite[Proposition~4.2]{CO-algebra}.

\begin{lemma} \label{lem:BVFrob-BVui}
Every odd BV Frobenius algebra is an odd BV unital infinitesimal bialgebra.
\end{lemma}
\begin{proof}
The basic properties of the operations are the same in both cases, and the unital infinitesimal relation follows directly from the Frobenius relation \eqref{eq:Frobenius} above. This means that it only remains to deduce the 7-term relation for $\Delta$ and $\lambda$ and the 11-term relation for $\Delta$, $\mu$, $\lambda$ and $\eta$ from the given properties.
Recall that the 7-term identity for $\Delta$ and $\mu$ has the form
\begin{equation}\label{eq:7_term_in_proof}
\Delta \mu(\mu \otimes 1) = \left[ \mu (\Delta\mu \otimes 1) - \mu(\mu \otimes 1)(\Delta \otimes 1 \otimes 1)\right](1+\sigma + \sigma^2).
\end{equation}

For the proof of the 7-term identity for $\Delta$ and $\lambda$, we precompose this identity with $\phi:=(1 \otimes 1 \otimes \lambda \eta \otimes 1)(1 \otimes \lambda \eta)$. Relatively straightforward computations using \eqref{eq:Frobenius} and \eqref{eq:Frobenius_with_eta} show that
\begin{align*}
([\Delta \mu (\mu \otimes 1)] \otimes 1 \otimes 1)\phi
&=(\Delta \otimes 1 \otimes 1) (\lambda \otimes 1) \lambda\\
([\mu (\Delta\mu \otimes 1)] \otimes 1 \otimes 1)\phi
&= -(\lambda \Delta \otimes 1) \lambda\\
([\mu (\Delta\mu \otimes 1) \sigma] \otimes 1 \otimes 1) \phi
&= - \sigma^2(\lambda \Delta \otimes 1)\lambda\\
([\mu  (\Delta\mu \otimes 1) \sigma^2] \otimes 1 \otimes 1)\phi
&= -  \sigma(\lambda \Delta \otimes 1)\lambda\\
([\mu (\mu \otimes 1) (\Delta \otimes 1 \otimes 1)] \otimes 1 \otimes 1) \phi
&= (\lambda \otimes 1) \lambda \Delta\\
([\mu (\mu \otimes 1) (\Delta \otimes 1 \otimes 1)\sigma] \otimes 1 \otimes 1) \phi
&= \sigma(\Delta \otimes 1 \otimes 1)(\lambda \otimes 1) \lambda \quad\text{and}\\
([\mu (\mu \otimes 1) (\Delta \otimes 1 \otimes 1) \sigma^2]  \otimes 1 \otimes 1) \phi
&= \sigma^2(\Delta \otimes 1 \otimes 1)(\lambda \otimes 1) \lambda.
\end{align*}
Combining these terms with the signs according to \eqref{eq:7_term_in_proof} yields
\begin{align*}
(\Delta \otimes 1 \otimes 1) (\lambda \otimes 1)\lambda
&= -(1+\sigma+\sigma^2)[(\lambda \Delta \otimes 1)\lambda]\\
&\phantom{==} - (\lambda \otimes 1)\lambda \Delta - (\sigma+\sigma^2)(\Delta \otimes 1 \otimes 1)(\lambda \otimes 1)\lambda,
\end{align*}
which is a rearrangement of the 7-term relation for $\Delta$ and $\lambda$ as stated in \eqref{eq:7_term_for_Delta_and_lambda}.

For the proof of the 11-term identity, we observe that by \eqref{eq:Frobenius_with_eta} the last 4 terms in the 11-term identity cancel in pairs, and so the identity to be proven reduces to
\begin{align*}
\lambda \Delta \mu
&=\lambda \mu (\Delta \otimes 1) + \lambda \mu (1 \otimes \Delta)
- (\Delta \otimes 1)\lambda \mu  - (1 \otimes \Delta) \lambda \mu \\
&\phantom{=}\,\, + (\mu \otimes 1)(1 \otimes \Delta \otimes 1)(1 \otimes \lambda)(1+\tau)
\end{align*}
To deduce this version, we precompose the 7-term identity \eqref{eq:7_term_in_proof}
with  $\psi:=(1 \otimes 1 \otimes \lambda \eta)$. This time, straightforward computations give
\begin{align*}
(\Delta \mu \otimes 1) (\mu \otimes 1 \otimes 1) \psi
&= (\Delta \otimes 1)\lambda  \mu\\
(\mu \otimes 1) (\Delta\mu \otimes 1 \otimes 1) \psi
&= - \lambda \Delta \mu\\
(\mu \otimes 1) (\Delta\mu \otimes 1 \otimes 1) (\sigma \otimes 1) \psi
=  (\mu \otimes 1) &(1 \otimes \Delta \otimes 1) (1 \otimes \lambda) \tau\\
(\mu \otimes 1) (\Delta\mu \otimes 1 \otimes 1) (\sigma^2 \otimes 1) \psi
=  (\mu \otimes 1) &(1 \otimes \Delta \otimes 1) (1 \otimes \lambda)\\
(\mu \otimes 1) (\mu \otimes 1 \otimes 1) (\Delta \otimes 1 \otimes 1 \otimes 1) \psi
&= - \lambda \mu (\Delta \otimes 1)\\
(\mu \otimes 1) (\mu \otimes 1 \otimes 1) (\Delta \otimes 1 \otimes 1 \otimes 1) (\sigma \otimes 1) \psi
&= (1 \otimes \Delta) \lambda \mu \quad\text{and}\\
(\mu \otimes 1) (\mu \otimes 1 \otimes 1) (\Delta \otimes 1 \otimes 1 \otimes 1)(\sigma^2 \otimes 1) \psi
&= - \lambda \mu (1 \otimes \Delta).
\end{align*}
Combining these terms according to the signs in \eqref{eq:7_term_in_proof}, we get
\begin{align*}
(\Delta \otimes 1) \lambda \mu
&= -\lambda \Delta \mu + (\mu \otimes 1) (1 \otimes \Delta \otimes 1) (1 \otimes \lambda) \tau \\
&\phantom{=+}  + (\mu \otimes 1) (1 \otimes \Delta \otimes 1) (1 \otimes \lambda) \\
&\phantom{=+} + \lambda\mu(\Delta \otimes 1) + \lambda \mu (1 \otimes \Delta) -(1 \otimes \Delta) \lambda \mu.
\end{align*}
Again, this is a rearrangement of the relation we wanted to prove. 
\end{proof}

The following result will be proved in Appendix~\ref{app:BVui-BVTateFrob}.

\begin{proposition} \label{prop:BVui-BVFrob}
If  $(A, \mu, \eta, \lambda, \Delta)$ is an odd BV unital infinitesimal bialgebra satisfying $\lambda \eta =0$, there is a canonical construction of an odd BV Frobenius algebra structure on $D(A):=A \oplus A^\vee[-|\lambda|]$ such that $(A,\mu, \eta)$ is a unital subalgebra and $(A^\vee[-|\lambda|], \lambda^\vee)$ is also a subalgebra, and such that the pairing is the canonical one.
\end{proposition}

\begin{remark}
This is a graded version of the Drinfel'd double construction, first discussed in~\cite{Drinfeld-quantum-groups}, see also Zhelyabin~\cite{Zhelyabin1997},  Aguiar~\cite{Aguiar2001} and Bai~\cite{Bai2010}. In the absence of the BV structure, Proposition~\ref{prop:BVui-BVFrob} answers a question left open in~\cite{CO-algebra} and echoes the constructions from~\cite[\S7]{CO-cones}.
\end{remark}

\begin{remark}
When $A$ is infinite dimensional, the statement of Proposition~\ref{prop:BVui-BVFrob} makes sense in the context of Tate vector spaces, see Appendix~\ref{sec:Tate}. The same is true for the statement of Lemma~\ref{lem:BVFrob-BVui}.
\end{remark}

\begin{remark}
We expect the definition of the Drinfel'd double to generalize to BV unital infinitesimal bialgebras such that $\lambda\eta$ is not necessarily zero. The formulas should be akin to those of~\cite[Theorem~7.8]{CO-reduced}. See also Conjecture~\ref{conj:H-doubles}(iii) and its discussion.
\end{remark} 

\begin{remark} The data of a BV algebra is equivalent to that of a Gerstenhaber algebra in $H_*(S^1)$-modules~\cite[Theorem~6.1]{Salvatore-Wahl}. This characterization goes through the description of the framed little 2-disc operad as a semi-direct product of the little 2-disc operad and the circle~\cite{Getzler-BV}. An anonymous referee has suggested that the data of a BV infinitesimal bialgebra might be equivalent to the data of (some kind of) an infinitesimal bialgebra in $H_*(S^1)$-modules. This question is related to Remark~\ref{rem:no-intermediate1} 
and an answer would involve a finer understanding of the corresponding dioperads. At the time of writing, we do not know whether such a description exists.
\end{remark}

\section{Moduli spaces of spheres}\label{sec:moduli}

The fact that moduli spaces of spheres with marked points and marked tangent directions play a fundamental role in constructing certain algebraic structures has been known for at least 30 years, as can be seen from the work of Kimura, Stasheff and Voronov~\cite{Kimura_Stasheff_Voronov_1995} and the references therein, 
in particular Getzler~\cite{Getzler-BV}.
These moduli spaces and their higher genus analogues underlie real enumerative geometry~\cite{LiuMelissa}, symplectic field theory~\cite{EGH}, Floer theory~\cite{Seidel07,Seidel-disjoinable,Ekholm-Oancea,Abouzaid-Groman-Varolgunes} and string topology~\cite{Sullivan-background}.

In the next two subsections, we will review some of the relevant work, mostly for motivation and to set up notation for what will come later. The third subsection contains a first (fairly benign) extension obtained by adding weights. Finally, in the last subsection, we discuss the special case of two inputs and two weighted outputs, which marks a departure from the strictly operadic world. It is clear that our discussion can be embedded into the more general framework of dioperads, but this is beyond the scope of this paper. 

\subsection{Moduli spaces with several inputs and one output} \label{ssec:moduli_kto1}

We denote by $\Ccal_{k,1}$ the moduli space of spheres with $k+1$ marked points, called punctures in the following, with one of them designated as ``output'' and the other $k$ considered as ``inputs'' and labelled by $1$, \dots, $k$, with its {\em real} compactification. This means that we add all stable configurations of nodal spheres where at each node we include a complex anti-linear identification of the tangent spaces 
at the two corresponding points, considered up to positive homothety. 
Readers more used to the Deligne-Mumford compactification should think of a real blow up of all boundary divisors. In particular, $\Ccal_{k,1}$ is a compact manifold with boundary and corners of real dimension $2(k-2)$.

Next, we denote by $\Ncal_{k,1}$ the moduli space of stable configurations as above with the additional data of a tangent direction at the output, called \emph{marker}. Thus $\Ncal_{k,1}$ is a principal $S^1$-bundle over $\Ccal_{k,1}$.

\begin{remark}
We will repeatedly make use of the operation of ``transporting'' a marker from an output puncture to some input puncture or vice versa, which we define as follows: Two points on the Riemann sphere together with a tangent direction at one of them determine a unique oriented circle, and the induced tangent direction at the other point is defined to be tangent to that circle, giving it the opposite orientation from the original one. 

In a model which conformally identifies the twice punctured sphere with the cylinder $\R \times S^1$, the given marker at one end picks out a line $\R \times \{pt\}$, and the marker at the other end is simply chosen tangent to this same line. The angular information at the nodes allows transfer between components and so makes this well-defined even for stable curves in the compactification.
\end{remark}

The bundle $\Ncal_{k,1} \to \Ccal_{k,1}$ admits sections, and for what follows we make the (non-canonical) choice to define $\zeta:\Ccal_{k,1} \to \Ncal_{k,1}$ 
by letting the tangent marker at the output of a given stable configuration be the unique direction tangent to the circle passing through the output and the first two inputs and pointing towards the first input. 
Again, the angular information at the nodes in our compactification make this section well-defined. In any case, it allows us to identify $\Ncal_{k,1} \cong \Ccal_{k,1} \times S^1$.

In contrast to the spaces $\Ccal_{k,1}$, the spaces $\Ncal_{k,1}$ form a topological operad, simply by formal composition creating a new node. Indeed, the required complex anti-linear identification of tangent spaces at the new node is obtained as follows: on the configuration for which the nodal point is an input, we induce a marker there by transporting the marker from the output as described above, and, up to positive homothety, there is a unique complex anti-linear identification of tangent spaces that matches this marker with the output marker on the other side.

Finally, we denote by $\Mcal_{k,1}$ the moduli spaces of stable configurations as above with (independent) marked tangent directions {\em at all points}. 
The space $\Mcal_{k,1}$ forms a principal $(S^1)^k$-bundle over $\Ncal_{k,1}$. This bundle admits a ``canonical" section $\xi:\Ncal_{k,1} \to \Mcal_{k,1}$ by transporting the marker from the output to all of the inputs as above, giving rise to a trivialization $\Mcal_{k,1}\cong \Ncal_{k,1}\times (S^1)^k$.

To summarize, we have the following tower of trivial fibrations with structure groups $S^1$, resp. $(S^1)^k$. The section $\xi$ is canonical and the section $\zeta$ depends on the choice of two distinct ordered input punctures. 

$$
\xymatrix{
\Mcal_{k,1}\ar[d]_{(S^1)^k} \\ \Ncal_{k,1} \ar[d]_{S^1} \ar@/_2pc/[u]_\xi \\ \Ccal_{k,1} \ar@/_2pc/[u]_\zeta
}
$$

\begin{remark}
The spaces $\Ncal_{k,1}$ form a model for the topological operad 
of (unframed) little 2-disks, and the spaces $\Mcal_{k,1}$ form a model for the topological operad of framed little 2-disks. 
Moreover, the maps $\xi:\Ncal_{k,1} \to \Mcal_{k,1}$ assemble into a map of operads.
\end{remark}

\subsection{Moduli spaces with 2 or 3 inputs and one output} \label{ssec:moduli_2or3to1}

We now discuss in detail the cases $k=2$ and $k=3$, which are used in this paper. 
When the various (chains in) moduli spaces appearing below are used to construct operations (as we will do in section \ref{sec:reduced} and \ref{sec:RFH} for Floer theory), the story that we tell in this subsection has the following interpretation: 
\begin{itemize}
\item The generators $P$ in degree 0 and $Q$ in degree 1 for $H_*(\Ncal_{2,1})$ give rise to the product $\mu$ and the bracket $\beta$.
\item The homology of each of the boundary components of $\Ncal_{3,1}$ can be described in terms of products of these cycles. In the implementation, such products parametrize compositions of the corresponding operations.
\item The Poisson relation~\eqref{eq:Poisson} is a consequence of the (obvious) fact that $\zeta(\p\Ccal_{3,1})$ is null-homologous in $\Ncal_{3,1}$ (Corollary~\ref{cor:Poisson_moduli}). 
\item The Jacobi relation~\eqref{eq:Jacobi-graded} is a consequence of the (obvious) fact that the sum of the three boundary components of $\Ncal_{3,1}$ is nullhomologous (Corollary~\ref{cor:Bi_as_products}).
\item The 7-term relation~\eqref{eq:7_term_for_Delta_and_mu} is a consequence of the (obvious) fact that $\xi\zeta(\p\Ccal_{3,1})$ is null-homologous in $\Mcal_{3,1}$ (Corollary~\ref{cor:cycles_7term}).
\end{itemize}

The space $\Ccal_{2,1}$ consists of a single point, the space $\Ncal_{2,1}$ is a circle and $\Mcal_{2,1}$ is a 3-torus. Note that the $\Z_2$-action of switching the labels at the inputs is free on $\Ncal_{2,1}$ and $\Mcal_{2,1}$, corresponding to the antipodal map on $S^1$ and its lift to $(S^1)^3$, respectively.

The homology of $\Ncal_{2,1}$ is generated by the class of a point 
$$
P := \zeta(\Ccal_{2,1})
$$ 
in degree $0$ and the class of the fundamental chain 
$$
Q := \Ncal_{2,1}
$$ 
in degree $1$. For definiteness, we orient $\Ncal_{2,1}$ by rotating the marker at the output counterclockwise. Note that the action of $\Z_2$ on $\Ncal_{2,1}$ by switching the input labels is trivial on homology, i.e., $[\tau P]=[P]$ and $[\tau Q] = [Q]$.

Similarly, the homology of $\Mcal_{2,1}$ in degree $0$ is generated by the class of $\hat P:=\xi(P)$, and in degree 1 there are three chains $D_1$, $D_2$ and $D_{\out}$ which correspond to rotation of the marker at the corresponding puncture, again counterclockwise. With these conventions and our convention of how markers are coupled, we have the relation
\begin{equation}\label{eq:lift_of_Q}
\xi_*[Q] = [D_\out] -[D_1] -[D_2]
\end{equation}
in $H_1(\Mcal_{2,1})$. In degree $2$ the homology is generated by pairwise products of the 1-dimensional cycles, and in degree $3$ we have the fundamental chain $D_\out \times D_1 \times D_2$.

Next we turn to $k=3$. The space $\Ccal_{3,1}$ is the real blow up of a 2-sphere at three points (corresponding to the three stable configurations consisting of two components each carrying two of the marked points), so we identify it with a 2-sphere with three disjoint open disks removed (i.e., a pair of pants with boundary). We label the boundary components of $\Ccal_{3,1}$ by $b_1$, $b_2$ and $b_3$, respectively, according to which input is on the same component as the output. Being a complex manifold, the interior of $\Ccal_{3,1}$ carries a preferred orientation, and we orient the boundary components accordingly. 

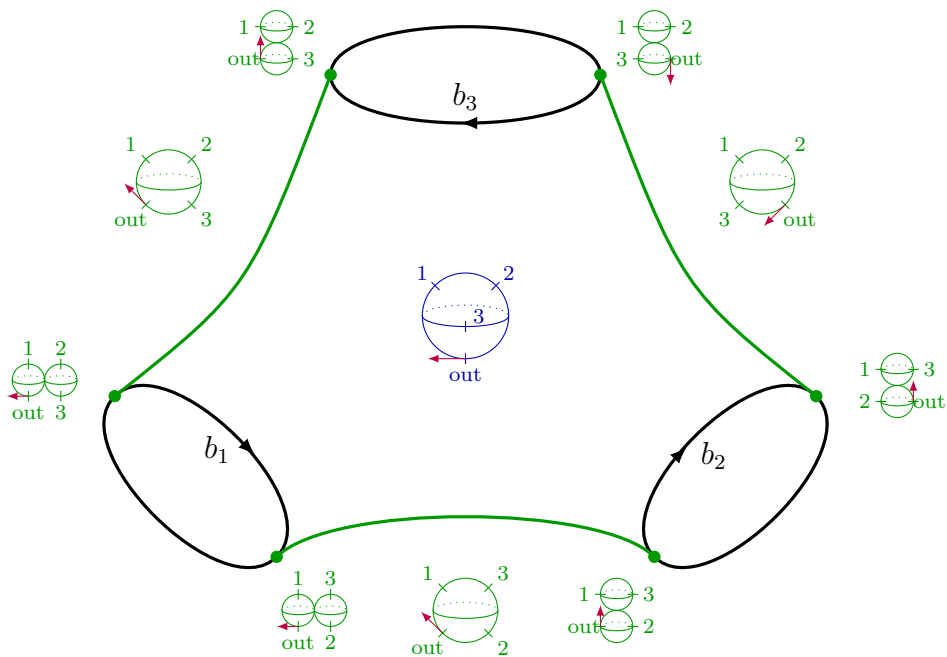
\begin{figure}
\begin{center}
\begin{tikzpicture}[scale=.71]
\definecolor{darkgreen}{rgb}{0,0.6,0};
\definecolor{darkblue}{rgb}{0,0,0.7};

\draw[very thick] (0,0) .. controls (0,1.2) and (5,1.2) .. (5,0) .. controls (5,-1.2) and (0,-1.2) .. (0,0);
\draw[very thick] (-4,-6) .. controls (-3,-5) and (0,-8) .. (-1,-9) .. controls (-2,-10) and (-5,-7) .. (-4,-6);
\draw[very thick] (6,-9) .. controls (5,-8) and (8,-5) .. (9,-6) .. controls (10,-7) and (7,-10) .. (6,-9);


\draw[very thick, -latex] (2.5,-.9) -- (2.4,-.9);
\draw[very thick, -latex] (-1.5,-7) -- (-1.4,-7.1);
\draw[very thick, -latex] (6.5,-7.05) -- (6.6,-6.95);


\node at (2.5,-0.4) {$b_3$};
\node at (-2.1,-7) {$b_1$};
\node at (7.1,-7.1) {$b_2$};

\draw[darkgreen, very thick] (0,0) .. controls (-1.5,-4) .. (-4,-6);
\draw[darkgreen, very thick] (5,0) .. controls (6.5,-4) .. (9,-6);
\draw[darkgreen, very thick] (-1,-9) .. controls (0,-8) and (5,-8) .. (6,-9);


\draw[darkgreen] (-3,-2) circle (0.6);
\draw[darkgreen, dotted] (-3.6,-2) arc (180:0: 0.6 and 0.15);
\draw[darkgreen] (-3.6,-2) arc (180:360: 0.6 and 0.15);
\draw[darkgreen] (-3.5,-1.5) -- (-3.35,-1.65);
\node[darkgreen] at (-3.7, -1.3) {\tiny$1$};
\draw[darkgreen] (-2.5,-1.5) -- (-2.65,-1.65);
\node[darkgreen] at (-2.3, -1.3) {\tiny$2$};
\draw[darkgreen] (-3.5,-2.5) -- (-3.35,-2.35);
\node[darkgreen] at (-3.7, -2.7) {\tiny$\out$};
\draw[purple, -latex] (-3.425,-2.425) -- (-3.825,-2.025);
\draw[darkgreen] (-2.5,-2.5) -- (-2.65,-2.35);
\node[darkgreen] at (-2.3, -2.7) {\tiny$3$};

\draw[darkgreen] (8,-2) circle (0.6);
\draw[darkgreen, dotted] (7.4,-2) arc (180:0: 0.6 and 0.15);
\draw[darkgreen] (7.4,-2) arc (180:360: 0.6 and 0.15);
\draw[darkgreen] (7.5,-1.5) -- (7.65,-1.65);
\node[darkgreen] at (7.3, -1.3) {\tiny$1$};
\draw[darkgreen] (8.5,-1.5) -- (8.35,-1.65);
\node[darkgreen] at (8.7, -1.3) {\tiny$2$};
\draw[darkgreen] (8.5,-2.5) -- (8.35,-2.35);
\node[darkgreen] at (8.7, -2.7) {\tiny$\out$};
\draw[purple, -latex] (8.425,-2.425) -- (8.025,-2.825);
\draw[darkgreen] (7.5,-2.5) -- (7.65,-2.35);
\node[darkgreen] at (7.3, -2.7) {\tiny$3$};

\draw[darkgreen] (2.5,-10) circle (0.6);
\draw[darkgreen, dotted] (1.9,-10) arc (180:0: 0.6 and 0.15);
\draw[darkgreen] (1.9,-10) arc (180:360: 0.6 and 0.15);
\draw[darkgreen] (2,-9.5) -- (2.15,-9.65);
\node[darkgreen] at (1.8, -9.3) {\tiny$1$};
\draw[darkgreen] (2,-10.5) -- (2.15,-10.35);
\node[darkgreen] at (1.8, -10.7) {\tiny$\out$};
\draw[purple, -latex] (2.075,-10.425) -- (1.675,-10.025);
\draw[darkgreen] (3,-10.5) -- (2.85,-10.35);
\node[darkgreen] at (3.2, -10.7) {\tiny$2$};
\draw[darkgreen] (3,-9.5) -- (2.85,-9.65);
\node[darkgreen] at (3.2, -9.3) {\tiny$3$};


\draw[darkblue] (2.5,-4.5) circle (0.8);
\draw[darkblue, dotted] (1.7,-4.5) arc (180:0: 0.8 and 0.2);
\draw[darkblue] (1.7,-4.5) arc (180:360: 0.8 and 0.2);
\draw[darkblue] (1.87,-3.87) -- (2.02,-4.02);
\node[darkblue] at (1.7, -3.7) {\tiny$1$};
\draw[darkblue] (2.98,-4.02) -- (3.13,-3.87);
\node[darkblue] at (3.3, -3.7) {\tiny$2$};
\draw[darkblue] (2.5, -4.6) -- (2.5, -4.8);
\node[darkblue] at (2.75, -4.5) {\tiny$3$};
\draw[darkblue] (2.5,-5.2) -- (2.5,-5.4);
\node[darkblue] at (2.5, -5.6) {\tiny$\out$};
\draw[purple, -latex] (2.5,-5.3) -- (1.8,-5.3);

\node[circle,darkgreen,draw,fill,inner sep=1.5pt] at (0,0) {};
\node[circle,darkgreen,draw,fill,inner sep=1.5pt] at (5,0) {};
\node[circle,darkgreen,draw,fill,inner sep=1.5pt] at (9,-6) {};
\node[circle,darkgreen,draw,fill,inner sep=1.5pt] at (6,-9) {};
\node[circle,darkgreen,draw,fill,inner sep=1.5pt] at (-1,-9) {};
\node[circle,darkgreen,draw,fill,inner sep=1.5pt] at (-4,-6) {};


\draw[darkgreen] (-1,.9) circle (0.3);
\draw[darkgreen, dotted] (-1.3,0.9) arc (180:0: 0.3 and 0.08);
\draw[darkgreen] (-1.3,0.9) arc (180:360: 0.3 and 0.08);
\draw[darkgreen] (-1,.3) circle (0.3);
\draw[darkgreen, dotted] (-1.3,0.3) arc (180:0: 0.3 and 0.08);
\draw[darkgreen] (-1.3,0.3) arc (180:360: 0.3 and 0.08);
\draw[darkgreen] (-1.4,.9 ) -- (-1.2,.9);
\node[darkgreen] at (-1.6,.9) {\tiny$1$};
\draw[darkgreen] (-.8,.9 ) -- (-.6,.9);
\node[darkgreen] at (-.4,.9) {\tiny$2$};
\draw[darkgreen] (-1.4,.3 ) -- (-1.2,.3);
\node[darkgreen] at (-1.6,.3) {\tiny$\out$};
\draw[purple, -latex] (-1.3,0.3) -- (-1.3,0.75);
\draw[darkgreen] (-.8,.3 ) -- (-.6,.3);
\node[darkgreen] at (-.4,.3) {\tiny$3$};

\draw[darkgreen] (-5.6,-5.7) circle (0.3);
\draw[darkgreen, dotted] (-5.9,-5.7) arc (180:0: 0.3 and 0.08);
\draw[darkgreen] (-5.9,-5.7) arc (180:360: 0.3 and 0.08);
\draw[darkgreen] (-5,-5.7) circle (0.3);
\draw[darkgreen, dotted] (-5.3,-5.7) arc (180:0: 0.3 and 0.08);
\draw[darkgreen] (-5.3,-5.7) arc (180:360: 0.3 and 0.08);
\draw[darkgreen] (-5.6,-5.5 ) -- (-5.6,-5.3);
\node[darkgreen] at (-5.6,-5.1) {\tiny$1$};
\draw[darkgreen] (-5.6,-5.9 ) -- (-5.6,-6.1);
\node[darkgreen] at (-5.6,-6.3) {\tiny$\out$};
\draw[purple, -latex] (-5.6,-6) -- (-6,-6);
\draw[darkgreen] (-5,-5.5 ) -- (-5,-5.3);
\node[darkgreen] at (-5,-5.1) {\tiny$2$};
\draw[darkgreen] (-5,-5.9 ) -- (-5,-6.1);
\node[darkgreen] at (-5,-6.3) {\tiny$3$};

\draw[darkgreen] (-.6,-10) circle (0.3);
\draw[darkgreen, dotted] (-.9,-10) arc (180:0: 0.3 and 0.08);
\draw[darkgreen] (-.9,-10) arc (180:360: 0.3 and 0.08);
\draw[darkgreen] (0,-10) circle (0.3);
\draw[darkgreen, dotted] (-.3,-10) arc (180:0: 0.3 and 0.08);
\draw[darkgreen] (-.3,-10) arc (180:360: 0.3 and 0.08);
\draw[darkgreen] (-.6,-9.8 ) -- (-.6,-9.6);
\node[darkgreen] at (-.6,-9.4) {\tiny$1$};
\draw[darkgreen] (-.6,-10.2 ) -- (-.6,-10.4);
\node[darkgreen] at (-.6,-10.6) {\tiny$\out$};
\draw[purple, -latex] (-.6,-10.3) -- (-1,-10.3);
\draw[darkgreen] (0,-9.8 ) -- (0,-9.6);
\node[darkgreen] at (0,-9.4) {\tiny$3$};
\draw[darkgreen] (0,-10.2 ) -- (0,-10.4);
\node[darkgreen] at (0,-10.6) {\tiny$2$};

\draw[darkgreen] (5.3,-9.7) circle (0.3);
\draw[darkgreen, dotted] (5,-9.7) arc (180:0: 0.3 and 0.08);
\draw[darkgreen] (5,-9.7) arc (180:360: 0.3 and 0.08);
\draw[darkgreen] (5.3,-10.3) circle (0.3);
\draw[darkgreen, dotted] (5,-10.3) arc (180:0: 0.3 and 0.08);
\draw[darkgreen] (5,-10.3) arc (180:360: 0.3 and 0.08);
\draw[darkgreen] (4.9,-9.7) -- (5.1,-9.7);
\node[darkgreen] at (4.7,-9.7) {\tiny$1$};
\draw[darkgreen] (5.5,-9.7) -- (5.7,-9.7);
\node[darkgreen] at (5.9,-9.7) {\tiny$3$};
\draw[darkgreen] (4.9,-10.3 ) -- (5.1,-10.3);
\node[darkgreen] at (4.7,-10.3) {\tiny$\out$};
\draw[purple, -latex] (5,-10.3) -- (5,-9.9);
\draw[darkgreen] (5.5,-10.3 ) -- (5.7,-10.3);
\node[darkgreen] at (5.9,-10.3) {\tiny$2$};

\draw[darkgreen] (10.5,-5.5) circle (0.3);
\draw[darkgreen, dotted] (10.2,-5.5) arc (180:0: 0.3 and 0.08);
\draw[darkgreen] (10.2,-5.5) arc (180:360: 0.3 and 0.08);
\draw[darkgreen] (10.5,-6.1) circle (0.3);
\draw[darkgreen, dotted] (10.2,-6.1) arc (180:0: 0.3 and 0.08);
\draw[darkgreen] (10.2,-6.1) arc (180:360: 0.3 and 0.08);
\draw[darkgreen] (10.1,-5.5) -- (10.3,-5.5);
\node[darkgreen] at (9.9,-5.5) {\tiny$1$};
\draw[darkgreen] (10.7,-5.5) -- (10.9,-5.5);
\node[darkgreen] at (11.1,-5.5) {\tiny$3$};
\draw[darkgreen] (10.1,-6.1 ) -- (10.3,-6.1);
\node[darkgreen] at (9.9,-6.1) {\tiny$2$};
\draw[darkgreen] (10.7,-6.1 ) -- (10.9,-6.1);
\node[darkgreen] at (11.1,-6.1) {\tiny$\out$};
\draw[purple, -latex] (10.8,-6.1) -- (10.8,-5.7);

\draw[darkgreen] (6,.9) circle (0.3);
\draw[darkgreen, dotted] (5.7,0.9) arc (180:0: 0.3 and 0.08);
\draw[darkgreen] (5.7,0.9) arc (180:360: 0.3 and 0.08);
\draw[darkgreen] (6,.3) circle (0.3);
\draw[darkgreen, dotted] (5.7,0.3) arc (180:0: 0.3 and 0.08);
\draw[darkgreen] (5.7,0.3) arc (180:360: 0.3 and 0.08);
\draw[darkgreen] (5.6,.9 ) -- (5.8,.9);
\node[darkgreen] at (5.4,.9) {\tiny$1$};
\draw[darkgreen] (6.2,.9 ) -- (6.4,.9);
\node[darkgreen] at (6.6,.9) {\tiny$2$};
\draw[darkgreen] (5.6,.3 ) -- (5.8,.3);
\node[darkgreen] at (5.4,.3) {\tiny$3$};
\draw[darkgreen] (6.2,.3 ) -- (6.4,.3);
\node[darkgreen] at (6.6,.3) {\tiny$\out$};
\draw[purple, -latex] (6.3,0.3) -- (6.3,-0.2);

\end{tikzpicture}
\caption{{\bf The space $\Ccal_{3,1}$.} The arrow at the output puncture signifies the image of our chosen section $\zeta:\Ccal_{3,1} \to \Ncal_{3,1}$. The three green edges represent the configurations where all four points lie on a common circle. \label{fig:C_3_1}
}
\end{center}
\end{figure}

Note that, upon moving in the positive direction along a boundary component, 
the circles defined by the three special points on each component turn clockwise relative to each other at the node. An argument is explained in Figure~\ref{fig:BV-upper}.
\begin{figure}
\begin{center}
\begin{tikzpicture}[scale=.45]
\definecolor{darkgreen}{rgb}{0,0.6,0};
\definecolor{darkblue}{rgb}{0,0,0.7};
\definecolor{darkred}{rgb}{0.7,0,0};

\draw[very thick] (0,0) .. controls (0,1.2) and (5,1.2) .. (5,0) .. controls (5,-1.2) and (0,-1.2) .. (0,0);
\draw[very thick] (-4,-6) .. controls (-3,-5) and (0,-8) .. (-1,-9) .. controls (-2,-10) and (-5,-7) .. (-4,-6);
\draw[very thick] (6,-9) .. controls (5,-8) and (8,-5) .. (9,-6) .. controls (10,-7) and (7,-10) .. (6,-9);


\draw[very thick, -latex] (2.5,-.9) -- (2.4,-.9);
\draw[very thick, -latex] (-1.5,-7) -- (-1.4,-7.1);
\draw[very thick, -latex] (6.5,-7.05) -- (6.6,-6.95);


\node at (2.5,-0.2) {$b_3$};
\node at (-2.1,-7.2) {$b_1$};
\node at (7.1,-7.2) {$b_2$};

\draw[very thick] (0,0) .. controls (-1.5,-4) .. (-4,-6);
\draw[very thick] (5,0) .. controls (6.5,-4) .. (9,-6);
\draw[very thick] (-1,-9) .. controls (0,-8) and (5,-8) .. (6,-9);

\draw[red, very thick] (0.2,-0.7) .. controls (-0.4,-2) .. (-1,-3.3);
\draw[red, very thick] (4.8,-0.7) .. controls (5.4,-2) .. (6,-3.3);
\draw[red, very thick] (4.8,-0.7) .. controls (3.8,-1.5) and (1.2,-1.5) .. (0.2,-0.7);
\draw[red, very thick, -latex] (2.4,-1.315) -- (2.3,-1.315);


\draw[darkblue] (2.5,-4.5) circle (0.8);
\draw[darkblue, dotted] (1.7,-4.5) arc (180:0: 0.8 and 0.2);
\draw[darkblue] (1.7,-4.5) arc (180:360: 0.8 and 0.2);
\draw[darkblue] (1.87,-3.87) -- (2.02,-4.02);
\node[darkblue] at (1.6, -3.6) {\tiny$1$};
\draw[darkblue] (2.98,-4.02) -- (3.13,-3.87);
\node[darkblue] at (3.4, -3.6) {\tiny$2$};
\draw[darkblue] (2.5, -4.6) -- (2.5, -4.8);
\node[darkblue] at (2.75, -4.4) {\tiny$3$};
\draw[darkblue] (2.5,-5.2) -- (2.5,-5.4);
\node[darkblue] at (2.5, -5.6) {\tiny$\out$};
\draw[blue, -latex] (2.5,-5.3) -- (1.8,-5.3);

\node at (-3,-2) {$\Ccal_{3,1}$};


\pgftransformxshift{11cm}

\draw[very thick] (4.5,-4.5) circle (4);
\draw[thick, dotted] (0.5,-4.5) arc (180:0: 4 and 1);
\draw[thick] (0.5,-4.5) arc (180:360: 4 and 1);


\draw[thick] (4.5,-3.7) -- (4.5,-3.3);
\node at (4.5,-4) {\tiny $1$};
\draw[thick] (8,-4.8) -- (8,-5.2);
\node at (8,-5.5) {\tiny $2$};
\draw[thick] (4.5,-5.3) -- (4.5,-5.7);
\node at (4.5,-6) {\tiny $\out$};
\draw[blue, very thick, -latex] (4.5,-5.5) -- (3,-5.5);


\draw[darkred, very thick] (2.5,-5) .. controls (4.5,-5.2)  .. (6.5,-5);
\draw[darkred, very thick, -latex] (4.4,-5.15) -- (4.6,-5.15);

\draw[darkgreen,  thick, dotted] (4.5,-5.5) .. controls (2.6,0.5)  .. (4.5,-3.5);
\draw[darkgreen,  thick, dotted] (4.5,-5.5) .. controls (6.4,0.5)  .. (4.5,-3.5);
\draw[darkgreen, -latex] (3.37,-0.31) arc (105:75: 4.3);

\end{tikzpicture}

\caption{As one moves along the red curve in $\Ccal_{3,1}$, the input $3$ moves along the dark red path in the sphere with the other three points fixed. So the circle through the output and inputs 3 and (say) 1 rotates clockwise with respect to the given marker at the output.} \label{fig:BV-upper}
\end{center}
\end{figure}
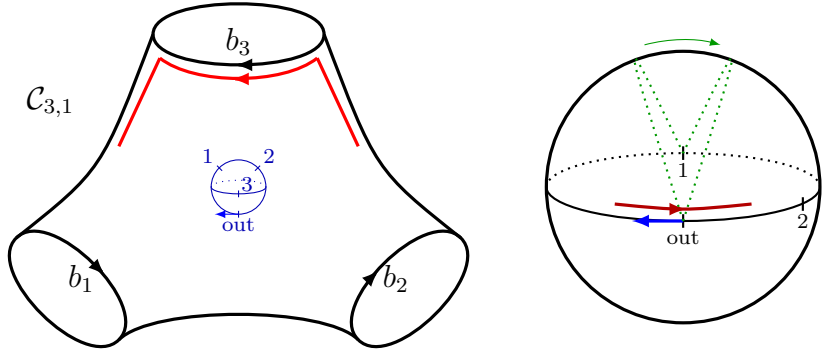

As we saw above, the sections $\zeta$ and $\xi$ give rise to identifications $\Ncal_{3,1}\cong \Ccal_{3,1} \times S^1$ and $\Mcal_{3,1} \cong \Ccal_{3,1} \times (S^1)^4$. We denote the boundary components of $\Ncal_{3,1}$ by $B_1$, $B_2$ and $B_3$, where $B_i=\pi^{-1}(b_i)\cong b_i \times S^1$. The homology of $\Ncal_{3,1}$ can be represented completely by using cycles in the boundary. Because each of the boundary components of $\Ncal_{3,1}$ can be described as a suitable product of two copies of $\Ncal_{2,1}$, any cycle in the boundary of $\Ncal_{3,1}$ admits a description as the image of a product of cycles from $\Ncal_{2,1}$. Below, we will use the notation $P$ and $Q$ introduced before for $0$- and $1$-cycles in $\Ncal_{2,1}$, but decorated with the labels of the input and output punctures of the corresponding sphere component (we use $*$ in place of the puncture corresponding to the node, and ``out'' for the final output).
For example, $P^{23}_*$ corresponds to the point class in the moduli space $\Ncal_{2,1}$ with input labels $2$ and $3$ (in that order), and output at the node. Because of the symmetries $[P^{ij}_k]=[P^{ji}_k]$ and $[Q^{ij}_k]=[Q^{ji}_k]$, the precise choice of order for the labels will not be important at this stage.
This will change for the discussion of chains in the moduli spaces with several outputs in \S\ref{ssec:moduli_1toell}.

The degree $0$ homology of the three boundary components is generated by the homology classes of the following three 0-cycles:
\begin{itemize}
\item $P^{1*}_\out \times P^{23}_*$ in $B_1$, 
\item $P^{*2}_\out \times P^{13}_*$ in $B_2$, and 
\item $P^{*3}_\out \times P^{12}_*$ in $B_3$. 
\end{itemize}
They are all homologous in $\Ncal_{3,1}$.
The homology in degree $1$ of each boundary component has rank $2$, and is generated by the homology classes of the following 1-cycles:
\begin{itemize}
\item $P^{1*}_\out \times Q^{23}_*$ and $Q^{1*}_\out \times P^{23}_*$ for $B_1$,
\item $P^{*2}_\out \times Q^{13}_*$ and $Q^{*2}_\out \times P^{13}_*$ for $B_2$, and
\item $P^{*3}_\out \times Q^{12}_*$ and $Q^{*3}_\out \times P^{12}_*$ for $B_3$.
\end{itemize}
Alternatively, the homology in degree 1 of the boundary components is also generated by the homology classes of the cycle $\zeta(b_i)$ and the circle fiber $F$ (corresponding to fixing a point in one of the boundary components $b_i$ of $\Ccal_{3,1}$ and moving the output marker counterclockwise) in terms of our chosen homology bases.

For later reference, the next couple of lemmas describe the relations between these two choices of generators.
\begin{lemma}\label{lem:bi_as_products}
With the notation introduced above, we have
\begin{align*}
\zeta_*[b_1] &= -[P^{1*}_\out \times Q^{23}_*], \\ 
\zeta_*[b_2] &= -[P^{*2}_\out \times Q^{13}_*], \qquad \text{and}\\ 
\zeta_*[b_3] &= [Q^{*3}_\out \times P^{12}_*]. 
\end{align*}
\end{lemma}

\begin{proof} 
We start with the description of $\zeta_*[b_1]$. Consider first the component containing the output and the first input. From the point of view of this component, the second input puncture sits at the node, and so the circle through the three special points does not change along $b_1$. As along $\zeta(b_1)$ the marker at the output is determined from this circle, it also does not move along this cycle. At the same time, since at the node the two circles determined by the three special points on the two components move clockwise with respect to each other, from the point of view of the other component the marker at its output (which is transported from the total output) moves clockwise relative to its special circle. Therefore $\zeta_*[b_1] = -[P^{1*}_\out \times Q^{23}_*]$ as claimed.

The second case is treated quite similarly, so we next discuss the description of $\zeta_*[b_3]$. This time, from the point of view of the component containing the first two input punctures, the output sits at the (outgoing) node and does not move, so the markers of these two input punctures remain tangent to the circle on this component along the whole cycle. In contrast, the marker at the (incoming) node of the other component moves clockwise in the positive direction of $\zeta(b_3)$, so the marker at the output moves counterclockwise relative to the special circle on this component. This explains the absence of a sign in $\zeta_*[b_3]= [Q^{*3}_\out \times P^{12}_*]$.
\end{proof}

\begin{corollary} \label{cor:Poisson_moduli}
In $H_1(\Ncal_{3,1})$, we have the relation
\begin{equation}\label{eq:cycles_poisson}
[Q^{*3}_\out \times P^{12}_*] = [P^{1*}_\out \times Q^{23}_*] + [P^{*2}_\out \times Q^{13}_*].
\end{equation}
\end{corollary}
\begin{proof}
The 2-chain $\zeta(\Ccal_{3,1})$ gives a homology between the left hand side and the right side. 
\end{proof}

\begin{lemma}
With the notation above, we have the following identity in $H_1(\Ncal_{3,1})$:
\begin{align*}
[F] &= [P^{1*}_\out \times Q^{23}_*] + [Q^{1*}_\out \times P^{23}_*] \\ 
&= [P^{*2}_\out \times Q^{13}_*] + [Q^{*2}_\out \times P^{13}_*]\\ 
&= [P^{*3}_\out \times Q^{12}_*] + [Q^{*3}_\out \times P^{12}_*]. 
\end{align*}
\end{lemma}

\begin{proof}
Along the cycle $Q^{1*}_{\out} \times P^{23}_*$, 
the marker at the output puncture turns counterclockwise, but at the same time the marker at the node, being coupled, turns {\em clockwise}. It follows that under the projection $\Ncal_{3,1} \to \Ccal_{3,1}$ this cycle projects to the boundary component $b_1$ traversed in the {\em positive} direction, and the projection to the fiber $S^1$ has degree $+1$. On the other hand, in the previous lemma we proved that
$$
\zeta_*[b_1] = -[P^{1*}_\out \times Q^{23}_*].
$$
Putting these two observations together, we obtain the first equality
$$
[F] = [P^{1*}_\out \times Q^{23}_*] + [Q^{1*}_\out \times P^{23}_*].
$$
The argument for the second presentation of $[F]$ is completely analogous, so we next consider the third presentation. Along the cycle $P^{*3}_\out \times Q^{12}_*$, the special circles of the two components move counterclockwise with respect to each other (because the identification is determined by the marker at the node of the component containing the first two input punctures, which makes this movement along $Q^{12}_*$). So this cycle projects to $b_3$ with its {\em negative} orientation. At the same time, the marker at the output section moves once counterclockwise relative to our chosen section $\zeta$, so
$$
[P^{*3}_\out \times Q^{12}_*] = [F] - \zeta_*[b_3].
$$
Combining this with the observation $\zeta_*(b_3)=[Q^{*3}_\out \times P^{12}_*]$
from the previous lemma, the claim again follows.
\end{proof}

\begin{corollary}\label{cor:Bi_as_products}
For the homology classes of the boundary components $[B_i] \in H_2(\Ncal_{3,1})$ we have
\begin{align*}
[B_1] &=  [Q^{1*}_\out \times Q^{23}_*],\\
[B_2] &=  [Q^{*2}_\out \times Q^{13}_*], \quad \text{and}\\
[B_3] &=  [Q^{*3}_\out \times Q^{12}_*].
\end{align*}
In particular, we have
\begin{equation} \label{eq:Jacobi-domains}
[Q^{1*}_\out \times Q^{23}_*] + [Q^{*2}_\out \times Q^{13}_*] + [Q^{*3}_\out \times Q^{12}_*] =0.
\end{equation}

\end{corollary}

\begin{proof}
By our conventions on the orientation of $\partial \Ncal_{3,1}$ we have
$$
[B_i] = (\zeta_*[b_i]) \times [F].
$$
Now the three equations follow from the description of $[b_i]$ and $[F]$ in the above two lemmas. The last assertion is a trivial consequence of the fact that $B_1+B_2+B_3$ is the oriented boundary of $\Ncal_{3,1}$.
\end{proof}

Using the section $\xi:\Ncal_{3,1} \to \Mcal_{3,1}$, we can lift the various homology classes and relations between them from $\Ncal_{3,1}$ to $\Mcal_{3,1}$. We only record one such observation which we will need later on.
\begin{corollary}\label{cor:cycles_7term}
In $H_1(\Mcal_{3,1})$, we have the relation
\begin{align*}
\lefteqn{
[D_\out \times \hat P^{*3}_\out \times \hat P^{12}_*] =}  \\
& [\hat P^{1*}_\out \times D_* \times \hat P^{23}_*]
+ [\hat P^{2*}_\out \times D_* \times \hat P^{31}_*]
+ [\hat P^{3*}_\out \times D_* \times \hat P^{12}_*]\\
&
- [\hat P^{1*}_\out \times \hat P^{23}_* \times D_1] 
- [\hat P^{1*}_\out \times \hat P^{23}_* \times D_2]
- [\hat P^{1*}_\out \times \hat P^{23}_* \times D_3].
\end{align*}
\end{corollary}

\begin{proof}
To distinguish instances of identity  \eqref{eq:lift_of_Q} for various input labels, we will write them as
$$
\xi_*[Q^{ij}_k] = [D_k \times \hat P^{ij}_k] - [\hat P^{ij}_k \times D_i] - [\hat P^{ij}_k \times D_j].
$$
With this notation, we can write the lift of the identity \eqref{eq:cycles_poisson} to $H_1(\Mcal_{3,1})$ as
\begin{align*}
\lefteqn{
[D_\out \times \hat P^{*3}_\out \times \hat P^{12}_*]
- [\hat P^{*3}_\out \times D_* \times \hat P^{12}_*]
- [\hat P^{*3}_\out \times D_3 \times \hat P^{12}_*] =}  \\
&
[\hat P^{1*}_\out \times D_* \times \hat P^{23}_*]
- [\hat P^{1*}_\out \times \hat P^{23}_* \times D_2]
- [\hat P^{1*}_\out \times \hat P^{23}_* \times D_3]\\
& + [\hat P^{*2}_\out \times D_* \times \hat P^{13}_*]
- [\hat P^{*2}_\out \times \hat P^{13}_* \times D_1] 
- [\hat P^{*2}_\out \times \hat P^{13}_* \times D_3]
\end{align*}
Now we have
$$
[\hat P^{*3}_\out \times D_3 \times \hat P^{12}_*]= [\hat P^{*3}_\out \times \hat P^{12}_* \times D_3]= [\hat P^{*2}_\out \times \hat P^{13}_* \times D_3],
$$
where the first two terms describe the same cycle in the boundary of $\Mcal_{3,1}$, and the second equality comes from the fact that the $0$-cycles $\hat P^{*3}_\out \times \hat P^{12}_*$ and $\hat P^{*2}_\out \times \hat P^{13}_*$ are homologous in the connected space $\Mcal_{3,1}$. So in the above 9-term relation, the last term on the left hand side cancels with the last term on the right hand side. Rearranging the remaining terms, we obtain the relation
\begin{align*}
\lefteqn{[D_\out \times \hat P^{*3}_\out \times \hat P^{12}_*]
=}  \\
&
[\hat P^{1*}_\out \times D_* \times \hat P^{23}_*]
+ [\hat P^{*3}_\out \times D_* \times \hat P^{12}_*]
+ [\hat P^{*2}_\out \times D_* \times \hat P^{13}_*]\\
& - [\hat P^{1*}_\out \times \hat P^{23}_* \times D_2]
- [\hat P^{1*}_\out \times \hat P^{23}_* \times D_3]
- [\hat P^{*2}_\out \times \hat P^{13}_* \times D_1] 
\end{align*}
Finally, using the symmetries $[\hat P^{ij}_k]=[\hat P^{ji}_k]$ and $[\hat P^{i*}_\out \times \hat P^{jk}_*] = [\hat P^{j*}_\out \times \hat P^{ki}_*]$, we find that this is equivalent to the 7-term relation as stated.
\end{proof}

\subsection{Moduli spaces with one input and several weighted outputs} \label{ssec:moduli_1toell}

Next we consider the analogous spaces with one input and several output punctures, where the output punctures carry non-negative weights adding up to $1$. This gives rise to a largely parallel story, with some important modifications.

We start from the three sequences of spaces $\Ccal^0_{1,\ell}$, $\Ncal^0_{1,\ell}$ and $\Mcal^0_{1,\ell}$ which are the exact analogues of the spaces we considered in the previous section, simply with the roles of inputs and outputs reversed.
In particular the space $\Ncal^0_{1,\ell}$ differs from $\Ccal^0_{1,\ell}$  by including a marker at the unique input puncture.
As before, we have a (non-canonical) section $\zeta:\Ccal^0_{1,\ell} \to \Ncal^0_{1,\ell}$ which is determined (quite similarly as before) by letting the marker at the input be the tangent direction of the circle passing through the input and the first two outputs and pointing to the first output. In this way, we get a trivialization $\Ncal^0_{1,\ell} \cong \Ccal^0_{1,\ell} \times S^1$.

Similarly, each space $\Mcal^0_{1,\ell}$ is a principal $(S^1)^\ell$-bundle over $\Ncal^0_{1,\ell}$. As before, this bundle admits a ``canonical'' section $\xi:\Ncal^0_{1,\ell} \to \Mcal^0_{1,\ell}$ by transporting the marker from the input to each of the output punctures, and so we get a trivialization $\Mcal^0_{1,\ell} \cong \Ncal^0_{1,\ell} \times (S^1)^\ell$.

What we are actually interested in are the spaces $\Ccal_{1,\ell}$, $\Ncal_{1,\ell}$ and $\Mcal_{1,\ell}$, which are obtained from $\Ccal^0_{1,\ell}$, $\Ncal^0_{1,\ell}$ and $\Mcal^0_{1,\ell}$ by adding weights at the outputs which add up to 1. More precisely, we let
$$
\simplex_{\ell-1} := \{ (s_1,\dots,s_\ell) \in \R_{\ge 0}^\ell \,|\, \sum s_j = 1 \}
$$
and set
$$
\Ccal'_{1,\ell} := \simplex_{\ell-1} \times \Ccal^0_{1,\ell}, \quad
\Ncal'_{1,\ell} := \simplex_{\ell-1} \times \Ncal^0_{1,\ell} \quad \text{and} \quad
\Mcal'_{1,\ell} := \simplex_{\ell-1} \times \Mcal^0_{1,\ell},
$$
where we consider the coordinate $s_i$ of a point in $\simplex_{\ell-1}$ as the weight at the $i$th output puncture. 
We orient the simplex $\simplex_{\ell-1}$ by the convention that the ordered sequence of vectors $(e_2-e_1,\dots, e_\ell-e_1)$ forms a positively oriented basis of $T_{e_1}\simplex_{\ell-1}$, where $e_i=(0,\dots,1,\dots,0)$ are the standard basis vectors in $\R^\ell$.

For $\ell=2$, we will have $\Ccal_{1,\ell} =\Ccal'_{1,\ell} $, $\Ncal_{1,2}=\Ncal'_{1,2}$ and $\Mcal_{1,2}=\Mcal'_{1,2}$. For $\ell \ge 3$, the spaces $\Ccal_{1,\ell}$ are obtained from $\Ccal'_{1,\ell}$ by taking real blow-ups of certain cycles in $\pi^{-1}(\p \Ccal_{1,\ell})$, and $\Ncal_{1,\ell}$ and $\Mcal_{1,\ell}$ are the corresponding (pull-back) bundles. Below we will give the details for the case $\ell=3$, which is the only one needed in this work.

In what follows, we will be interested in relations between certain homology classes in the relative homologies
\begin{align*}
H(\Ncal_{1,\ell},(\partial \simplex_{\ell-1}) \times \Ncal^0_{1,\ell})&\cong H(\simplex_{\ell-1},\partial \simplex_{\ell-1}) \otimes H(\Ncal^0_{1,\ell}), \quad \text{and} \\
H(\Mcal_{1,\ell},(\partial \simplex_{\ell-1}) \times \Mcal^0_{1,\ell})&\cong H(\simplex_{\ell-1},\partial \simplex_{\ell-1}) \otimes H(\Mcal^0_{1,\ell})
\end{align*}
for $\ell=2$ and $\ell=3$. 
Our computations will combine arguments from the previous section for homology classes coming from the second factors with considerations involving the weights. When the various (chains in) moduli spaces appearing in this subsection are used to construct operations, they have the following interpretations:
\begin{itemize}
\item The generators $R$ in degree 1 and $S$ in degree 2 for $H_*(\Ncal_{2,1}, (\p \simplex_1) \times \Ncal^0_{2,1})$ give rise to the coproduct $\lambda$ and the cobracket $\gamma$.
\item The boundary of $\Ncal_{3,1}$ will have three components $B_1$, $B_2$ and $B_3$ which are described in terms of products of these cycles. There will also be three additional faces $T_1$, $T_2$ and $T_3$ whose contribution to identities in implementations we will suppress by a dimensional restriction.
\item The coPoisson relation \eqref{eq:coPoisson} is a consequence of the (obvious) fact that $\zeta(\p \Ccal_{1,3})$ is zero in relative homology (Corollary~\ref{cor:cycles_copoisson}).
\item The coJacobi relation \eqref{eq:coJacobi} is a consequence of the (obvious) fact that the sum $B_1+B_2+B_3$ is nullhomologous in relative homology (Lemma~\ref{lem:cycles_coJacobi}).
\item The 7-term relation \eqref{eq:7_term_for_Delta_and_lambda} is a consequence of the (obvious) fact that the lift $\xi(B_1 +B_2 +B_3)$ to $\Mcal_{1,3}$ is still nullhomologous in relative homology (Corollary~\ref{cor:7-term-coBV}).
\end{itemize}  
As before, the space $\Ccal^0_{1,2}$ consists of a single point, the space $\Ncal^0_{1,2}$ is a circle and $\Mcal^0_{1,2}$ is a 3-torus. The homology of $\Ncal^0_{1,2}$ is generated by the class of a point  $C:=\zeta(\Ccal^0_{1,2})$ in degree 0 and the class of the fundamental chain $\Ncal^0_{1,2} \cong S^1$ in degree 1. It follows that
$H(\Ncal_{1,2}, (\partial \simplex_{1}) \times \Ncal^0_{1,2})$ is generated by the classes of the two relative cycles
$$
R:= \simplex_1 \times C \quad \text{and} \quad
S:= \simplex_1 \times \Ncal^0_{1,2}
$$
in degrees 1 and 2, respectively. Note that the action of $S_2\cong \Z_2$ by switching the labels of the outputs reverses the orientation of $\simplex_1$, so we have $[\tau R]=-[R]$ and $[\tau S]=-[S]$.

Similarly, the homology of $\Mcal^0_{1,2}$ in degree 0 is generated by $\xi_*(C)$, and in degree 1 there are three cycles $D_\inn$, $D_1$ and $D_2$ which correspond to rotation of the marker at the corresponding puncture in the counterclockwise direction while fixing all others. Then, analogously to the previous subsection, we have
$$
\xi_*[\Ncal^0_{1,2}] = [D_\inn] - [D_1] - [D_2]
$$
in $H_1(\Mcal^0_{1,2})$. Moreover, $H_1(\Mcal_{1,2}, (\partial \simplex_1) \times \Mcal^0_{1,2})$ is generated by the class of $\hat R:=\xi_*(R)$, and in degree 2 we have the relation
\begin{eqnarray}
\xi_*[S] &= [\hat R \times D_\inn] - [\hat R \times D_1] - [\hat R \times D_2] \notag\\
&= [\hat R \times D_\inn] + [D_1 \times \hat R] + [D_2 \times \hat R],\label{eq:lift_of_S}
\end{eqnarray}
where the signs in the second line correct for the switch in the order of the two 1-dimensional chains $R$ and $D_i$.

We next discuss $\ell=3$. The spaces $\Ccal^0_{1,3}$, $\Ncal^0_{1,3}$ and $\Mcal^0_{1,3}$ are diffeomorphic to the corresponding spaces $\Ccal_{3,1}$, $\Ncal_{3,1}$ and $\Mcal_{3,1}$ from the previous subsection, and so the relations in their homology follow the same pattern. 

Writing $b_i$ for the boundary circle of $\Ccal^0_{1,3}$ consisting of nodal configurations with the $i$th output puncture on the same sphere component as the input puncture, we have
$$
\partial \Ccal'_{1,3} = (\partial \simplex_2) \times \Ccal^0_{1,3} \cup \simplex_2 \times b_1 \cup \simplex_2 \times b_2 \cup \simplex_2 \times b_3.
$$
For reasons that will become clear later, we want to work with the modification $\Ccal_{1,3}$ obtained from blowing up the 1-dimensional submanifolds
\begin{align*}
V_1:=\{(1,0,0)\} \times b_1 &\subseteq \simplex_2 \times b_1, \\
V_2:=\{(0,1,0)\} \times b_2 &\subseteq \simplex_2 \times b_2, \quad \text{and} \\
V_3:=\{(0,0,1)\} \times b_3 &\subseteq \simplex_2 \times b_3
\end{align*}
of $\p \Ccal'_{1,3}$. 
As the blow-up occurs in the boundary, the space $\Ccal_{1,3}$ is a 4-dimensional manifold with boundary and (more) corners still homotopy equivalent to $\Ccal'_{1,3}$.

Neighborhoods of the submanifolds $V_i$ can be modelled by products of a quadrant in $\R^3$ with $b_i$, and so in the blow-up each of these 1-dimensional submanifolds of the boundary is replaced by a new boundary stratum $t_i$ which is a product of a triangle with $b_i$.

We now define $\Ncal_{1,3}$ to be the total space of the pull-back of the bundle $\Ncal'_{1,3} \to \Ccal'_{1,3}$ to the blow up. Set $T_i:= \pi^{-1}(t_i)$ and $T:=T_1 \cup T_2 \cup T_3$. Since the blow-up happened at submanifolds of $(\p \simplex_2) \times \Ccal^0_{1,3}$,  it is clear that we still have
$$
H(\Ncal_{1,3}, (\p \simplex_2) \times \Ncal^0_{1,3} \cup T) \cong H(\simplex_2, \p \simplex_2) \otimes H(\Ncal^0_{1,3}).
$$
The preimage of a boundary component $B'_i$ of $\Ncal'_{1,3}$ in the blow-up is the union of $T_i$ with a stratum $B_i$ of the form $\simplex_1 \times \simplex_1 \times b_i \times S^1$, meeting $T_i$ along a common face of the form $\simplex_1 \times b_i \times S^1$. 
Each of the $B_i$ is parametrized by a product of two copies of $\Ncal_{1,2}$. Also, it is still true that the relative homology $H(\Ncal_{1,3}, (\partial \simplex_2) \times \Ncal^0_{1,3} \cup T)$ is generated by the images of $H(B_i,\partial B_i)$ under the inclusion maps. These groups in turn are generated by products of cycles $R$ and/or $S$ from corresponding copies of $\Ncal_{1,2}$. As in the previous subsection, we will use superscripts and subscripts for the chains $R$ and $S$ in $\Ncal_{1,2}$ to specify the labels of the corresponding inputs and outputs.

The relative homology $H_2(\Ncal_{1,3}, (\partial \simplex_2) \times \Ncal^0_{1,3} \cup T)$ is clearly generated by the homology class of the cycle $\simplex_2 \times \{pt\} \subseteq \Ncal_{1,3}$, where $pt\in \Ncal^0_{1,3}$ is any point. The following lemma describes its relation to the relevant product cycles in the three boundary components $B_i$.
\begin{lemma}\label{lem:cycles_coassociativity}
In $H_2(\Ncal_{1,3}, (\partial \simplex_2) \times \Ncal^0_{1,3} \cup T)$ we have the relation
$$
[\simplex_2 \times \{pt\}]= - [R^*_{23} \times R^\inn_{1*}] = - [R^*_{31} \times R^\inn_{2*}] = [R^*_{12} \times R^\inn_{*3}].
$$
\end{lemma}

\begin{proof}
Up to sign, we can pick the image of $R^*_{23} \times R^\inn_{1*}$ under the product parametrization of this boundary component as a representative of the class $[\simplex_2 \times \{pt\}]$ in the boundary component $B_1$. Similar statements hold for the other two cases. To determine the signs, we need to understand the effect on orientations of the corresponding maps on weights.

For the boundary component $B_1$, the relevant blow-down map on the weights is given by
\begin{align}
\simplex_1 \times \simplex_1 &\to \simplex_2 \notag\\
((s_2,s_3),(s_1,s_*)) & \mapsto (s_1,s_*s_2,s_*s_3), \label{eq:weights_b1}
\end{align}
which is easily seen to be orientation reversing, proving the first equality.
For the boundary component $B_2$, the blow-down map on weights is given by
\begin{align}
\simplex_1 \times \simplex_1 &\to \simplex_2 \notag\\
((s_3,s_1),(s_2,s_*)) & \mapsto (s_*s_1,s_2,s_*s_3), \label{eq:weights_b2}
\end{align}
which is also orientation reversing, proving the second equality.
Finally, for the boundary component $B_3$, the blow-down map on weights is given by
\begin{align}
\simplex_1 \times \simplex_1 &\to \simplex_2 \notag\\
((s_1,s_2),(s_*,s_3)) & \mapsto (s_*s_1,s_*s_2,s_3), \label{eq:weights_b3}
\end{align}
which is orientation preserving, proving the last equality.
\end{proof}

The relative homologies $H_3(B_i,\partial B_i)$ have rank 2, and are generated by the homology classes of the following cycles:
\begin{itemize}
\item $R^*_{23} \times S^\inn_{1*}$ and $S^*_{23} \times R^\inn_{1*}$ for $B_1$,
\item $R^*_{31} \times S^\inn_{2*}$ and $S^*_{31} \times R^\inn_{2*}$ for $B_2$, and
\item $R^*_{12} \times S^\inn_{*3}$ and $S^*_{12} \times R^\inn_{*3}$ for $B_3$.
\end{itemize}
Our next goal is to describe the relative homology classes of the images of the relative cycles $\simplex_2 \times b_i \subseteq \Ccal_{1,3}$ 
under the section $\zeta:\Ccal_{1,3} \to \Ncal_{1,3}$.
\begin{lemma}
With the above notation, in $H_3(\Ncal_{1,3}, (\p \simplex_2 \times \Ncal^0_{1,3} \cup T)$ we have
\begin{align*}
\zeta_*[\simplex_2 \times b_1] &= - [S^*_{23} \times R^\inn_{1*}],\\
\zeta_*[\simplex_2 \times b_2] &= - [S^*_{31} \times R^\inn_{2*}], \quad\text{and}\\
\zeta_*[\simplex_2 \times b_3] &= [R^*_{12} \times S^\inn_{*3}].\\
\end{align*}
\end{lemma}

\begin{proof}
Recall that $R^k_{ij}= \simplex_1 \times C^k_{ij}$ and $S^k_{ij}= \simplex_1 \times K^k_{ij}$, where we use $C^k_{ij}$ and $K^k_{ij}$ as a notation for $\zeta(\Ccal^0_{1,2}) \in \Ncal^0_{1,2}$ and the fundamental chain of $\Ncal^0_{1,2}$ with input label $k$ and output labels $i$ and $j$, respectively. Now the argument in the proof of Lemma~\ref{lem:bi_as_products} shows that the relations
$$
\zeta_*[b_1] =-[K^*_{23} \times C^\inn_{1*}], \ 
\zeta_*[b_2] =-[K^*_{13} \times C^\inn_{*2}] \ \text{and} \ 
\zeta_*[b_3] = [C^*_{12} \times K^\inn_{*3}]
$$
hold in $H(\Ncal^0_{1,3})$. 
So in order to prove the lemma, it again remains to understand the effect that the parametrizations have on orientations. 

For the first two relations, note that the domain of a chain of type $S \times R$ is $\simplex_1 \times S^1 \times \simplex_1$ so there are two sources of signs: switching the order of the last two factors, and the effect of the map $\simplex_1 \times \simplex_1 \to \simplex_2$ from \eqref{eq:weights_b1} and \eqref{eq:weights_b2}, respectively. Since these are both orientation reversing, we obtain the first two relations as claimed.

For the third relation, the domain of $R^*_{12} \times S^\inn_{*3}$ is $\simplex_1 \times \simplex_1 \times S^1$, so the only source of a sign would be the effect of the map on weights in \eqref{eq:weights_b3}. As this map is orientation preserving, we conclude the third relation.
\end{proof}

\begin{corollary}\label{cor:cycles_copoisson}
In $H_3(\Ncal_{1,3}, (\partial \simplex_2) \times \Ncal^0_{1,3} \cup T)$, we have the relation
\begin{equation}\label{eq:cycles_copoisson}
[R^*_{12} \times S^\inn_{*3}] = [S^*_{23} \times R^\inn_{1*}] + [S^*_{31} \times R^\inn_{2*}].
\end{equation}
\end{corollary}

\begin{proof}
The 4-chain $\zeta(\Ccal_{1,3})$ gives a homology between the left hand side and the right hand side.
\end{proof}
Finally, we express the homology classes of the boundary components $B_i$ as products.
\begin{lemma}\label{lem:cycles_coJacobi}
We have
\begin{align*}
[B_1] &= [S^*_{23}\times S^\inn_{1*}],\\
[B_2] &= [S^*_{31} \times S^\inn_{2*}], \quad\text{and}\\
[B_3] &= - [S^*_{12} \times S^\inn_{*3}].
\end{align*}
In particular, the sum of these three classes vanishes in\break  $H_4(\Ncal_{1,3},(\partial \simplex_2)\times \Ncal^0_{1,3} \cup T)$.
\end{lemma}

\begin{remark}
We could make the statement more symmetric by writing the last equation as
$$
[B_3] = [S^*_{12} \times S^\inn_{3*}],
$$
using that $[S^\inn_{*3}]=-[S^\inn_{3*}]$.
\end{remark}
\begin{proof}
Denote by $\tilde B_i$ the boundary components of $\Ncal^0_{1,3}$, so that
$$
B_i = \simplex_2 \times \tilde B_i.
$$
Now an argument analogous to the one giving Corollary~\ref{cor:Bi_as_products} gives the relations
$$
[\tilde B_1] = [K^*_{23} \times K^\inn_{1*}], \quad
[\tilde B_2] = [K^*_{13} \times K^\inn_{*2}] \quad\text{and} \quad
[\tilde B_3] = [K^*_{12} \times K^\inn_{*3}].
$$
So in view of the relation $S^k_{ij}= \simplex_1 \times K^k_{ij}$, the proof of the lemma again comes down to analysing the effect of the parametrizations on orientations.

This time, the domains of the product chains of type $S \times S$ are $\simplex_1 \times S^1 \times \simplex_1 \times S^1$, so we always get a sign from switching the order of the middle two factors. From the maps on weights \eqref{eq:weights_b1}, \eqref{eq:weights_b2} and \eqref{eq:weights_b3} we get an additional sign in the first two cases, but not in the third. This proves the three relations stated in the lemma.

The final claim follows from the fact that $[B_1] + [B_2] + [B_3]= [\partial \Ncal_{1,3}]$ in the homology relative to $(\partial \simplex_2) \times \Ncal^0_{1,3} \cup T$.
\end{proof}

As was the case for $\Ncal_{1,3}$, the space $\Mcal_{1,3}$ is the total space of the pull-back of the bundle $\Mcal'_{1,3} \to \Ccal'_{1,3}$ to the blow-up. We call the union of new boundary faces of $\Mcal_{1,3}$ created by the blow-up $\mathbf T$. Then we can use the section $\xi:\Ncal_{1,3} \to \Mcal_{1,3}$ to lift relations in homology. 

\begin{corollary}\label{cor:7-term-coBV}
In $H_3(\Mcal_{1,3}, (\partial \simplex_2) \times \Mcal^0_{1,3} \cup \mathbf T)$ we have the relation
\begin{align*}
\lefteqn{[\hat R^*_{12} \times \hat R^\inn_{*3} \times D_\inn] =}&  \\
& - [\hat R^*_{12} \times D_* \times \hat R^\inn_{*3}]  - [\hat R^*_{23} \times D_* \times \hat R^\inn_{*1}]  - [\hat R^*_{31} \times D_* \times \hat R^\inn_{*2}] \\
& - [D_1 \times \hat R^*_{12} \times \hat R^\inn_{*3}] - [D_2 \times \hat R^*_{23} \times \hat R^\inn_{*1}] - [D_3 \times \hat R^*_{31} \times \hat R^\inn_{*2}].
\end{align*}
\end{corollary}

\begin{proof}
Similarly to the proof of Corollary~\ref{cor:cycles_7term}, we use \eqref{eq:lift_of_S} to rewrite the identity \eqref{eq:cycles_copoisson} as
\begin{align*}
\lefteqn{[\hat R^*_{12} \times \hat R^\inn_{*3} \times D_\inn] + [\hat R^*_{12} \times D_* \times \hat R^\inn_{*3}] + [\hat R^*_{12} \times D_3 \times \hat R^\inn_{*3}] =}& \\
& [\hat R^*_{23} \times D_* \times \hat R^\inn_{1*}] + [D_2 \times \hat R^*_{23} \times \hat R^\inn_{1*}] + [D_3 \times \hat R^*_{23} \times \hat R^\inn_{1*}]\\
& + [\hat R^*_{31} \times D_* \times \hat R^\inn_{2*}] + [D_3 \times \hat R^*_{31} \times \hat R^\inn_{2*}] + [D_1 \times \hat R^*_{31} \times \hat R^\inn_{2*}].
\end{align*}
Note that
$$
[\hat R^*_{12} \times D_3 \times \hat R^\inn_{*3}] = -[D_3 \times \hat R^*_{12} \times \hat R^\inn_{*3}] = [D_3 \times \hat R^*_{31} \times \hat R^\inn_{2*}],
$$
where in the second equality we have used the (lift of the) relation $[R^*_{12} \times R^\inn_{*3}]=-[R^*_{31} \times R^\inn_{2*}]$ from Lemma~\ref{lem:cycles_coassociativity}. It follows that the last term on the left hand side and the second to last term on the right hand side of the above relation cancel. The statement in the lemma is now obtained from rearranging terms and using $[R^k_{ij}]=-[R^k_{ji}]$ where necessary.
\end{proof}

\subsection{The moduli space with two inputs and two weighted outputs} \label{ssec:moduli_2to2}

In the previous cases of either only one output or only one input, we saw that there is a natural hierarchy formed by the spaces $\Ncal_{k,1}$ or $\Ncal_{1,\ell}$, respectively, which incorporate a (freely varying) marker at that preferred puncture, or, alternatively, consistently coupled markers at all the punctures. We do not see an analogous hierarchy in the case of simultaneously multiple inputs and multiple outputs, as now there seems to be no preferred way to consistently couple markers.

\begin{remark}\label{rem:no-intermediate2}
We believe that this absence of an intermediate hierarchy explains the absence of a corresponding intermediate algebraic structure between BV unital infinitesimal algebras and unital infinitesimal algebras mentioned in Remark~\ref{rem:no-intermediate1}.
\end{remark}
In this subsection, we will concentrate on the case $(k,\ell)=(2,2)$ of two inputs and two outputs, and we will only discuss relations immediately relevant to our goals.

The space $\Ccal_{2,2}^0$ is still naturally diffeomorphic to $\Ccal_{3,1}$, by simply interpreting the first two punctures as inputs and the forth and third punctures as the first and second outputs, respectively. We next define $\Ccal'_{2,2} := \simplex_1 \times \Ccal_{2,2}^0$, so that the space $\Ccal'_{2,2}$ is diffeomorphic to a thickening of the pair of pants depicted in Figure~\ref{fig:C_3_1}. In particular, the boundary of $\Ccal'_{2,2}$ is a genus two surface and can be written as
$$
\partial \Ccal'_{2,2} = (\partial \simplex_1) \times \Ccal_{2,2}^0 \cup \simplex_1 \times b_1 \cup \simplex_1 \times b_2 \cup \simplex_1 \times b_3,
$$
where the $b_i$ are the three boundary components of $\Ccal_{2,2}^0$.
To be more specific, we choose our notation so that
\begin{itemize}
\item on the boundary component $b_1$ the node separates the domain into two sphere components, one carrying the first input and first output, and the other carrying the second input and second output,
\item on the boundary component $b_2$ the node separates the domain into two sphere components, one carrying the first input and second output, and the other carrying the second input and first output, and
\item on the boundary component $b_3$, the node separates the component with the two inputs from the component with the two outputs.
\end{itemize}
At this stage, we confront a new phenomenon. In the cases with only one input or only one output, the trees underlying any nodal configuration come with a preferred direction, to the unique output or away from the unique input, respectively, so for each node it was always clear for which adjacent component this was to be considered as an input and for which component it was to be considered as an output, respectively. Now, on the boundary components $b_1$ and $b_2$, we have deformations of the conformal structure which separate the inputs and outputs into two input/output pairs, and the additional node could and should be interpreted in two ways.

\begin{figure}[h]
\begin{tikzpicture}[scale=0.6]

\draw[thick] (-1,0) circle (1);
\draw[thick] (1,0) circle (1);
\draw[thick, dotted] (-2,0) arc (180:0: 1 and 0.2);
\draw[thick] (-2,0) arc (180:360: 1 and 0.2);
\draw[thick, dotted] (0,0) arc (180:0: 1 and 0.2);
\draw[thick] (0,0) arc (180:360: 1 and 0.2);
\draw[thick] (-1,0.9) to (-1,1.1);
\node at (-1,1.4) {\tiny $1^+$};
\draw[thick] (-1,-0.9) to (-1,-1.1);
\node at (-1,-1.4) {\tiny $1^-$};
\draw[thick] (1,0.9) to (1,1.1);
\node at (1,1.4) {\tiny $2^+$};
\draw[thick] (1,-0.9) to (1,-1.1);
\node at (1,-1.4) {\tiny $2^-$};

\draw[thick] (-5,-0.7071) circle (1);
\draw[thick] (-6.4142,0.7071) circle (1);
\draw[thick, dotted] (-6,-0.7071) arc (180:0: 1 and 0.2);
\draw[thick] (-6,-0.7071,0) arc (180:360: 1 and 0.2);
\draw[thick, dotted] (-7.4142,0.7071) arc (180:0: 1 and 0.2);
\draw[thick] (-7.4142,0.7071) arc (180:360: 1 and 0.2);
\draw[thick] (-6.4142,1.6071) to (-6.4142,1.8071);
\node at (-6.4142,2.1071) {\tiny $1^+$};
\draw[thick] (-7.0213, 0.1) to (-7.2213,-0.1);
\node at (-7.4213,-0.4) {\tiny $1^-$};
\draw[thick] (-4.1929,0.1) to (-4.3929,-0.1);
\node at (-3.9929, 0.4) {\tiny $2^+$};
\draw[thick] (-5, -1.6071) to (-5,-1.8071);
\node at (-5,-2.1071) {\tiny $2^-$};
\draw[thick] (-5.8071,0.1) to (-5.6071,-0.1);
\node at (-5.4071,-0.3) {$*$};

\draw[thick] (5,-0.7071) circle (1);
\draw[thick] (6.4142,0.7071) circle (1);
\draw[thick, dotted] (4,-0.7071) arc (180:0: 1 and 0.2);
\draw[thick] (4,-0.7071,0) arc (180:360: 1 and 0.2);
\draw[thick, dotted] (5.4142,0.7071) arc (180:0: 1 and 0.2);
\draw[thick] (5.4142,0.7071) arc (180:360: 1 and 0.2);
\draw[thick] (6.4142,1.6071) to (6.4142,1.8071);
\node at (6.4142,2.1071) {\tiny $2^+$};
\draw[thick] (4.1929,0.1) to (4.3929,-0.1);
\node at (3.9929, 0.4) {\tiny $1^+$};
\draw[thick] (7.0213, 0.1) to (7.2213,-0.1);
\node at (7.4213,-0.4) {\tiny $2^-$};
\draw[thick] (5, -1.6071) to (5,-1.8071);
\node at (5,-2.1071) {\tiny $1^-$};
\draw[thick] (5.8071,0.1) to (5.6071,-0.1);
\node at (5.4071,-0.3) {$*$};

\end{tikzpicture}

\caption{All nodal configurations in $b_1$ split the inputs and outputs as described by the middle configuration, which a priori can be interpreted in at least the two ways described on the left and on the right.} 
\end{figure}
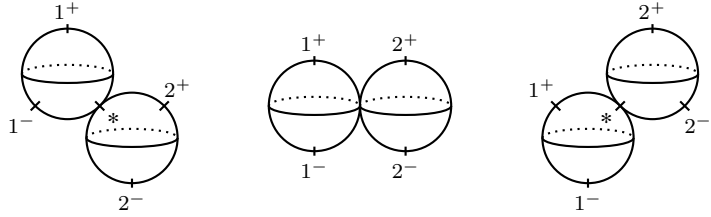
We solve this apparent puzzle by using the weights to make a consistent choice.
The idea, already proposed in \cite{Sullivan-background}, is to interpret the inputs as having equal weight, and the output weights as the fraction of the total received at the corresponding puncture. The effect is that for the component with the output whose weight is less than $\frac 12$ the node functions as an output, and for the component whose output has weight bigger than $\frac 12$, it functions as an input.

There is still one special case, namely the midpoint $(\frac 12, \frac 12) \in \simplex_1$ of the weight simplex, where the two components are balanced and the node is still indifferent. For reasons which will become clear in \S\ref{ssec:11-term-implemented}, we define $\Ccal_{2,2}$ as the blow-up of $\Ccal'_{2,2}$ along the two submanifolds $\{(\frac 12, \frac 12)\} \times b_1$ and $\{(\frac 12, \frac 12)\} \times b_2$ in its boundary. This blow-up has the effect of thickening the central circle of each of the annuli $\simplex_1 \times b_1$ and $\simplex_1 \times b_2$ into a new annulus $A_i$, which we identify with $[0,1] \times b_i$, respectively. When implemented in Floer theory, the additional parameter provided by the blow-up will allow to interpolate between Floer data coming from the left half, respectively from the right half of the simplex $\simplex_1$.
For convenience, we introduce the notation 
\begin{align*}
\simplex_\le&=\{(s_1,s_2) \in \simplex_1 \,:\, s_1 \le s_2\}
\quad \text{and} \quad\\
\simplex_\ge&=\{(s_1,s_2) \in \simplex_1 \,:\, s_1 \ge s_2\}
\end{align*}
to denote the two halves of the simplex $\simplex_1$. Using this notation, we schematically illustrate the boundary of $\Ccal_{2,2}$ in Figure~\ref{fig:boundary_C_2_2}.

\begin{figure}[h]
\begin{tikzpicture}
\definecolor{darkgreen}{rgb}{0,0.6,0};
\definecolor{darkblue}{rgb}{0,0,0.7};

\draw[darkgreen,very thick] (-5,0.6) to [out=80, in=170] (-2,2.1)
to [out=350, in=180] (0,1.8)
to [out=0, in=190] (2,2.1)
to [out=10, in=100] (5,0.6);
\draw[darkgreen,very thick] (5,-0.6)
to [out=260, in=350] (2,-2.1)
to [out=170, in=0] (0,-1.8)
to [out=180, in=10] (-2,-2.1)
to [out=190, in=280] (-5,-0.6);

\draw[darkgreen, very thick] (-3.8,0.6) to [out=70, in=110] (-0.6,0.3)
to [out=290, in=70] (-0.6,-0.3) to [out=250, in=290] (-3.8,-0.6);

\draw[darkgreen, very thick] (3.8,0.6) to [out=110, in=70] (0.6,0.3)
to [out=250, in=110] (0.6,-0.3) to [out=290, in=250] (3.8,-0.6);

\draw[thick, dashed, black, fill=yellow, fill opacity=0.1]
(-3.8,0.6) arc (0:180: 0.6 and 0.15)
to [out=260, in=100] (-5,-0.6)
arc (180:0: 0.6 and 0.15)
to [out=110, in=250] (-3.8,0.6);
\draw[very thick, black, , fill=yellow, fill opacity=0.2]
(-3.8,-0.6) arc (0:-180:0.6 and 0.15)
to [out=100, in=260] (-5,0.6)
arc (180:360: 0.6 and 0.15)
to [out=250, in=110] (-3.8,-0.6);
\draw[thick, black, , fill=red, fill opacity=0.1]
(-3.9,-0.15) arc (0:-180:0.57 and 0.15)
to [out=95, in=265] (-5.04,0.15)
arc (180:360: 0.57 and 0.15);
\draw[dashed, black, , fill=red, fill opacity=0.05]
(-3.9,0.15) arc (0:180:0.57 and 0.15)
to [out=95, in=265] (-5.04,-0.15)
arc (180:360: 0.57 and 0.15);

\draw[thick, dashed, black, fill=yellow, fill opacity=0.1]
(-0.6,0.3) arc (180:0:0.6 and 0.15)
to [out=250, in=110] (0.6,-0.3)
arc (0:180:0.6 and 0.15)
to [out=70, in=290] (-0.6,0.3);
\draw[very thick, black, fill=yellow, fill opacity=0.2]
(-0.6,0.3) arc (180:360: 0.6 and 0.15)
to [out=250, in=110] (0.6,-0.3)
arc (0:-180:0.6 and 0.15)
to [out=70, in=290] (-0.6,0.3);

\draw[thick, dashed, black, fill=yellow, fill opacity=0.1]
(5,0.6) arc (0:180:0.6 and 0.15)
to [out=290, in=70] (3.8,-0.6)
arc (180:0: 0.6 and 0.15)
to [out=80, in=280] (5, 0.6);
\draw[very thick, black, fill=yellow, fill opacity=0.2]
(5,0.6) arc (0:-180:0.6 and 0.15)
to [out=290, in=70] (3.8,-0.6)
arc (180:360:0.6 and 0.15)
to [out=80, in=280] (5, 0.6);
\draw[thick, black, , fill=red, fill opacity=0.1]
(5.04,-0.15) arc (0:-180:0.57 and 0.15)
to [out=85, in=275] (3.9,0.15)
arc (180:360: 0.57 and 0.15);
\draw[dashed, black, , fill=red, fill opacity=0.05]
(5.04,0.15) arc (0:180:0.57 and 0.15)
to [out=85, in=275] (3.9,-0.15)
arc (180:360: 0.57 and 0.15);

\node[black] at (-5.8,0.5) {\tiny$\simplex_\le \times b_1$};
\node[black] at (-5.5,0) {\tiny$A_1$};
\node[black] at (-5.8,-0.5) {\tiny$\simplex_\ge \times b_1$};
\node[black] at (-1.2,0) {\tiny$\simplex_1 \times b_3$};
\node[black] at (5.8,0.5) {\tiny$\simplex_\le \times b_2$};
\node[black] at (5.5,0) {\tiny$A_2$};
\node[black] at (5.8,-0.5) {\tiny$\simplex_\ge \times b_2$};

\node[darkgreen] at (0,1.3) {\tiny$\{(0,1)\} \times \Ccal_{2,2}^0$};
\node[darkgreen] at (0,-1.3) {\tiny$\{(1,0)\} \times \Ccal_{2,2}^0$};
\end{tikzpicture}

\caption{A schematic picture of $\partial \Ccal_{2,2}$. The space $\Ccal_{2,2}$ itself would correspond to the ``outside'' handlebody, in which the three green circles all bound disks.}\label{fig:boundary_C_2_2}
\end{figure}
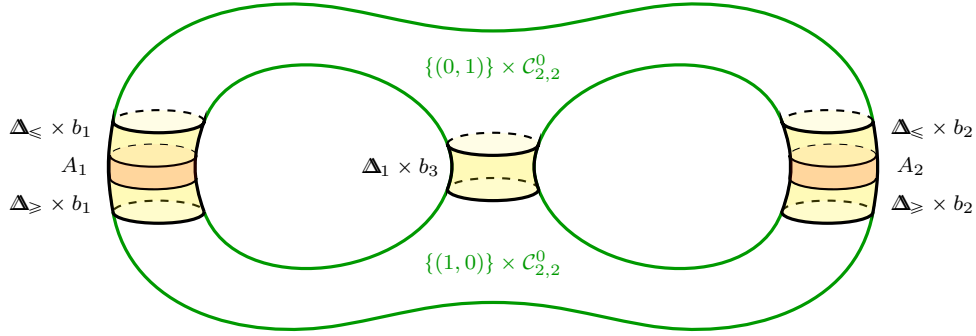

The space $\Mcal_{2,2}^0$ of configurations without weights but with markers at all punctures is a principal $(S^1)^4$-bundle over $\Ccal_{2,2}^0$, and the space $\Mcal_{2,2}$ is a principal $(S^1)^4$-bundle over $\Ccal_{2,2}$. 
Further below we will introduce and use a (noncanonical) section of this bundle in our arguments.

For $i\in \{1,2\}$, we set $\widetilde{b_i}:=\simplex_\le \times b_i \cup A_i \cup \simplex_\ge \times b_i \subseteq \Ccal_{2,2}$ and $\hat B_i:= \pi^{-1}(\widetilde{b_i}) \subseteq \Mcal_{2,2}$. We also set $\widetilde{b_3}:=\simplex_1 \times b_3\subseteq \Ccal_{2,2}$ and $\hat B_3:= \pi^{-1}(\widetilde{b_3}) \subseteq \Mcal_{2,2}$.
With this notation, the boundary of $\Mcal_{2,2}$ can be written as
$$
\partial \Mcal_{2,2} = (\partial \simplex_1) \times \Mcal_{2,2}^0 \cup \hat B_1 \cup \hat B_2 \cup \hat B_3.
$$
We are interested in relations in the homology $H_*(\Mcal_{2,2}, (\partial \simplex_1) \times \Mcal_{2,2}^0)$, which come from representing certain relative cycles in the $\hat B_i$ as (sums of) products of cycles from $\Mcal_{2,1}$ and $\Mcal_{1,2}$.

To keep track of labels, in what follows we write $\Mcal(i;j,k)$ for the space $\Mcal_{1,2}$ with input label $i$ and ouput labels $j$ and $k$ (in that order), and similarly $\Mcal(i,j;k)$ for the space $\Mcal_{2,1}$ with input labels $i$ and $j$ (in that order) and output label $k$. 

Recall that along $\hat B_1$ the domain splits into two components, with one carrying the first input and the first output, and the other carrying the second input and the second output. We now define a surjective map
$$
\Mcal(*,2;2) \times \Mcal(1;1,*) \to \pi^{-1}(\simplex_\le \times b_1) \subseteq \hat B_1
$$
by mapping a pair $(\Sigma_2,((s_1,s_2),\Sigma_1))$ to $((\frac {s_1} 2,\frac {1+s_2}2), \Sigma_2 \# \Sigma_1)$, where\break $\Sigma_2 \# \Sigma_1$ is the nodal configuration whose identification of tangent spaces at the node is given by matching the markers. The map factors as
$$
\Mcal(*,2;2) \times \Mcal(1;1,*) \to (\Mcal(*,2;2) \times \Mcal(1;1,*))/_{S^1} \to \pi^{-1}(\simplex_\le \times b_1), 
$$
where the circle action is given by simultaneously moving the marker at the node $*$ in both domains, and the second map is injective.
Similarly, there is a map
$$
\Mcal(1,*;1) \times \Mcal(2;*,2) \to \pi^{-1}(\simplex_\ge \times b_1) \subseteq \hat B_1
$$
by mapping a pair $(\Sigma_2,((s_1,s_2),\Sigma_1))$ to $((\frac {1+s_1} 2,\frac {s_2}2), \Sigma_2 \# \Sigma_1)$.
Finally, there is a parametrization of the ``middle part" created by the blow-up via a map
$$
\Mcal(1,*;1) \times \Mcal(*,2;2) \times [0,1] \to \pi^{-1}(A_1) \subseteq \hat B_1,
$$
which also factors through the diagonal action of $S^1$.
\begin{remark}
It might seem strange to parametrize $\pi^{-1}(A_1)$ by the fiber product of two copies of $\Mcal_{2,1}$. However, we will see in \S\ref{sec:reduced} below that, in the context of Floer homology, it is reasonable to reinterpret an output of weight zero as an input, and so this choice will be quite convenient.
\end{remark}
As a direct consequence of these parametrizations of $\hat B_1$, we see that the sum
$$
\hat P^{*2}_2 \times \hat R^1_{1*} + \hat P^{1*}_1 \times \hat P^{*2}_2\times [0,1] + \hat P^{1*}_1 \times \hat R^2_{*2}
$$
describes a generator of $H_1(\hat B_1, \partial \hat B_1)$.

In a completely analogous way, there are maps
\begin{align*}
\Mcal(*,1;2) \times \Mcal(2;1,*) &\to \pi^{-1}(\simplex_\ge \times b_2) \subseteq \hat B_2,\\
\Mcal(2,*;1) \times \Mcal(1;*,2) &\to \pi^{-1}(\simplex_\le \times b_2) \subseteq \hat B_2 \qquad \text{and}\\
\Mcal(2,*;1) \times \Mcal(*,1;2)\times [0,1] &\to \pi^{-1}(A_2) \subseteq \hat B_2.
\end{align*}

Just as before, they give rise to the description of a generator of $H_1(\hat B_2,\partial \hat B_2)$ as
$$
\hat P^{*1}_2 \times \hat R^2_{1*} + \hat P^{2*}_1 \times \hat P^{*1}_2 \times [0,1]+ \hat P^{2*}_1 \times \hat R^1_{*2}.
$$
In contrast, at the boundary component $\hat B_3$ the two inputs are separated from the two outputs by the node, so the whole boundary component is the image of a single map
$$
\Mcal(*;1,2) \times \Mcal(1,2;*) \to \hat B_3,
$$
so that the generator of $H_1(\hat B_3, \partial \hat B_3)$ can be represented by the chain $\hat R^*_{12} \times \hat P^{12}_*$.

\begin{lemma}\label{lem:rel_in_h1_of_m22}
In $H_1(\Mcal_{2,2}, (\partial \simplex_1) \times \Mcal_{2,2}^0)$, we have the relations
\begin{align*}
[\hat R^*_{12} \times \hat P^{12}_*]&=[\hat P^{*2}_2 \times \hat R^1_{1*} + \hat P^{1*}_1\times \hat P^{*2}_2 \times [0,1] + \hat P^{1*}_1 \times \hat R^2_{*2}]\\
&= [\hat P^{*1}_2 \times \hat R^2_{1*} + \hat P^{2*}_1\times \hat P^{*1}_2 \times [0,1] + \hat P^{2*}_1 \times \hat R^1_{*2}].
\end{align*}
\end{lemma}

\begin{proof}
All the relative cycles above represent the class $[\simplex_1 \times \{pt\}]$ for a suitable point in $\Mcal_{2,2}^0$, and as this later space is connected, they are homologous by chains of the form $\simplex \times c$ for suitable connecting paths $c \subseteq \Mcal_{2,2}^0$.
\end{proof}

Our next goal is to describe the boundary of a relative 3-chain that arises from a section  
\begin{equation} \label{eq:section-xi}
\xi:\Ccal_{2,2} \to \Mcal_{2,2}
\end{equation}
of the bundle $\Mcal_{2,2} \to \Ccal_{2,2}$ in terms of sums of product cycles. To specify $\xi$, we make the following choices:\footnote{While the argument below depends on our specific choice for $\xi$, other choices in fact lead to an equivalent outcome.}
\begin{itemize}
\item the marker at the first output is chosen tangent to the circle through the two inputs, pointing towards the first
\item the markers at the two inputs are coupled to the marker at the first output
\item the marker at the second output is coupled to the marker at the first input.
\end{itemize}
To describe $\xi(\partial \Ccal_{2,2})$, we will proceed in two steps.
First, by incorporating rotations at the node $*$, we obtain 2-cycles which project with degree $+1$ onto each of the boundary components of $\Ccal_{2,2}$. Then, we carefully study the corrections necessary to obtain the actual behavior of markers as in $\xi(\partial \Ccal_{2,2})$.

Above, we introduced the notation $\widetilde{b_i}$ for the three boundary components of $\Ccal_{2,2}$, each corresponding to a boundary component $b_i$ of $\Ccal^0_{2,2}$.
Recall that the $b_i$ are oriented in such a way that in the positive direction the two circles determined by the special points rotate \emph{clockwise} with respect to each other at the node. At the same time, along the cycle $D_*$ the marker moves \emph{counterclockwise}. This results in the following observations:
\begin{itemize}
\item In $\hat B_3$, a relative cycle projecting with degree 1 to $\widetilde{b_3} \subseteq \Ccal_{2,2}$ is given by $\hat R^*_{12} \times (-D_*) \times \hat P^{12}_*$.
\item In $\hat B_1$, a relative cycle projecting with degree 1 to $\widetilde{b_1} \subseteq \Ccal_{2,2}$ is given by
$$
\hat P^{*2}_2 \times D_* \times \hat R^1_{1*} + \hat P^{1*}_1 \times D_* \times \hat P^{*2}_2 \times [0,1] + \hat P^{1*}_1 \times D_* \times \hat R^2_{*2},
$$
because there are two sources of signs: one from the direction of rotation, and one from the fact that in $\simplex_1 \times b_1$ the interval comes before the circle factor in the domain, whereas in the above chain their order is reversed.
\item By the same argument, in $\hat B_2$, a relative cycle projecting with degree 1 to $\widetilde{b_2} \subseteq \Ccal_{2,2}$ is given by
$$
\hat P^{*1}_2 \times D_* \times \hat R^2_{1*} + \hat P^{2*}_1 \times D_* \times \hat P^{*1}_2 \times [0,1] + \hat P^{2*}_1 \times D_* \times \hat R^1_{*2}.
$$
\end{itemize}
In order to compare these chains to the corresponding pieces of the boundary of $\xi(\Ccal_{2,2})$, it remains to understand the behavior of the markers at the various punctures:
\begin{itemize}
\item Along $\hat B_3$, from the point of view of the two inputs, the first output is at the node of their component, so for the section $\xi$ their markers stay fixed and tangent to the special circle on their component (which is also what happens for the above cycle). In contrast, for the image of $\xi$  the markers at the outputs should both be coupled to the first input, whereas for the above cycle they are fixed to be tangent to their own special circle, which rotates with the marker at the node. This means that in order to describe the corresponding part of $\xi(\Ccal_{2,2})$, we need to correct the markers at the two outputs by turning them in the opposite direction. So
$\xi(\widetilde{b_3})$ is homologous to
\begin{equation}\label{eq:11-term-geom-first}
\hat R^*_{12} \times (-D_* +D_{1^-}+D_{2^-}) \times \hat P^{12}_*.
\end{equation}
\item Along $\hat B_1$, from the point of view of the component containing the first input and first output puncture, the second input sits at the node, and so for the image of $\xi$ their markers should always be tangent to the special circle on that component (which is what happens in the above relative cycle). In contrast, for the image of $\xi$ the other two markers should be coupled to these fixed markers, but instead for the above relative cycle they rotate with the marker at the node. We correct this by undoing this rotation, and so
$\xi(\widetilde{b_1})$ is homologous to
\begin{align}
&\hat P^{*2}_2 \times (D_*-D_{2^+}-D_{2^-}) \times \hat R^1_{1*}\notag\\
& +\hat P^{1*}_1 \times (D_*-D_{2^+}-D_{2^-}) \times \hat P^{*2}_{2}\times [0,1]\notag\\
&+ \hat P^{1*}_1 \times (D_*-D_{2^+}-D_{2^-}) \times \hat R^2_{*2}.
\label{eq:11-term-geom-second}
\end{align}
\item Along $\hat B_2$, from the point of view of the first output and the second input, the first input sits at the node of their component, and so for the image of $\xi$ their markers should always be tangent to the special circle on that component (which is again what happens in the above relative cycle). In contrast, for the image of $\xi$ the marker at the first input
should be coupled to the marker at the first output, whereas for the above relative cycle it moves with the marker at the node. We correct this by undoing that rotation. Because for the section $\xi$ the marker at the second output is coupled to the marker at the first input, this now in turn forces a correction of the marker at the second output in the opposite direction. We conclude that
$\xi(\widetilde{b_2})$ is homologous to
\begin{align}
&\hat P^{*1}_2 \times (D_*-D_{1^+}+D_{2^-}) \times \hat R^2_{1*}\notag\\
&+\hat P^{2*}_1 \times (D_* -D_{1^+}+D_{2^-}) \times \hat P^{*1}_{2}\times [0,1]\notag\\
&+ \hat P^{2*}_1 \times (D_*-D_{1^+}+D_{2^-}) \times \hat R^1_{*2}.
\label{eq:11-term-geom-third}
\end{align}
\begin{remark}\label{rem:11-term-geom}
The formulas \eqref{eq:11-term-geom-first}, \eqref{eq:11-term-geom-second}, and \eqref{eq:11-term-geom-third} are convenient for the following proof of Lemma~\ref{lem:simplified_11_term}. However, for the construction of Floer data in \S\ref{ssec:11-term-implemented} it is useful to observe that in each of the three cases, the markers that are rotated all belong to the same component. So we can actually give the following alternative product descriptions:
\begin{align*}
\xi(\widetilde{b_3}) &= - \hat S^{*}_{12} \times \hat P^{12}_* \qquad \text{(compare \eqref{eq:11-term-geom-first} with \eqref{eq:lift_of_S})}\\
\xi(\widetilde{b_1}) &= \hat N^{*2}_2 \times \hat R^1_{1*}
+ \hat P^{1*} \times \hat N^{*2}_2 \times [0,1]
+ \hat P^{1*}_1 \times \hat U^2_{*2}\\
\xi(\widetilde{b_2}) &= \hat M^{*1}_2 \times \hat R^2_{1*}
+ \hat P^{2*}_1 \times \hat L^{*1}_2 \times [0,1]
+ \hat P^{2*}_1 \times \hat V^1_{*2}
\end{align*}
with suitable 1-chains $M,N,L \subseteq \Mcal_{2,1}$ and 2-chains $U,V \subseteq \Mcal_{1,2}$.
\end{remark}
\end{itemize}
\begin{lemma} \label{lem:simplified_11_term}
In $H_2(\Mcal_{2,2}, (\partial \simplex_1) \times \Mcal_{2,2}^0)$ we have the relation
\begin{align*}
\lefteqn{[\hat R^*_{12} \times D_* \times \hat P^{12}_*] =}\\
&-[D_{1^-} \times \hat R^*_{12} \times \hat P^{12}_*] -[D_{2^-} \times \hat R^*_{12}  \times \hat P^{12}_*]\\
&+[\hat R^*_{12} \times \hat P^{12}_* \times D_{1^+}] +[\hat R^*_{12}  \times \hat P^{12}_* \times D_{2^+}]\notag\\
&+ [\hat P^{*2}_2 \times D_* \times \hat R^1_{1*} +\hat P^{1*}_1 \times D_* \times \hat P^{*2}_{2}\times [0,1]
+ \hat P^{1*}_1 \times D_* \times \hat R^2_{*2}]\notag\\
&+ [\hat P^{*1}_2 \times D_* \times \hat R^2_{1*} +\hat P^{2*}_1 \times D_* \times \hat P^{*1}_{2}\times [0,1]
+ \hat P^{2*}_1 \times D_* \times \hat R^1_{*2}].
\notag
\end{align*}
\end{lemma}
\begin{proof}
It follows from the above discussion that the homology classes of the relative cycles described in equations \eqref{eq:11-term-geom-first}, \eqref{eq:11-term-geom-second} and \eqref{eq:11-term-geom-third} sum to zero,
as these three relative cycles form the relative boundary of the 3-chain $\xi(\Ccal_{2,2})$. The cycle \eqref{eq:11-term-geom-first} can be rewritten as
\begin{align*}
\lefteqn{[\hat R^*_{12} \times (-D_* +D_{1^-}+D_{2^-}) \times \hat P^{12}_*]}&\phantom{blah}\\
&=-[\hat R^*_{12} \times D_* \times \hat P^{12}_*]
-[D_{1^-} \times \hat R^*_{12} \times \hat P^{12}_*]
-[D_{2^-} \times \hat R^*_{12} \times \hat P^{12}_*],
\end{align*}
where the signs in the second and third term come from switching the 1-dimensional domains of $D$ and $R$. Similarly, we have
\begin{align*}
\hat P^{*2}_2 \times (D_*&- D_{2^+}-D_{2^-}) \times \hat R^1_{1*}\\
&= \hat P^{*2}_2 \times D_* \times \hat R^1_{1*}
+ \hat P^{*2}_2 \times \hat R^1_{1*} \times D_{2^+}
- D_{2^-} \times \hat P^{*2}_2 \times \hat R^1_{1*},\\
\hat P^{1*}_1 \times (D_*&- D_{2^+}-D_{2^-}) \times \hat P^{*1}_{1}\times [0,1]\\
&= \hat P^{1*}_1 \times D_* \times \hat P^{*1}_{1}\times [0,1]
+ \hat P^{1*}_1 \times \hat P^{*1}_{1}\times [0,1] \times D_{2^+}\\
& \qquad \qquad \qquad \qquad \qquad \quad \ \  - D_{2^-} \times \hat P^{1*}_1 \times \hat P^{*1}_{1}\times [0,1],\\
\hat P^{1*}_1 \times (D_*&- D_{2^+}-D_{2^-}) \times \hat R^2_{*2}\\
&= \hat P^{1*}_1 \times D_* \times \hat R^2_{*2}
+ \hat P^{1*}_1 \times \hat R^2_{*2} \times D_{2^+}
- D_{2^-} \times \hat P^{1*}_1 \times \hat R^2_{*2},\\
\hat P^{*1}_2 \times (D_*&-D_{1^+}+D_{2^-}) \times \hat R^2_{1*}\\
&= \hat P^{*1}_2 \times D_* \times \hat R^2_{1*}
+ \hat P^{*1}_2 \times \hat R^2_{1*} \times D_{1^+}
+ D_{2^-} \times \hat P^{*1}_2 \times \hat R^2_{1*}, \\
\hat P^{2*}_1 \times (D_*&- D_{1^+}+D_{2^-}) \times \hat P^{*1}_{2}\times [0,1]\\
&= \hat P^{2*}_1 \times D_* \times \hat P^{*1}_{2}\times [0,1]
+ \hat P^{2*}_1 \times \hat P^{*1}_{2}\times [0,1] \times D_{1^+}\\
& \qquad \qquad \qquad \qquad \qquad \quad \ \ + D_{2^-} \times \hat P^{2*}_1 \times \hat P^{*1}_{2}\times [0,1], \\
\hat P^{2*}_1 \times (D_*&- D_{1^+}+D_{2^-}) \times \hat R^1_{*2}\\
&= \hat P^{2*}_1 \times D_* \times \hat R^1_{*2}
+ \hat P^{2*}_1 \times \hat R^1_{*2} \times D_{1^+}
+ D_{2^-} \times \hat P^{2*}_1 \times \hat R^1_{*2}.
\end{align*}
By using the relations from Lemma~\ref{lem:rel_in_h1_of_m22}, we deduce that
the class of the relative cycle \eqref{eq:11-term-geom-second} can be rewritten as
\begin{eqnarray*}
&[\hat P^{*2}_2 \times D_* \times \hat R^1_{1*} + \hat P^{1*}_1 \times D_* \times \hat P^{*1}_{1}\times [0,1] + \hat P^{1*}_1 \times D_* \times \hat R^2_{*2}]\\
&\phantom{333} + [\hat R^*_{12}  \times \hat P^{12}_* \times D_{2^+}] - [D_{2^-} \times \hat R^*_{12}  \times \hat P^{12}_*],
\end{eqnarray*}
and the class of the relative cycle \eqref{eq:11-term-geom-third} can be rewritten as
\begin{eqnarray*}
& [\hat P^{*1}_2 \times D_* \times \hat R^2_{1*} + \hat P^{2*}_1 \times D_* \times \hat P^{*1}_{2}\times [0,1] + \hat P^{2*}_1 \times D_* \times \hat R^1_{*2}]\\
&\phantom{333} + [\hat R^*_{12}  \times \hat P^{21}_* \times D_{1^+}] + [D_{2^-} \times \hat R^*_{12}  \times \hat P^{21}_*] \\
&= [\hat P^{*1}_2 \times D_* \times \hat R^2_{1*} + \hat P^{2*}_1 \times D_* \times \hat P^{*1}_{2}\times [0,1] + \hat P^{2*}_1 \times D_* \times \hat R^1_{*2}]\\
&\phantom{333} + [\hat R^*_{12}  \times \hat P^{12}_* \times D_{1^+}] + [D_{2^-} \times \hat R^*_{12}  \times \hat P^{12}_*].
\end{eqnarray*}
Here the last equality uses the symmetry $[\hat P^{ij}_k]=[\hat P^{ji}_k]$. Now collecting terms yields the desired result.
\end{proof}

\section{BV bialgebra structure on reduced symplectic homology}\label{sec:reduced}

We fix a field of coefficients $R$ and work with a Weinstein domain $W$ of dimension $2n$ that is \emph{strongly $R$-essential} in the sense of~\cite{CO-reduced}, as explained below. We use notation and terminology from~\cite{CO} and~\cite{CO-reduced}.

\begin{definition}[{\cite[Definition~2.1]{CO-reduced}}] \label{defi:R-essential} 
An \emph{$R$-essential Morse function} on $W$ is a defining Morse function $K:W\to\R$ without critical points of index $>n$ and such that the number of its index $n$ critical points is equal to the rank of $H_n(W;R)$. 
We say that $W$ is {\em $R$-essential} if it admits an $R$-essential Morse function $K$.
\end{definition}

To the Weinstein domain $W$ are associated the following two symplectic homology groups. The \emph{reduced symplectic homology} is $\ol{SH}_*(W)=\coker\,\eps$, where $\eps:SH^{-*}(W)\to SH_*(W)$ is the canonical continuation map from symplectic cohomology to symplectic homology. The \emph{symplectic homology relative to the continuation map} is $SH_*(W;\im c)=\colim_{\!\!H}  H_*(\coker\, c_H)$, where $c_H:FC_*(-H)\to FC_*(H)$ is the chain level continuation map associated to a choice of continuation data from $-H$ to $H$, where $H$ is a Hamiltonian that is admissible for symplectic homology and that restricts to an $R$-essential Morse function on $W$. There is a canonical map $\ol{SH}_*(W)\to SH_*(W;\im c)$.

\begin{definition}[{\cite[Definition~5.1]{CO-reduced}}] \label{defi:strongly-R-essential}
A \emph{strongly $R$-essential Weinstein domain} is an $R$-essential Weinstein domain $W$
such that the canonical map $\ol{SH}_*(W)\to SH_*(W;\im c)$ is an isomorphism. 
\end{definition}

Let \emph{shifted reduced symplectic homology} be $\ol{S\H}_*(W)=\ol{SH}_{*+n}(W)$. It was proved in~\cite{CO-reduced} that the pair of pants product on $SH_*(W)$ descends to a unital product $\mu$ of degree $0$ on $\ol{S\H}_*(W)$, with unit denoted $\eta$. (This does not use the assumption of $R$-essentiality and holds for any Liouville domain.) It was also proved in~\cite{CO-reduced} that, given a choice of continuation data $\cD$ from $-K$ to $K$, where $K$ is an $R$-essential Morse function on the strongly $R$-essential Weinstein domain $W$, the coproduct on $SH_*^{>0}(W)$ extends (non-canonically) to a coproduct $\lambda_\cD$ of odd degree $1-2n$ on $\ol{S\H}_*(W)$. Moreover: 

\begin{theorem}[{\cite[Theorem~6.4]{CO-reduced}}] \label{thm:unital_infinitesimal_reduced} Let $W$ be a strongly $R$-essential Weinstein domain of dimension $2n\ge 6$. Shifted reduced symplectic homology 
$$
(\ol{S\H}_*(W),\mu,\lambda_\cD,\eta)
$$ 
is a commutative and cocommutative unital infinitesimal bialgebra.  
\end{theorem}

In this section we prove the BV generalization of this result. 

\begin{theorem} \label{thm:BV_unital_infinitesimal_reduced} 
Let $W$ be a strongly $R$-essential Weinstein domain of dimension $2n\ge 8$. Shifted reduced symplectic homology carries a natural BV operator $\Delta$ of degree $1$ such that
$$
(\ol{S\H}_*(W),\mu,\lambda_\cD,\eta,\Delta)
$$ 
is a (commutative and cocommutative) BV unital infinitesimal bialgebra.  
\end{theorem}
As we discuss in Remark~\ref{rmk:11-term-is-9-term-for-SH} below, for reduced symplectic homology the 11 term relation in the definition of a BV unital infinitesimal bialgebra always reduces to a 9 term relation.

\subsection{The BV operator} \label{sec:BVoperator}
Symplectic homology carries a canonical BV operator, which turns $S\H_*(W)=SH_{*+n}(W)$ into a BV algebra (see~\cite[Chapter~10]{Abouzaid-cotangent} for a proof 
in the case of unit cotangent bundles, which goes through verbatim to arbitrary Liouville domains, and also~\cite{Seidel07}). Here we briefly review the construction, mainly in preparation of its use in our proofs.

\subsubsection{The moduli space $\cM_{1,1}$} The existence of the BV operator is a consequence of the topology of the moduli space $\cM_{1,1}$ of Riemann surfaces with 2 marked points, one input and one output, and with independent markers at the marked points. This moduli space is diffeomorphic to $S^1$, as we now explain. A Riemann surface $\Sigma$ as above is unstable: its group of automorphisms is infinite, canonically isomorphic to $\R$, and its elements are described as translations once we pick a conformal identification of $\Sigma$ minus the two marked points with the cylinder $\R\times S^1$, $S^1=\R/\Z$.  
In the cylinder description, the choice of a marker at, say, the positive puncture, determines a line $\R\times \{\theta\}$ on the cylinder and therefore a marker at the negative puncture.
Measuring the angle at the negative puncture from the marker transported from the positive puncture to the given marker gives a parametrization
$$
\cM_{1,1}\simeq S^1.
$$
Indeed, the automorphism group of the cylinder is $\R\times S^1$, and rotations preserve the relative position of the markers. Note that the construction of this parametrization is symmetric, i.e., one can alternatively measure the angle from  the marker transported from the negative puncture to the given marker at the positive puncture to describe the same map to $S^1$.

The moduli space $\cM_{1,1}$ carries a distinguished point, represented by a surface for which the markers at the positive and negative punctures are aligned. It also carries a canonical $S^1$-action given by rotating the marker at one of the punctures counterclockwise with respect to the orientation in that tangent space, while leaving the other one fixed. 

\subsubsection{Description of the BV operator at chain level in Hamiltonian Floer homology} \label{sss:chainlevel_BV}
Let $\wh W=W\sqcup [1,\infty)\times \p W$ be the symplectic completion of $W$ and $S^1=\R/\Z$. Let $H:S^1\times \wh W\to\R$ be a nondegenerate admissible Hamiltonian, also denoted $H=(H^t)$, $t\in S^1$. Let $J=(J^t)$, $t\in S^1$ be an admissible almost complex structure that is regular for $H$. 

Given a Riemann surface $\Sigma$ with positive punctures $\{z_i^+\}$, negative punctures $\{z_j^-\}$ and markers at the punctures, a collection of \emph{cylindrical ends compatible with the markers}, or a collection of \emph{compatible cylindrical ends}, is a choice of mutually disjoint neighborhoods of the punctures together with biholomorphic identifications of those neighborhoods with $[0,\infty)\times S^1$ for the positive punctures, resp. $(-\infty,0]\times S^1$ for the negative punctures, such that the directions $[0,\infty)\times\{0\}$ and $(-\infty,0]\times \{0\}$ are tangent to the markers.  

We set up a Floer problem parametrized by $\cM_{1,1}$ as follows. Let $\Sigma_\theta$, $\theta\in \cM_{1,1}$ be a family of representatives for the elements of $\cM_{1,1}$, decorated with a smoothly varying in $\theta$ family of compatible cylindrical ends at the punctures, and a smoothly varying in $\theta$ family of $1$-forms $\beta_\theta\in\Omega^1(\Sigma_\theta)$ such that $d\beta_\theta=0$ and $\beta_\theta=dt$ in the cylindrical ends.
\footnote{If we identify $\Sigma_\theta$ with the cylinder $\R\times S^1$,  
one possible such choice is $\beta_\theta=dt$.} 
Choose a smoothly varying family of Hamiltonians $H_\theta^z:\wh W\to \R$, $\theta\in\cM_{1,1}$, $z\in\Sigma_\theta$ with constant slope on $[1,\infty)\times\p W$ and such that $H_\theta^z=H^t$ in the cylindrical ends of $\Sigma_\theta$, with respect to the canonical coordinate $z=(s,t)$. Choose a smoothly varying family of admissible almost complex structures $J_\theta^z$, $\theta\in\cM_{1,1}$, $z\in\Sigma_\theta$ such that $J_\theta^z=J^t$ in the cylindrical ends of $\Sigma_\theta$ for $z=(s,t)$. 

Given $x^\pm\in \cP(H)$ we consider the moduli space of solutions to the Floer problem parametrized by $\cM_{1,1}$ given by 
\begin{align*}
\cM^\Delta(x^+,x^-) = \big\{ (\theta,u)\, : \, \ & \theta\in \cM_{1,1}, \, u:\Sigma_\theta\to \wh W,\\
& (du-\beta_\theta\otimes X_{H_\theta^z})^{0,1} = 0,\\
& u(\pm\infty,\cdot)=x^\pm\big\},
\end{align*}
where the anti-holomorphic part at a point $z\in\Sigma_\theta$ is computed with respect to $J_\theta^z$. For a generic choice of the pair $(H_\theta^z,J_\theta^z)$ this moduli space is a smooth manifold of dimension $\CZ(x^+)-\CZ(x^-)+1$. Its Floer compactification $\ol\cM^\Delta(x^+,x^-)$ is a manifold with boundary with corners such that 
$$
\p \ol\cM^\Delta(x^+,x^-)=\bigcup_{y} \ol\cM(x^+,y)\times \ol\cM^\Delta(y,x^-) \cup \bigcup_{y} \ol\cM^\Delta(x^+,y)\times \ol\cM(y,x^-).
$$
Here $\ol\cM(x^+,y)$ and $\ol\cM(y,x^-)$ are Floer compactifications of moduli spaces of Floer trajectories.
A choice of coherent orientations of the Floer moduli spaces allows for algebraic counts of their elements in the $0$-dimensional case, and these counts define the map 
$$
\Delta:FC_*(H,J)\to FC_{*+1}(H,J),
$$
$$
\Delta x^+ = \sum_{\CZ(x^-)=\CZ(x^+)+1} \#\cM^\Delta(x^+,x^-)x^-.
$$
This is a degree $1$ chain map, i.e., $\p \Delta + \Delta\p =0$, 
where $\p$ is the Floer differential. Moreover $\Delta^2$ is homotopic to zero and therefore $\Delta^2=0$ in homology. (As explained in~\cite{Abouzaid-cotangent}, the map $\Delta^2$ is described by a Floer problem parametrized by a 2-torus $\cM_{1,1}\times\cM_{1,1}\simeq S^1\times S^1$, and the Floer data can be extended to a solid torus with boundary $S^1\times S^1$.)

\subsubsection{The BV operator on (reduced) symplectic homology}
\label{sss:BV_on_limit}
The BV operator defined for a single admissible Hamiltonian $H$ commutes with the continuation maps defined by monotone homotopies~\cite[Chapter~10, \S2.3]{Abouzaid-cotangent} and induces in the direct limit the BV operator $\Delta$ on $SH_*(W)$. That $\Delta$ vanishes on the energy 0 sector is a consequence of the fact that, given an admissible Hamiltonian that is a $C^2$-small Morse function on $W$ and has small slope, there exist regular admissible almost complex structures $J$ that are time-independent~\cite[Theorem~7.3]{SZ92}, and the corresponding Floer moduli problem parametrized by $\cM_{1,1}\simeq S^1$, with asymptotes necessarily given by constant orbits, can be chosen to be regular and to have an $S^1$-symmetry. As a result there are no $0$-dimensional nonempty moduli spaces and, for such a choice of Floer data, the operator $\Delta$ is zero at chain level. The outcome of this discussion is that the BV operator on $SH_*(W)$ descends to $\ol{SH}_*(W)$.

\subsection{Product and coproduct}\label{ssec:prod_and_coprod}

The construction of the product in symplectic homology is standard, see e.g. \cite{Seidel07} or \cite[Chapter 10]{Abouzaid-cotangent}, and the construction of the coproduct was given in detail in \cite{CO-reduced}.
Here we briefly relate these constructions to the discussion in~\S\ref{sec:moduli}. In particular, we discuss the proof of the 7-term relation between $\lambda_\cD$ and $\Delta$, which has not appeared in the literature before.

\subsubsection{The product}
For $i\in \{0,1,2\}$, let nondegenerate admissible\break Hamiltonians $H_i=(H_i^t)$, $t\in S^1$ with positive slopes $b_i$ at infinity satisfying $b_1+b_2 \le b_0$ and admissible almost complex structures $J_i=(J_i^t)$ on $\widehat W$ be given. To define the product
$$
\mu:FC_*(H_1,J_1) \otimes FC_*(H_2,J_2) \to FC_*(H_0,J_0),
$$
we choose the following data:
\begin{itemize}
\item a sphere $\Sigma$ with two ordered input punctures $z_1$ and $z_2$ and one output puncture $z_\out$, with fixed asymptotic markers at the punctures and compatible cylindrical ends,
\item a 1-form $\alpha$ on $\Sigma$ which equals $dt$ in the standard coordinates $(s,t) \in \R \times S^1$ on the cylindrical ends,
\item a family of admissible Hamiltonians $H^z$ with slopes $b(z)$ at infinity parametrized by $z\in \Sigma$ which in the coordinates on the cylindrical ends agree with $H_1^t$, $H_2^t$ and $H_0^t$ near $z_1$, $z_2$ and $z_\out$, respectively, satisfying
$$
d(b(z) \alpha) \le 0,
$$
and
\item a family of admissible almost complex structures $J^z$ parametrized by $z\in \Sigma$ which in the coordinates on the cylindrical ends agree with $J_1^t$, $J_2^t$ and $J_0^t$ near $z_1$, $z_2$ and $z_\out$, respectively.
\end{itemize}
Then the product $\mu$ is defined from the count of finite energy solutions $u: \Sigma \to \widehat W$ to the Floer equation
$$
(du- \alpha \otimes X_{H^z})^{0,1}=0.
$$
To fit with the discussion of~\S\ref{ssec:moduli_2or3to1}, it is convenient to choose the surface $\Sigma$ with its asymptotic markers used in the definition of $\mu$ to correspond to the point $\hat P\in \xi(\Ncal_{2,1})\subseteq \Mcal_{2,1}$. Extending the choices of compatible cylindrical ends, 1-forms, admissible Hamiltonians and admissible almost complex structures smoothly over all domains parametrized by $\xi(\Ncal_{2,1})$, the resulting count of rigid solutions to the Floer equation for domains varying in the 1-dimensional parameter space $\xi(\Ncal_{2,1})$ now give rise to the bracket operation
$$
\beta: FC_*(H_1,J_1) \otimes FC_*(H_2,J_2) \to FC_*(H_0,J_0).
$$
The choices can be further extended to all of $\Mcal_{2,1}$. Notice that we can view the chains $D_1$, $D_2$ and $D_\out$ in $\Mcal_{2,1}$ as suitable products of $\hat P^{12}_\out$ with $\Mcal_{1,1}$. It follows that the solutions of Floer's equation with domains in these chains can be interpreted as suitable compositions of $\mu$ and $\Delta$, and so equation~\eqref{eq:lift_of_Q} translates to the well-known relation
$$
\beta = \Delta \mu - \mu (\Delta \otimes 1 + 1 \otimes \Delta)
$$
between the bracket and the BV-operator.

For the remaining properties of $\mu$, $\beta$ and $\Delta$, we use the relations in the homology of $\Ncal_{3,1}$ and $\Mcal_{3,1}$ obtained in~\S\ref{ssec:moduli_2or3to1}. Now we assume that we are given Hamiltonians $H_1$, $H_2$, $H_3$, $H_{12}$, $H_{23}$, $H_{13}$ and $H_0$ with corresponding slopes at infinity satisfying
$$
b_i+b_j \le b_{ij} \quad\text{and} \quad \max_{\{i,j,k\}=\{1,2,3\}}( b_{ij}+b_{jk}) \le b_0.
$$
As $\p \Ncal_{3,1}$ is identified with a union of products of copies of $\Ncal_{2,1}$, all the choices from before give choices of compatible cylindrical ends, 1-forms, admissibile Hamiltonians and admissible almost complex structures on the components of the nodal surfaces in $\xi(\p \Ncal_{3,1})$, which can be extended to all of $\Mcal_{3,1}$ compatibly with the gluing description of the compactification. Now
\begin{itemize}
\item the count of rigid solutions of Floer's equation for maps $u: \Sigma \to \widehat W$ whose domains $\Sigma$ vary in a 1-dimensional chain in $\Mcal_{3,1}$ with boundary $\hat P^{1*}_\out \times \hat P^{23}_* - \hat P^{*3}_\out \times \hat P^{12}_*$ realizes the chain homotopy for associativity of $\mu$, and
\item the count of rigid solutions of Floer's equation for maps $u: \Sigma \to \widehat W$ whose domains $\Sigma$ vary in the 2-dimensional chain $\xi(\zeta(\Ccal_{3,1}))$ realizes the chain homotopy for the Poisson identity for $\beta$ and $\mu$ (compare equation \eqref{eq:cycles_poisson} of Corollary~\ref{cor:Poisson_moduli}).
\item Reinterpreting the resulting relation as in Corollary~\ref{cor:cycles_7term} yields the 7-term relation between the product $\mu$ and the BV operator $\Delta$.
\end{itemize}
The constructions just outlined give the above relations for fixed choices of Hamiltonians $H_i$. By standard arguments, the products and brackets\break $HF(H_1,J_1) \otimes HF(H_2,J_2) \to HF(H_0,J_0)$ for various choices of the $H_i$ are compatible with continuation maps in a way that preserves the relations, and so eventually one obtains the corresponding relations on the limit, i.e., on (reduced) symplectic homology.
More details are found in \cite[Chapter~10]{Abouzaid-cotangent}, where the detailed construction of the BV structure first appeared. See also \cite{Abouzaid-Groman-Varolgunes} for the construction of a full chain level framed $E_2$ structure.

\subsubsection{The coproduct}

The coproduct $\lambda=\lambda_\cD$ will depend on a choice of a Hamiltonian $K:\widehat W \to \R$ with slope $\varepsilon>0$ smaller than half the smallest action of $\p W$, which restricts to an $R$-essential $C^2$-small Morse function on $W$, and a choice of regular continuation data $\cD$ from $-K$ to $K$ (for a discussion of this dependence, see~\cite[\S4]{CO-reduced}). We fix $\cD$ for the rest of this section, and denote the resulting continuation map by $c_\cD:FC(-K) \to FC(K)$. Below, we will use the generic notation $\im c \subseteq FC(H,J)$ for the image of any continuation map that factors through $c_\cD$.

For $i\in\{0,1,2\}$, let nondegenerate admissible Hamiltonians $H_i=(H_i^t)$, $t\in S^1$ with slopes $b_i$ satisfying $\min(b_1,b_2) \ge b_0+\eps$ and admissible almost complex structures $J_i=J_i^t$ on $\widehat W$ be given. To define the coproduct
$$
\lambda = \lambda_\cD:FC(H_0,J_0) \to FC(H_1,J_1) \otimes FC(H_2,J_2)
$$
we need the following additional data:
\begin{itemize}
\item a sphere $\Sigma$ with one input puncture $z_\inn$ and two ordered output punctures $z_1$ and $z_2$, with fixed asymptotic markers at the punctures and compatible cylindrical ends,
\item a 1-form $\alpha$ on $\Sigma$ which equals $dt$ in the standard coordinates $(s,t)\in \R \times S^1$ on the cylindrical ends,
\item a family of admissible Hamiltonians $H^{(s_1,s_2),z}$ that are parametrized by $((s_1,s_2),z)\in \simplex_1 \times \Sigma$ with slopes $b(s_1,s_2,z)$. As $(s_1,s_2) \in \simplex_1$ varies from $(0,1)$ to $(1,0)$, they interpolate between the following two limiting configurations:
\begin{itemize}
\item at $(0,1)$, we consider a Hamiltonian $H^{(0,1),z}$ parame\-trized by $z\in \Sigma$ which agrees with $H_0^t$ near $z_\inn$ and with the Hamiltonian $H_2^t$ near $z_2$, and agrees with $-K$ in the standard coordinates at $z_1$. This is formally glued to a continuation cylinder from $-K$ to $K$ using the data $\cD$ above, and another continuation cylinder from $K$ to $H_1^t$.
\item at $(1,0)$,  we consider a Hamiltonian $H^{(1,0),z}$ parame\-trized by $z\in \Sigma$ which agrees with $H_0^t$ near $z_\inn$ and with the Hamiltonian $H_1^t$ near $z_1$, and agrees with $-K$ in the standard coordinates at $z_2$. Again, the latter is formally glued to a continuation cylinder from $-K$ to $K$ using the data $\cD$ above, and a continuation cylinder from $K$ to $H_2^t$.
\end{itemize}
For $(s_1,s_2) \in \mathrm{int}(\simplex_1)$, the Hamiltonian $H^{s_1,z}$ is defined on $\Sigma$ and agrees with $H_i^t$ sufficiently far out in the standard coordinates near the puncture $z_i$ for $i\in \{1,2\}$. Moreover, for all $(s_1,s_2) \in \simplex_1$ we assume that
$$
d(b(s_1,s_2,z)\alpha) \le 0.
$$

\begin{figure}[h]
\begin{tikzpicture}[scale=0.65]

\draw[thick, fill = blue, fill opacity=0.2] (-3.08,.92) circle (1.5);
\draw[thick, dotted] (-4.58,0.92) arc (180:0: 1.5 and 0.3);
\draw[thick] (-4.58,0.92) arc (180:360: 1.5 and 0.3);

\draw[thick, fill = yellow, fill opacity=0.2] (-5,-0.7071) circle (1);
\draw[thick, dotted] (-6,-0.7071) arc (180:0: 1 and 0.2);
\draw[thick] (-6,-0.7071) arc (180:360: 1 and 0.2);

\draw[thick, fill = red, fill opacity=0.7] (-5,-2.7071) circle (1);
\draw[thick, dotted] (-6,-2.7071) arc (180:0: 1 and 0.2);
\draw[thick] (-6,-2.7071) arc (180:360: 1 and 0.2);

\draw[thick] (-3.08,2.32) to (-3.08,2.52);
\node at (-3.08,2.82) {\tiny $\inn$};
\node at (-2.35, 2.6) {\tiny $H_0$};

\draw[thick] (-4.1929,0.1) to (-4.3929,-0.1);
\node at (-5.1,0.6) {\tiny $-K$};

\draw[thick] (-2, 0.06) to (-1.8,-0.1);
\node at (-1.6,-0.3) {\tiny $2$};
\node at (-1.3, 0.2) {\tiny $H_2$};

\draw[thick] (-5, -1.6071) to (-5,-1.8071);
\node at (-5.8,-1.7) {\tiny $K$};

\draw[thick] (-5, -3.6071) to (-5,-3.8071);
\node at (-5,-4.1071) {\tiny $1$};
\node at (-5.7,-3.9) {\tiny $H_1$};

\draw[thick, fill=blue, fill opacity=0.2] (1,-2) circle (1.5);
\draw[thick, dotted] (-0.5,-2) arc (180:0: 1.5 and 0.3);
\draw[thick] (-0.5,-2) arc (180:360: 1.5 and 0.3);

\draw[thick] (1,-0.62) to (1,-0.38);
\node at (1,-0.1) {\tiny $\inn$};
\node at (0.3,-0.3) {\tiny $H_0$};

\draw[thick] (-0.08, -2.82) to (-0.3,-3.02);
\node at (-0.5,-3.3) {\tiny $1$};
\node at (-0.7, -2.82) {\tiny $H_1$};

\draw[thick] (2.08, -2.82) to (2.3,-3.02);
\node at (2.5,-3.3) {\tiny $2$};
\node at (2.7, -2.82) {\tiny $H_2$};

\pgftransformxshift{2cm}

\draw[thick, fill=blue, fill opacity=0.2] (3.08,0.92) circle (1.5);
\draw[thick, dotted] (1.58,0.92) arc (180:0: 1.5 and 0.3);
\draw[thick] (1.58,0.92) arc (180:360: 1.5 and 0.3);

\draw[thick, fill=yellow, fill opacity=0.2] (5,-0.7071) circle (1);
\draw[thick, dotted] (4,-0.7071) arc (180:0: 1 and 0.2);
\draw[thick] (4,-0.7071) arc (180:360: 1 and 0.2);

\draw[thick, fill=red, fill opacity=0.7] (5,-2.7071) circle (1);
\draw[thick, dotted] (4,-2.7071) arc (180:0: 1 and 0.2);
\draw[thick] (4,-2.7071) arc (180:360: 1 and 0.2);

\draw[thick] (3.08,2.32) to (3.08,2.52);
\node at (3.08,2.82) {\tiny $\inn$};
\node at (3.78,2.6) {\tiny $H_0$};

\draw[thick] (2., 0.06) to (1.8,-0.1);
\node at (1.6,-0.3) {\tiny $1$};
\node at (1.3, 0.2) {\tiny $H_1$};

\draw[thick] (4.1929,0.1) to (4.3929,-0.1);
\node at (5.1,0.6) {\tiny $-K$};

\draw[thick] (5, -1.6071) to (5,-1.8071);
\node at (5.8,-1.7) {\tiny $K$};

\draw[thick] (5, -3.6071) to (5,-3.8071);
\node at (5,-4.1071) {\tiny $2$};
\node at (5.7,-3.9) {\tiny $H_2$};

\pgftransformxshift{-2cm}
\draw[thick] (-5,-6) to (7,-6);
\draw[thick] (-5,-5.9) to (-5,-6.1);
\draw[thick] (7,-5.9) to (7,-6.1);

\node at (-5,-5.5) {\tiny $(s_1,s_2)=(0,1)$};
\node at (7,-5.5) {\tiny $(s_1,s_2)=(1,0)$};
\node at (1,-5) {\tiny $(s_1,s_2) \in \mathrm{int}(\simplex_1)$};
\end{tikzpicture}

\caption{A schematic picture of the Floer data for $(s_1,s_2) \in \simplex_1$ in the definition of $\lambda_\cD$. The punctures and nodal points are labelled with the Hamiltonians used as asymptotic data. \label{fig:coproduct}} 
\end{figure}
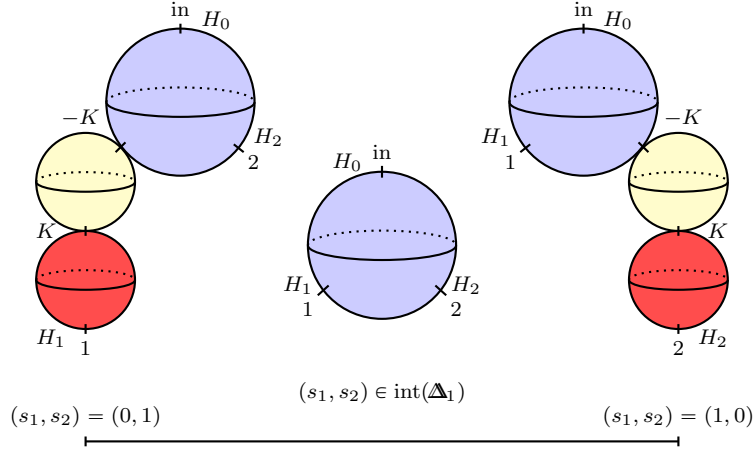

\noindent

\item a family of admissible almost complex structures $J^{(s_1,s_2),z}$ parame\-trized
by $((s_1,s_2),z) \in \simplex_1 \times \Sigma$ which in the standard coordinates on the cylindrical ends agree with $J^t_i$ near $z_i$. In particular, for $(s_1,s_2)=(0,1)$, we use $J_1^t$ on the continuation cylinders near $z_1$, and for $(s_1,s_2)=(1,0)$ we use $J_2^t$ on the continuation cylinders near $z_2$.
\end{itemize}
Given this data, the coproduct $\lambda_\cD$ is defined from the count of rigid finite energy solutions $u: \Sigma \to \widehat W$ of the Floer problem parametrized by $(s_1,s_2) \in \simplex_1$. To conclude that under the dimensional restriction $2n \ge 6$ we in fact get a well-defined map
$$
\lambda_\cD: FC(H_0,J_0)/(\im c) \to FC(H_1,J_1)/(\im c) \otimes FC(H_2,J_2)/(\im c),
$$
we use the index argument of \cite[Lemma~4.1]{CO-reduced} to see that $\im c \subseteq FC(H_0,J_0)$ is mapped to $\im c \otimes FC(H_2,J_2) + FC(H_1,J_1) \otimes \im c$.
\begin{remark}
As we work in Floer theory relative to $\im c$, the contributions of $\p \simplex_1$ in the definition of $\lambda_\cD$ are zero. This is why we worked with homology relative to the boundary of the simplex-factor in~\S\ref{ssec:moduli_1toell}.
\end{remark}
\begin{remark} \label{rmk:11-term-is-9-term-for-SH}
In~\S\ref{sss:BV_on_limit} we have seen that the BV operator vanishes on low energy classes. As the unit $\eta$ comes from $FC_*(K)$, a standard argument using Stokes' theorem shows that its coproduct $\lambda_\cD\eta$ is supported at total energy $\le \eps$. We can therefore conclude that $(1\otimes\Delta)\lambda\eta=0$, so that on reduced symplectic homology the 11-term relation reduces to a 9-term relation.  
\end{remark}

To fit with the discussion of~\S\ref{ssec:moduli_1toell}, it is convenient to choose the surface with its asymptotic markers used in the definition of $\lambda$ to correspond to the point $\xi(\zeta(\Ccal^0_{1,2})) \in \Mcal^0_{1,2}$, so that the parameter space of our parametrized Floer problem is identified with $\hat R= \xi_*(\simplex_1 \times \zeta(\Ccal^0_{1,2}))$. As before in the product case, we extend all the choices consistently to $\xi(\Ncal_{1,2})$. Then the relation $[R]=-[\tau R]$ in the relative homology of $\Ncal_{1,2}$ yields the cocommutativity of $\lambda_\cD$. Moreover, the count of rigid solutions to the Floer problem parametrized by $\hat S=\xi(\Ncal_{1,2}) \in \Mcal_{1,2}$ gives rise to the cobracket operation
$$
\gamma_\cD:FC(H_0,J_0)/(\im c) \to FC(H_1,J_1)/(\im c) \otimes FC(H_2,J_2)/(\im c).
$$
After further extending our choices of data to all of $\Mcal_{1,2}$, the identity \eqref{eq:lift_of_S} gives rise to the algebraic relation
$$
\gamma_\cD = (\Delta \otimes 1 + 1 \otimes \Delta) \lambda_\cD + \lambda_\cD\Delta.
$$

For the remaining properties of $\lambda_\cD$, $\gamma_\cD$ and $\Delta$ we use the relations in the homology of $\Ncal_{1,3}$ and $\Mcal_{1,3}$ obtained in~\S\ref{ssec:moduli_1toell}. Now we fix Hamiltonians $H_0$, $H_1$, $H_2$, $H_3$, $H_{12}=H_{21}$, $H_{13}=H_{31}$ and $H_{23}=H_{32}$ with corresponding slopes satisfying 
\begin{align}
b_i &\ge b_0+2 \eps \qquad \text{for }i\in \{1,2,3\},\label{eq:conditions_on_slopes}\\
\min(b_i,b_j) &\ge b_{ij}+\eps \qquad \text{and} \notag\\
\min(b_i,b_{jk}) &\ge b_0+\eps \qquad \text{for any bijection } \{i,j,k\} \cong \{1,2,3\}. \notag
\end{align}
We also fix generic almost complex structures $J_i$ and $J_{ij}$.

Recall that the boundary of $\Ncal_{1,3}$ splits as
$$
\p \Ncal_{1,3} = (\p \simplex_2) \times \Ncal_{1,3}^0 \cup\bigcup_{i=1}^3 \left(B_i \cup T_i\right),
$$
and each of the $B_i$ is in bijective correspondence with a suitable product of two copies of $\Ncal_{1,2}$, so we have choices of compatible cylindrical ends, 1-forms, admissible Hamiltonians and admissible almost-complex structures on the components of the nodal surfaces corresponding to points in $\xi(B_i)$. There is no issue with extending cylindrical ends and almost-complex structures over the remaining boundary parts and into the interior. We claim that the blow-up also makes it possible to extend the Hamiltonians and 1-forms first to all of $\xi(\p \Ncal_{1,3})$, in such a way that whenever the weight at one of the output punctures is $0$, the corresponding Floer data near that puncture carries two extra cylindrical components, with one a continuation cylinder from $-K$ to $K$ using the fixed data $\cD$, and the other a continuation cylinder from $K$ to the relevant Hamiltonian $H_i$ corresponding to that puncture. 

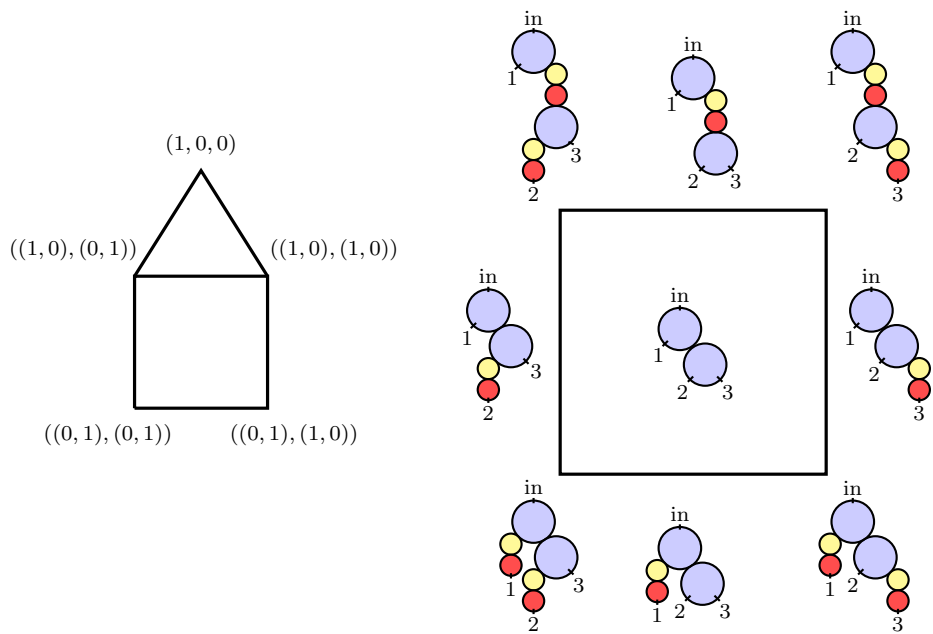
\begin{figure}[ht]
\begin{center}
\begin{tikzpicture}[scale=0.35]

\draw[very thick] (5,5) -- (10,5) -- (10,10) -- (5, 10) -- (5,5);
\draw[very thick] (5,10) -- (7.5,14) -- (10,10);

\node at (4,4) {\tiny $((0,1),(0,1))$};
\node at (11,4) {\tiny $((0,1),(1,0))$};
\node at (2.7,11) {\tiny $((1,0),(0,1))$};
\node at (12.5,11) {\tiny $((1,0),(1,0))$};
\node at (7.5,15) {\tiny $(1,0,0)$};

\pgftransformxshift{21cm}
\draw[very thick] (0,2.5) -- (10,2.5) -- (10,12.5) -- (0,12.5) -- cycle;


\draw[thick, fill = blue, fill opacity=0.2] (-1,18.5) circle (0.8);
\draw[thick, fill = yellow, fill opacity=0.4] (-0.15,17.65) circle (0.4);
\draw[thick, fill = red, fill opacity=0.7] (-0.15,16.85) circle (0.4);
\draw[thick, fill = blue, fill opacity=0.2] (-0.15,15.65) circle (0.8);
\draw[thick, fill = yellow, fill opacity=0.4] (-1,14.80) circle (0.4);
\draw[thick, fill = red, fill opacity=0.7] (-1,14) circle (0.4);

\draw[thick] (-1,19.2) -- (-1,19.4);
\node at (-1,19.8) {\tiny $\inn$};
\draw[thick] (-1.47,18.04) -- (-1.7,17.84);
\node at (-1.8,17.5) {\tiny $1$};
\draw[thick] (0.31,15.09) -- (0.51,14.99);
\node at (0.6, 14.6) {\tiny $3$};
\draw[thick] (-1,13.7) -- (-1,13.5);
\node at (-1,13.1) {\tiny $2$};

\draw[thick, fill = blue, fill opacity=0.2] (5,17.5) circle (0.8);
\draw[thick, fill = yellow, fill opacity=0.4] (5.85,16.65) circle (0.4);
\draw[thick, fill = red, fill opacity=0.7] (5.85,15.85) circle (0.4);
\draw[thick, fill = blue, fill opacity=0.2] (5.85,14.65) circle (0.8);

\draw[thick] (5,18.2) -- (5,18.4);
\node at (5,18.8) {\tiny $\inn$};
\draw[thick] (4.53,17.04) -- (4.33,16.84);
\node at (4.2,16.5) {\tiny $1$};
\draw[thick] (6.32,14.18) -- (6.52,13.98);
\node at (6.6, 13.6) {\tiny $3$};
\draw[thick] (5.43,14.18) -- (5.23,13.98);
\node at (5,13.6) {\tiny $2$};

\draw[thick, fill = blue, fill opacity=0.2] (11,18.5) circle (0.8);
\draw[thick, fill = yellow, fill opacity=0.4] (11.85,17.65) circle (0.4);
\draw[thick, fill = red, fill opacity=0.7] (11.85,16.85) circle (0.4);
\draw[thick, fill = blue, fill opacity=0.2] (11.85,15.65) circle (0.8);
\draw[thick, fill = yellow, fill opacity=0.4] (12.7,14.8) circle (0.4);
\draw[thick, fill = red, fill opacity=0.7] (12.697,14) circle (0.4);

\draw[thick] (11,19.2) -- (11,19.4);
\node at (11,19.8) {\tiny $\inn$};
\draw[thick] (10.53,18.04) -- (10.33,17.84);
\node at (10.2,17.5) {\tiny $1$};
\draw[thick] (11.35,15.19) -- (11.15,14.99);
\node at (11, 14.6) {\tiny $2$};
\draw[thick] (12.7,13.7) -- (12.7,13.5);
\node at (12.7,13.1) {\tiny $3$};


\draw[thick, fill = blue, fill opacity=0.2] (-2.7,8.7) circle (0.8);
\draw[thick, fill = blue, fill opacity=0.2] (-1.85,7.35) circle (0.8);
\draw[thick, fill = yellow, fill opacity=0.4] (-2.7,6.5) circle (0.4);
\draw[thick, fill = red, fill opacity=0.7] (-2.7,5.7) circle (0.4);

\draw[thick] (-2.7,9.4) -- (-2.7,9.6);
\node at (-2.7,10) {\tiny $\inn$};
\draw[thick] (-3.16,8.23) -- (-3.36,8.03);
\node at (-3.5,7.65) {\tiny $1$};
\draw[thick] (-1.38,6.88) -- (-1.19,6.68);
\node at (-0.9, 6.4) {\tiny $3$};
\draw[thick] (-2.7,5.4) -- (-2.7,5.2);
\node at (-2.7,4.8) {\tiny $2$};


\draw[thick, fill = blue, fill opacity=0.2] (-1,0.7) circle (0.8);
\draw[thick, fill = yellow, fill opacity=0.4] (-1.85,-0.15) circle (0.4);
\draw[thick, fill = red, fill opacity=0.7] (-1.85,-0.95) circle (0.4);
\draw[thick, fill = blue, fill opacity=0.2] (-0.15,-0.65) circle (0.8);
\draw[thick, fill = yellow, fill opacity=0.4] (-1,-1.5) circle (0.4);
\draw[thick, fill = red, fill opacity=0.7] (-1,-2.3) circle (0.4);

\draw[thick] (-1,1.4) -- (-1,1.6);
\node at (-1,2) {\tiny $\inn$};
\draw[thick] (-1.85,-1.25) -- (-1.85,-1.45);
\node at (-1.85,-1.85) {\tiny $1$};
\draw[thick] (-1,-2.6) -- (-1,-2.8);
\node at (-1, -3.2) {\tiny $2$};
\draw[thick] (0.3,-1.1) -- (0.5,-1.3);
\node at (0.7,-1.7) {\tiny $3$};


\draw[thick, fill = blue, fill opacity=0.2] (4.5,-0.3) circle (0.8);
\draw[thick, fill = yellow, fill opacity=0.4] (3.65,-1.15) circle (0.4);
\draw[thick, fill = red, fill opacity=0.7] (3.65,-1.95) circle (0.4);
\draw[thick, fill = blue, fill opacity=0.2] (5.35,-1.65) circle (0.8);

\draw[thick] (4.5,0.4) -- (4.5,0.6);
\node at (4.5,1) {\tiny $\inn$};
\draw[thick] (3.65,-2.25) -- (3.65,-2.45);
\node at (3.65,-2.85) {\tiny $1$};
\draw[thick] (4.9,-2.1) -- (4.7,-2.3);
\node at (4.5, -2.7) {\tiny $2$};
\draw[thick] (5.8,-2.1) -- (6,-2.3);
\node at (6.2,-2.7) {\tiny $3$};


\draw[thick, fill = blue, fill opacity=0.2] (11,0.7) circle (0.8);
\draw[thick, fill = yellow, fill opacity=0.4] (10.15,-0.15) circle (0.4);
\draw[thick, fill = red, fill opacity=0.7] (10.15,-0.95) circle (0.4);
\draw[thick, fill = blue, fill opacity=0.2] (11.85,-0.65) circle (0.8);
\draw[thick, fill = yellow, fill opacity=0.4] (12.7,-1.5) circle (0.4);
\draw[thick, fill = red, fill opacity=0.7] (12.7,-2.3) circle (0.4);

\draw[thick] (11,1.4) -- (11,1.6);
\node at (11,2) {\tiny $\inn$};
\draw[thick] (10.15,-1.25) -- (10.15,-1.45);
\node at (10.15,-1.85) {\tiny $1$};
\draw[thick] (11.2,-1.3) -- (11.4,-1.1);
\node at (11, -1.7) {\tiny $2$};
\draw[thick] (12.7,-2.6) -- (12.7,-2.8);
\node at (12.7,-3.2) {\tiny $3$};


\draw[thick, fill = blue, fill opacity=0.2] (11.7,8.7) circle (0.8);
\draw[thick, fill = blue, fill opacity=0.2] (12.65,7.35) circle (0.8);
\draw[thick, fill = yellow, fill opacity=0.4] (13.5,6.5) circle (0.4);
\draw[thick, fill = red, fill opacity=0.7] (13.5,5.7) circle (0.4);

\draw[thick] (11.7,9.4) -- (11.7,9.6);
\node at (11.7,10) {\tiny $\inn$};
\draw[thick] (11.24,8.23) -- (11.04,8.03);
\node at (10.9,7.65) {\tiny $1$};
\draw[thick] (12.2,6.9) -- (12,6.7);
\node at (11.8, 6.3) {\tiny $2$};
\draw[thick] (13.5,5.4) -- (13.5,5.2);
\node at (13.5,4.8) {\tiny $3$};


\draw[thick, fill = blue, fill opacity=0.2] (4.5,8) circle (0.8);
\draw[thick, fill = blue, fill opacity=0.2] (5.45,6.65) circle (0.8);

\draw[thick] (4.5,8.7) -- (4.5,8.9);
\node at (4.5,9.3) {\tiny $\inn$};
\draw[thick] (4.05,7.55) -- (3.85,7.35);
\node at (3.7,7) {\tiny $1$};
\draw[thick] (5,6.2) -- (4.8,6);
\node at (4.6, 5.6) {\tiny $2$};
\draw[thick] (5.9,6.2) -- (6.1,6);
\node at (6.3,5.6) {\tiny $3$};

\end{tikzpicture}
\caption{According to \S\ref{ssec:moduli_1toell}, the blow up of the boundary component $B'_1 \subset \Ncal'_{1,3}$ is a union $B_1 \cup T_1$, where $B_1$ is the product of a 2-torus with a square $\simplex_1 \times \simplex_1$ and $T_1$ is the product of a 2-torus with a triangle. The square and triangle fit together as depicted on the left. On the right, we show the Floer data on the square, which comes from its parametrization by products of copies of $\Ncal_{1,2}$. The colors agree with Figure~\ref{fig:coproduct}. \label{fig:square_labelled}} 
\end{center}
\end{figure}

To explain how this extension works near the blown up boundary component $B_1$, we use Figures~\ref{fig:square_labelled} and \ref{fig:triangle_labelled}. Consider first the left hand side of Figure~\ref{fig:square_labelled}, where we have labelled the vertices of the square by their weights $((s_1,s_*),(s_2,s_3))$, and the top vertex of the triangle by its weights $(s_1,s_2,s_3)$. Note that without the blow-up, the weights $s_2$ and $s_3$ from the parametrization of $B_1$ via the product of two copies of $\Ncal_{1,2}$ would not extend continuously to the rest of $\Ncal_{1,3}$. On the right hand side, we have drawn the data that comes from the product parametrization of the square. Here we use the same color conventions as in Figure~\ref{fig:coproduct}, i.e., blue for ``main'' components, yellow for continuation cylinders from $-K$ to $K$ using the fixed choice of data $\cD$, and red for continuation cylinders from $K$ to the Hamiltonian at the relevant puncture or nodal point. In particular, at the input puncture of the second main component we always use the Hamiltonian $H_{23}$ as asymptotic data.

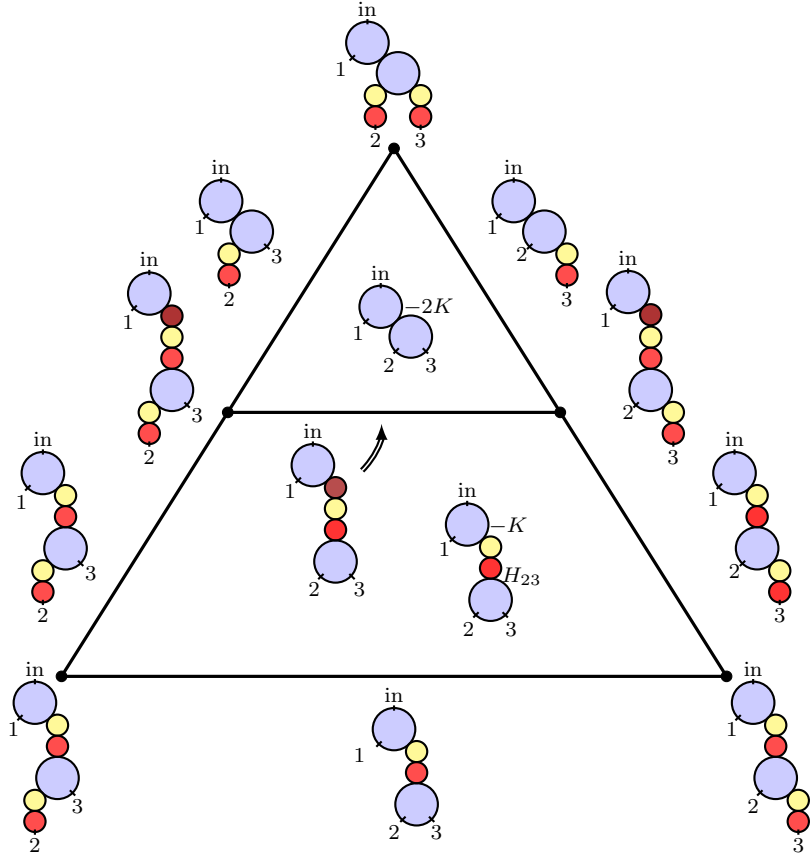
\begin{figure}[h]
\begin{center}
\begin{tikzpicture}[scale=0.35]
\definecolor{darkred}{rgb}{0.6,0,0};

\draw[very thick] (0,5) -- (25,5) -- (12.5,25) -- (0,5);
\draw[very thick] (6.25,15) -- (18.75,15);

\draw[fill] (0,5) circle (0.2);
\draw[fill] (25,5) circle (0.2);
\draw[fill] (12.5,25) circle (0.2);

\draw[fill] (6.25,15) circle (0.2);
\draw[fill] (18.75,15) circle (0.2);

\draw[thick, fill = blue, fill opacity=0.2] (-1,4) circle (0.8);
\draw[thick, fill = yellow, fill opacity=0.4] (-0.1515,3.1515) circle (0.4);
\draw[thick, fill = red, fill opacity=0.7] (-0.1515,2.3515) circle (0.4);
\draw[thick, fill = blue, fill opacity=0.2] (-0.1515,1.1515) circle (0.8);
\draw[thick, fill = yellow, fill opacity=0.4] (-1,0.303) circle (0.4);
\draw[thick, fill = red, fill opacity=0.7] (-1,-0.497) circle (0.4);

\draw[thick] (-1,4.7) -- (-1,4.9);
\node at (-1,5.3) {\tiny $\inn$};
\draw[thick] (-1.4657,3.5344) -- (-1.657,3.3344);
\node at (-1.8,3) {\tiny $1$};
\draw[thick] (0.3142,0.6858) -- (0.5142,0.4858);
\node at (0.6, 0.1) {\tiny $3$};
\draw[thick] (-1,-.8) -- (-1,-1);
\node at (-1,-1.4) {\tiny $2$};

\draw[thick, fill = blue, fill opacity=0.2] (12.5,3) circle (0.8);
\draw[thick, fill = yellow, fill opacity=0.4] (13.3485,2.1515) circle (0.4);
\draw[thick, fill = red, fill opacity=0.7] (13.3485,1.3515) circle (0.4);
\draw[thick, fill = blue, fill opacity=0.2] (13.3485,0.1515) circle (0.8);

\draw[thick] (12.5,3.7) -- (12.5,3.9);
\node at (12.5,4.3) {\tiny $\inn$};
\draw[thick] (12.0343,2.5344) -- (11.8343,2.3344);
\node at (11.2,2) {\tiny $1$};
\draw[thick] (13.8142,-0.3142) -- (14.0142,-0.5142);
\node at (14.1, -.9) {\tiny $3$};
\draw[thick] (12.9343,-0.3142) -- (12.7343,-0.5142);
\node at (12.5,-0.9) {\tiny $2$};

\draw[thick, fill = blue, fill opacity=0.2] (26,4) circle (0.8);
\draw[thick, fill = yellow, fill opacity=0.4] (26.8485,3.1515) circle (0.4);
\draw[thick, fill = red, fill opacity=0.7] (26.8485,2.3515) circle (0.4);
\draw[thick, fill = blue, fill opacity=0.2] (26.8485,1.1515) circle (0.8);
\draw[thick, fill = yellow, fill opacity=0.4] (27.697,0.303) circle (0.4);
\draw[thick, fill = red, fill opacity=0.7] (27.697,-0.497) circle (0.4);

\draw[thick] (26,4.7) -- (26,4.9);
\node at (26,5.3) {\tiny $\inn$};
\draw[thick] (25.5343,3.5344) -- (25.3343,3.3344);
\node at (25.2,3) {\tiny $1$};
\draw[thick] (26.3518,0.6858) -- (26.1518,0.4858);
\node at (26, 0.1) {\tiny $2$};
\draw[thick] (27.7,-.8) -- (27.7,-1);
\node at (27.7,-1.4) {\tiny $3$};


\draw[thick, fill = blue, fill opacity=0.2] (-0.7,12.7) circle (0.8);
\draw[thick, fill = yellow, fill opacity=0.4] (0.1515,11.85) circle (0.4);
\draw[thick, fill = red, fill opacity=0.7] (0.1515,11.05) circle (0.4);
\draw[thick, fill = blue, fill opacity=0.2] (0.1515,9.8485) circle (0.8);
\draw[thick, fill = yellow, fill opacity=0.4] (-0.7,9) circle (0.4);
\draw[thick, fill = red, fill opacity=0.7] (-0.7,8.2) circle (0.4);

\draw[thick] (-0.7,13.4) -- (-0.7,13.6);
\node at (-0.7,14) {\tiny $\inn$};
\draw[thick] (-1.166,12.24) -- (-1.366,12.04);
\node at (-1.5,11.6) {\tiny $1$};
\draw[thick] (0.6192,9.3828) -- (0.8192,9.1828);
\node at (1.1, 8.9) {\tiny $3$};
\draw[thick] (-0.7,7.9) -- (-0.7,7.7);
\node at (-0.7,7.3) {\tiny $2$};


\draw[thick, fill = blue, fill opacity=0.2] (3.3,19.5) circle (0.8);
\draw[thick, fill = darkred, fill opacity=0.8] (4.15,18.65) circle (0.4);
\draw[thick, fill = yellow, fill opacity=0.4] (4.15,17.85) circle (0.4);
\draw[thick, fill = red, fill opacity=0.7] (4.15,17.05) circle (0.4);
\draw[thick, fill = blue, fill opacity=0.2] (4.15,15.85) circle (0.8);
\draw[thick, fill = yellow, fill opacity=0.4] (3.3,15) circle (0.4);
\draw[thick, fill = red, fill opacity=0.7] (3.3,14.2) circle (0.4);

\draw[thick] (3.3,20.2) -- (3.3,20.4);
\node at (3.3,20.8) {\tiny $\inn$};
\draw[thick] (2.83,19.05) -- (2.63,18.85);
\node at (2.5,18.4) {\tiny $1$};
\draw[thick] (4.62,15.38) -- (4.82,15.18);
\node at (5.1, 14.9) {\tiny $3$};
\draw[thick] (3.3,13.9) -- (3.3,13.7);
\node at (3.3,13.3) {\tiny $2$};


\draw[thick, fill = blue, fill opacity=0.2] (6,23) circle (0.8);
\draw[thick, fill = blue, fill opacity=0.2] (7.1515,21.8485) circle (0.8);
\draw[thick, fill = yellow, fill opacity=0.4] (6.3,21) circle (0.4);
\draw[thick, fill = red, fill opacity=0.7] (6.3,20.2) circle (0.4);

\draw[thick] (6,23.7) -- (6,23.9);
\node at (6,24.3) {\tiny $\inn$};
\draw[thick] (5.534,22.5344) -- (5.334,22.3344);
\node at (5.2,22) {\tiny $1$};
\draw[thick] (7.6192,21.3828) -- (7.8192,21.1828);
\node at (8.1, 20.9) {\tiny $3$};
\draw[thick] (6.3,19.9) -- (6.3,19.7);
\node at (6.3,19.3) {\tiny $2$};


\draw[thick, fill = blue, fill opacity=0.2] (11.5,29) circle (0.8);
\draw[thick, fill = blue, fill opacity=0.2] (12.6515,27.8485) circle (0.8);
\draw[thick, fill = yellow, fill opacity=0.4] (11.8,27) circle (0.4);
\draw[thick, fill = red, fill opacity=0.7] (11.8,26.2) circle (0.4);
\draw[thick, fill = yellow, fill opacity=0.4] (13.5,27) circle (0.4);
\draw[thick, fill = red, fill opacity=0.7] (13.5,26.2) circle (0.4);

\draw[thick] (11.5,29.7) -- (11.5,29.9);
\node at (11.5,30.3) {\tiny $\inn$};
\draw[thick] (11.03,28.53) -- (10.83,28.33);
\node at (10.5,28) {\tiny $1$};
\draw[thick] (11.8,25.9) -- (11.8,25.7);
\node at (11.8, 25.3) {\tiny $2$};
\draw[thick] (13.5,25.9) -- (13.5,25.7);
\node at (13.5,25.3) {\tiny $3$};


\draw[thick, fill = blue, fill opacity=0.2] (17,23) circle (0.8);
\draw[thick, fill = blue, fill opacity=0.2] (18.1515,21.8485) circle (0.8);
\draw[thick, fill = yellow, fill opacity=0.4] (19,21) circle (0.4);
\draw[thick, fill = red, fill opacity=0.7] (19,20.2) circle (0.4);

\draw[thick] (17,23.7) -- (17,23.9);
\node at (17,24.3) {\tiny $\inn$};
\draw[thick] (16.5344,22.5344) -- (16.3344,22.3344);
\node at (16.2,22) {\tiny $1$};
\draw[thick] (17.7,21.3828) -- (17.5,21.1828);
\node at (17.3, 21) {\tiny $2$};
\draw[thick] (19,19.9) -- (19,19.7);
\node at (19,19.3) {\tiny $3$};


\draw[thick, fill = blue, fill opacity=0.2] (21.3,19.56) circle (0.8);
\draw[thick, fill = darkred, fill opacity=0.8] (22.15,18.7) circle (0.4);
\draw[thick, fill = yellow, fill opacity=0.4] (22.15,17.85) circle (0.4);
\draw[thick, fill = red, fill opacity=0.7] (22.15,17.05) circle (0.4);
\draw[thick, fill = blue, fill opacity=0.2] (22.15,15.85) circle (0.8);
\draw[thick, fill = yellow, fill opacity=0.4] (23,15) circle (0.4);
\draw[thick, fill = red, fill opacity=0.7] (23,14.2) circle (0.4);

\draw[thick] (21.3,20.26) -- (21.3,20.46);
\node at (21.3,20.85) {\tiny $\inn$};
\draw[thick] (20.83,19.09) -- (20.63,18.89);
\node at (20.45,18.5) {\tiny $1$};
\draw[thick] (21.7,15.3828) -- (21.5,15.1828);
\node at (21.3, 15) {\tiny $2$};
\draw[thick] (23,13.9) -- (23,13.7);
\node at (23,13.3) {\tiny $3$};


\draw[thick, fill = blue, fill opacity=0.2] (25.3,12.7) circle (0.8);
\draw[thick, fill = yellow, fill opacity=0.4] (26.15,11.85) circle (0.4);
\draw[thick, fill = red, fill opacity=0.8] (26.15,11.05) circle (0.4);
\draw[thick, fill = blue, fill opacity=0.2] (26.15,9.85) circle (0.8);
\draw[thick, fill = yellow, fill opacity=0.4] (27,9) circle (0.4);
\draw[thick, fill = red, fill opacity=0.8] (27,8.2) circle (0.4);

\draw[thick] (25.3,13.4) -- (25.3,13.6);
\node at (25.3,14) {\tiny $\inn$};
\draw[thick] (24.83,12.3) -- (24.63,12.1);
\node at (24.5,11.7) {\tiny $1$};
\draw[thick] (25.7,9.38) -- (25.5,9.18);
\node at (25.3, 9) {\tiny $2$};
\draw[thick] (27,7.9) -- (27,7.7);
\node at (27,7.3) {\tiny $3$};


\draw[thick, double, -latex] (11.3, 12.8) arc (-45:0: 2.5);

\draw[thick, fill = blue, fill opacity=0.2] (9.45,13) circle (0.8);
\draw[thick, fill = darkred, fill opacity=0.7] (10.3,12.15) circle (0.4);
\draw[thick, fill = yellow, fill opacity=0.4] (10.3,11.35) circle (0.4);
\draw[thick, fill = red, fill opacity=0.8] (10.3,10.55) circle (0.4);
\draw[thick, fill = blue, fill opacity=0.2] (10.3,9.35) circle (0.8);

\draw[thick] (9.45,13.7) -- (9.45,13.9);
\node at (9.45,14.3) {\tiny $\inn$};
\draw[thick] (9,12.53) -- (8.8,12.33);
\node at (8.6,12) {\tiny $1$};
\draw[thick] (9.85,8.88) -- (9.65,8.68);
\node at (9.5, 8.3) {\tiny $2$};
\draw[thick] (10.82,8.88) -- (11.02,8.68);
\node at (11.2,8.3) {\tiny $3$};


\draw[thick, fill = blue, fill opacity=0.2] (15.25,10.75) circle (0.8);
\draw[thick, fill = yellow, fill opacity=0.4] (16.13,9.9) circle (0.4);
\draw[thick, fill = red, fill opacity=0.8] (16.13,9.1) circle (0.4);
\draw[thick, fill = blue, fill opacity=0.2] (16.13,7.9) circle (0.8);

\draw[thick] (15.25,11.45) -- (15.25,11.65);
\node at (15.25,12.05) {\tiny $\inn$};
\draw[thick] (14.8,10.3) -- (14.6,10.1);
\node at (14.4,9.8) {\tiny $1$};
\draw[thick] (15.68, 7.43) -- (15.48,7.23);
\node at (15.3, 6.8) {\tiny $2$};
\draw[thick] (16.58, 7.43) -- (16.78,7.23);
\node at (16.9,6.8) {\tiny $3$};
\node at (16.8,10.7) {\tiny $-K$};
\node at (17.3,8.8) {\tiny $H_{23}$};


\draw[thick, fill = blue, fill opacity=0.2] (12,19) circle (0.8);
\draw[thick, fill = blue, fill opacity=0.2] (13.13,17.87) circle (0.8);

\draw[thick] (12,19.7) -- (12,19.9);
\node at (12,20.3) {\tiny $\inn$};
\node at (13.8,19) {\tiny $-2K$};
\draw[thick] (11.55,18.55) -- (11.35,18.35);
\node at (11.15, 18) {\tiny $1$};
\draw[thick] (12.68, 17.43) -- (12.48,17.23);
\node at (12.3, 16.8) {\tiny $2$};
\draw[thick] (13.58, 17.43) -- (13.78, 17.23);
\node at (13.9,16.8) {\tiny $3$};


\end{tikzpicture}
\caption{The Floer data selection along the triangle. \label{fig:triangle_labelled}} 
\end{center}
\end{figure}
Figure~\ref{fig:triangle_labelled} shows the extension of this Floer data to the triangle. On the bottom edge of the triangle, we (necessarily) start off with a copy of the top edge of the square from Figure~\ref{fig:square_labelled}. The lower half of the triangle is another product chain: in the vertical direction, only the data on the top main component varies, splitting off a continuation cylinder from $-2K$ to $-K$ at its second output puncture. In the horizontal direction, only the lower component varies, in the same way as for the coproduct $\lambda_\cD$. On the top half of the triangle, the data on the top component stays constant (with asymptotics $-2K$ at the node), and the data on the lower main component splits off continuation cylinders at both outputs as we approach the top corner. Notice that the conditions in \eqref{eq:conditions_on_slopes} ensure that on all components we can find the appropriate families of Hamiltonians and 1-forms $\alpha$ satisfying the condition
$$
d(b(s,z)\alpha)) \le 0,
$$
which is needed to ensure the relevant maximum principle for solutions to Floer's equation.
These choices can be continued towards other points in $\xi(\p \Ncal_{1,3})$ with weights $(1,0,0)$, by gluing the two main components into a single component with 2 inputs and 2 outputs as we move away from the boundary component $B_1$, but keeping the continuation cylinders at the outputs with weight $0$.
Once we have extended the data to all of $\xi(\p \Ncal_{1,3})$, there is no difficulty extending it further into all of $\xi(\Ncal_{1,3})$.

Consider now a curve $c$ in $\Ccal_{1,3}^0$ connecting a point in the boundary component $b_1$ with a point in the boundary component $b_3$. Its preimage $\widetilde c$ under the weight-forgetting projection $\Ccal_{1,3} \to \Ccal^0_{1,3}$ 
consists of a triangular prism $\simplex_2 \times c$, where at each end one of the vertices of the simplex has been blow up, producing a square and a new triangle in the boundary. The polytope $\widetilde c$ already appeared in \cite[\S6.3]{CO-reduced} (and was named ``The House'' there). As explained in {\em loc. cit.}, it gives rise to the homotopy for coassociativity of $\lambda_\cD$, i.e., commutativity of the diagram
$$
{\scriptsize 
\xymatrix
@C=40pt
{
\overline{HF}(H_0,J_0) \ar[r]^-{\lambda_\cD}\ar[d]_{\lambda_\cD}  & \overline{HF}(H_1,J_1) \otimes \overline{HF}(H_{23},J_{23}) \ar[d]^{-1 \otimes \lambda_\cD}\\
\overline{HF}(H_{12},J_{12}) \otimes \overline{HF}(H_3,J_3) \ar[r]^-{\lambda_\cD \otimes 1} & \overline{HF}(H_1,J_1) \otimes \overline{HF}(H_2,J_2) \otimes \overline{HF}(H_3,J_3)
}
}
$$
Indeed, from the point of view adopted in the present work, one studies the Floer problems corresponding to the chain $\zeta(\widetilde c) \subseteq \Ncal_{1,3}$. The two compositions in the diagram correspond to the square faces at the two ends. The boundary terms that come from the sides of the prism factor through at least one continuation map from $-K$ to $K$, and so they vanish for the reduced theory. To finish the proof of coassociativity, we need to prove that the triangle parts do not contribute. For this argument, we refer back to Figure~\ref{fig:triangle_labelled}. In the bottom half of the triangle, the input of the second main component comes (via continuation) from $FC(-K)$, so for action reasons it has to be supported by critical points in the interior of $W$. The same is also true for the output of that operation, and so the index argument from the proof of \cite[Lemma~4.1]{CO-reduced} shows that under the condition $2n\ge 6$ there are no solutions to be counted here. On the top half of the triangle, the argument is completely analogous, as the input of the second main component being from $FC(-2K)$ means that both it and the outputs also have to be supported by critical points in the interior of $W$. In summary, this shows that the triangular components do not contribute and so coassociativity holds as stated, if we assume $2n \ge 6$.
From the point of view of \S\ref{ssec:moduli_1toell}, this is a translation of one of the relations of Lemma~\ref{lem:cycles_coassociativity}. 

There is no issue extending the choices of Floer data from $\xi(\Ncal_{1,3})$ to all of $\Mcal_{1,3}$. Then the relation in Corollary~\ref{cor:7-term-coBV} translates into the 7-term relation between $\lambda_\cD$ and $\Delta$. At this point we need that Floer problems parametrized by families of triangles as in Figure~\ref{fig:triangle_labelled} over the boundary components of $\cC^0_{1,3}$ do not admit solutions, and hence do not contribute to the boundary. An index argument similar to that of~\cite[Lemma~4.1]{CO-reduced} shows that this is indeed the case under the dimensional assumption $2n\ge 8$.
 
As was the case for the product, all the constructions here are compatible with continuation maps, and so the identities for finite Hamiltonians eventually pass to the limit to yield the same identities in reduced symplectic homology. We have proven
\begin{proposition}\label{prop:coBV}
Let $W$ be a strongly R-essential Weinstein domain of dimension $2n\ge 8$. Then the coproduct $\lambda_\cD$ and the BV operator $\Delta$ on reduced symplectic homology satisfy the 7-term relation \eqref{eq:7_term_for_Delta_and_lambda}. \hfill $\qed$
\end{proposition}

\subsection{BV unital infinitesimal bialgebra structure}\label{ssec:11-term-implemented}

In this section we prove Theorem~\ref{thm:BV_unital_infinitesimal_reduced}. The key step is the following proposition, proven in \S\ref{ssec:alg_consequences}.

\begin{proposition}\label{prop:11-term}
Let $W$ be a strongly $R$-essential Weinstein domain of dimension $2n \ge 6$. The product $\mu$, the unit $\eta$, the coproduct $\lambda_\cD$ and the BV operator $\Delta$ on reduced symplectic homology satisfy the 11-term relation~\eqref{eq:11term}. 
\end{proposition}

\subsubsection{Construction of the Floer data}
As in the previous subsection, the main point in the proof is to construct the Floer data on $\Mcal_{2,2}$ in a way which is compatible with the description of the lift $\xi(\p\Ccal_{2,2})$ in terms of products of cycles, so that Lemma~\ref{lem:simplified_11_term} from \S\ref{ssec:moduli_2to2} becomes applicable. 

We fix Hamiltonians $H_{1^\pm}$ and $H_{2^\pm}$ with slopes $b_{1^\pm}$ and $b_{2^\pm}$ for use at the corresponding punctures and, for $i\in \{1,2,3\}$, we also fix Hamiltonians $H_i$ with slopes $b_i$ for use at the nodes in the nodal configurations at the boundary component $\hat B_i$ of $\Mcal_{2,2}$. To give admissible data for the various product and coproduct constructions, the slopes are required to satisfy the following inequalities:
\begin{align*}
b_{1^-}+b_{2^-} &\ge b_3 \ge b_{1^+}+b_{2^+},\\
b_1 \ge \max(b_{1^+},b_{2^+}) +\eps, \quad b_{1^-} &\ge b_1+b_{1^+}, \quad b_{2^-} \ge b_1+b_{2^+}, \quad \text{and}\\
b_2 \ge \max(b_{1^+},b_{2^+}) +\eps, \quad b_{1^-} &\ge b_2+b_{2^+}, \quad b_{2^-} \ge b_2+b_{1^+}.
\end{align*}
We also choose corresponding admissible almost complex structure $J_{1^\pm}$, $J_{2^\pm}$ and $J_i$ for $i \in \{1,2,3\}$.

Note that the fiber product description of $\hat B:=\hat B_1 \cup \hat B_2 \cup \hat B_3$ given in \S\ref{ssec:moduli_2to2} does not interact well with our choice of Floer data. Indeed, as the choice of Hamiltonian at the node is sensitive to the position of the marker, simultaneously rotating the markers
produces an equivalent nodal domain with {\em different} (namely rotated) data at the node.

To rectify this, we will first construct the Floer data on $\xi(\widetilde{b_1}\cup \widetilde{b_2} \cup \widetilde{b_3}) \subseteq \p \Mcal_{2,2}$, where the $\widetilde{b_i}$ are the pieces of the boundary of $\Ccal_{2,2}$ corresponding to the boundary components $b_i$ of $\Ccal^0_{2,2}$, and $\xi:\Ccal_{2,2} \to \Mcal_{2,2}$ is the section considered in \S\ref{ssec:moduli_2to2}. This can be extended to all of $\hat B$ by appropriately adapting to the rotating markers near the inputs and outputs. It will be clear that the resulting data can be extended further to $\p \simplex_1 \times \Mcal^0_{2,2}$ in such a way that at the output with zero weight we always keep the Hamiltonian $-K$ at the corresponding puncture of the main component, followed by continuation cylinders to $K$ and to the corresponding Hamiltonian $H_i^-$. With the Floer data set up on $\p \Mcal_{2,2}$, there is  then no difficulty extending it into the interior of $\Mcal_{2,2}$.

We have given parametrizations of $\xi(\p \Ccal_{2,2})$ as products of chains from the moduli space $\Mcal_{2,1}$ and $\Mcal_{1,2}$ in Remark~\ref{rem:11-term-geom}.
In particular, these parametrizations fix the choice of Floer data on $\xi(\widetilde{b_3})$, and also on the images of the outer thirds $\simplex_\le \times b_i$ and $\simplex_\ge \times b_i$ for the other two components.

It remains to explain how to extend these choices to the two middle annuli $A_1\subseteq \widetilde{b_1}$ and $A_2 \subseteq \widetilde{b_2}$. We give the details for $A_1$, with the other case being similar.

According to our choices in \S\ref{ssec:prod_and_coprod}, the above parametrizations result in the two parametrized Floer problems over $\simplex_\le \times b_1$ and $\simplex_\ge \times b_1$ depicted in Figure~\ref{fig:data_for_halves}.
\begin{figure}[h]
\begin{tikzpicture}[scale=0.25]
\definecolor{darkblue}{rgb}{0,0,0.7};

%
%

\draw[thick, fill = blue, fill opacity=0.2] (-3.1,.9) circle (1.5);
\draw[thick, dotted] (-4.6,0.9) arc (180:0: 1.5 and 0.3);
\draw[thick] (-4.6,0.9) arc (180:360: 1.5 and 0.3);

\draw[thick, fill = blue, fill opacity=0.2] (-0.9,-1.2) circle (1.5);
\draw[thick, dotted] (-2.4,-1.2) arc (180:0: 1.5 and 0.3);
\draw[thick] (-2.4,-1.2) arc (180:360: 1.5 and 0.3);

\draw[thick, fill = yellow, fill opacity=0.3] (-5,-0.7071) circle (1);
\draw[thick, dotted] (-6,-0.7071) arc (180:0: 1 and 0.2);
\draw[thick] (-6,-0.7071,0) arc (180:360: 1 and 0.2);

\draw[thick, fill = red, fill opacity=0.7] (-5,-2.7071) circle (1);
\draw[thick, dotted] (-6,-2.7071) arc (180:0: 1 and 0.2);
\draw[thick] (-6,-2.7071,0) arc (180:360: 1 and 0.2);

\draw[thick] (-3.1,2.3) to (-3.1,2.5);
\node at (-3.0,3) {\tiny $1$};

\draw[thick] (-4.2,0.1) to (-4.4,-0.1);
\node at (-5.7,0.8) {\tiny $-K$};

\draw[thick] (-2.1, 0) to (-1.9,-0.2);
\node at (-0.8 , 0.9) {\tiny $H_1$};

\draw[thick] (0.3, -0.1) to (0.1,-0.3);
\node at (0.5, 0.3) {\tiny $2$};

\draw[thick] (-5, -1.60) to (-5,-1.80);
\node at (-6.4,-1.7) {\tiny $K$};

\draw[thick] (-5, -3.6) to (-5,-3.80);
\node at (-5,-4.3) {\tiny $1$};

\draw[thick] (-0.9, -2.6) to (-0.9,-2.8);
\node at (-0.9,-3.3) {\tiny $2$};

\draw[thick, fill=blue, fill opacity=0.2] (4,-2) circle (1.5);
\draw[thick, dotted] (2.5,-2) arc (180:0: 1.5 and 0.3);
\draw[thick] (2.5,-2) arc (180:360: 1.5 and 0.3);

\draw[thick, fill=blue, fill opacity=0.2] (6.2,-4.1) circle (1.5);
\draw[thick, dotted] (4.7,-4.1) arc (180:0: 1.5 and 0.3);
\draw[thick] (4.7,-4.1) arc (180:360: 1.5 and 0.3);

\draw[thick] (4,-0.62) to (4,-0.38);
\node at (4,0.1) {\tiny $1$};

\draw[thick] (3.1, -3) to (2.9,-3.2);
\node at (2.7,-3.7) {\tiny $1$};

\draw[thick] (5.,-3) to (5.2,-3.2); 
\node at (6.3, -2) {\tiny $H_1$};

\draw[thick] (7.4, -3) to (7.2,-3.2);
\node at (7.6,-2.5) {\tiny $2$};

\draw[thick] (6.2, -5.5) to (6.2,-5.7);
\node at (6.2 ,-6.2) {\tiny $2$};


\draw[thick, fill=blue, fill opacity=0.2] (10.1,0.9) circle (1.5);
\draw[thick, dotted] (8.6,0.9) arc (180:0: 1.5 and 0.3);
\draw[thick] (8.6,0.9) arc (180:360: 1.5 and 0.3);

\draw[thick, fill=yellow, fill opacity=0.3] (12,-0.72) circle (1);
\draw[thick, dotted] (11,-0.72) arc (180:0: 1 and 0.2);
\draw[thick] (11,-0.72,0) arc (180:360: 1 and 0.2);

\draw[thick, fill=red, fill opacity=0.7] (12,-2.72) circle (1);
\draw[thick, dotted] (11,-2.72) arc (180:0: 1 and 0.2);
\draw[thick] (11,-2.72,0) arc (180:360: 1 and 0.2);

\draw[thick, fill=blue, fill opacity=0.2] (13.9,-4.4) circle (1.5);
\draw[thick, dotted] (12.4,-4.4) arc (180:0: 1.5 and 0.3);
\draw[thick] (12.4,-4.4) arc (180:360: 1.5 and 0.3);

\draw[thick] (10.1,2.3) to (10.1,2.5);
\node at (10.1,3) {\tiny $1$};

\draw[thick] (9, 0.1) to (8.8,-0.1);
\node at (8.5,-0.5) {\tiny $1$};

\draw[thick] (11.2,0.1) -- (11.4,-0.1);
\node at (12.6,0.7) {\tiny $-K$};

\draw[thick] (12,-1.6) -- (12,-1.85);
\node at (10.7,-1.7) {\tiny $K$};

\draw[thick] (12.6,-3.34) -- (12.8,-3.54);
\node at (13.8, -2.2) {\tiny $H_1$};

\draw[thick] (14.9, -3.5) to (15.1,-3.3);
\node at (15.4, -2.8) {\tiny $2$};

\draw[thick] (13.9, -5.8) to (13.9,-6);
\node at (13.9,-6.5) {\tiny $2$};


\draw[thick] (-3,-8) to (12,-8);
\draw[thick] (-3,-7.9) to (-3,-8.1);
\draw[thick] (12,-7.9) to (12,-8.1);

\node at (-3,-9) {\tiny $(0,1)$};
\node at (12,-9) {\tiny $(\frac 12, \frac 12)$};

\pgftransformxshift{24cm}

%
%

\draw[thick, fill = darkblue, fill opacity=0.7] (-1.1,.9) circle (1.5);
\draw[thick, dotted] (-2.6,0.9) arc (180:0: 1.5 and 0.3);
\draw[thick] (-2.6,0.9) arc (180:360: 1.5 and 0.3);

\draw[thick, fill = yellow, fill opacity=0.3] (-3,-0.71) circle (1);
\draw[thick, dotted] (-4,-0.71) arc (180:0: 1 and 0.2);
\draw[thick] (-4,-0.71,0) arc (180:360: 1 and 0.2);

\draw[thick, fill = red, fill opacity=0.7] (-3,-2.71) circle (1);
\draw[thick, dotted] (-4,-2.71) arc (180:0: 1 and 0.2);
\draw[thick] (-4,-2.71,0) arc (180:360: 1 and 0.2);

\draw[thick, fill = darkblue, fill opacity=0.7] (-4.9,-4.4) circle (1.5);
\draw[thick, dotted] (-6.4,-4.4) arc (180:0: 1.5 and 0.3);
\draw[thick] (-6.4,-4.4) arc (180:360: 1.5 and 0.3);

\draw[thick] (-1.1,2.3) to (-1.1,2.5);
\node at (-1.1,3) {\tiny $2$};

\draw[thick] (-2.2,0.1) to (-2.4,-0.1);
\node at (-3.6,0.8) {\tiny $-K$};

\draw[thick] (0, 0.1) to (0.2,-0.1);
\node at (0.5, -0.4) {\tiny $2$};

\draw[thick] (-3, -1.6) to (-3,-1.8);
\node at (-4.4,-1.7) {\tiny $K$};

\draw[thick] (-5.9, -3.5) to (-6.1,-3.3);
\node at (-6.4,-2.8) {\tiny $1$};

\draw[thick] (-3.9, -3.5) to (-3.65,-3.25);
\node at (-2.6,-4.4) {\tiny $H_1$};

\draw[thick] (-4.9, -5.8) to (-4.9,-6);
\node at (-4.9,-6.5) {\tiny $1$};

\draw[thick, fill=darkblue, fill opacity=0.7] (6.2,-2) circle (1.5);
\draw[thick, dotted] (4.7,-2) arc (180:0: 1.5 and 0.3);
\draw[thick] (4.7,-2) arc (180:360: 1.5 and 0.3);

\draw[thick, fill=darkblue, fill opacity=0.7] (4,-4.1) circle (1.5);
\draw[thick, dotted] (2.5,-4.1) arc (180:0: 1.5 and 0.3);
\draw[thick] (2.5,-4.1) arc (180:360: 1.5 and 0.3);

\draw[thick] (6.2,-0.62) to (6.2,-0.38);
\node at (6.2,0.1) {\tiny $2$};

\draw[thick] (3.1, -3.1) to (2.9,-2.9);
\node at (2.6,-2.4) {\tiny $1$};

\draw[thick] (5,-3.15) to (5.2,-2.95); 
\node at (3.9, -2) {\tiny $H_1$};

\draw[thick] (7.4, -3.1) to (7.2,-2.9);
\node at (7.6,-3.6) {\tiny $2$};

\draw[thick] (4, -5.5) to (4,-5.7);
\node at (4 ,-6.3) {\tiny $1$};


\draw[thick, fill=darkblue, fill opacity=0.7] (13.1,0.9) circle (1.5);
\draw[thick, dotted] (11.6,0.9) arc (180:0: 1.5 and 0.3);
\draw[thick] (11.6,0.9) arc (180:360: 1.5 and 0.3);

\draw[thick, fill=darkblue, fill opacity=0.7] (11,-1.3) circle (1.5);
\draw[thick, dotted] (9.5,-1.3) arc (180:0: 1.5 and 0.3);
\draw[thick] (9.5,-1.3) arc (180:360: 1.5 and 0.3);

\draw[thick, fill=yellow, fill opacity=0.2] (15,-0.72) circle (1);
\draw[thick, dotted] (14,-0.72) arc (180:0: 1 and 0.2);
\draw[thick] (14,-0.72,0) arc (180:360: 1 and 0.2);

\draw[thick, fill=red, fill opacity=0.7] (15,-2.72) circle (1);
\draw[thick, dotted] (14,-2.72) arc (180:0: 1 and 0.2);
\draw[thick] (14,-2.72,0) arc (180:360: 1 and 0.2);

\draw[thick] (13.1,2.3) to (13.1,2.5);
\node at (13.1,3) {\tiny $2$};

\draw[thick] (10, 0) to (10.2,-0.2);
\node at (9.8, 0.5) {\tiny $1$};

\draw[thick] (12.1,0) -- (11.9,-0.2);
\node at (10.8, 0.8) {\tiny $H_1$};

\draw[thick] (14.2,0.1) -- (14.4,-0.1);
\node at (15.6,0.7) {\tiny $-K$};

\draw[thick] (15,-1.6) -- (15,-1.85);
\node at (16.3,-1.7) {\tiny $K$};

\draw[thick] (15, -3.62) to (15,-3.82);
\node at (15, -4.4) {\tiny $2$};

\draw[thick] (11, -2.7) to (11,-2.9);
\node at (11,-3.5) {\tiny $1$};


\draw[thick] (-3,-8) to (12,-8);
\draw[thick] (-3,-7.9) to (-3,-8.1);
\draw[thick] (12,-7.9) to (12,-8.1);

\node at (-3,-9) {\tiny $(\frac 12, \frac 12)$};
\node at (12,-9) {\tiny $(1,0)$};

\end{tikzpicture}

\caption{A schematic picture of the Floer data on the boundary component $\hat B_1$. \label{fig:data_for_halves}}
\end{figure}
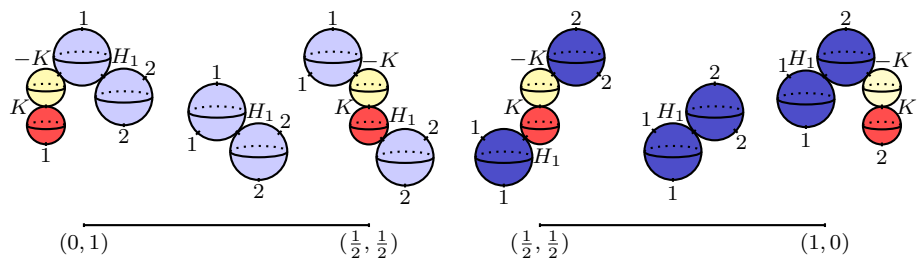
We use different shades of blue for the main components on the two halves because they have different asymptotics, and so their Floer data looks rather different. In any case, the data clearly does not match at the end points corresponding to $(\frac 12, \frac 12) \in \simplex_1$.
We use the annulus $A_1$ created by blowing up this midpoint of the simplex to interpolate between these choices.
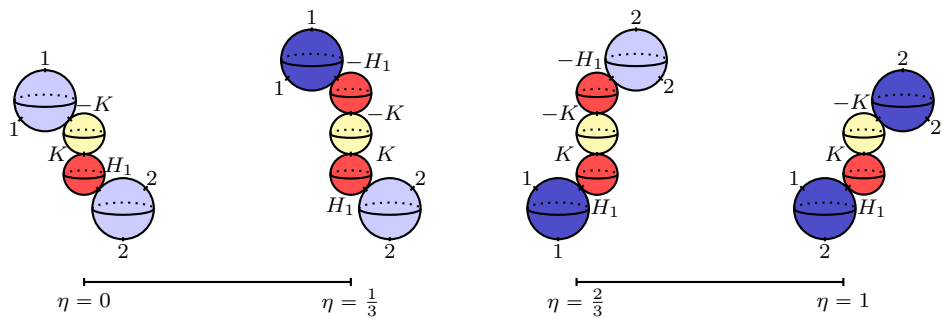
\begin{figure}[h]

\begin{tikzpicture}[scale=0.27]
\definecolor{darkblue}{rgb}{0,0,0.7};


\draw[thick, fill=blue, fill opacity=0.2] (2.1,0.9) circle (1.5);
\draw[thick, dotted] (0.6,0.9) arc (180:0: 1.5 and 0.3);
\draw[thick] (0.6,0.9) arc (180:360: 1.5 and 0.3);

\draw[thick, fill=yellow, fill opacity=0.3] (4,-0.72) circle (1);
\draw[thick, dotted] (3,-0.72) arc (180:0: 1 and 0.2);
\draw[thick] (3,-0.72) arc (180:360: 1 and 0.2);

\draw[thick, fill=red, fill opacity=0.7] (4,-2.72) circle (1);
\draw[thick, dotted] (3,-2.72) arc (180:0: 1 and 0.2);
\draw[thick] (3,-2.72) arc (180:360: 1 and 0.2);

\draw[thick, fill=blue, fill opacity=0.2] (5.9,-4.4) circle (1.5);
\draw[thick, dotted] (4.4,-4.4) arc (180:0: 1.5 and 0.3);
\draw[thick] (4.4,-4.4) arc (180:360: 1.5 and 0.3);

\draw[thick] (2.1,2.3) to (2.1,2.5);
\node at (2.1,3) {\tiny $1$};

\draw[thick] (1, 0.1) to (0.8,-0.1);
\node at (.6,-0.5) {\tiny $1$};

\draw[thick] (3.2,0.1) -- (3.4,-0.1);
\node at (4.5,0.7) {\tiny $-K$};

\draw[thick] (4,-1.6) -- (4,-1.85);
\node at (2.7,-1.7) {\tiny $K$};

\draw[thick] (4.6,-3.34) -- (4.8,-3.54);
\node at (5.7, -2.3) {\tiny $H_1$};

\draw[thick] (6.9, -3.5) to (7.1,-3.3);
\node at (7.3, -2.9) {\tiny $2$};

\draw[thick] (5.9, -5.8) to (5.9,-6);
\node at (5.9,-6.5) {\tiny $2$};


\draw[thick, fill=darkblue, fill opacity=0.7] (15.1,2.9) circle (1.5);
\draw[thick, dotted] (13.6,2.9) arc (180:0: 1.5 and 0.3);
\draw[thick] (13.6,2.9) arc (180:360: 1.5 and 0.3);

\draw[thick, fill=red, fill opacity=0.7] (17,1.28) circle (1);
\draw[thick, dotted] (16,1.28) arc (180:0: 1 and 0.2);
\draw[thick] (16,1.28) arc (180:360: 1 and 0.2);

\draw[thick, fill=yellow, fill opacity=0.3] (17,-0.72) circle (1);
\draw[thick, dotted] (16,-0.72) arc (180:0: 1 and 0.2);
\draw[thick] (16,-0.72,0) arc (180:360: 1 and 0.2);

\draw[thick, fill=red, fill opacity=0.7] (17,-2.72) circle (1);
\draw[thick, dotted] (16,-2.72) arc (180:0: 1 and 0.2);
\draw[thick] (16,-2.72,0) arc (180:360: 1 and 0.2);

\draw[thick, fill=blue, fill opacity=0.2] (18.9,-4.4) circle (1.5);
\draw[thick, dotted] (17.4,-4.4) arc (180:0: 1.5 and 0.3);
\draw[thick] (17.4,-4.4) arc (180:360: 1.5 and 0.3);

\draw[thick] (15.1,4.3) to (15.1,4.5);
\node at (15.1,5) {\tiny $1$};

\draw[thick] (14, 2.1) to (13.8,1.9);
\node at (13.5,1.5) {\tiny $1$};

\draw[thick] (16.2,2.1) -- (16.4,1.9);
\node at (17.9,2.8) {\tiny $-H_1$};

\draw[thick] (17,0.16) -- (17,0.4);
\node at (18.7,0.28) {\tiny $-K$};

\draw[thick] (17,-1.84) -- (17,-1.6);
\node at (18.7, -1.72) {\tiny $K$};

\draw[thick] (17.6,-3.34) -- (17.8,-3.54);
\node at (16.5, -4.3) {\tiny $H_1$};
\draw[thick] (19.9, -3.5) to (20.1,-3.3);
\node at (20.4, -2.9) {\tiny $2$};

\draw[thick] (18.9, -5.8) to (18.9,-6);
\node at (18.9,-6.5) {\tiny $2$};


\draw[thick, fill = blue, fill opacity=0.2] (30.9,2.9) circle (1.5);
\draw[thick, dotted] (29.4,2.9) arc (180:0: 1.5 and 0.3);
\draw[thick] (29.4,2.9) arc (180:360: 1.5 and 0.3);

\draw[thick, fill=red, fill opacity=0.7] (29,1.28) circle (1);
\draw[thick, dotted] (28,1.28) arc (180:0: 1 and 0.2);
\draw[thick] (28,1.28) arc (180:360: 1 and 0.2);

\draw[thick, fill = yellow, fill opacity=0.3] (29,-0.71) circle (1);
\draw[thick, dotted] (28,-0.71) arc (180:0: 1 and 0.2);
\draw[thick] (28,-0.71) arc (180:360: 1 and 0.2);

\draw[thick, fill=red, fill opacity=0.7] (29,-2.72) circle (1);
\draw[thick, dotted] (28,-2.72) arc (180:0: 1 and 0.2);
\draw[thick] (28,-2.72) arc (180:360: 1 and 0.2);

\draw[thick, fill = darkblue, fill opacity=0.7] (27.1,-4.4) circle (1.5);
\draw[thick, dotted] (25.6,-4.4) arc (180:0: 1.5 and 0.3);
\draw[thick] (25.6,-4.4) arc (180:360: 1.5 and 0.3);

\draw[thick] (30.9,4.3) to (30.9,4.5);
\node at (30.9,5) {\tiny $2$};

\draw[thick] (29.8,2.1) to (29.6,1.9);
\node at (28.2,2.9) {\tiny $-H_1$};

\draw[thick] (32, 2.1) to (32.2,1.9);
\node at (32.5, 1.6) {\tiny $2$};

\draw[thick] (29, 0.4) to (29,0.15);
\node at (27.2,0.28) {\tiny $-K$};

\draw[thick] (29, -1.62) to (29,-1.82);
\node at (27.4,-1.72) {\tiny $K$};

\draw[thick] (28.1, -3.5) to (28.35,-3.25);
\node at (29.4,-4.4) {\tiny $H_1$};

\draw[thick] (26.1, -3.5) to (25.9,-3.3);
\node at (25.6,-2.9) {\tiny $1$};

\draw[thick] (27.1, -5.8) to (27.1,-6);
\node at (27.1,-6.5) {\tiny $1$};

\draw[thick, fill = darkblue, fill opacity=0.7] (43.9,.9) circle (1.5);
\draw[thick, dotted] (42.4,0.9) arc (180:0: 1.5 and 0.3);
\draw[thick] (42.4,0.9) arc (180:360: 1.5 and 0.3);

\draw[thick, fill = yellow, fill opacity=0.3] (42,-0.71) circle (1);
\draw[thick, dotted] (41,-0.71) arc (180:0: 1 and 0.2);
\draw[thick] (41,-0.71,0) arc (180:360: 1 and 0.2);

\draw[thick, fill = red, fill opacity=0.7] (42,-2.71) circle (1);
\draw[thick, dotted] (41,-2.71) arc (180:0: 1 and 0.2);
\draw[thick] (41,-2.71,0) arc (180:360: 1 and 0.2);

\draw[thick, fill = darkblue, fill opacity=0.7] (40.1,-4.4) circle (1.5);
\draw[thick, dotted] (38.6,-4.4) arc (180:0: 1.5 and 0.3);
\draw[thick] (38.6,-4.4) arc (180:360: 1.5 and 0.3);

\draw[thick] (43.9,2.3) to (43.9,2.5);
\node at (43.9,3) {\tiny $2$};

\draw[thick] (42.8,0.1) to (42.6,-0.1);
\node at (41.4,0.8) {\tiny $-K$};

\draw[thick] (45, 0.1) to (45.2,-0.1);
\node at (45.5, -0.4) {\tiny $2$};

\draw[thick] (42, -1.6) to (42,-1.8);
\node at (40.6,-1.7) {\tiny $K$};

\draw[thick] (39.1, -3.5) to (38.9,-3.3);
\node at (38.6,-2.9) {\tiny $1$};

\draw[thick] (41.1, -3.5) to (41.35,-3.25);
\node at (42.4,-4.4) {\tiny $H_1$};

\draw[thick] (40.1, -5.8) to (40.1,-6);
\node at (40.1,-6.5) {\tiny $1$};


\draw[thick] (4,-8) to (17,-8);
\draw[thick] (28,-8) to (41,-8);

\draw[thick] (4,-7.8) to (4,-8.2);
\node at (4,-9) {\tiny $\eta=0$};

\draw[thick] (17,-7.8) to (17,-8.2);
\node at (17,-9) {\tiny $\eta=\frac 13$};

\draw[thick] (28,-7.8) to (28,-8.2);
\node at (28,-9) {\tiny $\eta=\frac 23$};

\draw[thick] (41,-7.8) to (41,-8.2);
\node at (41,-9) {\tiny $\eta=1$};

\end{tikzpicture}
\caption{Floer data in the outer two thirds of the transition region.\label{fig:transition}}
\end{figure}

The following observation is crucial for our construction: Given Hamiltonians $F$, $G$ and $H$ with slopes $b_H \ge b_F+b_G$, admissible Floer data on a pair of pants with two inputs using $F$ and $G$ and one output using $H$ can be reinterpreted, by turning the $G$-puncture into an output, as admissible Floer data on a pair of pants with one input using $F$ and two outputs using $-G$ and $H$, and vice versa.

Keeping this in mind, as the parameter $\eta\in [0,1]$ in the middle interval increases from $0$ to $\frac 13$, we keep the lower main component unchanged but deform the data on the upper main component to the composition of the reinterpreted data from the lower dark blue component with its adjacent continuation cylinder turned upside down. Similarly, as the parameter $\eta$ decreases from $1$ to $\frac 23$, we deform the data on the upper main component to match the composition of the reinterpreted data from the lower light blue component with the adjacent continuation cylinder turned upside down. These choices are illustrated in Figure~\ref{fig:transition}.

We can now reinterpret the configuration at $\eta=\frac 23$ by turning it completely upside down, as in Figure~\ref{fig:upside_down}.
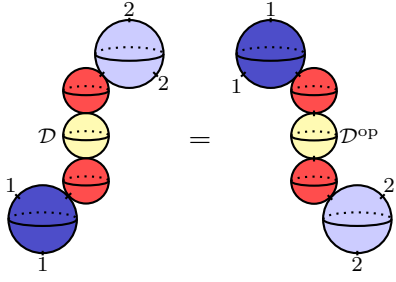
\begin{figure}[h]

\begin{tikzpicture}[scale=0.30]
\definecolor{darkblue}{rgb}{0,0,0.7};

\draw[thick, fill = blue, fill opacity=0.2] (19.9,2.9) circle (1.5);
\draw[thick, dotted] (18.4,2.9) arc (180:0: 1.5 and 0.3);
\draw[thick] (18.4,2.9) arc (180:360: 1.5 and 0.3);

\draw[thick, fill = red, fill opacity=0.7] (18,1.28) circle (1);
\draw[thick, dotted] (17,1.28) arc (180:0: 1 and 0.2);
\draw[thick] (17,1.28) arc (180:360: 1 and 0.2);

\draw[thick, fill = yellow, fill opacity=0.3] (18,-0.71) circle (1);
\draw[thick, dotted] (17,-0.71) arc (180:0: 1 and 0.2);
\draw[thick] (17,-0.71) arc (180:360: 1 and 0.2);

\draw[thick, fill = red, fill opacity=0.7] (18,-2.71) circle (1);
\draw[thick, dotted] (17,-2.71) arc (180:0: 1 and 0.2);
\draw[thick] (17,-2.71) arc (180:360: 1 and 0.2);

\draw[thick, fill = darkblue, fill opacity=0.7] (16.1,-4.4) circle (1.5);
\draw[thick, dotted] (14.6,-4.4) arc (180:0: 1.5 and 0.3);
\draw[thick] (14.6,-4.4) arc (180:360: 1.5 and 0.3);

\draw[thick] (19.9,4.3) to (19.9,4.5);
\node at (19.9,4.9) {\tiny $2$};

\draw[thick] (18.8,2.1) to (18.6,1.9);

\draw[thick] (20.95, 2.1) to (21.15,1.9);
\node at (21.4, 1.6) {\tiny $2$};

\draw[thick] (17.1, -3.5) to (17.35,-3.25);

\node at (16.3,-0.72) {\tiny $\cD$};

\draw[thick] (15.1, -3.5) to (14.9,-3.3);
\node at (14.7,-2.9) {\tiny $1$};

\draw[thick] (16.1, -5.8) to (16.1,-6);
\node at (16.1,-6.4) {\tiny $1$};

\node at (23,-1) {$=$};


\draw[thick, fill=darkblue, fill opacity=0.7] (26.1,2.9) circle (1.5);
\draw[thick, dotted] (24.6,2.9) arc (180:0: 1.5 and 0.3);
\draw[thick] (24.6,2.9) arc (180:360: 1.5 and 0.3);

\draw[thick, fill=red, fill opacity=0.7] (28,1.28) circle (1);
\draw[thick, dotted] (27,1.28) arc (180:0: 1 and 0.2);
\draw[thick] (27,1.28) arc (180:360: 1 and 0.2);

\draw[thick, fill=yellow, fill opacity=0.3] (28,-0.72) circle (1);
\draw[thick, dotted] (27,-0.72) arc (180:0: 1 and 0.2);
\draw[thick] (27,-0.72) arc (180:360: 1 and 0.2);

\draw[thick, fill=red, fill opacity=0.7] (28,-2.72) circle (1);
\draw[thick, dotted] (27,-2.72) arc (180:0: 1 and 0.2);
\draw[thick] (27,-2.72,0) arc (180:360: 1 and 0.2);

\draw[thick, fill=blue, fill opacity=0.2] (29.9,-4.4) circle (1.5);
\draw[thick, dotted] (28.4,-4.4) arc (180:0: 1.5 and 0.3);
\draw[thick] (28.4,-4.4) arc (180:360: 1.5 and 0.3);

\draw[thick] (26.1,4.3) to (26.1,4.5);
\node at (26.1,4.9) {\tiny $1$};

\draw[thick] (25, 2.1) to (24.8,1.9);
\node at (24.6,1.5) {\tiny $1$};

\draw[thick] (28, 0.15) to (28,0.4);

\draw[thick] (28, -1.85) to (28,-1.6);

\draw[thick] (27.2,2.1) -- (27.4,1.9);

\draw[thick] (28.6,-3.34) -- (28.8,-3.54);

\node at (30,-.72) {\tiny $\cD^{\mathrm{op}}$};

\draw[thick] (30.9, -3.5) to (31.1,-3.3);
\node at (31.3, -2.9) {\tiny $2$};

\draw[thick] (29.9, -5.8) to (29.9,-6);
\node at (29.9,-6.4) {\tiny $2$};

\end{tikzpicture}
\caption{Reinterpreting the Floer problem at $\eta=\frac 23$.\label{fig:upside_down}}
\end{figure}
By the choices we made above, the data on the main components and the first and third continuation cylinders will match the corresponding data for $\eta=\frac 13$ exactly. However, since the continuation cylinder from $-K$ to $K$ has also been turned upside down, we get the ``upside down'' version $\cD^{\mathrm{op}}$ of the continuation data $\cD$ that we fixed for that continuation map. So to get a smooth moduli problem for $\eta \in [0,1]$ interpolating between the two given configurations at $\eta=0$ and $\eta=1$, it only remains to fill in the middle third of the interval by a generic homotopy from $\cD$ to $\cD^{\mathrm{op}}$.

The conclusion of this discussion is that on $\xi(\widetilde{b_1})$
we now have a smoothly varying family of Floer data. Note that the end points of the middle interval give rise to corners in the blow-up, so smoothness there is not an issue. The discussion of $\xi(\widetilde{b_2})$
is completely analogous, just with the labels of the output punctures exchanged.
As we explained close to the start of this subsection, extending the data to all of $\Mcal_{2,2}$ is now straightforward.

\subsubsection{Algebraic consequences}\label{ssec:alg_consequences}
To draw a first consequence  of this construction, consider a path $c$ in $\Mcal^0_{2,2}$ connecting the point $c_3\in \hat B^0_3$ for which the circles through the 3 special points on the two components align so that the inputs 1 and 2 are followed by the outputs 2 and 1 in the resulting cyclic order, and the point $c_1\in \hat B^0_1$ with the same property.
The chain $\simplex_1 \times c \subseteq \Mcal'_{2,2}$ is transformed by the blow-up into a hexagon. There is one side over $c_3$, which in our parametrization of the boundary corresponds to the chain $-\hat R^*_{12} \times \hat P^{12}_*$ (because we chose $c_3$ as the starting point of $c$). We have seen that this gives rise to the algebraic term $\lambda_\cD \mu$. The adjacent two sides correspond to $(\p \simplex_1) \times c$. Their contributions vanish for reduced symplectic homology, because by construction they factor through the continuation map from $-K$ to $K$.
Over the point $c_1$, we have three more sides (due to the blow-up). The two halves of $\simplex_1$ are identified in our parametrization of $\hat B_1$ with $\hat P^{*2}_2 \times \hat R^1_{1*}$ and $\hat P^{1*}_1 \times \hat R^2_{*2}$, respectively. They give rise to the algebraic terms $(1 \otimes \mu)(\lambda_\cD \otimes 1)$ and $(\mu \otimes 1)(1\otimes \lambda_\cD)$.
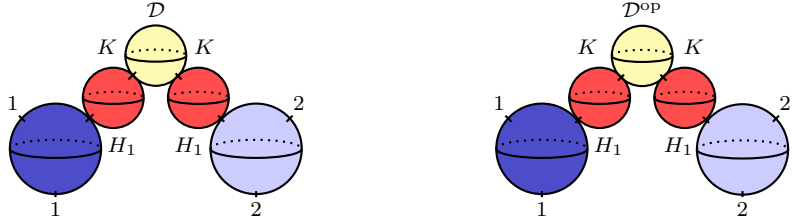
\begin{figure}[h]
\begin{tikzpicture}[scale=0.40]
\definecolor{darkblue}{rgb}{0,0,0.7};


\draw[thick, fill=darkblue, fill opacity=0.7] (24.7,-3.8) circle (1.5);
\draw[thick, dotted] (23.2,-3.8) arc (180:0: 1.5 and 0.3);
\draw[thick] (23.2,-3.8) arc (180:360: 1.5 and 0.3);

\draw[thick, fill=red, fill opacity=0.7] (26.6,-2.11) circle (1);
\draw[thick, dotted] (25.6,-2.11) arc (180:0: 1 and 0.2);
\draw[thick] (25.6,-2.11) arc (180:360: 1 and 0.2);

\draw[thick, fill=yellow, fill opacity=0.3] (28,-0.72) circle (1);
\draw[thick, dotted] (27,-0.72) arc (180:0: 1 and 0.2);
\draw[thick] (27,-0.72,0) arc (180:360: 1 and 0.2);

\draw[thick, fill=red, fill opacity=0.7] (29.4,-2.11) circle (1);
\draw[thick, dotted] (28.4,-2.11) arc (180:0: 1 and 0.2);
\draw[thick] (28.4,-2.11) arc (180:360: 1 and 0.2);

\draw[thick, fill=blue, fill opacity=0.2] (31.3,-3.82) circle (1.5);
\draw[thick, dotted] (29.8,-3.82) arc (180:0: 1.5 and 0.3);
\draw[thick] (29.8,-3.82) arc (180:360: 1.5 and 0.3);

\node at (28,0.8) {\tiny $\cD^{\mathrm{op}}$};

\draw[thick] (23.7, -2.9) to (23.5,-2.7);
\node at (23.3,-2.3) {\tiny $1$};

\draw[thick] (24.7,-5.2) to (24.7,-5.4);
\node at (24.7,-5.8) {\tiny $1$};

\draw[thick] (27.2,-1.54) -- (27.4,-1.34);
\node at (26.2,-0.4) {\tiny $K$};

\draw[thick] (25.8,-2.94) -- (26,-2.74);
\node at (26.9,-3.7) {\tiny $H_1$};

\draw[thick] (28.6,-1.34) -- (28.8,-1.54);
\node at (29.7, -0.4) {\tiny $K$};

\draw[thick] (30,-2.74) -- (30.2,-2.94);
\node at (29.2, -3.7) {\tiny $H_1$};

\draw[thick] (32.3, -2.9) to (32.5,-2.7);
\node at (32.7, -2.3) {\tiny $2$};

\draw[thick] (31.3, -5.2) to (31.3,-5.4);
\node at (31.3,-5.8) {\tiny $2$};

\pgftransformxshift{-6cm}


\draw[thick, fill = blue, fill opacity=0.2] (21.3,-3.8) circle (1.5);
\draw[thick, dotted] (19.8,-3.8) arc (180:0: 1.5 and 0.3);
\draw[thick] (19.8,-3.8) arc (180:360: 1.5 and 0.3);

\draw[thick, fill = red, fill opacity=0.7] (16.6,-2.11) circle (1);
\draw[thick, dotted] (15.6,-2.11) arc (180:0: 1 and 0.2);
\draw[thick] (15.6,-2.11) arc (180:360: 1 and 0.2);

\draw[thick, fill = yellow, fill opacity=0.3] (18,-0.71) circle (1);
\draw[thick, dotted] (17,-0.71) arc (180:0: 1 and 0.2);
\draw[thick] (17,-0.71) arc (180:360: 1 and 0.2);

\draw[thick, fill = red, fill opacity=0.7] (19.4,-2.11) circle (1);
\draw[thick, dotted] (18.4,-2.11) arc (180:0: 1 and 0.2);
\draw[thick] (18.4,-2.11) arc (180:360: 1 and 0.2);

\draw[thick, fill = darkblue, fill opacity=0.7] (14.7,-3.8) circle (1.5);
\draw[thick, dotted] (13.2,-3.8) arc (180:0: 1.5 and 0.3);
\draw[thick] (13.2,-3.8) arc (180:360: 1.5 and 0.3);

\node at (18,0.8) {\tiny $\cD$};

\draw[thick] (21.3,-5.2) to (21.3,-5.4);
\node at (21.3,-5.8) {\tiny $2$};

\draw[thick] (22.3, -2.9) to (22.5,-2.7);
\node at (22.7, -2.3) {\tiny $2$};

\draw[thick] (17.1, -1.5) to (17.35,-1.25);
\node at (16.4,-0.4) {\tiny $K$};

\draw[thick] (15.7, -2.9) to (15.95,-2.65);
\node at (16.9,-3.7) {\tiny $H_1$};

\draw[thick] (18.85,-1.5) to (18.65,-1.3);
\node at (19.6,-0.4) {\tiny $K$};

\draw[thick] (20.25,-2.9) to (20.05,-2.7);
\node at (19.1,-3.7) {\tiny $H_1$};

\draw[thick] (13.7, -2.9) to (13.5,-2.7);
\node at (13.3,-2.3) {\tiny $1$};

\draw[thick] (14.7, -5.2) to (14.7,-5.4);
\node at (14.7,-5.8) {\tiny $1$};

\end{tikzpicture}
\caption{The middle third of the middle interval can also be interpreted as a homotopy between these two configurations. \label{fig:upside_down_2}}
\end{figure}

The contributions of the outer thirds of the middle interval, depicted in Figure~\ref{fig:transition}, can be ignored in reduced symplectic homology. 
This is because we can choose the Floer data for the coproduct near the endpoints of the parametrizing interval, where it degenerates  through $-K$, to be very close to a further degeneration through $-H_1$. As a consequence, the Floer data on those intervals is close to being constant and therefore brings no index $-1$ contributions. We are therefore reduced to analyzing the contribution of the middle third of the middle interval.
As explained in \cite[\S 4.2]{CO-reduced}, after turning the ``upper'' components upside down once more (as in Figure~\ref{fig:upside_down_2}), the homotopy of continuation data from $\cD^{\mathrm{op}}$ to $\cD$ can be reinterpreted algebraically as giving rise to the term $\lambda_\cD \eta$. As we need to use the homotopy in the opposite direction, there is an additional sign, and so algebraically this edge gives rise to the term $-(\mu \otimes \mu)(1 \otimes \lambda_\cD \eta \otimes 1)$. In summary, we see that the Floer problem associated to this hexagon gives rise to the chain homotopy for the unital infinitesimal relation, whose proof along these lines first appeared in \cite[\S 6.4]{CO-reduced}.
\begin{proof}[Proof of Proposition~\ref{prop:11-term}]
Consider the 3-chain $\xi(\Ccal_{2,2}) \subseteq \Mcal_{2,2}$. According to Lemma~\ref{lem:simplified_11_term}, its boundary is homologous to
\begin{align*}
&-\hat R^*_{12} \times D_* \times \hat P^{12}_*
-D_{1^-} \times \hat R^*_{12} \times \hat P^{12}_* -D_{2^-} \times \hat R^*_{12}  \times \hat P^{12}_*\\
&+\hat R^*_{12} \times \hat P^{12}_* \times D_{1^+} +\hat R^*_{12}  \times \hat P^{12}_* \times D_{2^+}+ \hat P^{*2}_2 \times D_* \times \hat R^1_{1*}\\
&+ \hat P^{1*}_1 \times D_* \times \hat R^2_{*2}
+ \hat P^{*1}_2 \times D_* \times \hat R^2_{1*}+ \hat P^{2*}_1 \times D_* \times \hat R^1_{*2}\\
& + \hat P^{1*}_1 \times D_* \times \hat P^{*2}_{2}\times [0,1]
+\hat P^{2*}_1 \times D_* \times \hat P^{*1}_{2}\times [0,1]
\end{align*}
as a cycle in $\Mcal_{2,2}$ relative to $(\p \simplex_1) \times \Mcal^0_{2,2}$. We claim that these 11 summands correspond direclty to the terms in the 11-term relation. This is clear for the first 9 terms.
It remains to  identify the count of solutions to the Floer problems parametrized by the last two terms with $-(\mu \otimes \mu) (1 \otimes 1 \otimes \Delta \otimes 1) (1 \otimes \lambda \eta \otimes 1)$ and its version precomposed with $\tau$, respectively.

The two cases are similar, so we discuss $\hat P^{1*}_1 \times D_* \times \hat P^{*2}_{2}\times [0,1]$. Recall from the above discussion of the unital infinitesimal relation that the 1-chain $\hat P^{1*}_1 \times \hat P^{*2}_{2}\times [0,1]$ gives rise to the algebraic term $-(\mu \otimes \mu)(1 \otimes \lambda_\cD \eta \otimes 1)$. The contribution of the chain $D_*$ is an additional insertion of $\Delta$ at one of the two outputs of $\lambda_\cD \eta$. In fact, a moment's reflection shows that these two choices are chain homotopic, and so we choose to put it at the second ouput, which results in the desired algebraic translation $-(\mu \otimes \mu) (1 \otimes 1 \otimes \Delta \otimes 1) (1 \otimes \lambda_\cD \eta \otimes 1)$ for the contribution of this term.
\end{proof} 
\begin{proof}[Proof of Theorem~\ref{thm:BV_unital_infinitesimal_reduced}]
We have reviewed the definitions of $\Delta$, $\mu$ and $\lambda_\cD$ in \S\ref{sec:BVoperator} and~\ref{ssec:prod_and_coprod}. In particular, we discussed how to get the basic properties of $\mu$ and $\lambda_\cD$ (symmetries, associativity and coassociativity and the unital infinitesimal relation from \cite[\S6.4]{CO-reduced}) from the point of view taken in \S\ref{sec:moduli}, and we also discussed the 7-term relations for $\Delta$ and $\mu$ (originally due to \cite[Chapter~10, \S5.2]{Abouzaid-cotangent}) and for $\Delta$ and $\lambda_\cD$ (Proposition~\ref{prop:coBV}).

The 11-term relation for $\Delta$, $\mu$, $\lambda_\cD$ and $\eta$ is the content of Proposition~\ref{prop:11-term}.
\end{proof}

\section{BV Frobenius algebra structure on Rabinowitz Floer homology}\label{sec:RFH}

We use in this section notation from~\cite[\S5]{CHO-PD}. Let $W$ be a Liouville domain. Rabinowitz Floer homology $SH_*(\p W)$ of its boundary is defined using the family  
$$
\cH=\{H_{u,v}\, : \, u,v\in\R\}
$$ 
of Hamiltonians on $\widehat W$, equal to $0$ on $\p W$, linear in $r$ of slope $v$ on $[1,\infty)\times \p W$, linear in $r$ of slope $u$ on $[1/2,1]\times \p V$, and constant equal to $-u/2$ on $\{r\le 1/2\}$.  
Given $-\infty<a<b<\infty$ we define (see~\cite{Cieliebak-Frauenfelder-Oancea,CO})
$$
SH_*^{(a,b)}(\p W)=\lim^{\longrightarrow}_{v\to\infty,\, u\to-\infty}FH_*^{(a,b)}(H_{u,v}),
$$
and further
$$
SH_*(\p W)=\lim^{\longrightarrow}_{b\to\infty}\lim^{\longleftarrow}_{a\to-\infty}SH_*^{(a,b)}(\p W).
$$

We denote $S\H_*(\p W)=SH_{*+n}(\p W)$ and call it \emph{degree shifted Rabinowitz Floer homology}. 

\begin{theorem}[{\cite[Theorem~1.2]{CHO-PD}}] \label{thm:RFH-Frob} Degree shifted Rabinowitz Floer homology $S\H_*(\p W)$ is a graded Frobenius algebra. \qed
\end{theorem}

This is the main result proved in~\cite{CHO-PD}. In this section we improve it to 
\begin{theorem}\label{thm:BVFrob}
Degree shifted Rabinowitz Floer homology $S\H_*(\p W)$ is an odd BV Frobenius algebra. 
\end{theorem} 

{\bf Notation.} We will denote all operations in Rabinowitz Floer homology by boldface letters $\boldlambda$, $\boldeta$, $\boldc$ etc.

\begin{proof}
In view of Theorem~\ref{thm:RFH-Frob} and Definition~\ref{defi:odd-Frobenius}, we need to prove that $S\H_*(\p W)$ is a BV algebra, and also that the BV Frobenius relation 
\begin{equation} \label{eq:BV-Frob-relation}
(\boldDelta\otimes 1)\boldc = (1\otimes\boldDelta)\boldc
\end{equation}
holds, where $\boldc=\boldlambda\boldeta$ is the canonical nondegenerate copairing on Rabinowitz Floer homology. 

The BV algebra property for $S\H_*(\p W)$ is proved exactly as for the symplectic homology of a Liouville domain (see \S\ref{sec:reduced} or \cite[\S10]{Abouzaid-cotangent}), and we omit any further details.

We now prove the BV Frobenius relation~\eqref{eq:BV-Frob-relation}. Let $H=H_{u,v}$ with $u<0<v$ and let $L=H_{u,u}$. We denote $c_L:FC_*(-H)\to FC_*(H)$ the continuation map obtained by concatenating the continuation maps $FC_*(-H)\to FC_*(L)\to FC_*(H)$, and we denote by 
$c_{-L}:FC_*(-H)\to FC_*(H)$ the continuation map obtained by concatenating continuation maps $FC_*(-H)\to FC_*(-L)\to FC_*(H)$. We fix in the sequel a finite action window $(a,b)$ with $a<0<b$ and consider Floer complexes supported in that action window. We also consider parameters $u\ll a$ and $v \gg b$.  Let $\vec \boldc:FC_*^{(a,b)}(-H)\to FC_{*+1}^{(a,b)}(H)$ be the secondary continuation map interpolating between $c_{-L}$ and $c_L$ as in~\cite[\S5.6]{CHO-PD} (the action truncated ``primary" continuation maps $c_{-L}$ and $c_L$ vanish for $u\ll a$ and $v \gg b$). After identifying $FC_*(-H)$ with $FC_*(H)^\vee$, one can view the copairing $\boldc=\boldlambda\boldeta$ in Rabinowitz Floer homology as being implicitly induced from $\vec \boldc$ via dualization
$$
\vec \boldc = (\ev\otimes 1)(1\otimes \boldc).
$$

We define BV operators $\Delta_{-H}:FC_*(-H)\to FC_{*+1}(-H)$ and $\Delta_H:FC_*(H)\to FC_*(H)$ by slightly perturbing the autonomous Hamiltonians $\pm H$ to nondegenerate (and non-autonomous) ones. Such BV operators have the property that they preserve action filtrations. The proof from~\cite[Chapter~10, \S2.3]{Abouzaid-cotangent} showing that BV operators commute with continuation maps adapts in a straightforward way in order to show that the BV operators $\Delta_{\pm H}$ commute up to homotopy with the secondary continuation map $\vec \boldc$, i.e., 
\begin{equation} \label{eq:Delta-vecc}
\vec \boldc\circ \Delta_{-H} + \Delta_H\circ \vec \boldc=[\p,\alpha_1],
\end{equation}
for a suitable map $\alpha_1:FC_*^{(a,b)}(-H)\to FC_{*+3}^{(a,b)}(H)$. The plus sign is a consequence of the fact that the operations $\vec \boldc$ and $\Delta_{\pm H}$ arise from counts of solutions of parametrized Floer problems with odd dimensional parameter spaces. In our case the parameter spaces are $[-1,1]$ and $S^1$, and switching the order of the parameters results in a change of orientation of the product parameter space $S^1\times [-1,1]$. 

The complex $FC_*^{(a,b)}(-H)$ is naturally identified with $FC_*^{(-b,-a)}(H)^\vee$. Via this identification the BV operator $\Delta_{-H}$ corresponds to the dual of the BV operator $\Delta_H$, i.e., 
\begin{equation*} 
\Delta_{-H}\equiv \Delta_H^\vee.
\end{equation*}
This can be seen by writing explicit models for the moduli spaces that define the BV operators, and noting that, in the definition of the BV operator $\Delta$, inserting a positive rotation at the output is equivalent to inserting a negative rotation at the input. 

In view of this identification equation~\eqref{eq:Delta-vecc} becomes
\begin{equation} \label{eq:Deltavee-vecc}
\vec \boldc\circ \Delta_H^\vee + \Delta_H\circ \vec \boldc=[\p,\alpha_2],
\end{equation}
where $\alpha_2:FC_*^{(-b,-a)}(H)^\vee\to FC_{*+3}^{(a,b)}(H)$ is the map that corresponds to $\alpha_1$ under the identification. 

We now use the action filtered copairing $\boldc$ defined by $\vec \boldc=(\ev\otimes 1)(1\otimes \boldc)$. 
Inserting this into~\eqref{eq:Deltavee-vecc} we obtain
\begin{align*}
-[\p,\alpha_2]&=
-(\ev\otimes 1) (1\otimes \boldc)\Delta_H^\vee + \Delta_H(\ev\otimes 1)(1\otimes \boldc) \\
& = (\ev\otimes 1)(\Delta_H^\vee\otimes \boldc)+(\ev\otimes \Delta_H)(1\otimes \boldc)\\
& = (\ev(\Delta_H^\vee\otimes 1)\otimes 1)(1\otimes \boldc)+ (\ev\otimes 1)(1\otimes (1\otimes\Delta_H)\boldc)\\
& = (\ev(1\otimes \Delta_H)\otimes 1)(1\otimes \boldc)+ (\ev\otimes 1)(1\otimes (1\otimes\Delta_H)\boldc)\\
& = (\ev\otimes 1)(1\otimes (\Delta_H\otimes 1)\boldc) + (\ev\otimes 1)(1\otimes (1\otimes \Delta_H)\boldc)\\
& = (\ev\otimes 1)((\Delta_H\otimes 1)\boldc -(1\otimes\Delta_H)\boldc).
\end{align*} 
The second equality is the Koszul sign rule, and the forth equality is the adjunction formula for $\Delta_H^\vee$. 
We apply the map $1\otimes -$ to this equality and precompose with the chain map $\coev$. The left hand side writes $[\p,\alpha_3]$, with $\alpha_3=(1\otimes (-\alpha_2))\coev$. The right hand side writes 
$$
((1\otimes\ev)(\coev\otimes 1)\otimes 1)((\Delta_H\otimes 1)\boldc -(1\otimes\Delta_H)\boldc) = (\Delta_H\otimes 1)\boldc -(1\otimes\Delta_H)\boldc.
$$ 
We have thus proved the chain level equality 
$$
[\p,\alpha_3] = (\Delta_H\otimes 1)\boldc -(1\otimes\Delta_H)\boldc,
$$
for a suitable element $\alpha_3$. The BV Frobenius relation $(\boldDelta\otimes 1)\boldc =(1\otimes \boldDelta)\boldc$ therefore holds in action filtered homology. Passing to the limit over $H$ and letting $a\to-\infty$, $b\to\infty$ as in the definition of the Rabinowitz Floer homology groups we obtain the desired BV Frobenius relation on Rabinowitz Floer homology. 
\end{proof}

\section{Application to string topology}\label{sec:stringtop}

\subsection{Results and conjectures}
In this first subsection we present our results and state some conjectures. The following subsections will contain further details on some of the proofs. Let $Q$ be a closed oriented smooth manifold of dimension $n$. Let $\Lambda=\Lambda Q=\{\gamma:S^1\to Q\}$ be the space of free loops on $Q$, and let $\H_*\Lambda=H_{*+n}\Lambda$ be the \emph{loop homology of $Q$}. 

\subsubsection{Algebras} Loop homology carries the famous \emph{Chas-Sullivan product} $\mu$, which is associative, commutative, and unital, with unit denoted $\eta$ and represented by the fundamental class of $Q$ under the inclusion of constant loops $\Lambda_0\subset \Lambda$. It also carries a natural \emph{BV operator} $\Delta$ of degree one induced by the canonical $S^1$-action on $\Lambda$ by shift reparametrizations of loops in the source. The first foundational result of string topology is the following. 

\begin{theorem}[{Chas-Sullivan~\cite[Theorem~5.4]{CS}}] \label{thm:H-BV-algebra}
\qquad 

Loop homology $(\H_*\Lambda,\mu,\eta,\Delta)$ is a BV algebra. \qed
\end{theorem}

The \emph{string homology of $Q$} is the $S^1$-equivariant homology group   
$\H_*^{S^1}\Lambda= H_{*+n}^{S^1}\Lambda$. 
The \emph{string bracket} on $\H_*^{S^1}\Lambda$ is 
$$
\beta^{S^1}=E\mu(M\otimes M),
$$ 
where 
$E:\H_*\Lambda \to \H_*^{S^1}\Lambda$ and $M:\H_*^{S^1}\Lambda\to \H_{*+1}\Lambda$ 
are the ``erase'' and ``mark'' maps in the Gysin long exact sequence 
\begin{equation} \label{eq:Gysin}
\xymatrix{
\dots \ar[r] & \H_*\Lambda \ar[r]^-E & \H_*^{S^1}\Lambda \ar[r] & \H_{*-2}^{S^1}\Lambda \ar[r]^-M & \H_{*-1}\Lambda \ar[r] & \dots
}
\end{equation}
(The middle map is cap product with the Euler class of the $S^1$-bundle $\Lambda Q \times ES^1\to  \Lambda Q \times_{S^1} ES^1$.) Note that $\Delta = ME$ and $EM=0$, which implies $\Delta M = 0$ and $E\Delta = 0$. 

The second foundational result of string topology is the following. 

\begin{theorem}[{Chas-Sullivan~\cite[Theorem~6.1]{CS}}]  \label{thm:HS1-Lie-algebra}
\qquad 

String homology $(\H_*^{S^1}\Lambda,\beta^{S^1})$ is a Lie algebra. \qed
\end{theorem}

\subsubsection{Bialgebras} 
Our purpose in this section is to upgrade the previous two theorems to statements about bialgebras. These are Theorems~\ref{thm:H-BV-bialgebra} and~\ref{thm:HS1-Lie-bialgebra} below. 

Following~\cite{CHO-MorseFloerGH}, the relevant non-equivariant object is the \emph{reduced loop homology of $Q$},   
$$
\ol\H_*\Lambda=\H_*\Lambda/\chi(Q)[\mathrm{pt}].
$$
It was proved in~\cite{CHO-MorseFloerGH} that the Chas-Sullivan loop product on $\H_*\Lambda$ descends to a unital associative and commutative product $\mu$ on $\ol\H_*\Lambda$, with unit $\eta$ represented by the fundamental class. It was also proved in the same reference that, using field coefficients, a choice of continuation data $\cD$ as in~\S\ref{sec:reduced} determines in dimension $n\ge 3$ a coassociative and cocommutative \emph{loop coproduct} $\lambda_\cD$ on $\ol\H_*\Lambda$, and $(\ol\H_*\Lambda, \mu,\lambda_\cD,\eta)$ is a unital infinitesimal bialgebra. The BV operator on $\H_*\Lambda$ descends to $\ol\H_*\Lambda$ as well, denoted still by $\Delta$. 

\begin{theorem}  \label{thm:H-BV-bialgebra}
Reduced loop homology $(\ol\H_*\Lambda, \mu,\lambda_\cD,\eta,\Delta)$ is a BV unital infinitesimal bialgebra whenever $\dim Q\ge 4$.
\end{theorem}
We deduce this theorem from the corresponding results for reduced symplectic homology in \S\ref{sec:SH-H}.

To phrase the bialgebra counterpart of Theorem~\ref{thm:HS1-Lie-algebra} we introduce \emph{reduced string homology}
$$
\ol\H_*^{S^1}\Lambda=\H_*^{S^1}\Lambda/\chi(Q)[\mathrm{pt}].
$$
Here $\chi(Q)[\mathrm{pt}]\in\H_{-n}^{S^1}\Lambda$ is the image under the ``erase" map $E:\H_*\Lambda\to \H_*^{S^1}\Lambda$ of the class $\chi(Q)[\mathrm{pt}]\in\H_{-n}\Lambda$. As a consequence of the definition, the map $E$ induces a map between reduced homologies
$$
E:\ol\H_*\Lambda\to \ol\H_*^{S^1}\Lambda. 
$$
That the map $M$ also induces a map between reduced homologies 
$$
M:\ol\H_*^{S^1}\Lambda\to \ol\H_{*+1}\Lambda
$$
follows from the fact that the inclusion $i: Q\hookrightarrow \Lambda$, which is $S^1$-equivariant (for the trivial $S^1$-action on $Q$), induces a map of Gysin sequences 
$$
\xymatrix{
H_*\Lambda \ar[r]^E & H_*^{S^1}\Lambda \ar[r] & H_{*-2}^{S^1}\Lambda \ar[r]^M & H_{*-1}\Lambda \\
H_*Q \ar@{>->}[u]_{i_*} \ar@{>->}[r]^E & H_*^{S^1}Q \ar[u]_{i_*^{S^1}} \ar@{->>}[r] & H_{*-2}^{S^1}Q \ar[u]^{i_*^{S^1}} \ar[r]^0 & H_{*-1}Q \ar@{>->}[u]_{i_*} 
}
$$
The connecting map in the Gysin sequence for $Q$ vanishes because the $S^1$-action is trivial. As a consequence, the map $M$ factors through $\H_*^{S^1}\Lambda/i_*^{S^1}(\H_*^{S^1}Q)$, and in particular also through the quotient of $\H_*^{S^1}\Lambda$ by the smaller subgroup $i_*^{S^1}(\chi(Q)[\mathrm{pt}])$.
\begin{remark}
The maps $M$ and $E$ also induce maps with source, respectively target, $\H_*^{S^1}\Lambda/\chi(Q)\cdot i_*^{S^1}(\H_*^{S^1}(\mathrm{pt}))$. Our motivation to define reduced string homology as $\H_*^{S^1}\Lambda/\chi(Q)[\mathrm{pt}]$ is that this is the \emph{largest} quotient on which the maps $M$ and $E$ are defined with target, respectively source, $\ol\H_*\Lambda$. An algebraic structure that exists on $\ol\H_*^{S^1}\Lambda$ will also exist on smaller quotients. 
\end{remark}
\begin{definition}
The \emph{string bracket} and \emph{string cobracket} on $\ol\H_*^{S^1}\Lambda$ are defined as 
$$
\beta^{S^1} = E\mu (M\otimes M),\qquad \gamma_\cD^{S^1}=(E\otimes E)\lambda_\cD M,
$$
where $\mu$ and $\lambda_\cD$ are the loop product and loop coproduct on reduced loop homology, and $E$, $M$ are understood to act between reduced homologies.
\end{definition}

The string cobracket depends on the continuation data $\cD$ through the loop coproduct. Following~\cite{CHO-MorseFloerGH}, if $H_1Q=0$ then the loop coproduct does not depend on $\cD$, and therefore the string cobracket does not depend on $\cD$ either. It would be interesting to investigate more thoroughly the dependence of $\gamma_\cD^{S^1}$ on $\cD$, in the manner of~\cite{CHO-MorseFloerGH}. 

\begin{theorem} \label{thm:HS1-Lie-bialgebra}
\qquad 

Reduced string homology $(\ol\H_*^{S^1}\Lambda,\beta^{S^1},\gamma_\cD^{S^1})$ is a Lie bialgebra whenever $\dim Q \ge 4$. 
\end{theorem}

We prove this result in \S\ref{sec:SH-H}. Here we discuss two ramifications of this theorem.

(1) Theorem~\ref{thm:HS1-Lie-bialgebra} implies the main result of Chas and Sullivan from~\cite{CS-Lie} under the assumption that the map $i_*^{S^1}:H_*^{S^1}Q\to H_*^{S^1}\Lambda$ is injective. 

This assumption ensures that the relative homology group $\H_*^{S^1}(\Lambda,\Lambda_0)$ is identified with the quotient $\H_*^{S^1}\Lambda/\H_*^{S^1}Q$. Then the operation $E\mu(M\otimes M)$, which is \emph{a priori} defined on $\H_*^{S^1}\Lambda$, descends to an operation denoted $\beta^{S^1}$ on $\H_*^{S^1}\Lambda/\H_*^{S^1}Q=\H_*^{S^1}(\Lambda,\Lambda_0)$. On the other hand, the maps $M$ and $E$ have counterparts acting between loop/string homologies of the pair $(\Lambda,\Lambda_0)$, and the coproduct $\lambda_\cD$ descends to the Sullivan-Goresky-Hingston coproduct $\lambda$ on $H_*(\Lambda,\Lambda_0)$ (that does not depend on $\cD$), resulting in a well defined cobracket $\gamma^{S^1}=(E\otimes E)\lambda_\cD M$. 
With these definitions, we see that the operations on $\H_*^{S^1}(\Lambda,\Lambda_0)$ descend from the operations on $\ol\H_*^{S^1}\Lambda$. As a consequence, under the above injectivity assumption for $i_*^{S^1}$ and for $\dim Q \ge 4$, Theorem~\ref{thm:HS1-Lie-bialgebra} directly implies the following result. 

\begin{theorem}[{Chas-Sullivan~\cite{CS-Lie}}]  \label{cor:HS1-relQ-Lie-bialgebra}
\qquad 

String homology $(\H_*^{S^1}(\Lambda,\Lambda_0),\beta^{S^1},\gamma^{S^1})$ is a Lie bialgebra. \qed
\end{theorem}

\begin{remark} A proof of this result is sketched in~\cite{CS-Lie}. The original statement also contains involutivity of this Lie bialgebra structure, whose proof is beyond the scope of this paper. 
\end{remark}

(2) Theorem~\ref{thm:HS1-Lie-bialgebra} admits a generalization in symplectic homology. Let $W$ be a Liouville domain of dimension $2n$, and let $SH_*^{S^1}(W)$ be the $S^1$-equivariant symplectic homology of $W$ (see~\cite{BO3+4,BOGysin} and~\cite[\S8.2]{CO} for its definition and properties). This fits into the Gysin sequence
$$
{\tiny
\xymatrix
@C=15pt
{
SH_*(W) \ar[r]^E & SH_*^{S^1}(W) \ar[r] & SH_{*-2}^{S^1}W) \ar[r]^M & SH_{*-1}(W) \\
H_{*+n}(W,\p W) \ar @{->} [u]_{i_*} \ar@{>->}[r]^E & H_{*+n}^{S^1}(W,\p W) \ar[u]_{i_*^{S^1}} \ar@{->>}[r] & H_{*-2+n}^{S^1}(W,\p W) \ar[u]^{i_*^{S^1}} \ar[r]^0 & H_{*-1+n}(W,\p W) \ar@{->}[u]_{i_*} 
}
}
$$
Assume now that $W$ is a strongly $R$-essential Weinstein domain, and recall from~\S\ref{sec:reduced} the reduced symplectic homology group $\ol{SH}_*(W)$.
This can be equivalently written as $\ol{SH}_*(W)=SH_*(W)/i_*c_*H_{*+n}W$, where $c_*:H_{*+n}W\to H_{*+n}(W,\p W)$ is the canonical map induced by the inclusion $W\hookrightarrow (W,\p W)$. We define analogously the \emph{reduced $S^1$-equivariant symplectic homology group} 
$$
\ol{SH}_*^{S^1}(W)=SH_*^{S^1}(W)/i_*^{S^1}c_*^{S^1}H_{*+n}^{S^1}W,
$$
and further let $\ol{S\H}_*^{S^1}(W)=\ol{SH}_{*+n}^{S^1}(W)$.

Arguing as for $\ol{\H}_*^{S^1}\Lambda$, one shows that the ``erase" and ``mark" maps $E$ and $M$ descend to maps between reduced symplectic homologies $E:\ol{S\H}_*(W)\to \ol{S\H}_*^{S^1}(W)$ and $M:\ol{S\H}_*^{S^1}(W)\to \ol{S\H}_{*+1}(W)$, and we can define the \emph{bracket} $\beta^{S^1}$ and the \emph{cobracket} $\gamma^{S^1}$ on $\ol{S\H}_*^{S^1}(W)$ by 
$$
\beta^{S^1}=E\mu(M\otimes M),\qquad \gamma^{S^1}=(E\otimes E)\lambda_\cD M,
$$
where $\mu$ and $\lambda_\cD$ are the product and coproduct on $\ol{S\H}_*(W)$ as in~\S\ref{sec:reduced}. 

\begin{theorem} \label{thm:SHS1-Lie-bialgebra}
%
Reduced $S^1$-equivariant symplectic homology\\ $(\ol{S\H}_*^{S^1}(W),\beta^{S^1},\gamma_\cD^{S^1})$ is a Lie bialgebra whenever $\dim W\ge 8$. 
\end{theorem}

\subsubsection{Frobenius algebras} Rabinowitz loop homology $\widehat\H_*\Lambda$ was introduced in~\cite{CHO-PD} as a new object of interest in string topology. It was originally defined as Rabinowitz Floer homology $S\H_*(S^*Q;\sigma)$, where $S^*Q=\p D^*Q$ is the unit sphere cotangent bundle of $Q$, seen as the boundary of the unit disc cotangent bundle $D^*Q$, and $\sigma$ is a suitable local system on $\Lambda$, see~\S\ref{sec:SH-H} below. It was proved in~\cite{CHO-PD} that $\widehat\H_*\Lambda$ carries the structure of an odd Frobenius algebra with product $\mu$, coproduct $\lambda$, unit $\eta$, and counit $\eps$. Since it is defined as a symplectic homology group, it also carries a BV operator $\Delta$, see also the discussion in~\S\ref{sec:SH-H} below. 

\begin{theorem} \label{thm:RFH-BV-Frobenius}
Rabinowitz loop homology $(\widehat\H_*\Lambda,\mu,\lambda,\eta,\eps,\Delta)$ is an odd BV Frobenius algebra.  
\end{theorem}

\begin{proof}
By definition $\widehat\H_*\Lambda=S\H_*(S^*Q;\sigma)$, so this is a variant of Theorem~\ref{thm:BVFrob} in the presence of the local system $\sigma$. The proof is verbatim the same. The local system $\sigma$ is discussed in more detail in~\S\ref{sec:SH-H}.
\end{proof}

\begin{remark}
Rabinowitz loop homology was studied from a Morse theoretic perspective in~\cite{CHO-MorseFloerGH} using the formalism of cones and $A_2^+$ structures from~\cite{CO-cones}. It is an interesting question to upgrade the notion of an $A_2^+$ structure from~\cite{CO-cones} to the BV setting. See also the conjectures below.
\end{remark}

The next result follows essentially from~\cite[Theorem~7.8]{CO-reduced}. 

\begin{theorem} \label{thm:H-double} Assume $H_1Q=0$ and $\dim Q\ge 3$. The odd Frobenius algebra structure on $\widehat\H_*\Lambda$ is isomorphic to the double of the odd unital infinitesimal structure on $\ol\H_*\Lambda$. 
\end{theorem}

\begin{proof}
The assumption $H_1Q=0$ ensures that $\widehat\H_*\Lambda$ splits canonically as the double of $\ol\H_*\Lambda$. Moreover, the secondary continuation bivector that appears in the formulas from~\cite[Theorem~7.8]{CO-reduced} vanishes, and this shows that the product and the unit on $\widehat \H_*\Lambda$ coincide with the one on the double as described in the proof of Proposition~\ref{prop:BVui-BVFrob}. A similar computation shows that the coproduct on $\widehat\H_*\Lambda$ also coincides with the one on the double, and therefore the copairings coincide as well.
\end{proof}

The above theorem has a generalization in symplectic homology. 

\begin{theorem} \label{thm:SH-double} Let $W$ be a strongly essential Weinstein domain of dimension $2n\ge 6$ such that $H^{n-1}W=0$. The odd Frobenius algebra structure on $S\H_*(\p W)$ is isomorphic to the double of the odd unital infinitesimal structure on $\ol{S\H}_*(W)$. 
\end{theorem}

\begin{proof}
The proof is similar to that of Theorem~\ref{thm:H-double}. It relies on~\cite[Theorem~7.8]{CO-reduced}, which is phrased for symplectic homology. 
\end{proof}

The following conjecture is within reach of the methods used in~\cite{CO-cones,CO-reduced} and in this paper. 

\begin{conjecture} \label{conj:H-doubles}
\qquad 

(i) Assuming $H_1Q=0$ and $\dim Q\ge 4$, the BV operator on $\widehat\H_*\Lambda$ is isomorphic to the BV operator on the double of $\ol\H_*\Lambda$, so that Theorem~\ref{thm:H-double} can be upgraded to an isomorphism of odd BV Frobenius algebra structures. 

(ii) Assuming $H^{n-1}W=0$ and $2n\ge 8$, the BV operator on $S\H_*(\p W)$ is isomorphic to the BV operator on the double of $\ol{S\H}_*(W)$, so that Theorem~\ref{thm:SH-double} can be upgraded to an isomorphism of odd BV Frobenius algebra structures. 

(iii) The double construction from Proposition~\ref{prop:BVui-BVFrob} generalizes to all odd BV unital infinitesimal algebras, i.e.,  the assumption $\lambda\eta=0$ is redundant in that proposition. 

(iv) With respect to this generalized double construction, the assumptions $H_1Q=0$ in Theorem~\ref{thm:H-double} and $H^{n-1}W=0$ in Theorem~\ref{thm:SH-double} can be dropped, and the isomorphisms from these two theorems can be upgraded to BV.  
\end{conjecture}

A solution to Conjecture~\ref{conj:H-doubles}(iii) should be informed by the explicit formulas from~\cite[Theorem~7.8]{CO-reduced}, which hold without the assumption $H_1Q=0$ or $H^{n-1}(W)=0$.  A solution to Conjecture~\ref{conj:H-doubles}(i-ii-iv) would go through a BV enhancement of the cone formalism from~\cite{CO-cones}. 

We now discuss the $S^1$-equivariant analogue of Rabinowitz loop homology. We define \emph{Rabinowitz string homology} as the $S^1$-equivariant Rabinowitz Floer homology of $S^*Q$, i.e.,
$$
\widehat\H_*^{S^1}\Lambda=S\H_*^{S^1}(S^*Q;\sigma).
$$
More generally, for a Liouville domain $W$ one defines its $S^1$-equivariant Rabinowitz Floer homology $S\H_*^{S^1}(\p W)$ as in~\cite[\S8.2]{CO}.

The double construction was originally introduced by Drinfel'd in the context of Lie algebras~\cite{Drinfeld-quantum-groups}. As more recent references we point the reader to~\cite{Kosmann-Schwarzbach}, the lecture notes~\cite{JF,JF-Reshetikhin}, and~\cite[\S5 and Appendix]{Bai2010}. Given a finite dimensional Lie bialgebra $\mathfrak{g}$, its double $\mathfrak{g}\oplus\mathfrak{g}^*$ is a Lie algebra whose canonical symmetric pairing $\pi(x+a^*,y+b^*)=\langle x,b^*\rangle + \langle a^*,y\rangle$ is nondegenerate and invariant, i.e., $\pi([x,y],z)=\pi(x,[y,z])$~\cite[\S1.6]{Kosmann-Schwarzbach}. A little bit more is true: the double is a ``factorizable" Lie bialgebra~\cite[\S2.5]{Kosmann-Schwarzbach}. Note however that, in contrast to the non-equivariant setting, the canonical pairing on the double is not ``Frobenius" in the sense that it does not turn the bracket into the cobracket and vice-versa.\footnote{There is an established notion of ``Frobenius Lie algebra" in the literature~\cite{Ooms1980,Elashvili1982}, but this is not the structure present on the double of a Lie algebra.} We expect that the doubling construction extends to (filtered) graded Lie bialgebras, producing a (filtered) graded factorizable Lie bialgebra.

In analogy with the previous discussion on $\widehat\H_*\Lambda$ and $S\H_*(\p W)$, we conjecture the following. 

\begin{conjecture} \label{conj:HS1-doubles}
 (i) Rabinowitz string homology $\widehat\H_*^{S^1}\Lambda$, and more generally $S^1$-equivariant Rabinowitz Floer homology $S\H_*^{S^1}(\p W)$, are graded factorizable Lie bialgebras. 

 (ii) Assume $Q$ is orientable and $\dim Q\ge 4$. The graded factorizable Lie bialgebra structure on $\widehat \H_*^{S^1}\Lambda$ is isomorphic to the one on the double of the Lie bialgebra $\ol\H_*^{S^1}\Lambda$. 

 (iii) Assume $W$ is a strongly essential Weinstein domain of dimension $2n\ge 8$. The graded factorizable Lie bialgebra structure on $S\H_*^{S^1}(\p W)$ is isomorphic to the one on the double of the Lie bialgebra $\ol{S\H}_*^{S^1}(W)$.    
\end{conjecture}

The rest of the section is structured as follows. In~\S\ref{sec:SH-H} we prove Theorem~\ref{thm:H-BV-bialgebra} by specializing our results on reduced symplectic homology from~\S\ref{sec:reduced} to the case of cotangent bundles. In~\S\ref{sec:equivariant} we deduce Theorem~\ref{thm:HS1-Lie-bialgebra} from Theorem~\ref{thm:H-BV-bialgebra}, and we deduce Theorem~\ref{thm:SHS1-Lie-bialgebra} from Theorem~\ref{thm:BV_unital_infinitesimal_reduced}. 

\subsection{From symplectic homology to loop space homology} \label{sec:SH-H}

We describe in this subsection the passage from the symplectic setting to the topological setting. Let $D^*Q$ be the unit cotangent bundle of $Q$ (with respect to some Riemannian metric on $Q$). The free loop space $\Lambda$ carries a local system $\sigma$ of free rank one abelian groups obtained by transgressing the second Stiefel-Whitney class of $Q$~\cite{Kragh,Abouzaid-cotangent}. This induces canonically a local system, also denoted $\sigma$, on the free loop space of $T^*Q$, and the symplectic homology groups with local coefficients in this local system are denoted $S\H_*(D^*Q;\sigma)$, with corresponding \emph{reduced} versions $\ol{S\H}_*(D^*Q;\sigma)$ defined as in~\S\ref{sec:reduced}. Following~\cite[\S10.3-4]{Abouzaid-cotangent}, see also~\cite[Appendix~A]{CHO-MorseFloerGH}, $\ol{S\H}_*(D^*Q;\sigma)$ carries a unital associative and commutative product $\mu$, with unit denoted $\eta$.
With field coefficients and associated to a choice of continuation data $\cD$, it also carries a coassociative and cocommutative coproduct $\lambda_\cD$. It also carries a BV operator $\Delta$ constructed as in~\S\ref{sec:BVoperator} (the construction goes through because the monodromy of $\sigma$ is trivial along the orbits of the $S^1$-action, as explained in~\cite[\S10.2, eq.~(10.12)]{Abouzaid-cotangent}). We obtain

\begin{theorem} \label{thm:SH-sigma-BV-bialgebra} Reduced symplectic homology with local coefficients 
$(\ol{S\H}_*(D^*Q;\sigma),\mu,\lambda_\cD,\eta,\Delta)$ is a BV unital infinitesimal bialgebra. \qed
\end{theorem}

This is the 
analogue of Theorem~\ref{thm:BV_unital_infinitesimal_reduced} for twisted coefficients, and the proof is verbatim the same. Indeed, disc cotangent bundles of orientable manifolds are strongly essential Weinstein domains~\cite{CO-reduced}, and the only effect of the local system $\sigma$ is a modification of the signs when counting solutions of Floer-type equations in the definitions of the boundary operator and the various operations. Note that, since we assume our base manifold $Q$ to be orientable and oriented, there is no ``twist'' on the BV structure in the sense of~\cite[\S10]{Abouzaid-cotangent}.

Symplectic homology of $D^*Q$ with local coefficients in $\sigma$ is relevant in this context because of the \emph{Viterbo isomorphism}     
$$
\Psi:S\H_*(D^*Q;\sigma)\stackrel\simeq\longrightarrow \H_*\Lambda.
$$
While the original Viterbo isomorphism went through the intermediate of generating function homology~\cite{Viterbo-cotangent}, the previous map was constructed by Abbondandolo-Schwarz~\cite{AS-Legendre} and Abouzaid~\cite{Abouzaid-cotangent}, and it was revisited by Cieliebak and the second author of this paper in~\cite{CHO-MorseFloerGH}. 

\begin{theorem} \label{thm:Psi-reduced} The map $\Psi$ descends to an isomorphism of unital infinitesimal bialgebras 
$$
\Psi:\ol{S\H}_*(D^*Q;\sigma)\stackrel\simeq\longrightarrow \ol\H_*\Lambda
$$
which intertwines the BV operators. 
\end{theorem}

\begin{proof}
That $\Psi$ descends was proved in~\cite{CHO-MorseFloerGH}. That $\Psi$ intertwines the products was proved by Abbondandolo-Schwarz~\cite{AS-Legendre} and Abouzaid~\cite{Abouzaid-cotangent}. That $\Psi$ intertwines the coproducts (for the same choice of continuation data on both sides) was proved in~\cite{CHO-MorseFloerGH}. That $\Psi$ intertwines the BV operators was proved by Abouzaid~\cite{Abouzaid-cotangent}. 
\end{proof}

We infer Theorem~\ref{thm:H-BV-bialgebra} as a corollary. 

\begin{proof}[Proof of Theorem~\ref{thm:H-BV-bialgebra}]
By Theorem~\ref{thm:Psi-reduced} the map $\Psi$ intertwines the products, coproducts, units, and BV operator. By Theorem~\ref{thm:SH-sigma-BV-bialgebra} the domain of $\Psi$ is a BV unital infinitesimal bialgebra. Sine $\Psi$ is an isomorphism, the target of $\Psi$ is also a BV unital infinitesimal bialgebra. 
\end{proof}

\subsection{From nonequivariant to equivariant string topology} \label{sec:equivariant}

In this section we prove Theorem~\ref{thm:HS1-Lie-bialgebra} as a formal consequence of Theorem~\ref{thm:H-BV-bialgebra}.

\begin{proof}[Proof of Theorem~\ref{thm:HS1-Lie-bialgebra}] The Jacobi relation for $\beta^{S^1}$ is obtained from the 7-term relation for $\Delta$ and $\mu$ by precomposing with $M\otimes M\otimes M$ and postcomposing with $E$. Similarly, the coJacobi relation for $\gamma^{S^1}$ is obtained from the 7-term relation for $\Delta$ and $\lambda$ by precomposing with $M$ and postcomposing with $E\otimes E \otimes E$. 

It remains to prove Drinfel'd compatibility.  
Recall from Remark~\ref{rmk:11-term-is-9-term-for-SH} that, in the case of reduced symplectic homology, the 11-term identity reduces to the 9-term identity
\begin{center}
\begin{tikzpicture}[scale=.25]

\draw (0,0) -- (1,-1) -- (2,0);
\draw (1,-1) -- (1,-3);
\node at (0,-2) {\tiny $\Delta$};
\draw (0,-4) -- (1,-3) -- (2,-4);
\node at (4,-2) {$=$};

\pgftransformxshift{6cm}
\draw (0,0) -- (1,-1) -- (2,0);
\draw (1,-1) -- (1,-3);
\draw (0,-4) -- (1,-3) -- (2,-4);
\node at (0,0.8) {\tiny$\Delta$};
\node at (4,-2) {$+$};

\pgftransformxshift{6cm}
\draw (0,0) -- (1,-1) -- (2,0);
\draw (1,-1) -- (1,-3);
\draw (0,-4) -- (1,-3) -- (2,-4);
\node at (2,0.8) {\tiny$\Delta$};
\node at (4,-2) {$-$};

\pgftransformxshift{6cm}
\draw (0,0) -- (1,-1) -- (2,0);
\draw (1,-1) -- (1,-3);
\draw (0,-4) -- (1,-3) -- (2,-4);
\node at (0,-4.8) {\tiny$\Delta$};
\node at (4,-2) {$-$};

\pgftransformxshift{6cm}
\draw (0,0) -- (1,-1) -- (2,0);
\draw (1,-1) -- (1,-3);
\draw (0,-4) -- (1,-3) -- (2,-4);
\node at (2,-4.8) {\tiny$\Delta$};

\pgftransformxshift{-18cm}
\pgftransformyshift{-8cm}

\node at (0,-2) {$+$};
\pgftransformxshift{2cm}

\draw (0,0) -- (0,-2) -- (1,-3) -- (1,-4);
\draw (1,-3) -- (3,-1);
\node at (1.5,-1.5) {\tiny$\Delta$};
\draw (3,0) -- (3,-1) -- (4,-2) -- (4,-4);
\node at (7,-2.5) {$+$};

\pgftransformxshift{9cm}
\node at (0,.8) {\tiny$2$};
\node at (3,.8) {\tiny$1$};
\draw (0,0) -- (0,-2) -- (1,-3) -- (1,-4);
\draw (1,-3) -- (3,-1);
\node at (1.5,-1.5) {\tiny$\Delta$};
\draw (3,0) -- (3,-1) -- (4,-2) -- (4,-4);
\node at (7,-2.5) {$+$};

\pgftransformxshift{9cm}

\draw (1,0)  -- (1,-1) -- (0,-2) -- (0,-4);
\draw (1,-1) -- (3,-3);
\node at (2.5,-1.5) {\tiny$\Delta$};
\draw (4,0) -- (4,-2) -- (3,-3) -- (3,-4);
\node at (7,-2.5) {$+$};

\pgftransformxshift{9cm}
\node at (1,.8) {\tiny$2$};
\node at (4,.8) {\tiny$1$};
\draw (1,0)  -- (1,-1) -- (0,-2) -- (0,-4);
\draw (1,-1) -- (3,-3);
\node at (2.5,-1.5) {\tiny$\Delta$};
\draw (4,0) -- (4,-2) -- (3,-3) -- (3,-4);
\pgftransformxshift{9cm}
\end{tikzpicture}
\end{center}
By simultaneously pre-composing this identity with $M\otimes M=(M \otimes 1)(1 \otimes M)$ and post-composing it with $E\otimes E$, we obtain a modified version of this 9-term identity, which in turn implies Drinfel'd compatibility for the string bracket and cobracket in the form
\begin{equation}\label{eq:drinfeld_in_string_topology}
\gamma_\cD^{S^1} \beta^{S^1}=(\beta^{S^1} \otimes 1)(1 \otimes \gamma_\cD^{S^1})(1-\tau)
 + (1\otimes \beta^{S^1})(\gamma_\cD^{S^1} \otimes 1)(1-\tau).
\end{equation}
Indeed, it is clear that the left hand side of the modified 9-term identity gives the left hand term in \eqref{eq:drinfeld_in_string_topology}. On the right hand side of the modified 9-term identity the first four terms vanish, because $\Delta M = 0$ and $E\Delta = 0$.

The last four terms give rise to the right hand side of \eqref{eq:drinfeld_in_string_topology} as follows. 
There is no sign in the first term because we need to switch the order of the degree 1 map $(M \otimes 1)$ with both $(1 \otimes \lambda)$ and $(1 \otimes \Delta \otimes 1)$. Compared to the first term, the second term gets an additional sign from moving $M \otimes M$ past $\tau$.

In the third term, $(1 \otimes M)$ is moved past $(M\otimes 1)$ and $(\lambda \otimes 1)$, therefore no sign appears. 
Compared to that, the final term again has an additional sign from moving $M \otimes M$ past $\tau$.
\end{proof}

\begin{remark} Using tripods with square marks at the center to denote $\beta^{S^1}$ and $\gamma_\cD^{S^1}$, equation \eqref{eq:drinfeld_in_string_topology} is graphically represented by
\begin{center}
\begin{tikzpicture}[scale=.25]

\draw (0,0) -- (1,-1) -- (2,0);
\draw (1,-1) -- (1,-3);
\draw (0,-4) -- (1,-3) -- (2,-4);
\node[rectangle,draw,fill,inner sep=1.5pt] at (1,-1) {};
\node[rectangle,draw,fill,inner sep=1.5pt] at (1,-3) {};

\node at (4,-2) {$=$};

\pgftransformxshift{6cm}

\draw (0,0) -- (0,-2) -- (1,-3) -- (1,-4);
\draw (1,-3) -- (3,-1);
\draw (3,0) -- (3,-1) -- (4,-2) -- (4,-4);
\node[rectangle,draw,fill,inner sep=1.5pt] at (1,-3) {};
\node[rectangle,draw,fill,inner sep=1.5pt] at (3,-1) {};
\node at (7,-2.5) {$-$};

\pgftransformxshift{9cm}
\node at (0,.8) {\tiny$2$};
\node at (3,.8) {\tiny$1$};
\draw (0,0) -- (0,-2) -- (1,-3) -- (1,-4);
\draw (1,-3) -- (3,-1);
\draw (3,0) -- (3,-1) -- (4,-2) -- (4,-4);
\node[rectangle,draw,fill,inner sep=1.5pt] at (1,-3) {};
\node[rectangle,draw,fill,inner sep=1.5pt] at (3,-1) {};
\node at (7,-2.5) {$+$};

\pgftransformxshift{9cm}

\draw (1,0)  -- (1,-1) -- (0,-2) -- (0,-4);
\draw (1,-1) -- (3,-3);
\draw (4,0) -- (4,-2) -- (3,-3) -- (3,-4);
\node[rectangle,draw,fill,inner sep=1.5pt] at (1,-1) {};
\node[rectangle,draw,fill,inner sep=1.5pt] at (3,-3) {};
\node at (7,-2.5) {$-$};

\pgftransformxshift{9cm}
\node at (1,.8) {\tiny$2$};
\node at (4,.8) {\tiny$1$};
\draw (1,0)  -- (1,-1) -- (0,-2) -- (0,-4);
\draw (1,-1) -- (3,-3);
\draw (4,0) -- (4,-2) -- (3,-3) -- (3,-4);
\node[rectangle,draw,fill,inner sep=1.5pt] at (1,-1) {};
\node[rectangle,draw,fill,inner sep=1.5pt] at (3,-3) {};
\pgftransformxshift{9cm}
\end{tikzpicture}
\end{center}
\end{remark}

\begin{proof}[Proof of Theorem~\ref{thm:SHS1-Lie-bialgebra}] The proof is verbatim the same as the previous one of Theorem~\ref{thm:HS1-Lie-bialgebra}. The role of $\ol\H_*\Lambda$ is taken by $\ol{S\H}_*(W)$, which carries an odd BV unital infinitesimal bialgebra structure by Theorem~\ref{thm:BV_unital_infinitesimal_reduced}. 
\end{proof}

\section{Example: odd dimensional spheres} \label{sec:spheres}

For spheres $S^n$ of odd dimension $n\ge 3$, we have
$$
\H_*(\Lambda S^n) \cong \Lambda[A,U], \quad |A|=-n\,,\quad |U|=n-1,
$$
and the coproducts are given on additive generators by~\cite{CHO-MorseFloerGH}
\begin{align*}
\lambda(AU^k) &=\sum_{i+j=k-1} AU^i \otimes AU^j\\
\lambda(U^k) &=\sum_{i+j=k-1} AU^i \otimes U^j-U^i \otimes AU^j.
\end{align*}
In particular, $\lambda(\eta)=0$, so the 11-term relation reduces to 9 terms. Moreover, from \cite{Menichi} we know that
the BV operator $\Delta$ satisfies
$$
\forall\, k\ge 0:\,\Delta(U^k)=0 \quad\text{and}\quad \Delta (AU^k)= kU^{k-1}.
$$

These formulas allow explicit computations of the various terms appearing in the 9-term relation in this specific example. The cancellations are quite intricate, to the point that the relative signs in the 9-term relation are actually uniquely determined. 

To illustrate this point, we perform the computations for the input $AU^r \otimes U^s$. The interested reader may want to compute other cases.

For the left hand side of the equation, we have

\begin{minipage}{.1\textwidth}
\begin{tikzpicture}[scale=.25]
\draw (0,0) -- (1,-1) -- (2,0);
\draw (1,-1) -- (1,-3);
\node at (0,-2) {\tiny $\Delta$};
\draw (0,-4) -- (1,-3) -- (2,-4);
\end{tikzpicture}
\end{minipage}
\begin{minipage}{.89\textwidth}
\begin{align*}
\lambda \, \Delta \, \mu (AU^r \otimes U^s)
&= (r+s) \lambda (U^{r+s-1})\\
&= (r+s) \displaystyle\sum_{i+j=r+s-2} \!\!\! \left(AU^i \otimes U^j - U^i \otimes AU^j\right).
\end{align*}
\end{minipage}

The terms of the right hand side are computed as follows:

\begin{minipage}{.1\textwidth}
\begin{tikzpicture}[scale=.25]
\draw (0,0) -- (1,-1) -- (2,0);
\draw (1,-1) -- (1,-3);
\draw (0,-4) -- (1,-3) -- (2,-4);
\node at (0,0.8) {\tiny$\Delta$};
\end{tikzpicture}
\end{minipage}
\begin{minipage}{.85\textwidth}
\begin{align*}
\lambda  \, \mu & \, (\Delta \otimes 1) (AU^r \otimes U^s)\\
&= r \lambda (U^{r+s-1})
= r \sum_{i+j=r+s-2} \left(AU^i \otimes U^j - U^i \otimes AU^j\right).
\end{align*}
\end{minipage}

\begin{minipage}{.1\textwidth}
\begin{tikzpicture}[scale=.25]
\draw (0,0) -- (1,-1) -- (2,0);
\draw (1,-1) -- (1,-3);
\draw (0,-4) -- (1,-3) -- (2,-4);
\node at (2,0.8) {\tiny$\Delta$};
\end{tikzpicture}
\end{minipage}
\begin{minipage}{.85\textwidth}
$\lambda \, \mu \, (1 \otimes \Delta) (AU^r \otimes U^s) = 0.$
\end{minipage}

\begin{minipage}{.1\textwidth}
\begin{tikzpicture}[scale=.25]
\draw (0,0) -- (1,-1) -- (2,0);
\draw (1,-1) -- (1,-3);
\draw (0,-4) -- (1,-3) -- (2,-4);
\node at (0,-4.8) {\tiny$\Delta$};
\end{tikzpicture}

\phantom{blah}

\phantom{blah}
\end{minipage}
\begin{minipage}{.85\textwidth}
\begin{align*}
(\Delta \otimes 1) \,\lambda \, \mu (AU^r \otimes U^s)
&= (\Delta \otimes 1) \lambda (AU^{r+s})\\
&= \sum_{i'+j=r+s-1} (\Delta \otimes 1)(AU^{i'} \otimes AU^j)\\
&= \sum_{i+j=r+s-2} (i+1) U^i \otimes AU^j,
\end{align*}
\end{minipage}

Here we have used $\Delta(A)=0$ to reindex $i=i'-1$. We will make similar replacements below without explicit mention.

\begin{minipage}{.1\textwidth}
\begin{tikzpicture}[scale=.25]
\draw (0,0) -- (1,-1) -- (2,0);
\draw (1,-1) -- (1,-3);
\draw (0,-4) -- (1,-3) -- (2,-4);
\node at (2,-4.8) {\tiny $\Delta$};
\end{tikzpicture}

\phantom{blah}

\phantom{blah}
\end{minipage}
\begin{minipage}{.85\textwidth}
\begin{align*}
(1 \otimes \Delta) \,\lambda \, \mu (AU^r \otimes U^s)
&= (1 \otimes \Delta) \lambda (AU^{r+s})\\
&= \sum_{i+j=r+s-1} (1 \otimes \Delta)(AU^i \otimes AU^j)\\
&= - \sum_{i+j=r+s-2} (j+1) AU^i \otimes U^j.
\end{align*}
\end{minipage}

\begin{minipage}{.1\textwidth}
\begin{tikzpicture}[scale=.25]
\draw (0,0) -- (0,-2) -- (1,-3) -- (1,-4);
\draw (1,-3) -- (3,-1);
\node at (1.5,-1.5) {\tiny$\Delta$};
\draw (3,0) -- (3,-1) -- (4,-2) -- (4,-4);
\end{tikzpicture}

\phantom{blah}

\phantom{blah}

\phantom{blah}

\phantom{blah}

\phantom{blah}

\phantom{blah}

\phantom{blah}

\phantom{blah}
\end{minipage}
\begin{minipage}{.89\textwidth}
\begin{align*}
(\mu &\otimes 1)(1 \otimes \Delta \otimes 1)(1 \otimes \lambda)(AU^r \otimes U^s)\\
&=-(\mu \otimes 1)(1 \otimes \Delta \otimes 1) \\
& \qquad \qquad \sum_{i+j=s-1} \left(AU^r \otimes AU^i \otimes U^j - AU^r \otimes U^i \otimes AU^j \right) \\
&=(\mu \otimes 1) \sum_{i+j=s-2} (i+1)AU^r\otimes U^i \otimes U^j\\
&=\!\!\!\!\sum_{i+j=r+s-2 \atop i \ge r} \!\!\!\! (i+1-r) AU^i \otimes U^j
=\!\!\!\!\sum_{i+j=r+s-2 \atop i \ge r} \!\!\!\! (s-(j+1)) AU^i \otimes U^j.
\end{align*}
\end{minipage}

\begin{minipage}{.1\textwidth}
\begin{tikzpicture}[scale=.25]
\node at (0,.8) {\tiny$2$};
\node at (3,.8) {\tiny$1$};
\draw (0,0) -- (0,-2) -- (1,-3) -- (1,-4);
\draw (1,-3) -- (3,-1);
\node at (1.5,-1.5) {\tiny$\Delta$};
\draw (3,0) -- (3,-1) -- (4,-2) -- (4,-4);
\end{tikzpicture}

\phantom{blah}

\phantom{blah}

\phantom{blah}

\phantom{blah}

\phantom{blah}
\end{minipage}
\begin{minipage}{.85\textwidth}
\begin{align*}
\lefteqn{(\mu \otimes 1)(1 \otimes \Delta \otimes 1)(1 \otimes \lambda)(U^s \otimes AU^r)}\\
&=(\mu \otimes 1)(1 \otimes \Delta \otimes 1) \sum_{i+j=r-1} U^s \otimes AU^i \otimes AU^j  \\
&=(\mu \otimes 1) \sum_{i+j=r-2} (i+1)U^s\otimes U^i \otimes AU^j\\
&=\sum_{i+j=r+s-2 \atop i \ge s} (i+1-s) U^i \otimes AU^j.
\end{align*}
\end{minipage}

\begin{minipage}{.1\textwidth}
\begin{tikzpicture}[scale=.25]
\draw (1,0)  -- (1,-1) -- (0,-2) -- (0,-4);
\draw (1,-1) -- (3,-3);
\node at (2.5,-1.5) {\tiny$\Delta$};
\draw (4,0) -- (4,-2) -- (3,-3) -- (3,-4);
\end{tikzpicture}

\phantom{blah}

\phantom{blah}

\phantom{blah}

\phantom{blah}

\phantom{blah}

\phantom{blah}

\phantom{blah}

\phantom{blah}
\end{minipage}
\begin{minipage}{.85\textwidth}
\begin{align*}
\lefteqn{(1 \otimes \mu)(1 \otimes \Delta \otimes 1)(\lambda \otimes 1)(AU^r \otimes U^s)}\\
&=(1 \otimes \mu)(1 \otimes \Delta \otimes 1) \sum_{i+j=r-1} AU^i \otimes AU^j \otimes U^s  \\
&=-(1 \otimes \mu) \sum_{i+j=r-2} (j+1) AU^i\otimes U^j \otimes U^s\\
&=-\sum_{i+j=r+s-2 \atop j \ge s-1} (j+1-s) AU^i \otimes U^j\\
&=\sum_{i+j=r+s-2 \atop i \le r-1} (s-(j+1)) AU^i \otimes U^j.
\end{align*}
\end{minipage}

Including the index $j=s-1$ in the next to last line does not cause problems here, because $s-(j+1)=0$ in this case.

\begin{minipage}{.1\textwidth}
\begin{tikzpicture}[scale=.25]
\node at (1,.8) {\tiny$2$};
\node at (4,.8) {\tiny$1$};
\draw (1,0)  -- (1,-1) -- (0,-2) -- (0,-4);
\draw (1,-1) -- (3,-3);
\node at (2.5,-1.5) {\tiny$\Delta$};
\draw (4,0) -- (4,-2) -- (3,-3) -- (3,-4);
\end{tikzpicture}

\phantom{blah}

\phantom{blah}

\phantom{blah}

\phantom{blah}

\phantom{blah}

\phantom{blah}

\phantom{blah}

\phantom{blah}
\end{minipage}
\begin{minipage}{.85\textwidth}
\begin{align*}
(1 &\otimes \mu)(1 \otimes \Delta \otimes 1)(\lambda \otimes 1)(U^s \otimes AU^r)\\
&=(1 \otimes \mu)(1 \otimes \Delta \otimes 1) \sum_{i+j=s-1} \left(\dots - U^i \otimes AU^j \otimes AU^r \right) \\
&=-(1 \otimes \mu) \sum_{i+j=s-2} (j+1) U^i\otimes U^j \otimes AU^r\\
&=-\sum_{i+j=r+s-2 \atop j \ge r-1} \underbrace{(j+1-r)}_{=s-(i+1)} U^i \otimes AU^j\\
&=\sum_{i+j=r+s-2 \atop i \le s-1} (i+1-s) U^i \otimes AU^j.
\end{align*}
\end{minipage}

In total, this is consistent with the equality
\begin{center}
\begin{tikzpicture}[scale=.25]

\draw (0,0) -- (1,-1) -- (2,0);
\draw (1,-1) -- (1,-3);
\node at (0,-2) {\tiny $\Delta$};
\draw (0,-4) -- (1,-3) -- (2,-4);
\node at (4,-2) {$=$};

\pgftransformxshift{6cm}
\draw (0,0) -- (1,-1) -- (2,0);
\draw (1,-1) -- (1,-3);
\draw (0,-4) -- (1,-3) -- (2,-4);
\node at (0,0.8) {\tiny$\Delta$};
\node at (4,-2) {$\pm$};

\pgftransformxshift{6cm}
\draw (0,0) -- (1,-1) -- (2,0);
\draw (1,-1) -- (1,-3);
\draw (0,-4) -- (1,-3) -- (2,-4);
\node at (2,0.8) {\tiny$\Delta$};
\node at (4,-2) {$-$};

\pgftransformxshift{6cm}
\draw (0,0) -- (1,-1) -- (2,0);
\draw (1,-1) -- (1,-3);
\draw (0,-4) -- (1,-3) -- (2,-4);
\node at (0,-4.8) {\tiny$\Delta$};
\node at (4,-2) {$-$};

\pgftransformxshift{6cm}
\draw (0,0) -- (1,-1) -- (2,0);
\draw (1,-1) -- (1,-3);
\draw (0,-4) -- (1,-3) -- (2,-4);
\node at (2,-4.8) {\tiny$\Delta$};

\pgftransformxshift{-18cm}
\pgftransformyshift{-8cm}

\node at (0,-2) {$+$};
\pgftransformxshift{2cm}

\draw (0,0) -- (0,-2) -- (1,-3) -- (1,-4);
\draw (1,-3) -- (3,-1);
\node at (1.5,-1.5) {\tiny$\Delta$};
\draw (3,0) -- (3,-1) -- (4,-2) -- (4,-4);
\node at (7,-2.5) {$+$};

\pgftransformxshift{9cm}
\node at (0,.8) {\tiny$2$};
\node at (3,.8) {\tiny$1$};
\draw (0,0) -- (0,-2) -- (1,-3) -- (1,-4);
\draw (1,-3) -- (3,-1);
\node at (1.5,-1.5) {\tiny$\Delta$};
\draw (3,0) -- (3,-1) -- (4,-2) -- (4,-4);
\node at (7,-2.5) {$+$};

\pgftransformxshift{9cm}

\draw (1,0)  -- (1,-1) -- (0,-2) -- (0,-4);
\draw (1,-1) -- (3,-3);
\node at (2.5,-1.5) {\tiny$\Delta$};
\draw (4,0) -- (4,-2) -- (3,-3) -- (3,-4);
\node at (7,-2.5) {$+$};

\pgftransformxshift{9cm}
\node at (1,.8) {\tiny$2$};
\node at (4,.8) {\tiny$1$};
\draw (1,0)  -- (1,-1) -- (0,-2) -- (0,-4);
\draw (1,-1) -- (3,-3);
\node at (2.5,-1.5) {\tiny$\Delta$};
\draw (4,0) -- (4,-2) -- (3,-3) -- (3,-4);
\pgftransformxshift{9cm}
\end{tikzpicture}
\end{center}
The undetermined sign could be obtained from doing the calculation for $U^r \otimes AU^s$ instead. This shows that, already in this example, there is no freedom in the signs for the first 9 terms in the general 11-term relation in~\S\ref{ssec:definitions}. 

\vfill \pagebreak

\appendix 



\section{Tate vector spaces} \label{sec:Tate}
We give in this appendix a brief account of \emph{Tate vector spaces}, summarizing results from~\cite{CO-Tate,CO-algebra}. Tate vector spaces were first defined by Lefschetz~\cite[II]{Lefschetz-book} under the name  {\em locally linearly compact vector spaces} as linear analogues of locally compact abelian groups. 

\subsection{Tate vector spaces} 
Let $\bk$ be a field equipped with the discrete topology. 
A {\em linearly topologized vector space} is a $\bk$-vector space $V$
with a Hausdorff topology that is translation invariant and has a
basis of open neighbourhoods of $0$ consisting of linear subspaces. 
Morphisms between linearly topologized vector spaces are by definition 
continuous linear maps.
The kernel of a morphism $f:V\to W$ is $\ker f=f^{-1}(0)$ (this is a closed subspace), 
and the cokernel is $\coker f=W/\ol{f(V)}$ (the quotient by a closed subspace is Hausdorff).
The {\em completion} of a linearly topologized vector space $V$ is
$$
   \wh V := \lim_{U\subset V \mbox{\tiny  open linear}} V/U.
$$
This inverse limit is topologized as a subset of the product $\prod_U A/U$, and each $A/U$ is discrete with respect to the quotient topology.
The space $V$ is called {\em complete} if the canonical map $V \to \wh V$ is an isomorphism. A linearly topologized vector space $V$ is called \emph{discrete} if $\{0\}$ is an open set, \emph{linearly compact} if it is complete and $\dim(V/U)<\infty$ for each open linear subspace $U\subset V$, and \emph{Tate}, or \emph{locally linearly compact}, if it is complete and admits an open linearly compact subspace.

\begin{example}
The vector space $\bk[t^{-1},t]]$ of Laurent power series, with a basis of neighborhoods of $0$ given by $t^n\bk[[t]]$, $n\in\Z$, is Tate. The open subspace $\bk[[t]]$ is linearly compact. 
\end{example}

The topological dual $V^\vee=\Hom(V,\bk)$ is topologized by the compact-open topology, i.e., by defining a neighbourhood basis of the origin to be given by the linear subspaces
$K^\perp=\{\alpha\in V^\vee\mid \alpha|_K=0\}$, with $K\subset V$ an arbitrary linearly compact subspace.

The Lefschetz-Tate duality theorem from~\cite[(II.28.2-29.1)]{Lefschetz-book} states that, if $V$ is discrete, then $V^\vee$ is linearly compact, if $V$ is linearly compact, then $V^\vee$ is discrete, and if $V$ is Tate, then so is $V^\vee$ and the canonical map $V\to
  V^{\vee\vee}$ is a topological isomorphism.

\begin{example}
The space 
$\bk[t^{-1}]\simeq \bigoplus_\N \bk$ is discrete, the space $\bk[[t]]\simeq \prod_\N\bk$ is linearly compact, and they are dual of each other~\cite
{CO-Tate}. 
\end{example}  

There are two kinds of topological tensor products that are relevant in this context (Beilinson~\cite[\S1.1]{Beilinson}). The first, denoted $\hatotimes^*$, is naturally the source of the product maps in our bialgebra structures, and the second, denoted $\hatotimes^!$, is naturally the target of the coproduct maps. Both are defined as completions of the algebraic tensor product, with respect to two different topologies. For Tate vector spaces $V,W$, these topologies on $V\otimes W$ can be described as follows: 
\begin{itemize}
\item The \emph{$*$ topology} is the finest linear topology such that the canonical bilinear map $V\times W\to V\otimes W$ is continuous.
\item The \emph{$!$ topology} is the coarsest topology such that the canonical evaluation map 
$
V\otimes W\to B(V^\vee\times W^\vee,\bk)
$  
is continuous. Here $B(V^\vee\times W^\vee,\bk)$ is the space of continuous bilinear forms on $V^\vee\times W^\vee$ with the  compact-open topology. 
\end{itemize}
The $!$ topology is coarser than the $*$ topology, so there is a canonical morphism induced by the identity 
$V\hatotimes^* W\to V\hatotimes^! W$.  Also, the identity induces a canonical morphism $U\hatotimes^*(V\hatotimes^!W) \to (U\hatotimes^*V)\hatotimes^!W$.

For Tate spaces $V,W$ there are canonical isomorphisms $(V\hatotimes^* W)^\vee\simeq V^\vee\hatotimes^! W^\vee$ and $(V\hatotimes^! W)^\vee\simeq V^\vee\hatotimes^* W^\vee$, but note that none of the factors need be Tate, see Esposito and Penkov~\cite{Esposito-Penkov22,Esposito-Penkov23} as well as~\cite[\S5]{CO-Tate}.

\subsection{Graded Tate vector spaces} 
The previous discussion can be enhanced to a graded setting, and this is useful in the context of Floer homology. We follow~\cite[\S2]{CO-algebra}. 

A \emph{($\Z$-)graded Tate vector space} is a direct sum $V=\bigoplus_{i\in\Z}V_i$ with Tate summands. Morphism spaces in the category of graded Tate vector spaces are by definition  
\begin{align*}
  \Hom(V,W) = \bigoplus_{k\in\Z}\Hom_k(V,W),\quad
  \Hom_k(V,W) = \prod_i \Hom(V_i,W_{i+k}),
\end{align*}
and each $\Hom(V_i,W_{i+k})$ is topologized by the compact-open topology.
An element in $\Hom_k(V,W)$ is called a \emph{(homogeneous) morphism of degree $k$}, and we denote $|\phi|=k$ its degree. 

The \emph{dual} of a graded linearly topologized vector space is defined to be  
$$
  V^\vee = \bigoplus_{i\in\Z}V^\vee_i,\qquad V^\vee_i=\Hom(V_{-i},\bk).
$$
The $*$ and $!$ tensor products of graded Tate vector spaces are the graded linearly topologized vector spaces 
\begin{gather*}
V\otimes^* W = \bigoplus_{k\in\Z}(V\otimes^* W)_k,\qquad 
  (V\otimes^* W)_k = \bigoplus_{i+j=k}V_i\hatotimes^* W_j, \cr
V\otimes^! W = \bigoplus_{k\in\Z}(V\otimes^! W)_k,\qquad 
  (V\otimes^! W)_k = \prod_{i+j=k}V_i\hatotimes^! W_j.
\end{gather*}

Notice that these tensor products are again completed, with the completion happening separately in each constituent piece. To unburden the notation, we leave off the hats for the tensor products in the graded setting. As in the non-graded case, we have canonical isomorphisms
$(V\otimes^* W)^\vee\simeq V^\vee\otimes^! W^\vee$ and $(V\otimes^! W)^\vee\simeq V^\vee\otimes^* W^\vee$, and also a canonical morphism induced by the identity $U\otimes^*(V\otimes^!W) \to (U\otimes^*V)\otimes^!W$.

The main theorem in~\cite{CHO-PD} states that, given a Liouville domain $V$, Rabinowitz Floer homology $SH_*(\p V)$ and cohomology $SH^*(\p V)$ are
graded Tate vector spaces topologically dual to each other. 
Poincar\'e duality defines an isomorphism of graded Tate vector spaces.

\subsection{Frobenius algebras}

We follow~\cite{CO-algebra}. Let $V=\bigoplus_{i\in\Z}V_i$ be a graded Tate vector space. The \emph{twist} $\tau:V\otimes V\to V\otimes V$ defined by $\tau(a\otimes b)=(-1)^{|a||b|}b\otimes a$ induces maps $V\otimes^* V\to V\otimes^* V$ and $V\otimes^! V\to V\otimes^! V$, still denoted $\tau$.  

A morphism $\boldmu:V\otimes^* V\to V$ is called a \emph{product}.  
Commutativity, associativity, and the notion of a unit are defined in the usual way.

A morphism $\boldlambda:V\to V\otimes^! V$ is called a \emph{coproduct}. 
Cocommutativity, coassociativity, and the notion of a counit are defined in the usual way. 
A \emph{Tate Frobenius algebra}~\cite[\S5]{CO-algebra}, also called \emph{biunital coFrobenius bialgebra}, is a graded Tate vector space $V$ endowed with an associative degree zero product $\boldmu:V\otimes^* V\to V$ with unit $\boldeta\in V$, and a coassociative coproduct $\boldlambda:V\to V\otimes^! V$ with counit $\boldeps:V\to \bk$, which satisfy in addition the following relations:
\begin{itemize}
\item {\sc (Frobenius)} Define the \emph{copairing} by 
$$
\boldc = \boldlambda\boldeta,
$$
and the \emph{pairing} by 
$$
\boldp = (-1)^{|\boldlambda|}\boldeps\boldmu.
$$
Then the following Frobenius relations hold: 
$$
\left.\begin{array}{rrr}
\boldlambda = & (1\otimes \boldmu)(\boldc\otimes 1)   = & (\boldmu\otimes 1)(1\otimes \boldc),\\
\boldmu = & (-1)^{|\boldlambda|}(\boldp\otimes 1)(1\otimes \boldlambda)  = & (1\otimes \boldp)(\boldlambda\otimes 1).
\end{array}\right.
$$
\item {\sc (symmetry)}
\begin{align*}
\tau \boldc & = (-1)^{|\boldlambda|} \boldc, \\
\boldp \tau  & = \boldp. 
\end{align*}
\end{itemize}

The {\sc (Frobenius)} relations have to be interpreted in terms of the canonical maps $V\otimes^*(V\otimes^!V) \to (V\otimes^*V)\otimes^!V$ and $(V\otimes^! V)\otimes^* V\to V\otimes^!(V\otimes^* V)$, denoted $can$. E.g., the relation $\boldlambda=(\boldmu\otimes 1)(1\otimes \boldc)$ has to be interpreted as $\boldlambda$ being equal to the composition 
$$
V\stackrel{1\otimes\boldc}{\longrightarrow} V\otimes^* (V\otimes^! V) \stackrel{can}{\longrightarrow} (V\otimes^*V)\otimes^!V \stackrel{\boldmu\otimes 1}{\longrightarrow} V\otimes^! V.
$$

{\bf Terminology.} For readability, we refer in the paper to these algebras simply as \emph{Frobenius algebras}. In the finite dimensional case they coincide with the classical notion of Frobenius algebra. In the infinite dimensional case, they need to be understood in a Tate setting.

A Frobenius algebra is called {\em commutative and cocommutative} if $\boldmu$ is commutative, i.e., $\boldmu\tau=\boldmu$, and $\boldlambda$ is cocommutative, i.e., $\tau\boldlambda=(-1)^{|\boldlambda|}\boldlambda$.  
  
\begin{remark}\label{rmk:TateBVui} The notion of a \emph{BV unital infinitesimal bialgebra} from Definition~\ref{defi:BVui} also makes sense in the context of Tate vector spaces. E.g., the term $(\mu\otimes\mu)(1\otimes\lambda\eta\otimes1)$ from the unital infinitesimal relation has to be understood as acting from $A\otimes^*A$ to $A\otimes^!A$, being equal to the composition 
$A\otimes^*A\stackrel{1\otimes\lambda\eta\otimes1}\longrightarrow A\otimes^*(A\otimes^!A)\otimes^*A\stackrel{can\otimes1}{\longrightarrow} ((A\otimes^*A)\otimes^!A)\otimes^*A\stackrel{can}{\longrightarrow} (A\otimes^*A)\otimes^!(A\otimes^*A)\stackrel{\mu\otimes\mu}{\longrightarrow}A\otimes^!A$. We will not spell out further details here because, in our context, the structure arises on reduced symplectic homology and the underlying linearly topologized vector space is discrete. In this case, the $*$ tensor product and the $!$ tensor product are both discrete, and they both coincide with the algebraic tensor product~\cite[Lemma~5.14]{CO-Tate}. 
\end{remark}


\section{From BV unital infinitesimal to BV Frobenius}\label{app:BVui-BVTateFrob}

In this appendix we prove Proposition~\ref{prop:BVui-BVFrob}. Before starting the proof we discuss notation, and also some facts regarding the behavior of operations under dualization.

We work in the category of graded Tate vector spaces over a discrete field $\bk$, with the two tensor products $\otimes^*$ and $\otimes^!$.
Graded morphisms $f:A\to A'$ and $g:B\to B'$ induce morphisms $f\otimes g:A\otimes^* B\to A'\otimes^* B'$ and $f\otimes g:A\otimes^! B\to A'\otimes^! B'$ (we use the same notation $f\otimes g$ for both kinds of tensor product). We have $(g\otimes f)\tau = (-1)^{|f||g|}\tau(f\otimes g)$. 
In calculations with maps on tensor products with multiple factors, we sometimes denote the transposition $(ij)$ on the terms $i$ and $j$ in a tensor product by $\tau_{ij}$. For example $1\otimes\tau \otimes 1 = \tau_{23}:A^{\otimes^* 4} \to A^{\otimes^* 4}$.

In the explicit description of the various operations below we use the canonical evaluation map $\ev:A^\vee\otimes^* A\to \bk$~\cite[Proposition~4.12]{CO-Tate}. Identifying $\bk \cong \bk^\vee$, we denote its dual map $\coev=\ev^\vee: \bk \to (A^\vee \otimes^* A)^\vee \cong A  \otimes^! A^\vee$. 

It follows from the definitions that the maps $\ev$ and $\coev$ satisfy the identities
$$
(\ev \otimes 1_{A^\vee})(1_{A^\vee} \otimes \coev) = 1_{A^\vee}
\quad \text{and} \quad
(1_A \otimes \ev)(\coev \otimes 1_A) = 1_A.
$$
We also have
$$
(\ev\tau\otimes 1)(1\otimes \tau\coev)=1_A
\quad \text{and} \quad
(1\otimes \ev\tau)(\tau\coev\otimes 1)=1_{A^\vee}.
$$ 

Given a BV unital infinitesimal bialgebra $A$, we will define on $D(A)=A\oplus A^\vee[-|\lambda|]$ a product $\boldmu$ with unit $\boldeta$, a coproduct $\boldlambda$ with counit $\boldeps$, and a BV operator $\boldDelta$. The operations $\boldmu$, $\boldeta$, $\boldlambda$, $\boldeps$ will be depicted graphically using double lines
\begin{center}
\begin{tikzpicture}[scale=.25]

\draw[double, double distance=1.5pt] (0,0) -- (1,-1) -- (2,0);
\draw[double, double distance=1.5pt] (1,-1) -- (1,-2.5);
\node at (1,-5.5) {$\boldmu$};

\pgftransformxshift{6cm}
\draw[double, double distance=1.5pt] (1,-1) -- (1,-2);
\node[circle,draw,inner sep=1.4pt] at (1,-0.85) {};
\node at (1,-5.5) {$\boldeta$};

\pgftransformxshift{6cm}
\draw[double, double distance=1.5pt] (0,-2.5) -- (1,-1.5) -- (2,-2.5);
\draw[double, double distance=1.5pt] (1,0) -- (1,-1.5);
\node at (1,-5.3) {$\boldlambda$};

\pgftransformxshift{6cm}
\draw[double, double distance=1.5pt] (1,-1) -- (1,-2);
\node[circle,draw,inner sep=1.4pt] at (1,-2.15) {};
\node at (1,-5.5) {$\boldeps$};

\end{tikzpicture}
\end{center}

As before, these are read from top to bottom, i.e., inputs in $D(A)$ are at the top leaves and outputs in $D(A)$ are at the bottom leaves. In the definitions of these operations, we will use the original operations in $A$, for which we continue to use the previous graphical notation

\begin{center}
\begin{tikzpicture}[scale=.25]

\draw (0,0) -- (1,-1) -- (2,0);
\draw (1,-1) -- (1,-2.5);
\node at (1,-5.5) {$\mu$};

\pgftransformxshift{6cm}

\draw (1,-1) -- (1,-2);
\node[circle,draw,inner sep=1pt] at (1,-1) {};
\node at (1,-5.5) {$\eta$};

\pgftransformxshift{6cm}
\draw (1,0) -- (1,-1.5);
\draw (0,-2.5) -- (1,-1.5) -- (2,-2.5);
\node at (1,-5.3) {$\lambda$};

\end{tikzpicture}
\end{center}

Here inputs/outputs in $A$ can be interpreted as outputs/inputs in $A^\vee$ and so these operations in $A$ give rise to mixed operations involving inputs and outputs in $A$ and $A^\vee$. To clarify inputs and outputs of these mixed operations, we use arrows pointing to the tree (for inputs) and away from the tree (for outputs). Inputs of a 2-to-1 operation (product) are always read \emph{clockwise}, and outputs of a 1-to-2 operation (coproduct) are always read \emph{counterclockwise}. Inputs always live in $\otimes^*$ tensor products, and outputs always live in $\otimes^!$ tensor products. To illustrate these conventions, here are the graphical representations of $\mu^\vee:A^\vee \to A^\vee \otimes^! A^\vee$ and $\lambda^\vee:A^\vee \otimes^* A^\vee \to A^\vee$:
\begin{center}
\begin{tikzpicture}[scale=.25]

\draw (0,0) -- (1,-1) -- (2,0);
\draw (1,-1) -- (1,-2.5);
\node at (1,-6.5) {$\mu^\vee$};
\draw[->]  (0,.2)   -- (0,1);
\draw[->]  (2,.2)   -- (2,1);
\draw[->]  (1,-3.5)   -- (1,-2.7);
\node at (-.8,0) {\tiny $2$};
\node at (2.8,0) {\tiny $1$};

\pgftransformxshift{8cm}

\draw (1,0) -- (1,-1.5);
\draw (0,-2.5) -- (1,-1.5) -- (2,-2.5);
\node at (1,-6.35) {$\lambda^\vee$};
\draw[->]  (1,.2)   -- (1,1);
\draw[->]  (0,-3.5)   -- (0,-2.7);
\draw[->]  (2,-3.5)   -- (2,-2.7);
\node at (-.8,-2.5) {\tiny $2$};
\node at (2.8,-2.5) {\tiny $1$};

\end{tikzpicture}
\end{center}

We denote 
$$
\xymatrix{
A^\vee \ar@<.5ex>[r]^-s & A^\vee[-|\lambda|] \ar@<.5ex>[l]^-\omega
}
$$
the mutually inverse shifts, of degrees $|s|=|\lambda|$ and $|\omega|=-|\lambda|$. Operations involving inputs (resp. outputs) in $A^\vee$ can be converted to operations involving inputs (resp. outputs) in $A^\vee[-|\lambda|]$ by precomposing with $\omega$ (resp. postcomposing with $s$). For example, the dual map $\lambda^\vee:A^\vee\otimes^! A^\vee \simeq (A \otimes^* A)^\vee\to A^\vee$ has degree $|\lambda|$ and so $s\lambda^\vee (\omega\otimes\omega):A^\vee[-|\lambda|] \otimes^! A^\vee[-|\lambda|]\to A^\vee[-|\lambda|]$ has degree $0$.

\begin{proof}[Proof of Proposition~\ref{prop:BVui-BVFrob}] In the sign computations involved in the proof it is only the parity of $|\lambda|$ that matters. We therefore assume w.l.o.g. that $|\lambda|=-1$ and work on $D(A)=A\oplus A^\vee[1]$. By assumption we have $\lambda\eta=0$.

\smallskip
{\it Step~1. Definition of the product $\boldmu$ and proof of commutativity and associativity.}

In principle, the product $\boldmu$ can be written as a sum of eight components, according to whether the inputs and the output lie in $A$ or in $A^\vee[1]$.
However, the components $A\otimes^* A\to A^\vee[1]$ and $A^\vee[1]\otimes^* A^\vee[1]\to A$ vanish, and so we have the following preliminary graphical description:

\begin{center}

\begin{tikzpicture}[scale=.25]

\draw[double, double distance=1.5pt] (0,0) -- (1,-1) -- (2,0);
\draw[double, double distance=1.5pt] (1,-1) -- (1,-2.5);
\node at (6,-1) {=};

\pgftransformxshift{14cm}
\draw (0,0) -- (1,-1) -- (2,0);
\draw (1,-1) -- (1,-2.5);
\node at (1,-5.5) {\tiny $A\otimes^* A\to A$};
\node at (7,-1) {+};
\draw[->]        (0,1)   -- (0,.2);
\draw[->]        (2,1)   -- (2,.2);
\draw[->]        (1,-2.7)   -- (1,-3.5);
\node at (-.8,0) {\tiny $1$};
\node at (2.8,0) {\tiny $2$};

\pgftransformxshift{14cm}
\draw (0,0) -- (1,-1) -- (2,0);
\draw (1,-1) -- (1,-2.5);
\node at (1,-5.5) {\tiny $A^\vee[1] \otimes^* A\to A^\vee[1]$};
\node at (7,-1) {+};
\draw[->]        (0,1)   -- (0,.2);
\draw[->]        (2,.2) -- (2,1);
\draw[->]        (1,-3.5) -- (1,-2.7);
\node at (-.8,0) {\tiny $2$};
\node at (0.3,-2.5) {\tiny $1$};

\pgftransformxshift{14cm}
\draw (0,0) -- (1,-1) -- (2,0);
\draw (1,-1) -- (1,-2.5);
\node at (1,-5.5) {\tiny $A \otimes^* A^\vee[1]\to A^\vee[1]$};
\draw[->]        (0,.2) -- (0,1);
\draw[->]        (2,1)   -- (2,.2);
\draw[->]        (1,-3.5) -- (1,-2.7);
\node at (2.8,0) {\tiny $1$};
\node at (1.8,-2.5) {\tiny $2$};

\pgftransformxshift{-28cm}
\pgftransformyshift{-12cm}
\draw (1,0) -- (1,-1.5);
\draw (0,-2.5) -- (1,-1.5) -- (2,-2.5);
\node at (1,-5.3) {\tiny $A^\vee[1]\otimes^* A^\vee[1]\to A^\vee[1]$};
\node at (-5,-1) {$+$};
\node at (7,-1) {$-$};
\draw[->]        (1,.2) -- (1,1);
\draw[->]        (0,-3.5)  -- (0,-2.7);
\draw[->]        (2,-3.5)  -- (2,-2.7);
\node at (-.8,-2.5) {\tiny $2$};
\node at (2.8,-2.5) {\tiny $1$};

\pgftransformxshift{14cm}
\draw (1,0) -- (1,-1.5);
\draw (0,-2.5) -- (1,-1.5) -- (2,-2.5);
\node at (1,-5.3) {\tiny $A^\vee[1] \otimes^* A\to A$};
\node at (7,-1) {+};
\draw[->]        (1,1) -- (1,.2);
\draw[->]        (0,-3.5)  -- (0,-2.7);
\draw[->]        (2,-2.7)  -- (2,-3.5);
\node at (-0.8,-2.5) {\tiny $1$};
\node at (0.2,0) {\tiny $2$};

\pgftransformxshift{14cm}
\draw (1,0) -- (1,-1.5);
\draw (0,-2.5) -- (1,-1.5) -- (2,-2.5);
\node at (1,-5.3) {\tiny $A \otimes^* A^\vee[1]\to A$};
\draw[->]        (1,1) -- (1,.2);
\draw[->]        (0,-2.7) -- (0,-3.5) ;
\draw[->]        (2,-3.5)  -- (2,-2.7);
\node at (1.8,0) {\tiny $1$};
\node at (2.8,-2.5) {\tiny $2$};

\end{tikzpicture}
\end{center}

For the more precise formulas below, we note that whenever an operation has an input (resp. an output) in $A^\vee[1]$, one must first precompose with $\omega$ (resp. last postcompose with $s$). Moreover, we use $\ev$ to turn outputs in $A$ into inputs in $A^\vee$, and $\coev$ to turn inputs in $A$ into outputs in $A^\vee$.

The component $A \otimes^* A\to A$ of the product $\boldmu$ is simply given by $\mu$. The explicit formulas for the other components are the following. 

\bigskip
\begin{minipage}{.8\textwidth}
\begin{tikzpicture}[scale=.25]
\draw (0,0) -- (1,-1) -- (2,0);
\draw (1,-1) -- (1,-2.5);
\node at (8,-1) {=};
\draw[->]        (0,1)   -- (0,.2);
\draw[->]        (2,.2) -- (2,1);
\draw[->]        (1,-3.5) -- (1,-2.7);

\pgftransformxshift{12cm}
\draw (0,0) -- (1,-1) -- (2,0);
\draw (1,-1) -- (1,-2.5);
\node at (8,-1) {=};
\node at (2.2,-1) {\tiny $\mu^\vee$};
\draw[->]        (2,.2) -- (2,2); \node at (2,2.5) {\tiny $s$};
\draw (0,2) -- (0,1.5); \draw[>-<]  (0,1.8)   -- (0,.4); \draw (0,.7) -- (0,.2); \node at (-.8,1) {\tiny $\ev$};
\draw[->]        (1,-4.5) -- (1,-2.7); \node at (1,-5) {\tiny $\omega$};

\pgftransformxshift{12cm}
\draw (0,0) -- (1,-1) -- (2,0);
\draw (1,-1) -- (1,-2.5);
\node at (2.2,-1) {\tiny $\mu$};
\draw (2,2) -- (2,1.5); \draw[<->]  (2,1.8)   -- (2,.4); \draw (2,.7) -- (2,.2); \node at (3.8,1) {\tiny $\coev$}; \node at (2,2.5) {\tiny $s$};
\draw[->]        (0,2) -- (0,.2); 
\draw (1,-2.7) -- (1,-3.2); \draw[>-<]  (1,-2.9)   -- (1,-4.3); \draw (1,-4) -- (1,-4.5); \node at (1.8,-3.5) {\tiny $\ev$}; \node at (1,-5) {\tiny $\omega$};

\end{tikzpicture}
\end{minipage}
\begin{align*}
\boldmu |_{A^\vee[1] \otimes^* A\to A^\vee[1]}
&= s(1\otimes\ev)(\mu^\vee\otimes 1)(\omega\otimes 1)\\
&= s(\ev\otimes 1)(1\otimes\mu\otimes 1)(1\otimes 1\otimes\coev)(\omega\otimes 1).\\
\end{align*}

\begin{minipage}{.8\textwidth}
\begin{tikzpicture}[scale=.25]
\draw (0,0) -- (1,-1) -- (2,0);
\draw (1,-1) -- (1,-2.5);
\node at (8,-1) {=};
\draw[->]        (0,.2) -- (0,1);
\draw[->]        (2,1)   -- (2,.2);
\draw[->]        (1,-3.5) -- (1,-2.7);

\pgftransformxshift{12cm}
\draw (0,0) -- (1,-1) -- (2,0);
\draw (1,-1) -- (1,-2.5);
\node at (8,-1) {=};
\node at (2.2,-1) {\tiny $\mu^\vee$};
\draw[->]        (0,.2) -- (0,2); \node at (0,2.5) {\tiny $s$};
\draw (2,2) -- (2,1.5); \draw[>-<]  (2,1.8)   -- (2,.4); \draw (2,.7) -- (2,.2); \node at (2.8,1) {\tiny $\ev$};
\draw[->]        (1,-4.5) -- (1,-2.7); \node at (1,-5) {\tiny $\omega$};

\pgftransformxshift{12cm}
\draw (0,0) -- (1,-1) -- (2,0);
\draw (1,-1) -- (1,-2.5);
\node at (2.2,-1) {\tiny $\mu$};
\draw (0,2) -- (0,1.5); \draw[<->]  (0,1.8)   -- (0,.4); \draw (0,.7) -- (0,.2); \node at (-1.8,1) {\tiny $\coev$}; \node at (0,2.5) {\tiny $s$};
\draw[->]        (2,2) -- (2,.2); 
\draw (1,-2.7) -- (1,-3.2); \draw[>-<]  (1,-2.9)   -- (1,-4.3); \draw (1,-4) -- (1,-4.5); \node at (1.8,-3.5) {\tiny $\ev$}; \node at (1,-5) {\tiny $\omega$};

\end{tikzpicture}
\end{minipage}
\begin{align*}
\boldmu|_{A \otimes^* A^\vee[1]\to A^\vee[1]}
& = s(\ev\tau\otimes 1)(1\otimes\mu^\vee)(1\otimes\omega)\\
& = s(1\otimes\ev\tau)(1\otimes\mu\otimes 1)(\tau\coev\otimes 1\otimes 1)(1\otimes\omega)\\
& = s(\ev \otimes 1) (1 \otimes \mu\tau \otimes 1) (1 \otimes 1 \otimes \coev) \tau (1 \otimes \omega).
\end{align*}


\begin{minipage}{.8\textwidth}
\begin{tikzpicture}[scale=.25]

\node at (-3,1) {\phantom{-}};

\draw (1,0) -- (1,-1.5);
\draw (0,-2.5) -- (1,-1.5) -- (2,-2.5);
\node at (8,-1) {$=$};
\draw[->]        (1,.2) -- (1,1);
\draw[->]        (0,-3.5)  -- (0,-2.7);
\draw[->]        (2,-3.5)  -- (2,-2.7);
\node at (-.8,-2.5) {\tiny $2$};
\node at (2.8,-2.5) {\tiny $1$};

\pgftransformxshift{12cm}
\draw (1,0) -- (1,-1.5);
\draw (0,-2.5) -- (1,-1.5) -- (2,-2.5);
\node at (8,-1) {$=$};
\node at (2.3,-1) {\tiny $\lambda^\vee$};
\draw[->] (1,.2) -- (1,2); \node at (1,2.5) {\tiny $s$};
\draw[->] (0,-4.5) -- (0,-2.7); \node at (0,-5) {\tiny $\omega$};
\draw[->] (2,-4.5) -- (2,-2.7); \node at (2,-5) {\tiny $\omega$};
\node at (-.8,-4.5) {\tiny $2$};
\node at (2.8,-4.5) {\tiny $1$};

\pgftransformxshift{12cm}
\draw (1,0) -- (1,-1.5);
\draw (0,-2.5) -- (1,-1.5) -- (2,-2.5);
\node at (1.8,-1) {\tiny $\lambda$};
\draw (1,2) -- (1,1.5); \draw[<->]  (1,1.8)   -- (1,.4); \draw (1,.7) -- (1,.2); \node at (-.8,1) {\tiny $\coev$}; \node at (1,2.5) {\tiny $s$};
\draw (0,-2.7) -- (0,-3.2); \draw[>-<]  (0,-2.9)   -- (0,-4.3); \draw (0,-4) -- (0,-4.5); \node at (-.8,-3.5) {\tiny $\ev$}; \node at (0,-5) {\tiny $\omega$};
\draw (2,-2.7) -- (2,-3.2); \draw[>-<]  (2,-2.9)   -- (2,-4.3); \draw (2,-4) -- (2,-4.5); \node at (2.8,-3.5) {\tiny $\ev$}; \node at (2,-5) {\tiny $\omega$};
\node at (-.8,-4.5) {\tiny $2$};
\node at (2.8,-4.5) {\tiny $1$};

\end{tikzpicture}
\end{minipage}

\begin{align*}
\MoveEqLeft{\boldmu|_{A^\vee[1] \otimes^* A^\vee[1]\to A^\vee[1]}}\\
& = s \lambda^\vee(\omega\otimes\omega) \\
& = s(1\otimes\ev\tau\otimes\ev\tau)\tau_{34}\tau_{23}(1\otimes\lambda\otimes 1 \otimes 1)(\tau\coev\otimes 1\otimes 1)(\omega\otimes\omega).\\
&  = s (\ev \otimes \ev \otimes 1) \tau_{23} \tau_{34}(1 \otimes 1 \otimes \lambda \otimes 1) (1 \otimes 1 \otimes \coev)(\omega \otimes \omega)\\
&= - s (\ev \otimes \ev \otimes 1) \tau_{23} (1 \otimes 1 \otimes \lambda \otimes 1) (1 \otimes 1 \otimes \coev)(\omega \otimes \omega).   
\end{align*}

\begin{minipage}{.8\textwidth}
\begin{tikzpicture}[scale=.25]
\draw (1,0) -- (1,-1.5);
\draw (0,-2.5) -- (1,-1.5) -- (2,-2.5);
\node at (-3,-1) {$-$};
\node at (8,-1) {=};
\draw[->]        (1,1) -- (1,.2);
\draw[->]        (0,-3.5)  -- (0,-2.7);
\draw[->]        (2,-2.7)  -- (2,-3.5);

\pgftransformxshift{16cm}
\draw (1,0) -- (1,-1.5);
\draw (0,-2.5) -- (1,-1.5) -- (2,-2.5);
\node at (-3,-1) {$-$};
\node at (1.8,-1) {\tiny $\lambda$};
\draw[->] (1,2) -- (1,.2); 
\draw (0,-2.7) -- (0,-3.2); \draw[>-<]  (0,-2.9)   -- (0,-4.3); \draw (0,-4) -- (0,-4.5); \node at (-.8,-3.5) {\tiny $\ev$}; \node at (0,-5) {\tiny $\omega$};
\draw[->] (2,-2.7) -- (2,-4.5); 

\end{tikzpicture}
\end{minipage}

$$
\boldmu|_{A^\vee[1]\otimes^* A\to A}=-(\ev\otimes 1)(1\otimes\lambda)(\omega\otimes 1)
$$


\begin{minipage}{.8\textwidth}
\begin{tikzpicture}[scale=.25]

\node at (-3,1) {\phantom{-}};

\draw (1,0) -- (1,-1.5);
\draw (0,-2.5) -- (1,-1.5) -- (2,-2.5);
\node at (8,-1) {=};
\draw[->]        (1,1) -- (1,.2);
\draw[->]        (0,-2.7) -- (0,-3.5) ;
\draw[->]        (2,-3.5)  -- (2,-2.7);

\pgftransformxshift{12cm}
\draw (1,0) -- (1,-1.5);
\draw (0,-2.5) -- (1,-1.5) -- (2,-2.5);
\node at (1.8,-1) {\tiny $\lambda$};
\draw[->] (1,2) -- (1,.2); 
\draw[->] (0,-2.7) -- (0,-4.5); 
\draw (2,-2.7) -- (2,-3.2); \draw[>-<]  (2,-2.9)   -- (2,-4.3); \draw (2,-4) -- (2,-4.5); \node at (2.8,-3.5) {\tiny $\ev$}; \node at (2,-5) {\tiny $\omega$};

\end{tikzpicture}
\end{minipage}

$$
\boldmu|_{A \otimes^* A^\vee[1]\to A}= (1\otimes\ev\tau)(\lambda\otimes 1)(1\otimes\omega).
$$

The product $\boldmu$ has degree $0$ and is commutative. For example, for the restriction to $A^\vee[1]\otimes^* A^\vee[1]$ we have
\begin{align*}
\boldmu\tau&|_{A^\vee[1] \otimes^* A^\vee[1]}\\
& = s\lambda^\vee(\omega\otimes\omega)\tau = -s\lambda^\vee\tau(\omega\otimes\omega)=s\lambda^\vee(\omega\otimes\omega)=\boldmu|_{A^\vee[1]\otimes^* A^\vee[1]},
\end{align*}
where we have used the graded commutativity $\lambda^\vee\tau=-\lambda^\vee$ for $\lambda^\vee$ and the relation $(\omega\otimes\omega)\tau=-\tau(\omega\otimes\omega)$. 
Similarly, for the component mapping $A^\vee[1] \otimes^* A$ to $A$ we have
\begin{align*}
\boldmu \tau|_{A^\vee[1] \otimes^* A \to A}& = \boldmu|_{A \otimes^* A^\vee[1] \to A} \tau\\
&= (1\otimes\ev\tau)(\lambda\otimes 1)(1\otimes\omega)\tau \\
&= (1\otimes\ev\tau)(\lambda\otimes 1)\tau(\omega\otimes 1) \\
&= (1\otimes\ev\tau)(1 \otimes \tau)(\tau \otimes 1)(1 \otimes \lambda)(\omega\otimes 1) \\
&= (1\otimes\ev)(\tau \otimes 1)(1 \otimes \lambda)(\omega\otimes 1) \\
&= (\ev \otimes 1)(1 \otimes \tau)(1 \otimes \lambda)(\omega\otimes 1) \\
&= -(\ev \otimes 1)(1 \otimes \lambda)(\omega\otimes 1) \\
&= \boldmu |_{A^\vee[1]\otimes^* A \to A}.
\end{align*}
The computations for the other components are analogous.

A direct, though tedious, computation shows that the product $\boldmu$ is associative. This is a consequence of the associativity of $\mu$, the coassociativity of $\lambda$, the infinitesimal relation $\lambda\mu=(1\otimes\mu)(\lambda\otimes 1) + (\mu\otimes 1)(1\otimes\lambda)$, the commutativity of $\mu$ and the cocommutativity of $\lambda$. Strictly speaking, the last two properties are only used via the fact that, together with the infinitesimal relation, they imply the so-called ``anti-symmetry relation" for $\mu$ and $\lambda$, cf.~\cite[\S3, Remark~3.4]{CO-algebra}. In the ungraded case the proof of the associativity of the product $\boldmu$ can be found in Zhelyabin~\cite{Zhelyabin1997} and Aguiar~\cite{Aguiar2001}. We give here one sample computation to illustrate the appearance of signs in the graded case, namely we show how the associativity of $\boldmu$ on the component $A^\vee[1]^{ \otimes^* 3}\to A^\vee[1]$ amounts to graded coassociativity of $\lambda$. Noting that, on $A^\vee[1] \otimes^* A^\vee[1]$, we have $\boldmu=s\lambda^\vee(\omega\otimes \omega)$, we compute
\begin{align*}
\boldmu(\boldmu\otimes 1) & = s\lambda^\vee(\omega\otimes\omega)(s\lambda^\vee(\omega\otimes\omega)) \\
& = s\lambda^\vee(\omega\otimes 1)(s\lambda^\vee(\omega\otimes\omega)\otimes\omega) \\
& = s\lambda^\vee(\lambda^\vee\otimes 1)(\omega\otimes\omega\otimes\omega).\\
\boldmu(1\otimes\boldmu) & = s\lambda^\vee(\omega\otimes\omega) (1\otimes s\lambda^\vee(\omega\otimes\omega))\\
&=s\lambda^\vee(\omega\lambda^\vee(\omega\otimes\omega))\\
&=-s\lambda^\vee(1\otimes\lambda^\vee)(\omega\otimes\omega\otimes\omega).
\end{align*}
In the last line the Koszul sign comes from exchanging the maps $\omega$ and $\lambda^\vee$, both of odd degree. Thus even degree associativity $\boldmu(\boldmu\otimes 1)=\boldmu(1\otimes\boldmu)$ for $\boldmu$ is equivalent to odd degree associativity $\lambda^\vee(\lambda^\vee\otimes 1)=-\lambda^\vee(1\otimes\lambda^\vee)$ for $\lambda^\vee$, which is in turn equivalent to odd degree coassociativity $(\lambda\otimes 1)\lambda=-(1\otimes\lambda)\lambda$ for $\lambda$. 

\smallskip
{\it Step~2. Definition of the coproduct $\boldlambda$ and proof of coassociativity.} As for the product, out of the eight possible components of the coproduct the two components $A \to A^\vee[1] \otimes^! A^\vee[1]$ and $A^\vee[1]\to A \otimes^! A$ vanish, and so we get the following preliminary description:

\begin{center}

\begin{tikzpicture}[scale=.25]

\draw[double, double distance=1.5pt] (0,-2.5) -- (1,-1.5) -- (2,-2.5);
\draw[double, double distance=1.5pt] (1,0) -- (1,-1.5);
\node at (6,-1) {=};

\pgftransformxshift{14cm}
\draw (1,0) -- (1,-1.5);
\draw (0,-2.5) -- (1,-1.5) -- (2,-2.5);
\node at (1,-5.3) {\tiny $A\to A \otimes^! A$};
\node at (7,-1) {$+$};
\draw[->]        (1,1) -- (1,.2);
\draw[->]        (0,-2.7)   -- (0,-3.5);
\draw[->]        (2,-2.7)   -- (2,-3.5);
\node at (-.8,-2.5) {\tiny $1$};
\node at (2.8,-2.5) {\tiny $2$};

\pgftransformxshift{14cm}
\draw (1,0) -- (1,-1.5);
\draw (0,-2.5) -- (1,-1.5) -- (2,-2.5);
\node at (1,-5.3) {\tiny $A^\vee[1]\to A\otimes^! A^\vee[1]$};
\node at (7,-1) {+};
\draw[<-]        (1,1) -- (1,.2);
\draw[->]        (0,-3.5)  -- (0,-2.7);
\draw[->]        (2,-2.7)  -- (2,-3.5);
\node at (2.8,-2.5) {\tiny $1$};
\node at (1.8,0) {\tiny $2$};

\pgftransformxshift{14cm}
\draw (1,0) -- (1,-1.5);
\draw (0,-2.5) -- (1,-1.5) -- (2,-2.5);
\node at (1,-5.3) {\tiny $A^\vee[1]\to A^\vee[1]\otimes^! A$};
\draw[<-]        (1,1) -- (1,.2);
\draw[->]        (0,-2.7) -- (0,-3.5) ;
\draw[->]        (2,-3.5)  -- (2,-2.7);
\node at (0.2,0) {\tiny $1$};
\node at (-0.8,-2.5) {\tiny $2$};

\pgftransformxshift{-28cm}
\pgftransformyshift{-12cm}
\draw (0,0) -- (1,-1) -- (2,0);
\draw (1,-1) -- (1,-2.5);
\node at (1,-5.5) {\tiny $A^\vee[1]\to A^\vee[1]\otimes^! A^\vee[1]$};
\node at (-5,-1) {$+$};
\node at (7,-1) {$-$};
\draw[<-]        (0,1)   -- (0,.2);
\draw[<-]        (2,1)   -- (2,.2);
\draw[<-]        (1,-2.7)   -- (1,-3.5);
\node at (-.8,0) {\tiny $2$};
\node at (2.8,0) {\tiny $1$};

\pgftransformxshift{14cm}
\draw (0,0) -- (1,-1) -- (2,0);
\draw (1,-1) -- (1,-2.5);
\node at (1,-5.5) {\tiny $A\to A \otimes^! A^\vee[1]$};
\node at (7,-1) {+};
\draw[->]        (0,1)   -- (0,.2);
\draw[->]        (2,.2) -- (2,1);
\draw[<-]        (1,-3.5) -- (1,-2.7);
\node at (1.8,-2.5) {\tiny $1$};
\node at (2.8,0) {\tiny $2$};

\pgftransformxshift{14cm}
\draw (0,0) -- (1,-1) -- (2,0);
\draw (1,-1) -- (1,-2.5);
\node at (1,-5.5) {\tiny $A\to A^\vee[1]\otimes^! A$};
\draw[->]        (0,.2) -- (0,1);
\draw[->]        (2,1)   -- (2,.2);
\draw[<-]        (1,-3.5) -- (1,-2.7);
\node at (-0.8,0) {\tiny $1$};
\node at (0.2,-2.5) {\tiny $2$};

\end{tikzpicture}
\end{center}

For the more precise descriptions of the individual terms below, we remind the reader that whenever an operation has an input (resp. an output) in $A^\vee[1]$, one must first precompose with $\omega$ (resp. last postcompose with $s$), and that we use $\ev$ to turn outputs in $A$ into inputs in $A^\vee$ and $\coev$ to turn inputs in $A$ into outputs in $A^\vee$. 

The component $A\to A \otimes^! A$ of the coproduct $\boldlambda$ is simply given by $\lambda$. The formulas for the other components are the following. 

\begin{minipage}{.8\textwidth}
\begin{tikzpicture}[scale=.25]
\node at (-3,0) {\phantom{blah}};
\draw (1,0) -- (1,-1.5);
\draw (0,-2.5) -- (1,-1.5) -- (2,-2.5);
\node at (8,-1) {=};
\draw[<-]        (1,1) -- (1,.2);
\draw[->]        (0,-3.5)  -- (0,-2.7);
\draw[->]        (2,-2.7)  -- (2,-3.5);

\pgftransformxshift{12cm}
\draw (1,0) -- (1,-1.5);
\draw (0,-2.5) -- (1,-1.5) -- (2,-2.5);
\node at (1.8,-1) {\tiny $\lambda$};
\draw (1,2) -- (1,1.5); \draw[<->]  (1,1.8)   -- (1,.4); \draw (1,.7) -- (1,.2); \node at (-.8,1) {\tiny $\coev$}; \node at (1,2.5) {\tiny $s$};
\draw (0,-2.7) -- (0,-3.2); \draw[>-<]  (0,-2.9)   -- (0,-4.3); \draw (0,-4) -- (0,-4.5); \node at (-.8,-3.5) {\tiny $\ev$}; \node at (0,-5) {\tiny $\omega$};
\draw[->] (2,-2.7) -- (2,-4.5);

\end{tikzpicture}
\end{minipage}

$$
\boldlambda|_{A^\vee[1]\to A \otimes^! A^\vee[1]}=(1\otimes s)(\ev\otimes 1\otimes 1)(1\otimes\lambda\otimes 1)(1\otimes\coev)\omega
$$


\begin{minipage}{.8\textwidth}
\begin{tikzpicture}[scale=.25]
\node at (-3,0) {\phantom{blah}};
\draw (1,0) -- (1,-1.5);
\draw (0,-2.5) -- (1,-1.5) -- (2,-2.5);
\node at (8,-1) {=};
\draw[<-]        (1,1) -- (1,.2);
\draw[->]        (0,-2.7) -- (0,-3.5) ;
\draw[->]        (2,-3.5)  -- (2,-2.7);

\pgftransformxshift{12cm}
\draw (1,0) -- (1,-1.5);
\draw (0,-2.5) -- (1,-1.5) -- (2,-2.5);
\node at (1.8,-1) {\tiny $\lambda$};
\draw (1,2) -- (1,1.5); \draw[<->]  (1,1.8)   -- (1,.4); \draw (1,.7) -- (1,.2); \node at (-.8,1) {\tiny $\coev$}; \node at (1,2.5) {\tiny $s$};
\draw[->] (0,-2.7) -- (0,-4.5); 
\draw (2,-2.7) -- (2,-3.2); \draw[>-<]  (2,-2.9)   -- (2,-4.3); \draw (2,-4) -- (2,-4.5); \node at (2.8,-3.5) {\tiny $\ev$}; \node at (2,-5) {\tiny $\omega$};

\end{tikzpicture}
\end{minipage}

\begin{align*}
\boldlambda|_{A^\vee[1]\to A^\vee[1] \otimes^! A}
&=(s\otimes 1)(1\otimes 1\otimes \ev\tau)(1\otimes\lambda\otimes 1)(\tau\coev\otimes 1)\omega\\
&=(1\otimes s)(\ev \otimes 1\otimes 1)(1\otimes\tau\lambda\otimes 1)(1\otimes \coev)\omega
\end{align*}


\begin{minipage}{.8\textwidth}
\begin{tikzpicture}[scale=.25]
\node at (-3,0) {\phantom{blah}}; 
\draw (0,0) -- (1,-1) -- (2,0);
\draw (1,-1) -- (1,-2.5);
\node at (8,-1) {$=$};
\draw[<-]        (0,1)   -- (0,.2);
\draw[<-]        (2,1)   -- (2,.2);
\draw[<-]        (1,-2.7)   -- (1,-3.5);
\node at (-.8,0) {\tiny $2$};
\node at (2.8,0) {\tiny $1$};

\pgftransformxshift{14cm}
\draw (0,0) -- (1,-1) -- (2,0);
\draw (1,-1) -- (1,-2.5);
\node at (2,-1.3) {\tiny $\mu^\vee$};
\draw[<-]        (0,1)   -- (0,.2); \node at (0,1.5) {\tiny $s$};
\draw[<-]        (2,1)   -- (2,.2); \node at (2,1.5) {\tiny $s$};
\draw[<-]        (1,-2.7)   -- (1,-3.5); \node at (1,-4) {\tiny $\omega$};
\node at (-.8,0) {\tiny $2$};
\node at (2.8,0) {\tiny $1$};

\node at (8,-1) {$=$};

\pgftransformxshift{14cm}
\draw (0,0) -- (1,-1) -- (2,0);
\draw (1,-1) -- (1,-2.5);
\node at (1.8,-1) {\tiny $\mu$};
\draw (0,2) -- (0,1.5); \draw[<->]  (0,1.8)   -- (0,.4); \draw (0,.7) -- (0,.2); \node at (-1.5,1) {\tiny $\coev$}; \node at (0,2.5) {\tiny $s$};
\draw (2,2) -- (2,1.5); \draw[<->]  (2,1.8)   -- (2,.4); \draw (2,.7) -- (2,.2); \node at (3.5,1) {\tiny $\coev$}; \node at (2,2.5) {\tiny $s$};
\draw (1,-2.7) -- (1,-3.2); \draw[>-<]  (1,-2.9)   -- (1,-4.3); \draw (1,-4) -- (1,-4.5); \node at (1.8,-3.5) {\tiny $\ev$}; \node at (1,-5) {\tiny $\omega$};
\node at (-.8,0) {\tiny $2$};
\node at (2.8,0) {\tiny $1$};

\end{tikzpicture}
\end{minipage}

\begin{align*}
\boldlambda&|_{A^\vee[1]\to A^\vee[1] \otimes^! A^\vee[1]}\\
&= (s \otimes s) \mu^\vee \omega\\
&=(s\otimes s) (\ev\otimes 1 \otimes 1)(1\otimes\mu\otimes 1\otimes 1)\tau_{23} \tau_{34} (1\otimes\coev\otimes\coev)\omega
\end{align*}



\begin{minipage}{.8\textwidth}
\begin{tikzpicture}[scale=.25]
\node at (-3,0) {\phantom{blah}}; 
\draw (0,0) -- (1,-1) -- (2,0);
\draw (1,-1) -- (1,-2.5);
\node at (-3,-1) {$-$};
\node at (8,-1) {$=$};
\draw[->]        (0,1)   -- (0,.2);
\draw[->]        (2,.2) -- (2,1);
\draw[<-]        (1,-3.5) -- (1,-2.7);

\pgftransformxshift{14cm}
\draw (0,0) -- (1,-1) -- (2,0);
\draw (1,-1) -- (1,-2.5);
\node at (-3,-1) {$-$};
\node at (1.8,-1) {\tiny $\mu$};
\draw[->] (0,2) -- (0,.2); 
\draw (2,2) -- (2,1.5); \draw[<->]  (2,1.8)   -- (2,.4); \draw (2,.7) -- (2,.2); \node at (3.5,1) {\tiny $\coev$}; \node at (2,2.5) {\tiny $s$};
\draw[->] (1,-2.7) -- (1,-4.5);

\end{tikzpicture}
\end{minipage}

$$
\boldlambda|_{A\to A \otimes^! A^\vee[1]}=-(1\otimes s)(\mu\otimes 1)(1\otimes\coev)
$$


\begin{minipage}{.8\textwidth}
\begin{tikzpicture}[scale=.25]
\node at (-3,0) {\phantom{blah}}; 
\draw (0,0) -- (1,-1) -- (2,0);
\draw (1,-1) -- (1,-2.5);
\node at (8,-1) {$=$};
\draw[->]        (0,.2) -- (0,1);
\draw[->]        (2,1)   -- (2,.2);
\draw[<-]        (1,-3.5) -- (1,-2.7);

\pgftransformxshift{14cm}
\draw (0,0) -- (1,-1) -- (2,0);
\draw (1,-1) -- (1,-2.5);
\node at (1.8,-1) {\tiny $\mu$};
\draw[->] (2,2) -- (2,.2); 
\draw (0,2) -- (0,1.5); \draw[<->]  (0,1.8)   -- (0,.4); \draw (0,.7) -- (0,.2); \node at (-1.5,1) {\tiny $\coev$}; \node at (0,2.5) {\tiny $s$};
\draw[->] (1,-2.7) -- (1,-4.5);

\end{tikzpicture}
\end{minipage}

$$
\boldlambda|_{A\to A^\vee[1] \otimes^! A}=(s\otimes 1)(1\otimes\mu)(\tau\coev\otimes 1)
$$

The coproduct $\boldlambda$ has degree $-1$ and graded cocommutativity (meaning $\tau\boldlambda=-\boldlambda$) is established by computations analogous to those showing commutativity of the product $\boldmu$.
 
Just like in the proof of associativity of $\boldmu$, a direct and equally tedious computation shows that the coproduct $\boldlambda$ is graded coassociative, i.e., $(\boldlambda\otimes 1)\boldlambda=-(1\otimes\boldlambda)\boldlambda$. For the direct computation, this is a consequence of the associativity of $\mu$, the coassociativity of $\lambda$, the infinitesimal relation $\lambda\mu=(1\otimes\mu)(\lambda\otimes 1) + (\mu\otimes 1)(1\otimes\lambda)$, the commutativity of $\mu$ and the cocommutativity of $\lambda$ (through the ``anti-symmetry relation" for $\mu$ and $\lambda$). Alternatively, coassociativity for $\boldlambda$ can be deduced formally by dualization from the associativity of $\boldmu$. 

\smallskip
{\it Step~3. Definition of $\boldeta$, $\boldeps$, $\boldDelta$ and proof of the BV Frobenius relation.} 

The \emph{unit} $\boldeta\in A\oplus A^\vee[1]$ is defined to be 
$$
\boldeta = \eta \in A. 
$$
The relations $\boldmu(\boldeta\otimes 1)=1=\boldmu(1\otimes\boldeta)$ are straightforward from the definition of $\boldmu$. 

The \emph{counit} $\boldeps\in (A\oplus A^\vee[1])^\vee$ is defined to be 
$$
\boldeps=\eta\omega=\ev(\eta\otimes\omega) \in (A^\vee[1])^\vee.
$$
In other words, we view $\eta\in A$ as an element of $A^{\vee\vee}$ via the canonical 
 identification $A\simeq A^{\vee\vee}$,
and $\boldeps$ acts on $A^\vee[1]$ by $\alpha\mapsto \langle \eta,\omega\alpha\rangle=\ev(\eta\otimes\omega\alpha)$. That $\boldeps$ is a counit means that $(\boldeps\otimes 1)\boldlambda=1=-(1\otimes\boldeps)\boldlambda$, and these relations can be verified from the definition of $\boldeps$. For example, on the component $A^\vee[1]\to A^\vee[1]$ we have 
\begin{align*}
(& 1\otimes\boldeps)\boldlambda  \\
& = (1\otimes\boldeps)\boldlambda|_{A^\vee[1]\otimes A^\vee[1]} \\
& = (1\otimes \eta\omega)(s\otimes s) (\ev\otimes 1 \otimes 1)\\
& \qquad \qquad \qquad (1\otimes\mu\otimes 1\otimes 1)\tau_{23} \tau_{34} (1\otimes\coev\otimes\coev)\omega\\
& = - (s\otimes \eta)(\ev\otimes 1 \otimes 1)(1\otimes\mu\otimes 1\otimes 1)\tau_{23} \tau_{34} (1\otimes\coev\otimes\coev)\omega\\
& = -  s(\ev\otimes 1)(1\otimes \mu(1\otimes\eta)\otimes 1)(1\otimes \coev) \omega \\
& = -1.
\end{align*}

The \emph{BV operator} $\boldDelta:A\oplus A^\vee[1]\to A\oplus A^\vee[1]$ is defined to be 
$$
\boldDelta = \begin{pmatrix} \Delta & 0 \\ 0 & -s\Delta^\vee \omega \end{pmatrix}.
$$
Note that $\boldDelta$ has degree $1$.
The operator $\Delta^\vee$ is defined by the adjunction equality 
$$
\ev(\Delta^\vee\otimes 1)=\ev(1\otimes \Delta),
$$
or, dually,
\begin{equation}\label{eq:coev_and_delta}
(1\otimes\Delta^\vee)\coev = (\Delta\otimes 1)\coev.
\end{equation}

We now prove the BV Frobenius relation~\eqref{eq:BVFrobenius-relation}, i.e.,  
\begin{equation} \label{eq:Delta-Frobenius}
(\boldDelta\otimes 1) \boldlambda\boldeta = (1\otimes\boldDelta)\boldlambda\boldeta.
\end{equation}
The definition of $\boldeta$ and $\boldlambda$ implies 
\begin{align*}
\boldlambda\boldeta & = -(1\otimes s)(\mu\otimes 1)(1\otimes\coev)\eta + (s\otimes 1)(1\otimes\mu)(\tau \coev\otimes 1) \eta \\
& = -(1\otimes s)\coev + (s\otimes 1)\tau\coev\\
& = -(1\otimes s)\coev + \tau (1 \otimes s)\coev.
\end{align*}
Using the definition of $\boldDelta$ and the fact that both $\Delta$ and $s$ have odd degree, we now get
\begin{align*}
(\boldDelta \otimes 1) \boldlambda \boldeta
& = - (\Delta \otimes 1)(1 \otimes s) \coev + (-s\Delta^\vee\omega \otimes 1)\tau (1 \otimes s) \coev \\
& =  (1 \otimes s)(\Delta \otimes 1) \coev - \tau (1 \otimes s)(1 \otimes \Delta^\vee) \coev \\
&= (1 \otimes s)(1 \otimes \Delta^\vee) \coev - \tau (1 \otimes s)(\Delta \otimes 1) \coev \\
&= (1 \otimes s\Delta^\vee) \coev + \tau (\Delta \otimes 1)(1 \otimes s) \coev \\
&= (1 \otimes s\Delta^\vee\omega) (1 \otimes s)\coev + (1 \otimes \Delta)\tau (1 \otimes s) \coev \\
&= (1 \otimes \boldDelta) \boldlambda \boldeta,
\end{align*}
proving \eqref{eq:Delta-Frobenius}.
Here, in the third step, we have used \eqref{eq:coev_and_delta} twice.
\footnote{If the second diagonal term in $\boldDelta$ is chosen to be $s\Delta^\vee\omega$, then we get the relation $(\boldDelta\otimes1)\boldlambda\boldeta=-(1\otimes\boldDelta)\boldlambda\boldeta$.}

\smallskip
{\it Step~4. Proof of the Frobenius relation~\eqref{eq:Frobenius}.} 

The Frobenius relation 
$$
\boldlambda\boldmu=(\boldmu\otimes 1)(1\otimes\boldlambda)=(1\otimes\boldmu)(\boldlambda\otimes 1)
$$
is proved by a direct computation. This is very much similar to the computations involved in the proof of associativity of $\boldmu$ and coassociativity of $\boldlambda$, and we omit the details. In the ungraded case, the Frobenius relation was proved by Bai~\cite{Bai2010}. 

\smallskip
{\it Step~5. Proof of the 7-term relation for $\boldDelta$ and $\boldmu$.} 

The 7-term relation for $\boldDelta$ and $\boldmu$ holds as a consequence of the 7-term relation for $\Delta$ and $\mu$, the 7-term relation for $\Delta$ and $\lambda$, and the 9-term relation for $\Delta$, $\lambda$ and $\mu$. The 7-term relation for $\boldDelta$ and $\boldmu$ involves 16 components, and we give the details of the proof on the component $A \otimes^* A \otimes^* A^\vee[1]\to A$ in order to illustrate the appearance of the 9-term relation for $\Delta$, $\lambda$ and $\mu$.

Recall that the 7-term relation for $\boldDelta$ and $\boldmu$ writes 
\begin{center}
\begin{tikzpicture}[scale=.23]
\draw[double, double distance=1.5pt] (0,0) -- (1,-1) -- (2,0);
\draw[double, double distance=1.5pt] (1,-1) -- (1,-2);
\draw[double, double distance=1.5pt] (1,-2) -- (2,-3) -- (4,-1);
\draw[double, double distance=1.5pt] (2,-3) -- (2,-4);
\node at (2,-5) {\tiny$\boldDelta$};

\node at (5.5,-2) {$+$};

\pgftransformxshift{7cm}

\draw[double, double distance=1.5pt] (0,0) -- (1,-1) -- (2,0);
\draw[double, double distance=1.5pt] (1,-1) -- (1,-2);
\draw[double, double distance=1.5pt] (1,-2) -- (2,-3) -- (4,-1);
\draw[double, double distance=1.5pt] (2,-3) -- (2,-4);
\node at (0,1) {\tiny$\boldDelta$};

\node at (5.5,-2) {$+$};
\pgftransformxshift{7cm}

\draw[double, double distance=1.5pt] (0,0) -- (1,-1) -- (2,0);
\draw[double, double distance=1.5pt] (1,-1) -- (1,-2);
\draw[double, double distance=1.5pt] (1,-2) -- (2,-3) -- (4,-1);
\draw[double, double distance=1.5pt] (2,-3) -- (2,-4);
\node at (0,.8) {\tiny$\boldDelta$};
\node at (0,2) {\tiny$2$};
\node at (2,1) {\tiny$3$};
\node at (4,0) {\tiny$1$};

\node at (5.5,-2) {$+$};
\pgftransformxshift{7cm}

\draw[double, double distance=1.5pt] (0,0) -- (1,-1) -- (2,0);
\draw[double, double distance=1.5pt] (1,-1) -- (1,-2);
\draw[double, double distance=1.5pt] (1,-2) -- (2,-3) -- (4,-1);
\draw[double, double distance=1.5pt] (2,-3) -- (2,-4);
\node at (0,.8) {\tiny$\boldDelta$};
\node at (0,2) {\tiny$3$};
\node at (2,1) {\tiny$1$};
\node at (4,0) {\tiny$2$};

\node at (5.5,-2) {$-$};
\pgftransformxshift{7cm}

\draw[double, double distance=1.5pt] (0,0) -- (1,-1) -- (2,0);
\draw[double, double distance=1.5pt] (1,-1) -- (1,-2);
\draw[double, double distance=1.5pt] (1,-2) -- (2,-3) -- (4,-1);
\draw[double, double distance=1.5pt] (2,-3) -- (2,-4);
\node at (0.4,-2.4) {\tiny$\boldDelta$};

\node at (5.5,-2) {$-$};
\pgftransformxshift{7cm}

\draw[double, double distance=1.5pt] (0,0) -- (1,-1) -- (2,0);
\draw[double, double distance=1.5pt] (1,-1) -- (1,-2);
\draw[double, double distance=1.5pt] (1,-2) -- (2,-3) -- (4,-1);
\draw[double, double distance=1.5pt] (2,-3) -- (2,-4);
\node at (0.4,-2.4) {\tiny$\boldDelta$};
\node at (0,1) {\tiny$2$};
\node at (2,1) {\tiny$3$};
\node at (4,0) {\tiny$1$};
\node at (5.5,-2) {$-$};
\pgftransformxshift{7cm}

\draw[double, double distance=1.5pt] (0,0) -- (1,-1) -- (2,0);
\draw[double, double distance=1.5pt] (1,-1) -- (1,-2);
\draw[double, double distance=1.5pt] (1,-2) -- (2,-3) -- (4,-1);
\draw[double, double distance=1.5pt] (2,-3) -- (2,-4);
\node at (0,1) {\tiny$3$};
\node at (2,1) {\tiny$1$};
\node at (4,0) {\tiny$2$};
\node at (0.4,-2.4) {\tiny$\boldDelta$};

\node at (6.5,-2) {=\ \ 0.};
\end{tikzpicture}
\end{center}

For the \underline{first term}, the result of the first multiplication must land in $A$, and so it gives rise to a single term 
\begin{center}
\begin{tikzpicture}[scale=.25]
\draw (0,0) -- (1,-1) -- (2,0);
\draw (1,-1) -- (1,-2);
\draw (0,-3) -- (1,-2) -- (2,-3);
\node at (0,-4) {\tiny$\Delta$};
\draw (2,-3.2) -- (2,-3.7); \draw[>-<]  (2,-3.4)   -- (2,-4.8); \draw (2,-4.5) -- (2,-5); \node at (2.8,-4) {\tiny $\ev$}; \node at (2,-5.5) {\tiny $\omega$};
\node at (0,.5) {\tiny$1$};
\node at (2,.5) {\tiny$2$};
\node at (2.5,-3) {\tiny$3$}; 
\end{tikzpicture}
\raisebox{1cm}{\tiny $=\Delta (1\otimes \ev\tau)(\lambda\otimes 1)(1\otimes\omega)(\mu\otimes 1) = (1\otimes \ev\tau)[((\Delta\otimes 1)\lambda\mu)\otimes 1](1\otimes1\otimes\omega)$.} 
\end{center}

Similarly, the \underline{second term} gives rise to a single term 
\begin{center}
\begin{tikzpicture}[scale=.25]
\draw (0,0) -- (1,-1) -- (2,0);
\draw (1,-1) -- (1,-2);
\draw (0,-3) -- (1,-2) -- (2,-3);
\node at (0,.7) {\tiny$\Delta$};
\draw (2,-3.2) -- (2,-3.7); \draw[>-<]  (2,-3.4)   -- (2,-4.8); \draw (2,-4.5) -- (2,-5); \node at (2.8,-4) {\tiny $\ev$}; \node at (2,-5.5) {\tiny $\omega$};
\node at (-.5,-.5) {\tiny$1$};
\node at (2.5,-.5) {\tiny$2$};
\node at (2.5,-3) {\tiny$3$}; 
\end{tikzpicture}
\raisebox{1cm}{\begin{minipage}{10cm}\tiny \begin{align*}&=(1\otimes\ev\tau)(\lambda\otimes 1)(1\otimes\omega)[\mu(\Delta\otimes 1)\otimes 1]\\ & = -(1\otimes \ev\tau)[\lambda\mu(\Delta\otimes 1)](1\otimes1\otimes\omega).\end{align*}\end{minipage}} 
\end{center}

For the \underline{third term}, the output of the first product can land in either $A$ or $A^\vee[1]$, and so it gives rise to the sum
\begin{center}
\begin{tikzpicture}[scale=.25]
\draw (1,0) -- (1,-1);
\draw (-1,-3) -- (1,-1) -- (2,-2);
\node at (1.2,.7) {\tiny$\Delta$};
\draw (2,-2.2) -- (2,-2.7); \draw[>-<]  (2,-2.4)   -- (2,-3.8); \draw (2,-3.5) -- (2,-4); \node at (2.8,-3) {\tiny $\ev$}; \node at (2,-4.5) {\tiny $\omega$};
\node at (.5,0) {\tiny$2$};
\node at (2.5,-2) {\tiny$3$}; 
\draw (-1,-3) -- (2,-6) -- (5,-3);
\draw (2,-6) -- (2,-9);
\node at (5,-2.5) {\tiny$1$};
\node at (8,-3) {$-$};

\pgftransformxshift{14cm}
\draw (1,3) -- (1,0) -- (5,-3) -- (1,0) -- (-1,-2);
\draw (-1,-2.2) -- (-1,-2.7); \draw[>-<]  (-1,-2.4)   -- (-1,-3.8); \draw (-1,-3.5) -- (-1,-4); \node at (0,-3) {\tiny $\ev$}; \node at (-0.2,-3.7) {\tiny $\omega$};
\draw (-1,-4.2) -- (-1,-4.7); \draw[<->]  (-1,-4.4)   -- (-1,-5.8); \draw (-1,-5.5) -- (-1,-6); \node at (0.5,-5.2) {\tiny $\coev$}; \node at (-0.3,-4.4) {\tiny $s$};
\draw (-1,-6) -- (1,-8) -- (3,-6) -- (1,-8) -- (1,-10);
\draw (1,-10.2) -- (1,-10.7); \draw[>-<]  (1,-10.4)   -- (1,-11.8); \draw (1,-11.5) -- (1,-12); \node at (2,-11) {\tiny $\ev$}; \node at (1.8,-11.7) {\tiny $\omega$};
\node at (1,3.5) {\tiny $1$};
\node at (3.5,-5.5) {\tiny $\Delta$};
\node at (2.5,-5.7) {\tiny $2$};
\node at (1.5,-10) {\tiny $3$};
\node at (8,-3) {$=$};
\node at (11,-3) {$-$};

\pgftransformxshift{16cm}
\draw (1,0) -- (1,-2) -- (3,-4); 
\draw (1,-2) -- (-1,-4) -- (-3,-2) -- (-1,-4) -- (-1,-6);
\draw (3,-4.2) -- (3,-4.7); \draw[>-<]  (3,-4.4)   -- (3,-5.8); \draw (3,-5.5) -- (3,-6); \node at (3.8,-5) {\tiny $\ev$}; \node at (3,-6.5) {\tiny $\omega$};
\node at (1.2,.7) {\tiny$\Delta$};
\node at (-3,-1.5) {\tiny$1$};
\node at (.5,0) {\tiny$2$};
\node at (3.5,-4) {\tiny$3$}; 
\node at (8,-3) {$-$};

\pgftransformxshift{12cm}
\draw (1,0) -- (1,-2) -- (-1,-4) -- (-1,-6);
\draw (1,-2) -- (3,-4) -- (5,-2) -- (3,-4) -- (3,-6);
\draw (3,-6.2) -- (3,-6.7); \draw[>-<]  (3,-6.4)   -- (3,-7.8); \draw (3,-7.5) -- (3,-8); \node at (3.8,-7) {\tiny $\ev$}; \node at (3,-8.5) {\tiny $\omega$};
\node at (1,.5) {\tiny $1$};
\node at (5.5,-1.5) {\tiny $\Delta$};
\node at (4.5,-1.7) {\tiny $2$};
\node at (3.5,-6) {\tiny $3$};

\end{tikzpicture}
\end{center}
Here the individual terms on left and right hand side can be described explicitly as
{\tiny \begin{align*}&\mu   
\big([(1\otimes\ev\tau)(\lambda\otimes 1)(1\otimes\omega)]\otimes 1\big)
(\Delta\otimes 1 \otimes 1)(1\otimes\tau)(\tau\otimes 1)
\\ & = -(1\otimes \ev\tau)\big([(\mu\otimes 1)(1\otimes\lambda)(1\otimes\Delta)]\otimes 1\big)(1\otimes1\otimes\omega).\end{align*}}
and
{\tiny \begin{align*}
-(\ev\otimes 1)(1\otimes&\lambda)(\omega\otimes 1) 
\Big[
\big(
s(1\otimes\ev\tau)(1\otimes\mu\otimes1)(\tau\coev\otimes 1\otimes 1)(1\otimes\omega)
\big) 
\otimes 1 \Big]\\
& (\Delta \otimes 1\otimes 1)(1\otimes\tau)(\tau\otimes 1)
 = - (1\otimes\ev\tau)\big([(1\otimes\mu)(\lambda\otimes 1)(1\otimes\Delta)]\otimes1\big)(1\otimes 1 \otimes\omega).
\end{align*}}

In the second computation we use that $\tau\lambda=-\lambda$ and we also get a minus sign from inverting the order of $\omega$ and $\Delta$. These two signs cancel each other, so that the overall sign remains unchanged. 
Using the infinitesimal relation we finally find that the third term is equal to 
\begin{tikzpicture}[scale=.25]
\draw (0,0) -- (1,-1) -- (2,0);
\draw (1,-1) -- (1,-2);
\draw (0,-3) -- (1,-2) -- (2,-3);
\node at (-2,-1.5) {$-$};
\node at (2,.7) {\tiny$\Delta$};
\draw (2,-3.2) -- (2,-3.7); \draw[>-<]  (2,-3.4)   -- (2,-4.8); \draw (2,-4.5) -- (2,-5); \node at (2.8,-4) {\tiny $\ev$}; \node at (2,-5.5) {\tiny $\omega$};
\node at (-.5,-.5) {\tiny$1$};
\node at (2.5,-.5) {\tiny$2$};
\node at (2.5,-3) {\tiny$3$}; 
\end{tikzpicture}
\raisebox{1cm}{\begin{minipage}{10cm}
$=-(1\otimes\ev\tau)([\lambda\mu(1\otimes\Delta)]\otimes 1)(1\otimes1\otimes \omega)$.
\end{minipage}}

The \underline{fourth term} gives rise to the sum 
\begin{center}
\begin{tikzpicture}[scale=.2]
\node at (-3,-3) {$-$};
\draw (1,0) -- (1,-2) -- (-1,-4) -- (-1,-6);
\draw (1,-2) -- (3,-4) -- (5,-2) -- (3,-4) -- (3,-6);
\draw (-1,-6.2) -- (-1,-6.7); \draw[>-<]  (-1,-6.4)   -- (-1,-7.8); \draw (-1,-7.5) -- (-1,-8); \node at (-.2,-7) {\tiny $\ \ev$}; \node at (-.2,-8) {\tiny $\, \omega$};
\draw (-1,-8.2) -- (-1,-9); \draw[<-] (-1,-9) -- (-1,-10); \node at (1.5,-10) {\tiny $\ -s\Delta^\vee\omega$};
\node at (1,.7) {\tiny $1$};
\node at (5,-1.3) {\tiny $2$};
\node at (-1.7,-6) {\tiny $3$};
\node at (8,-3) {$-$};

\pgftransformxshift{14cm}
\draw (1,0) -- (1,-2) -- (3,-4); 
\draw (1,-2) -- (-1,-4) -- (-3,-2) -- (-1,-4) -- (-1,-6);
\draw (-1,-6.2) -- (-1,-6.7); \draw[>-<]  (-1,-6.4)   -- (-1,-7.8); \draw (-1,-7.5) -- (-1,-8); \node at (-.2,-7) {\tiny $\ \ev$}; \node at (-.2,-8) {\tiny $\, \omega$};
\draw (-1,-8.2) -- (-1,-9); \draw[<-] (-1,-9) -- (-1,-10); \node at (1.5,-10) {\tiny $\ -s\Delta^\vee\omega$};
\node at (-3,-1.3) {\tiny$1$};
\node at (.5,0) {\tiny$2$};
\node at (-1.7,-6) {\tiny $3$};
\node at (7,-3) {$=$};
\node at (10,-3) {$-$};

\pgftransformxshift{12cm}
\draw (0,0) -- (2,-2) -- (4,0);
\draw (2,-2) -- (2,-4) -- (0,-6) -- (2,-4) -- (4,-6);
\draw (0,-6.2) -- (0,-6.7); \draw[>-<]  (0,-6.4)   -- (0,-7.8); \draw (0,-7.5) -- (0,-8); \node at (0.8,-7) {\tiny $\ \ev$}; \node at (0.8,-8) {\tiny $\, \omega$};
\draw (0,-8.2) -- (0,-9); \draw[<-] (0,-9) -- (0,-10); \node at (2.5,-10) {\tiny $\ -s\Delta^\vee\omega$};
\node at (-0.5,0) {\tiny$1$};
\node at (4.5,0) {\tiny$2$};
\node at (-0.7,-6) {\tiny $3$};
\node at (7,-3) {$=$};
\node at (10,-3) {$-$};

\pgftransformxshift{12cm}
\draw (0,0) -- (2,-2) -- (4,0);
\draw (2,-2) -- (2,-4) -- (0,-6) -- (2,-4) -- (4,-6);
\draw (0,-7.2) -- (0,-7.7); \draw[>-<]  (0,-7.4)   -- (0,-8.8); \draw (0,-8.5) -- (0,-9); \node at (0.8,-8) {\tiny $\ \ev$}; \node at (0.8,-9) {\tiny $\, \omega$};
\node at (0,-6.6) {\tiny$\Delta$};
\node at (-0.5,0) {\tiny$1$};
\node at (4.5,0) {\tiny$2$};
\node at (-0.7,-6) {\tiny $3$};
\node at (7,-3) {$=$};
\node at (10,-3) {$+$};

\pgftransformxshift{12cm}
\draw (0,0) -- (2,-2) -- (4,0);
\draw (2,-2) -- (2,-4) -- (0,-6) -- (2,-4) -- (4,-6);
\draw (4,-7.2) -- (4,-7.7); \draw[>-<]  (4,-7.4)   -- (4,-8.8); \draw (4,-8.5) -- (4,-9); \node at (4.8,-8) {\tiny $\ \ev$}; \node at (4.8,-9) {\tiny $\, \omega$};
\node at (4,-6.6) {\tiny$\Delta$};
\node at (-0.5,0) {\tiny$1$};
\node at (4.5,0) {\tiny$2$};
\node at (4.7,-6) {\tiny $3$};

\end{tikzpicture}
\end{center}
The first equality in the graphical representation is a consequence of the infinitesimal relation. The second equality is a consequence of $\ev(\omega(-s\Delta^\vee \omega)\otimes 1)=\ev(\omega\otimes 1)(1\otimes \Delta)$. The third equality is a consequence of the graded cocommutativity of $\lambda$. The outcome is 
$$
+(1\otimes\ev\tau)[(1\otimes\Delta)\mu\lambda](1\otimes1\otimes\omega).
$$

The \underline{fifth term} gives rise to the single term
\begin{center}
\begin{tikzpicture}[scale=.25]
\node at (-3,-1.5) {$-$};
\draw (0,0) -- (1,-1) -- (2,0);
\draw (1,-1) -- (1,-2);
\draw (0,-3) -- (1,-2) -- (2,-3);
\node at (1.5,-1.5) {\tiny$\Delta$};
\draw (2,-3.2) -- (2,-3.7); \draw[>-<]  (2,-3.4)   -- (2,-4.8); \draw (2,-4.5) -- (2,-5); \node at (2.8,-4) {\tiny $\ev$}; \node at (2,-5.5) {\tiny $\omega$};
\node at (0,.5) {\tiny$1$};
\node at (2,.5) {\tiny$2$};
\node at (2.5,-3) {\tiny$3$}; 
\end{tikzpicture}
\raisebox{1.1cm}{\tiny $=-(1\otimes \ev\tau)(\lambda\otimes 1)(1\otimes\omega)(\Delta\mu\otimes 1) =  (1\otimes \ev\tau)[(\lambda\Delta\mu)\otimes 1](1\otimes1\otimes\omega)$.} 
\end{center}
The sign change arises from inverting the order of $\omega$ and $\Delta$. 

The \underline{sixth term} gives rise to the sum
\begin{center}
\begin{tikzpicture}[scale=.25]
\node at (-3,-3) {$-$};
\draw (1,0) -- (1,-1);
\draw (-1,-3) -- (1,-1) -- (2,-2);
\node at (-1.5,-3) {\tiny$\Delta$};
\draw (2,-2.2) -- (2,-2.7); \draw[>-<]  (2,-2.4)   -- (2,-3.8); \draw (2,-3.5) -- (2,-4); \node at (2.8,-3) {\tiny $\ev$}; \node at (2,-4.5) {\tiny $\omega$};
\node at (.5,0) {\tiny$2$};
\node at (2.5,-2) {\tiny$3$}; 
\draw (-1,-3) -- (2,-6) -- (5,-3);
\draw (2,-6) -- (2,-9);
\node at (5,-2.5) {\tiny$1$};
\node at (8,-3) {$+$};

\pgftransformxshift{12cm}
\draw (1,3) -- (1,0) -- (5,-3) -- (1,0) -- (-1,-2);
\draw (-1,-2.2) -- (-1,-2.7); \draw[>-<]  (-1,-2.4)   -- (-1,-3.8); \draw (-1,-3.5) -- (-1,-4); \node at (0,-3) {\tiny $\ev$}; \node at (-0.2,-3.7) {\tiny $\omega$};
\draw (-1,-4.2) -- (-1,-5); \draw[<-] (-1,-5) -- (-1,-6); \node at (1.5,-5) {\tiny $-s\Delta^\vee\omega$};
\draw (-1,-6.2) -- (-1,-6.7); \draw[<->]  (-1,-6.4)   -- (-1,-7.8); \draw (-1,-7.5) -- (-1,-8); \node at (0.5,-7.2) {\tiny $\coev$}; \node at (-0.3,-6.4) {\tiny $s$};
\draw (-1,-8) -- (1,-10) -- (3,-8) -- (1,-10) -- (1,-12);
\draw (1,-12.2) -- (1,-12.7); \draw[>-<]  (1,-12.4)   -- (1,-13.8); \draw (1,-13.5) -- (1,-14); \node at (2,-13) {\tiny $\ev$}; \node at (1.8,-13.7) {\tiny $\omega$};
\node at (1,3.5) {\tiny $1$};
\node at (2.5,-7.7) {\tiny $2$};
\node at (1.5,-12) {\tiny $3$};
\node at (8,-3) {$=$};
\node at (11,-3) {$-$};

\pgftransformxshift{16cm}
\draw (1,0) -- (1,-2) -- (3,-4); 
\draw (1,-2) -- (-1,-4) -- (-3,-2) -- (-1,-4) -- (-1,-6);
\draw (3,-4.2) -- (3,-4.7); \draw[>-<]  (3,-4.4)   -- (3,-5.8); \draw (3,-5.5) -- (3,-6); \node at (3.8,-5) {\tiny $\ev$}; \node at (3,-6.5) {\tiny $\omega$};
\node at (0.5,-3.5) {\tiny$\Delta$};
\node at (-3,-1.5) {\tiny$1$};
\node at (.5,0) {\tiny$2$};
\node at (3.5,-4) {\tiny$3$}; 
\node at (8,-3) {$-$};

\pgftransformxshift{12cm}
\draw (1,0) -- (1,-2) -- (-1,-4) -- (-1,-6);
\draw (1,-2) -- (3,-4) -- (5,-2) -- (3,-4) -- (3,-6);
\draw (3,-6.2) -- (3,-6.7); \draw[>-<]  (3,-6.4)   -- (3,-7.8); \draw (3,-7.5) -- (3,-8); \node at (3.8,-7) {\tiny $\ev$}; \node at (3,-8.5) {\tiny $\omega$};
\node at (1,.5) {\tiny $1$};
\node at (2.5,-2.3) {\tiny $\Delta$};
\node at (4.5,-1.7) {\tiny $2$};
\node at (3.5,-6) {\tiny $3$};

\end{tikzpicture}
\end{center}
To transform the first term in the sum we use commutativity of $\mu$. Upon transforming the second term in the sum we pick up two minus signs, arising from inverting the order of $\Delta$ and $\lambda$, and from graded cocommutativity $\tau\lambda=-\lambda$. These cancel each other, so that the outcome is still a minus sign. The final result is 
$$
-(1\otimes\ev\tau)\Big[
(\mu \otimes 1)(1 \otimes \Delta \otimes 1)(1 \otimes \lambda)+ (1 \otimes \mu)(1 \otimes \Delta \otimes 1)(\lambda \otimes 1)
\Big]
(1\otimes 1\otimes\omega).
$$

The \underline{seventh term} gives rise to the sum 
\begin{center}
\begin{tikzpicture}[scale=.23]
\node at (-3,-3) {$+$};
\draw (1,0) -- (1,-2) -- (-1,-4) -- (-1,-6);
\draw (1,-2) -- (3,-4) -- (5,-2) -- (3,-4) -- (3,-6);
\draw (-1,-6.2) -- (-1,-6.7); \draw[>-<]  (-1,-6.4)   -- (-1,-7.8); \draw (-1,-7.5) -- (-1,-8); \node at (-.2,-7) {\tiny $\ev$}; \node at (-.2,-8) {\tiny $\omega$};
\node at (1,.7) {\tiny $1$};
\node at (5,-1.3) {\tiny $2$};
\node at (-1.7,-6) {\tiny $3$};
\node at (2.5,-2.3) {\tiny $\Delta$};
\node at (8,-3) {$+$};

\pgftransformxshift{14cm}
\draw (1,0) -- (1,-2) -- (3,-4); 
\draw (1,-2) -- (-1,-4);

\draw (-1,-4.2) -- (-1,-4.7); \draw[>-<]  (-1,-4.4)   -- (-1,-5.8); \draw (-1,-5.5) -- (-1,-6); \node at (-.2,-5) {\tiny $\ev$}; \node at (-.2,-6) {\tiny $\omega$};
\draw (-1,-6.2) -- (-1,-7); \draw[<-] (-1,-7) -- (-1,-8); \node at (1.5,-7) {\tiny $-s\Delta^\vee\omega$};
\draw (-1,-8.2) -- (-1,-8.7); \draw[<->]  (-1,-8.4)   -- (-1,-9.8); \draw (-1,-9.5) -- (-1,-10); \node at (0.5,-9.2) {\tiny $\coev$}; \node at (-0.3,-8.4) {\tiny $s$};

\draw (-1,-10) -- (-3,-12) -- (-5,-10) -- (-3,-12) -- (-3,-14);
\draw (-3,-14.2) -- (-3,-14.7); \draw[>-<]  (-3,-14.4)   -- (-3,-15.8); \draw (-3,-15.5) -- (-3,-16); \node at (-2.2,-15) {\tiny $\ev$}; \node at (-2.2,-16) {\tiny $\omega$};

\node at (-5,-9.3) {\tiny$1$};
\node at (.5,0) {\tiny$2$};
\node at (-3.7,-16) {\tiny $3$};
\node at (7,-3) {$=$};
\node at (10,-3) {$-$};

\pgftransformxshift{16cm}
\draw (1,0) -- (1,-2) -- (3,-4); 
\draw (1,-2) -- (-1,-4) -- (-3,-2) -- (-1,-4) -- (-1,-6);
\draw (3,-4.2) -- (3,-4.7); \draw[>-<]  (3,-4.4)   -- (3,-5.8); \draw (3,-5.5) -- (3,-6); \node at (3.8,-5) {\tiny $\ev$}; \node at (3,-6.5) {\tiny $\omega$};
\node at (0.5,-3.5) {\tiny$\Delta$};
\node at (-3,-1.3) {\tiny$2$};
\node at (.5,0) {\tiny$1$};
\node at (3.7,-4) {\tiny$3$}; 
\node at (8,-3) {$-$};

\pgftransformxshift{12cm}
\draw (1,0) -- (1,-2) -- (-1,-4) -- (-1,-6);
\draw (1,-2) -- (3,-4) -- (5,-2) -- (3,-4) -- (3,-6);
\draw (3,-6.2) -- (3,-6.7); \draw[>-<]  (3,-6.4)   -- (3,-7.8); \draw (3,-7.5) -- (3,-8); \node at (3.8,-7) {\tiny $\ev$}; \node at (3,-8.5) {\tiny $\omega$};
\node at (1,.7) {\tiny $2$};
\node at (2.5,-2.3) {\tiny $\Delta$};
\node at (4.5,-1.5) {\tiny $1$};
\node at (3.7,-6) {\tiny $3$};

\end{tikzpicture}
\end{center}

To transform the first term in the sum we use commutativity of $\mu$ and graded cocommutativity of $\lambda$. Upon transforming the second term in the sum we pick up three minus signs: one from $-s\Delta^\vee\omega$, another one from inverting the order of $\Delta$ and $\lambda$, and another one from graded cocommutativity $\tau\lambda=-\lambda$ (we also use commutativity $\mu\tau=\mu$, which does not introduce any new sign). The final result is 
$$
-(1\otimes\ev\tau)\Big[
(\mu \otimes 1)(1 \otimes \Delta \otimes 1)(1 \otimes \lambda)+ (1 \otimes \mu)(1 \otimes \Delta \otimes 1)(\lambda \otimes 1)
\Big] \tau 
(1\otimes 1\otimes\omega).
$$

To conclude, we sum up all these seven factors and we find precisely the 9-term relation for $\Delta$, $\lambda$ and $\mu$, precomposed with $1 \otimes 1 \otimes \omega$ and postcomposed with $1 \otimes \ev \tau$. 
\end{proof}


\bibliographystyle{abbrv}
\bibliography{000_bv}

\end{document}